\documentclass[11pt]{amsart}
\usepackage[top=1.5in, left=1.2in, right=1.2in, bottom=1.5in]{geometry}
\usepackage{amssymb}
\usepackage{comment}
\usepackage{tikz-cd}
\usepackage{bbm}
\usepackage{graphicx}
\usepackage{changepage}
\usepackage{mathtools}
\usepackage{csquotes}
\usepackage{hyperref}
\usepackage{enumerate}
\usepackage{float}
\usepackage{tabulary}

\usepackage{caption}
\usepackage{wrapfig}
\usepackage{bm}
\usepackage{array}
\usepackage{color}
\usepackage{pst-all}
\usepackage[caption=false]{subfig}

\usepackage{amsmath, amscd}
\usepackage{amsthm}
\usepackage[all]{xy}

\newtheorem*{theorem*}{Theorem}
\newtheorem*{question*}{Problem}
\newtheorem*{lemma*}{Lemma}
\newtheorem*{corollary*}{Corollary}
\newtheorem{theorem}{Theorem}[section]
\newtheorem{lemma}[theorem]{Lemma}

\newtheorem{corollary}[theorem]{Corollary}
\newtheorem{proposition}[theorem]{Proposition}

\theoremstyle{definition}
\newtheorem{definition}[theorem]{Definition}

\begin{document}

\title{Hausdorff dimension of Kuperberg minimal sets}

\author{Daniel Ingebretson}

\maketitle

\begin{abstract}
In 1994, Kuperberg (\cite{Kup1}) constructed a smooth flow on a three-manifold with no periodic orbits.
It was later shown that a generic Kuperberg flow preserves a codimension one laminar minimal set.
We develop new techniques to study the symbolic dynamics and dimension theory of this minimal set, by relating it to the limit set of a graph directed pseudo-Markov system over a countable alphabet.
\end{abstract}

\setcounter{tocdepth}{1}
\tableofcontents

\vfill
\eject

\section{Introduction}
\label{Intro}
In this work, we study the dynamics and fractal geometry of the minimal sets for generic Kuperberg flows on 3-manifolds.
The minimal sets resemble, in many ways, the strange attractors that arise in physics, and one of the outstanding open problems is to understand the dimension theory of Kuperberg minimal sets, and its dependance on the dynamics.

Krystyna Kuperberg showed in the work \cite{Kup1} that every closed 3-manifold admits a smooth flow with no periodic orbits.
Her proof was based on the construction of a smooth aperiodic flow in a \textit{plug}, which is a compact three-manifold with boundary.
This plug is inserted in flows to break open periodic orbits.
It is known that the flow in the plug preserves a unique minimal set $ \mathcal{M} $, and that under generic assumptions, $ \mathcal{M} $ is a codimension one lamination with a Cantor transversal, as was shown in \cite{Hur}.
The dynamics of Kuperberg flows have been previously studied in \cite{Kup2}, \cite{Ghy}, \cite{Hur}, \cite{Hur2}, \cite{Hur3}, and \cite{Mat}, and it is known that the topology of $ \mathcal{M} $ is particularly complicated.

There have been many notable contributions to the dimension theory of limit sets of dynamical systems in dimensions higher than two; see \cite{Bot}, \cite{Sim}, \cite{Sim2}, \cite{Sim3} for some examples.
A common theme in these works is hyperbolicity in the dynamics and a reduction to one dimension via stable manifolds.
However, as stated in the survey \cite{Sch},

\vspace{0.1cm}
\begin{displayquote}
``Even for the simplest examples of higher dimension [than 2] we are far from a general theory of the Hausdorff dimension of limit sets.''
\end{displayquote}
\vspace{0.1cm}

The Kuperberg flow does not resemble these systems, because by a theorem of Katok \cite{Kat}, an aperiodic flow cannot preserve a hyperbolic measure.
Though the flow is not hyperbolic and has zero entropy, arbitrarily small perturbations of it are hyperbolic and have positive entropy (see \cite{Hur2}). For this reason, the dynamics of the Kuperberg flow are said to lie ``at the boundary of hyperbolicity."

In two dimensions, this type of behavior is present in H\'{e}non-like families and Kupka-Smale diffeomorphisms (see \cite{Bed3},\cite{Cao},\cite{Lep}).
Studying the fractal geometry and dimension theory of the Kuperberg minimal set makes a new contribution to a general dimension theory for limit sets in dimension three, in the absence of hyperbolicity.

Fortunately, the characterization of the minimal set as a codimension one lamination reduces the dimension theory to that of the transverse Cantor set.
Without this, the study of its fractal geometry and dimension theory would be completely intractable.

The dimension theory of Cantor sets in the the line has a vast literature, particularly for limit sets of iterated function systems, graph directed systems, and their generalizations.
However, the transverse Kuperberg minimal set poses new challenges in this direction as well.
These come from the complicated symbolic dynamics of the action of the holonomy pseudogroup associated to the flow, which is not semiconjugate to a subshift.

In this paper we propose a general framework for treating the symbolic dynamics of limit sets of pseudogroups, and apply this to a transverse section of the Kuperberg minimal set. 
We build a symbolic model of this transverse Cantor set and extract a graph directed subspace by analyzing the pseudogroup.
This allows the application of results from one-dimensional thermodynamic formalism to obtain dimension estimates, which are then extended to the minimal set via the product structure.

\subsection{The Kuperberg flow}
Kuperberg's construction is the first-- and only currently known-- smooth flow on $ S^3 $ with no periodic orbits.
This was discovered as a counterexample to Seifert's conjecture.

\subsubsection{Seifert's conjecture}
A vector field on a manifold is said to have a closed orbit if one of its integral curves is homeomorphic to $ S^1 $.
The Hopf vector field on $ S^3 $, whose integral curves form the Hopf fibration, has all orbits closed. 
In 1950, Seifert \cite{Sei} showed that every nonsingular vector field on $ S^3 $ sufficiently close to the Hopf vector field also has a closed orbit, and then asked if every continuous vector field on $ S^3 $ does.
The generalized Seifert conjecture asked this question for any compact orientable $ n $-manifold with Euler characteristic zero.

Counterexamples in dimension four and greater were discovered in 1966 by Wilson \cite{Wil}, who constructed the first \textit{plug}, the product of a closed rectangle with a torus, which carries a smooth vector field satisfying certain properties.
A plug is a manifold with boundary, together with a smooth vector field. 
If this local vector field satisfies some symmetry conditions, the plug can be inserted into a manifold carrying a global vector field, in such a way that the local dynamics in the plug are compatible with the global dynamics.
If the plug intersects a periodic orbit, the plug's interior dynamics can break it.

Using this method, Wilson constructed smooth counterexamples to Seifert's conjecture in dimension greater than or equal to four.
Seifert's conjecture is trivial in dimension two, so Seifert's conjecture only remained unsolved in dimension three, although Wilson did succeed in showing that on every closed connected three-manifold there exists a smooth vector field with only \textit{finitely many} closed orbits.

The first counterexample to Seifert's conjecture in dimension three was constructed in 1972 by Schweitzer \cite{Schw}. 
This counterexample used a plug supporting an aperiodic vector field of class $ C^1 $.
In 1988, Harrison \cite{Har} modified Schweitzer's construction to class $ C^2 $, but serious obstructions remained in extending to $ C^{\infty} $.
For an account of Schweitzer's and Harrison's constructions, see \cite{Ghy}.

\subsubsection{Kuperberg's plug}
In 1994, Kuperberg \cite{Kup1} constructed a $ C^{\infty} $ counterexample to Seifert's conjecture in dimension three.
This construction began with a modified Wilson plug embedded in $ \mathbb{R}^3 $ containing two periodic orbits.
Kuperberg then used self-intersections to break the periodic orbits inside Wilson's plug without creating new periodic orbits.
See \cite{Kup1}, \cite{Ghy}, \cite{Hur}, and \cite{Mat} for descriptions of Kuperberg's construction.

\subsubsection{Kuperberg's minimal set}
Ghys \cite{Ghy} showed that Kuperberg's plug contains a unique minimal set.
Using a numerical simulation due to B. Sevannec, he obtained an image of this minimal set on a transverse section of the plug.
See Figure \ref{matfig}.

\begin{figure}[h]
\includegraphics[width=0.6\linewidth, trim={2.95cm 1.4cm 0 0}, clip]{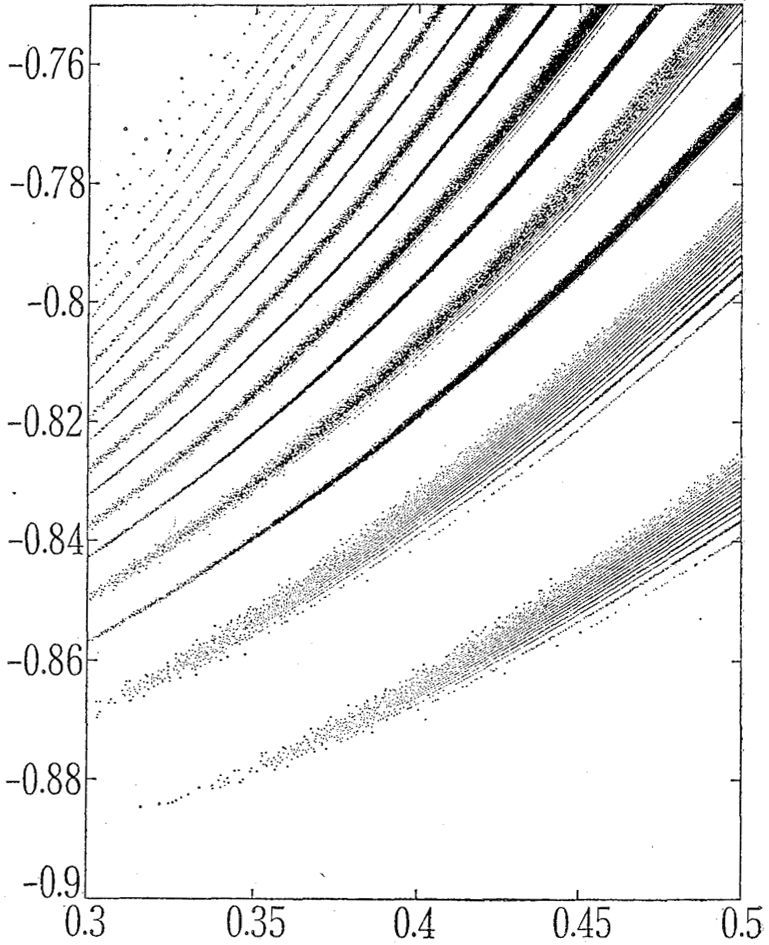}
\caption{Cross-section of Kuperberg minimal set, from Ghys (\cite{Ghy})}
\label{matfig}
\end{figure}

Ghys encouraged an investigation into the properties of this minimal set, and asked how such properties depend on Kuperberg's construction.
A closer study of the topology and dynamics of the minimal set was carried out by Hurder and Rechtman \cite{Hur}.
To answer Ghys' question, they defined a special class of flows called \textit{generic Kuperberg flows} that preserves a unique minimal set with the following characterization.

\begin{theorem}(\cite{Hur}, Theorem 17.1) 
\label{minchar3}
Let $ K $ be the Kuperberg plug, $ \psi_t : K \rightarrow K $ a generic Kuperberg flow, and $ \mathcal{M} \subset K $ the minimal set.
Then $ \mathcal{M} $ is a codimension one lamination with a Cantor transversal $ \tau $.
Furthermore, there exists a closed surface $ \mathcal{R}' \subset K $ such that
$$
\mathcal{M} = \overline{\bigcup_{-\infty < t < \infty} \psi_t(\mathcal{R}')}.
$$
\end{theorem}

The surface $ \mathcal{R}' $ is called the \textit{notched Reeb cylinder}. Because of Theorem \ref{minchar3}, the fractal geometry of $ \mathcal{M} $ can be studied by analyzing the orbit of the $ \mathcal{R}' $.
Perhaps the first question in this direction is the Hausdorff dimension of the minimal set.
Because of the local product structure implied by this theorem, the dimension theory of $ \mathcal{M} $ reduces to that of $ \tau $.
The study of dynamically defined Cantor sets and their dimension theory has a long history.

\subsection{Iterated function systems and limit sets of group actions}
A large class of fractals are the limit sets of \textit{iterated function systems}, which were introduced by Hutchinson \cite{Hut}.

\subsubsection{Iterated function systems}
Let $ X $ be a compact space, and $ E = \{1,\ldots,p\} $ a finite alphabet. 
An \textit{iterated function system} is a collection $ \{ \phi_i : X \rightarrow X \}_{i \in E} $ of injective contracting maps, with a common Lipschitz constant $ 0 < s < 1 $.

Each iterated function system has an invariant \textit{limit set}:
$$
J = \bigcap_{n = 1}^{\infty} \bigcup_{(i_1,\ldots,i_n) \in E^n} \phi_{i_1} \circ \cdots \circ \phi_{i_n} (X).
$$
With appropriate separation conditions, $ J $ is a Cantor set.
There is a $ p $-to-$ 1 $ expanding map $ S $ whose inverse branches are $ \phi_i $, and the dynamics of $ S |_J $ is conjugate to the one-sided shift on $ E^{\mathbb{N}} $.
For an introduction to iterated function systems, see chapter 9 of \cite{Fal2}.
There are many generalizations of iterated function systems, including graph-directed Markov systems.

\subsubsection{Graph directed Markov systems}
Let $ (V,E) $ be directed graph with finite vertex and edge sets $ V $ and $ E $, respectively.
Each edge $ e $ has an initial vertex $ i(e) \in V $ and terminal vertex $ t(e) \in V $.
Let $ A: E \times E \rightarrow \{0,1\} $ be the edge incidence matrix of this directed graph, so if $ A_{ee'} = 1 $, then $ t(e) = i(e') $.
For each $ v \in V $, let $ X_v $ be a metric space, and for each $ e \in E $ let $ \phi_e : X_{t(e)} \rightarrow X_{i(e)} $ be an injective contraction map.
If the maps $ \{ \phi_e \}_{e \in E} $ have a common Lipschitz constant $ 0 < s < 1 $, the collection is called a \textit{graph directed Markov system}.

For each $ n \geq 1 $, the matrix $ A $ determines the following space of \textit{admissible words} of length $ n $:
$$
E_A^n = \{ \omega \in E^{\mathbb{N}} : A_{\omega_i, \omega_{i+1}} = 1 \text{ for all } 1 \leq i \leq n-1 \}.
$$
In terms of these, the system has an invariant limit set:
$$
J = \bigcap_{n = 1}^{\infty} \bigcup_{(i_1,\ldots,i_n) \in E_A^n} \phi_{i_1} \circ \cdots \circ \phi_{i_n} \left(X_{t(\omega_n)} \right).
$$
As with iterated function systems, these limit sets are often Cantor sets, and their dynamics are conjugate to a subshift of finite type over the alphabet $ E $.

In some cases, the limit set of a discrete group $ \Gamma = \langle g_1, \ldots, g_n \rangle $ acting on a compact space $ X $ can be realized as the limit set of an graph-directed system defined by the generators $ g_i $ and their images $ g_i(X) $.
Here are some examples.

\begin{itemize}
\item \textit{Expanding maps}: A distance expanding map $ f: X \rightarrow X $ of a metric space $ X $ defines a semigroup action of $ \mathbb{N} $ on $ X $.
Such a map has a Markov partition of arbitrarily small diameter (see \cite{Rue}).
Defining the iterated function system to be the inverse branches of $ f $, the limit set of this action is the limit set of the graph directed system whose incidence matrix is the matrix defining the Markov partition.
\vspace{0.1cm}

\item \textit{Fuchsian groups}:  Let $ \Gamma $ be a Fuchsian group acting on the hyperbolic disc $ \mathbb{H}^2 $. Bowen \cite{Bow20} related the action of $ \Gamma $ on its boundary circle $ \partial \mathbb{H}^2 = S^1 $ to an expanding Markov map $ f : S^1 \rightarrow S^1 $.
This correspondence is called the \textit{Bowen-Series coding}; via this correspondence, these actions are orbit equivalent.
As above, the inverse branches of $ f $ form a graph directed system with admissible words coded by the matrix defining the Markov map $ f $.
\vspace{0.1cm}

\item \textit{Schottky groups}: Another example is the limit set of a finitely generated Kleinian group of Schottky type, acting on the Riemann sphere.
It can be shown that such a limit set is the limit set of an appropriately defined graph directed system. 
For details, see Chapter 5 of \cite{Mau3}.
\vspace{0.1cm}
\end{itemize}

\subsubsection{Infinitely generated function systems and pseudo-Markov systems}
There are many generalizations of iterated function systems and graph directed systems.
These include the infinite iterated function systems of Mauldin and Urba\'{n}ski \cite{Mau1} and the pseudo-Markov systems of Stratmann and Urba\'{n}ski (\cite{Str}).
The former can be used to describe sets of complex continued fractions (see \cite{Mau2}), and the latter are models of limit sets of infinitely generated Schottky groups (see \cite{Str}), among many other applications.

The dynamics of a graph directed function system on its limit set is semiconjugate to a shift over a sequence space of admissible words.
This is the domain of symbolic dynamics, and the ergodic properties of such systems is well studied.
One of the advantages of relating the limit set of a group to the limit set of a function system, is that the symbolic dynamics of the function system can then be used to study the symbolic dynamics of the group action.

Once such a connection has been made, the fractal geometry of the limit set of the group can be studied using techniques from iterated function systems.
The patterns that emerge when ``zooming in" to the fractal by applying maps in the function system, are the same as those that emerge by applying the generators of the group to a fundamental domain.
These regular patterns are captured by the incidence matrix determining the admissible words in the coding of the limit set.

\subsection{General function systems and limit sets of pseudogroup actions}
Pseudogroups are a generalization of groups of transformations of metric spaces (see \cite{Hae}).
A primary application of pseudogroups is in the dynamics of foliations and laminations.
Compositions of transition maps of a foliation or lamination comprise its \textit{holonomy} pseudogroup.
For a flow that does not admit a global section, the collection of first-return maps to a section also forms a pseudogroup.
For an exposition of the dynamics of pseudogroups see \cite{Hur0} and \cite{Wal}.

Limit sets of pseudogroup actions have a similar definition to those of group actions, but are generally more difficult to study.
They can be fractals, but they need not exhibit the same self-similarity evident in limit sets of groups.

In Chapter \ref{genfun}, we define the notion of a \textit{general function system}. 
The limit set of such a system is a fractal that need not be self-similar.
This provides a framework to relate the limit sets of pseudogroups to those of function systems.
The transverse Cantor set of the Kuperberg minimal set is the limit set of a pseudogroup action on the transversal.
The pseudogroup here is the holonomy of the foliation by flowlines of the Kuperberg flow.
In Chapter \ref{Transcant}, we will relate this set to the limit set of a general function system.

\subsection{Symbolic dynamics and thermodynamic formalism}
Let $ E $ be an alphabet (finite or infinite).
The dynamics of the shift map on invariant subspaces of the sequence space $ E^{\mathbb{N}} $ is well studied.
The shift map has an associated \textit{topological pressure} that is related to ergodic properties of measures supported on the space.
This is part of the thermodynamic formalism developed  by Sinai, Ruelle, and Bowen (see \cite{Sin}, \cite{Rue}, and \cite{Bow1}).
For generalized systems such as infinite iterated function systems and pseudo-Markov systems, there are extensions of the thermodynamic formalism (see \cite{Mau3}).
In Chapter \ref{Thermo}, we will define the topological pressure in an appropriate context.

\subsubsection{Symbolic dynamics of limit sets of graph directed systems}
For graph directed systems, there is a bijective coding map $ \pi : \Sigma \rightarrow J $, where $ \Sigma \subset E^{\mathbb{N}} $ is a compact shift-invariant subset, and $ J $ is the limit set of the system.
This map intertwines the system's dynamics on $ J $ with the shift on $ \Sigma $.
Following Barriera \cite{Bar1} we say that the function system is \textit{modeled} by the subshift $ \Sigma $.

In this way, symbolic quantities such as pressure have natural analogues defined entirely in terms of the function system.
If the function system is assumed to have regularity $ C^{1+\alpha} $ for some $ \alpha > 0 $, the pressure has additional uniformity properties that makes its definition particularly transparent.
In Chapter \ref{Conf} we will present the pressure in this context, and study these properties.

\subsubsection{Symbolic dynamics of limit sets of general function systems}
General function systems are coded by more general sequence spaces, including spaces that are not shift-invariant.
These are also introduced in Chapters \ref{Thermo} and \ref{Conf}.
In later chapters we will equate the transverse Kuperberg minimal set to the limit set $ J $ of a general function system, and show that there is a bijective correspondence $ \pi : \Sigma \rightarrow J $, where $ \Sigma \subset \mathbb{N}^{\mathbb{N}} $ is a sequence space that is not shift-invariant.
As with subshifts, we say that such a general function system is \textit{modeled} by this general symbolic space $ \Sigma $.

The definition of limit sets of general function systems resembles that of graph directed systems.
However, their fractal geometry is \textit{a priori} more complicated than their graph directed counterparts, and exhibits less self-similarity.
Applying the maps in the function system, we ``zoom in" on the fractal, but the regular patterns present in graph directed systems do not emerge, because the underlying dynamics are those of a pseudogroup rather than those of a group.

The limit sets of actions of pseudogroups is not as widely studied as those of groups and can exhibit substantially more pathology.
The ergodic theory and symbolic dynamics of these systems is still being developed (see \cite{Wal}).
Progress in this direction includes the entropy theory of Ghys, Langevin, and Walczak \cite{Ghy2}.
However, it is not at all clear how to develop a thermodynamic formalism or to define quantities such as pressure for limit sets of pseudogroups and of function systems that are coded by these general symbolic spaces.

\subsubsection{Dual symbolic spaces}
In his study of differentiable structures on Cantor sets, Sullivan \cite{Sul} defined the notion of a \textit{dual} Cantor set.
The symbolic description of the dual is given by simply reversing the coding and reading the words in the opposite order.

The distortion of a fractal in a metric space can be quantified by its \textit{ratio geometry}.
The ratio geometry is a sequence of real numbers that measure the self-similarity defect of the fractal; if the sequence is constant, the fractal is self-similar and its similarity coefficient is equal to this constant.
The asymptotic ratio geometry is called the \textit{scaling function} and is viewed as a function on the symbolic space coding the fractal.
Sullivan proved that for Cantor sets defined by $ C^{1+\alpha} $ function systems, the scaling function on the dual is an invariant of the differential structure.
In Chapter \ref{Kupmin} we will see that dual Cantor sets arise naturally in our study of the symbolic dynamics of the Kuperberg minimal set.
We present the dual of a symbolic space in Chapter \ref{dualcant} in the context of general symbolic spaces appropriate for coding the limit sets of general function systems.
For references on Sullivan's theorem and dual Cantor sets, see \cite{Bed2}, \cite{Prz1}, and \cite{Prz2}.

\subsection{Symbolic dynamics of the Kuperberg minimal set}
We now return to the Kuperberg flow, its minimal set, and the fractal geometry of the minimal set.
In Chapter \ref{Wilsflow} we briefly present the general theory of plugs, and summarize Wilson's construction \cite{Wil} of a vector field on a mirror-image plug with two periodic orbits.
In Chapter \ref{Kupsection}, we summarize Kuperberg's construction of a plug $ K $ \cite{Kup1}, using self-insertions to modify Wilson's plug.
The flow of the resulting vector field on $ K $ is called the Kuperberg flow $ \psi_t $.
The images of these periodic orbits under the quotient map are called the \textit{special orbits}.

To simplify the problem, it is necessary to make additional assumptions on the construction $ K $ and $ \psi_t $.
These assumptions are listed in Chapter \ref{InsertAssume}, and are compatible with the generic hypotheses on Kuperberg flows given in \cite{Hur}.
Under these assumptions, we can write the insertion maps in coordinates and explicitly integrate the Kuperberg vector field.

The dynamics of $ \psi_t $ are complicated, but there are several important notions that allow us to relate these to the simpler dynamics of the Wilson flow.
These notions are called \textit{transition} and \textit{level}; they were defined by Kuperberg in \cite{Kup1} and used extensively in \cite{Hur}, \cite{Ghy}, and \cite{Mat}.
We can decompose orbits of points in $ K $ by level, and relate each level set to an orbit in Wilson's plug.
We make this precise in Chapter \ref{levo}.

\subsubsection{The Kuperberg pseudogroup}
In Chapter \ref{Kuppseudo} we commence the study of the holonomy pseudogroup associated to $ \psi_t $.
This flow does not admit a global section, so we choose a convenient local section defined in Chapter \ref{InsertAssume}.
It is the union of two rectangles transverse to the flow, that lie in the entrances to the two insertion regions.
In Chapters \ref{Kupmin} through \ref{FunctionSystems}, we restrict to just one which we refer to as $ S $.

The map taking a point $ x \in S $ to its first return under $ \psi_t $ generates a pseudogroup $ \Psi $.
Using the theory of levels from Chapter \ref{Kupsection}, we first show that this pseudogroup is generated by the first-return maps of the Wilson flow, together with the insertion maps.

The intersection of the notched Reeb cylinder $ \mathcal{R}' $ with $ S $ is a curve $ \gamma $.
In view of Theorem \ref{minchar3}, the intersection $ \mathcal{M} \cap S $ is the closure of the orbit of the curve $ \gamma $ under this pseudogroup.
Because our assumptions in Chapter \ref{InsertAssume} allowed us to integrate the Kuperberg flow and write the insertion maps in coordinates, we then set out to explicitly parametrize the transition curves in the intersection $ \mathcal{M} \cap S $. 
We carry this out in Chapter \ref{Kupmin}.
See Figure \ref{TheMin} for a picture of some of these curves.

\begin{figure}[h]
\includegraphics[width=0.9\linewidth, trim={0 0.5 0 0.8cm}, clip]{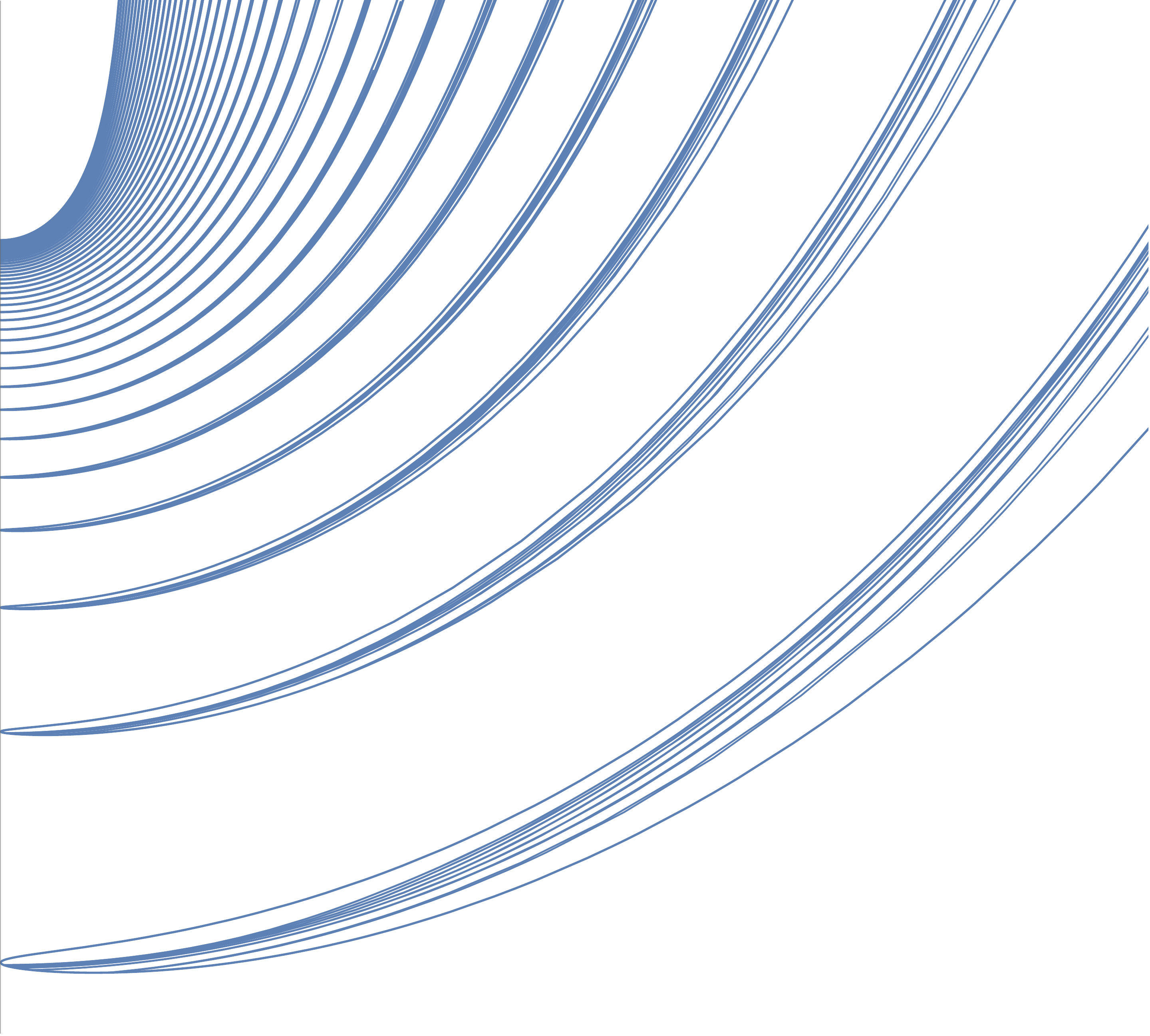}
\caption{The first two iterations in the recursive construction of the transverse minimal set. Compare with Figure \ref{matfig}.}
\label{TheMin}
\end{figure}

\subsubsection{Interlaced Cantor sets}
Through Chapter \ref{FunctionSystems}, we only consider the first return to $ S $, one of the two rectangular regions defined by the insertions.
To account for the entire minimal set, we must also consider points that enter the other insertion region, before intersecting $ S $.
These points also form a Cantor set in $ S $, and because of the symmetry of the plug, these Cantor sets are identical.
In Chapter \ref{Transcant} we will prove that these two Cantor sets are \textit{interlaced}, and that $ \mathcal{M} \cap S $ is equal to this interlaced Cantor set.
The symbolic dynamics of two interlaced Cantor sets modeled by sequence spaces $ \Sigma $ and $ \Xi $, is defined naturally by the induced dynamics on a \textit{joint} sequence space $ \Sigma \ast \Xi $.
These terms will be defined precisely in Chapter \ref{sectioninter}.

\subsubsection{Symbolic dynamics of the Kuperberg minimal set}
Using the theory of levels, we prove that each curve in $ \mathcal{M} \cap S $ is coded by a word $ \omega $ in an appropriate general sequence space, whose word length corresponds to the level of the curve.
These can be used to code the points in $ \tau \subset S $, the Cantor transversal of $ \mathcal{M} $.

The space $ \Sigma $ of admissible words is not shift-invariant, and depends delicately on the symbolic dynamics of the Kuperberg pseudogroup.
The number of words in each level depends on the \textit{escape times} of curves in $ \mathcal{M} \cap S $ under the pseudogroup.
In general, it is impossible to predict the exact escape times of all curves in $ \mathcal{M} \cap S $.
However, in Chapter \ref{Transcant} we give an iterative construction of the sequence space $ \Sigma $ in terms of these escape times.
In Chapter \ref{FunctionSystems}, we use the Kuperberg pseudogroup and projection maps along the leaves of the lamination $ \mathcal{M} $ to define a general function system on the transversal.
Using the symbolic dynamics developed in Chapters \ref{Kuppseudo} and \ref{Kupmin}, we show that this general function system is modeled by the dual of the sequence space $ \Sigma $, in the sense of Sullivan.
This allows us to prove the following theorem.

\begin{theorem*}[\textbf{A}]
Let $ \mathcal{M} $ be the Kuperberg minimal set with Cantor transversal $ \tau $.
There is a sequence space $ \Sigma \subset \mathbb{N}^{\mathbb{N}} $ and a $ C^{1+\alpha} $ general function system on $ [0,1] $ modeled by the dual $ \widetilde{\Sigma} $, with limit set $ \tau $.
\end{theorem*}

As we show in Chapter \ref{Conf}, limit sets of general function systems modeled by a sequence space have a bijective coding to the space.
Then as an immediate corollary to Theorem \textbf{A}, we obtain

\begin{corollary*}[\textbf{B}]
Let $ \mathcal{M} $ be the Kuperberg minimal set with Cantor transversal $ \tau $.
Then there exists a sequence space $ \Sigma \subset \mathbb{N}^{\mathbb{N}} $ and bijective coding map
$$
\pi : \Sigma \rightarrow \tau.
$$
\end{corollary*}

This coding of $ \tau $ by $ \Sigma $ will be crucial later, when estimating the dimension.

\subsection{Dimension theory of limit sets}
In his study of the limit sets Fuchsian groups, Bowen \cite{Bow2} related the thermodynamic formalism to dimension theory.
In this setting, the pressure defined by the symbolic dynamics depends only on a parameter $ t \in \mathbb{R} $, and can thus be viewed as a function $ p : \mathbb{R} \rightarrow \mathbb{R} $.
Bowen proved that this function has a unique zero that coincides with the Hausdorff dimension of the limit set.
This relation is known as \textit{Bowen's equation} for dimension.
This equation-- and its subsequent generalizations in other settings-- is now ubiquitous in the dimension theory of dynamical systems.

There is an immediate analogue of Bowen's equation for limit sets of graph directed Markov systems.
Similarly, there is an analogue for each generalization, including graph directed and pseudo-Markov systems.
In Chapter \ref{Dime} we will present the pressure function and Bowen's equation in the appropriate generality.
For a proof of Bowen's equation for limit sets of finite iterated function systems, see \cite{Bed1}.
For generalizations of Bowen's equation, see \cite{Mau1}, \cite{Mau3}, and \cite{Str}, in increasing order of generality.
For general expositions of applications of thermodynamic formalism to dimension theory see \cite{Pes2}, \cite{Fal1}, \cite{Prz1} and \cite{Sch}.

In the survey \cite{Sch}, Schmeling and Weiss point out how pervasive Bowen's ideas are in the dimension theory of dynamical systems.

\vspace{0.1cm}
\begin{displayquote}
``One of the most useful techniques in the subject is to obtain a \textit{Bowen formula} for the Hausdorff dimension of a set, i.e. to obtain the Hausdorff dimension as the zero of an expression involving the thermodynamic pressure. Most dimension formulas for limit sets of dynamical systems and geometric constructions in the literature are obtained, or can be viewed, as Bowen formulas.''
\end{displayquote}
\vspace{0.1cm}

For this reason, to study the dimension theory of a set as complicated as the transverse minimal set $ \tau $ in the Kuperberg plug, it seems necessary to have the full power of the thermodynamic formalism at our disposal.
However, we have already noted that for limit sets of pseudogroups and general function systems modeled by sequence spaces that are not shift-invariant, such a formalism does not exist.
Thus, it is necessary to relate $ \tau $ to a more tractable function system, for instance the pseudo-Markov systems of Stratmann and Urba\'{n}ski (\cite{Str}).

\subsection{A graph directed subspace of $ \Sigma $}
We carry out this analysis in Chapters \ref{Kupmin} and \ref{GDPMsub}.
For each $ \epsilon > 0 $, let $ S_{\epsilon} \subset S $ be a sub-rectangle of width $ \epsilon $.
By analyzing the parametrizations of these curves and their images under the generators of $ \Psi $, we obtain bounds (with error) on the escape times of curves in $ \mathcal{M} \cap S_{\epsilon} $. 
The error in these bounds decreases as $ \epsilon \rightarrow 0 $.
Because $ \Sigma $ is defined in terms of escape times, we thus extract a subspace $ \Sigma_{\epsilon} \subset \Sigma $ that we can determine explicitly for small $ \epsilon $.
We then show that the bijective coding $ \pi : \Sigma \rightarrow \tau $ restricts to a bijective coding $ \pi : \Sigma_{\epsilon} \rightarrow \tau_{\epsilon} $, where $ \tau_{\epsilon} $ is the intersection of $ \tau $ with an $ \epsilon $-neighborhood of the critical orbit in $ K $.

Fortunately, for small enough $ \epsilon > 0 $, the fractal $ \tau_{\epsilon} $ exhibits much more self-similarity than is evident in $ \tau $.
The next theorem exploits this self-similarity.

\begin{theorem*}[\textbf{C}]
Let $ \mathcal{M} $ be the Kuperberg minimal set, with Cantor transversal $ \tau $.
Let $ \tau_{\epsilon} $ be the intersection of $ \tau $ with an $ \epsilon $-neighborhood of the critical orbit in $ K $. 
For sufficiently small $ \epsilon > 0 $ there is a $ C^{1+\alpha} $ graph directed pseudo-Markov system on $ [0,\epsilon] $ with limit set $ \tau_{\epsilon} $.
\end{theorem*}

\subsection{Dimension theory of the Kuperberg minimal set}
Theorem \textbf{C} shows that the general function system modeled by $ \Sigma $ from Theorem \textbf{B} has a graph-directed subsystem modeled by $ \Sigma_{\epsilon} \subset \Sigma $.
Thus for small enough $ \epsilon $, we can invoke the dimension theory developed in Chapter \ref{Dime} for graph directed systems to obtain results about the dimension theory of $ \tau_{\epsilon} $.

\subsubsection{Properties of the dimension}
To relate this to the dimension theory of $ \tau $, we first state the following global-to-local result.

\begin{lemma*}[\textbf{D}]
Let $ \tau $ be the transverse Cantor set of the Kuperberg minimal set, and let $ \tau_{\epsilon} $ be the intersection of $ \tau $ with an $ \epsilon $-neighborhood of the critical orbit in $ K $.
Then for any $ \epsilon > 0 $,
$$
\text{dim}_H(\tau) = \text{dim}_H(\tau_{\epsilon}).
$$
\end{lemma*}

We prove this lemma in Chapter \ref{Dimtrans}.
Applying the thermodynamic formalism for graph directed systems from Chapter \ref{Dime}, we obtain the following theorem.

\begin{theorem*}[\textbf{E}]
Let $ \tau $ be the transverse Cantor set of the Kuperberg minimal set. Then the Lebesgue measure of $ \tau $ is zero, and $ 0 < \dim_H(\tau) < 1 $.
\end{theorem*}

\subsubsection{Numerical estimates for dimension}
Finally we turn to numerical dimension results.
The Kuperberg flow is defined in terms of several external parameters, the most important being its angular speed $ a > 0 $.
To numerically estimate dimension using Bowen's equation, it is necessary to calculate the pressure function and its zero explicitly.
Besides calculating the dimension, we are interested in its dependence on the parameter $ a > 0 $.
As we show in Chapter \ref{Dime}, the pressure function depends on the symbolic dynamics and the derivatives of the maps comprising the function system.
Both of these quantities depend on external parameters, including $ a $.

The symbolic dynamics are determined by the space $ \Sigma_{\epsilon} $, which we have calculated by virtue of Theorem \textbf{C}.
However, the function system on $ [0,1] $ from Theorem \textbf{B} is defined in terms of the Kuperberg pseudogroup and projection maps along leaves.
Explicit calculation of the derivatives of these maps seems impossible.

Fortunately, in regularity $ C^{1+\alpha} $, the derivatives of the maps can be related to ratio geometry of the limit set.
This is the \textit{bounded distortion property} from one-dimensional dynamics, used by Shub and Sullivan (\cite{Shu}), and is presented in Chapter \ref{Dime}.
This reduces the pressure calculation to the estimation of the ratio geometry of the transverse Cantor set $ \tau $.

A detailed study of this ratio geometry is carried out in Chapter \ref{Transversal}.
In this chapter, we use the parametrizations of the curves calculated in Chapter \ref{Kupmin} and study their intersections with the transversal.
As with the symbolic dynamics, by restricting to a suitably small $ \epsilon $-neighborhood of the critical orbit, we obtain explicit bounds on the ratio geometry.
The simplest type of ratio geometry is that of \textit{stationary} systems, such as iterated function systems whose maps are similarities.
Such systems have a clean numerical dimension theory that depends on the \textit{ratio coefficients} of the system (see \cite{Pes2}).

In this direction, we define in Chapter \ref{Dime} an \textit{asymptotically stationary function system with error $ a_{\delta} $} for some $ \delta $.
This error is a function $ a_{\delta} : \Sigma \rightarrow \mathbb{R}_{\geq 0} $ that decreases to zero as $ \delta $ does.
The ratio geometry of the limit set of such a function system differs from that of a stationary system by this error.
As long as the error satisfies a natural summability condition, the pressure function for an asymptotically stationary system approaches that of a stationary system and allows for numerical estimates.

In Chapter \ref{Transversal}, we show that for any $ \delta > 0 $, there exists $ \epsilon > 0 $ such that pseudo-Markov system whose limit set is $ \tau_{\epsilon} $ is asymptotically stationary with summable error $ a_{\delta} $.
This can be used to obtain the following dimension estimates.

\begin{theorem*}[\textbf{F}]
Let $ \tau $ be the Cantor transversal of the Kuperberg minimal set. 
Let $ t = \text{dim}_H(\tau) $ be its Hausdorff dimension, and $ a > 0 $ the angular speed of the Kuperberg flow.
\begin{itemize}
\item $ t = \text{dim}_H(\tau) $ is the unique zero of a dynamically defined pressure function, 
\item $ t $ depends continuously on $ a $, 
\item For any $ a $ we may compute $ t $ to a desired level of accuracy.
\end{itemize}
\end{theorem*}

We conclude Chapter \ref{Dimtrans} by extending the results of Theorems \textbf{E} and \textbf{F} to the entire minimal set $ \mathcal{M} $, using the product structure from Theorem \ref{minchar}.
In Chapter \ref{Furth}, we survey some remaining open questions related to the dimension theory of the Kuperberg minimal set.

\subsection{Acknowledgements}
The author owes a debt of gratitude to Steve Hurder for his guidance and support for the duration of this project.

\vfill
\eject

\section{Symbolic spaces over an infinite alphabet}
\label{Thermo}
In this chapter we will fix some important notation that will be used throughout the paper. 
The notation of graph-directed symbolic spaces is standard and we follow some commonly observed conventions.
The main reference here is \cite{Mau1} (see also \cite{Bow1}, \cite{Mau1} \cite{Rue}).
We then introduce general symbolic spaces and symbolic spaces of infinite type, which are natural generalizations of graph-directed symbolic spaces.
We conclude by presenting dual symbolic spaces.

\subsection{Countable alphabets}
Let $ E \subset \mathbb{N} $ be a countable alphabet, and let $ E^{\ast} = \bigcup_{n \geq 1} E^n $ and $ E^{\infty} = E^{\mathbb{N}} $ be the finite and infinite words in $ E $, respectively.
If $ \omega \in E^{\ast} $ then $ \omega \in E^n $ for some $ n $ and we say $ |\omega| = n $ is the word length of $ \omega $.
If $ \omega \in E^{\infty} $, we set $ |\omega|=\infty $.
If $ \omega \in E^{\ast} \cup E^{\infty} $ and $ n \leq |\omega| $, we denote by $ \omega |_n $ the truncated word $ (\omega_1, \ldots, \omega_n) $.
If $ \omega \in E^{\ast} $ is a finite word, we denote
$$
[\omega] = \{ \tau \in E^{\infty} : \tau |_{|\omega|} = \omega \}.
$$
We have a countable-to-one left shift map $ \sigma: E^{\infty} \rightarrow E^{\infty} $.
With the convention $ \frac{1}{2^{\infty}} = 0 $, the space $ E^{\ast} \cup E^{\infty} $ is metrizable in the usual metric
$$
d(\omega, \tau) = \frac{1}{2^{|c(\omega, \tau)|}},
$$
where $ c(\omega, \tau) $ is the longest common initial subword of $ \omega $ and $ \tau $.

\subsection{General and infinite type symbolic spaces}

\subsubsection{General symbolic spaces}
Let $ \Sigma \subset E^{\ast} $ be a collection of finite words.
Because the alphabet $ E $ is countable, in general $ \Sigma $ is infinite.
For each such $ \Sigma $ and $ n \geq 1 $, let
$$
\Sigma_n = \{ \omega \in \Sigma : |\omega| = n \}.
$$
The symbolic spaces that arise naturally in our applications will satisfy the following property.

\begin{definition}[Extension admissibility property]
\label{extadm}
We say that $ \Sigma \subset E^{\ast} $ satisfies the \textit{extension admissibility property} if $ \Sigma_n \neq \emptyset $ for all $ n \geq 1 $, and for all $ (\omega_1, \ldots, \omega_n) \in \Sigma_n $ with $ n > 1 $, we have $ (\omega_1, \ldots, \omega_{n-1}) \in \Sigma_{n-1} $.
\end{definition}

We will refer to spaces $ \Sigma \subset E^{\ast} $ satisfying the extension admissibility property as \textit{general symbolic spaces}.
These spaces have words of arbitrary length, and each word is comprised of admissible subwords.
Such spaces need not be shift-invariant, and the spaces we will consider in our applications will not be.

\subsubsection{Symbolic spaces of infinite type}
Let $ \Sigma \subset E^{\infty} $ be a closed subspace.
For each $ n \geq 1 $ define
$$
\Sigma_n = \{ \omega |_n : \omega \in \Sigma \}.
$$
This definition is compatible with the one given above for spaces of finite words.
There is a natural analogue of Definition \ref{extadm} for these spaces.

\begin{definition}[Restriction admissibility property]
\label{restadm}
We say that $ \Sigma \subset E^{\infty} $ satisfies the \textit{restriction admissibility property} if for all $ \omega \in \Sigma $ and for all $ n > 1 $ with $ \omega |_n \in \Sigma_n $, we have $ \omega |_{n-1} \in \Sigma_{n-1} $.
\end{definition}

We will refer to spaces $ \Sigma \subset E^{\infty} $ satisfying the restriction admissibility property as \textit{symbolic spaces of infinite type}.
There is a natural way of obtaining a space of infinite type from a general symbolic space, and vice versa, called extension and restriction.
There are versions of these notions for sequences of words, and those of spaces.

\subsubsection{Extension and restriction of words}
\label{extrestwords}
Fix a general symbolic space, and consider a sequence of finite words
$$
(\omega_1, \ldots, \omega_n) \in \Sigma_n,
$$
defined for all $ n \in \mathbb{N} $.
In terms of this, we define $ \omega \in E^{\infty} $ by
$$
\omega = (\omega_1, \omega_2, \ldots),
$$
so that $ \omega |_n = (\omega_1, \ldots, \omega_n) $.
The word $ \omega \in E^{\infty} $ is called the \textit{infinite extension} of the sequence $ (\omega_1, \ldots, \omega_n) $.

Similarly, if $ \Sigma \subset E^{\infty} $ is a symbolic space of infinite type, for each word $ \omega \in \Sigma $ we obtain a sequence $ \omega |_n \in \Sigma_n $ by truncating.
This is naturally a sequence in $ E^{\ast} $, and we call it the \textit{finite restriction} of $ \omega $.

Extension and restriction are naturally dual to each other. 
If $ (\omega_1, \ldots, \omega_n) \in \Sigma_n $ is a sequence in a general symbolic space, it is equal to the restriction of its extension.
If $ \omega \in \Sigma $ is a word in a space of infinite type, it is equal to the extension of its restriction.

\subsubsection{Extension and restriction of spaces}
For general symbolic spaces, we have the following analogue of the above notion, which we also refer to as infinite extension.

\begin{definition}[Infinite extension]
\label{infext}
Let $ \Sigma \subset E^{\ast} $ be a general symbolic space. The \textit{infinite extension} $ \Sigma^{\infty} $ is 
$$
\Sigma^{\infty} = \{ \omega \in E^{\infty} : \omega |_n \in \Sigma_n \text{ for all } n \in \mathbb{N} \}.
$$
\end{definition}
Thus the infinite extension $ \Sigma^{\infty} $ of a general symbolic space $ \Sigma $ consists of the infinite words whose finite truncations lie in $ \Sigma $.
Notice that $ \Sigma^{\infty} $ satisfies the restriction admissibility property because $ \Sigma $ is assumed to satisfy the extension admissibility property, so $ \Sigma^{\infty} $ is in fact a space of infinite type.
Similarly, we obtain a general space from a space of infinite type by \textit{finite restriction}.

\begin{definition}[Finite restriction]
\label{finrest}
Let $ \Sigma \subset E^{\infty} $ be symbolic space of infinite type. The \textit{finite restriction} $ \Sigma^{\ast} $ is 
$$
\Sigma^{\ast} = \bigcup_{n \geq 1} \Sigma_n
$$
\end{definition}
Thus the finite restriction $ \Sigma^{\ast} $ of a space of infinite type $ \Sigma $ consists of all the finite truncations of words in $ \Sigma $.
Notice that $ \Sigma^{\ast} $ satisfies the extension admissibility property because $ \Sigma $ is assumed to satisfy the restriction admissibility property, so $ \Sigma^{\ast} $ is in fact general symbolic space.

As with words and sequences, extension and restriction are naturally dual to each other. 
If $ \Sigma $ is a general symbolic space then $ (\Sigma^{\infty})^{\ast} = \Sigma $.
If $ \Sigma $ is a symbolic space of infinite type then $ (\Sigma^{\ast})^{\infty} = \Sigma $.

\subsection{Graph directed symbolic spaces}
Let $ (V,E) $ be a directed graph with countable vertex and edge sets $ V $ and $ E $.
For each edge $ e \in E $ let $ i(e) $ and $ t(e) \in V $ be its initial and terminal vertex, respectively.
Let $ A: E \times E \rightarrow \{0,1\} $ be the edge incidence matrix of this directed graph, i.e. if $ A_{ee'} = 1 $ then $ t(e)=i(e') $.

For $ n \geq 1 $, the admissible words of length $ n $ are
\begin{equation}
\label{admis}
E^n_A = \{ \omega \in E^n : A_{\omega_i \omega_{i+1}} = 1 \text{ for all } 1 \leq i \leq n-1 \}.
\end{equation}
Let $ E^{\ast}_A = \bigcup_{n \geq 1} E^n_A $ be the collection of all finite admissible words, and $ E^{\infty}_A $ the one-sided infinite admissible words.
It is easy to see that $ E_A^{\ast} $ satisfies the extension admissibility property, so it is a special case of a general symbolic space.
Because $ E_A^{\infty} $ is closed, it is a special case of a symbolic space of infinite type.
The infinite extension of $ E_A^{\ast} $ is $ E_A^{\infty} $ and the finite restriction of $ E_A^{\infty} $ is $ E_A^{\ast} $.
The left shift restricts to $ \sigma: E_A^{\infty} \rightarrow E_A^{\infty} $ because the admissible words $ E_A^{\infty} $ are invariant.

\subsection{Dual symbolic spaces}
\label{dualcant}
In this chapter we will define the dual of a symbolic space (see \cite{Sul}).
Consider the case $ E = \mathbb{N} $ so that $ E^{\infty} = \prod_{i=1}^{\infty} E $.
We define the space $ \widetilde{E}^{\infty} \subset \prod_{i=-\infty}^{-1} E $ as follows.
$$
\widetilde{E}^{\infty} = \{(\ldots, \omega_2, \omega_1) : (\omega_1,\omega_2,\ldots) \in E^{\infty} \}
$$
There is a natural bijection $ E^{\infty} \rightarrow \widetilde{E}^{\infty} $ given by
$$
(\omega_1, \omega_2, \ldots,) \mapsto (\ldots, \omega_2, \omega_1)
$$
This map is an isometry in the above metric.
It is also an involution, so we say that $ \widetilde{E}^{\infty} $ is the \textit{dual space} to $ E^{\infty} $.

Similarly, we define
$$
\widetilde{E}^n = \{ (\omega_n,\ldots,\omega_1) : (\omega_1,\ldots,\omega_n) \in E^n \},
$$
and $ \widetilde{E}^{\ast} = \bigcup_{n \geq 1} \widetilde{E}^n $.

For a graph directed symbolic space $ E_A^{\infty} $ as defined in Chapter \ref{Thermo}, we have a dual $ \widetilde{E}_A^{\infty} $ defined by
$$
\widetilde{E}_A^{\infty} = \{ (\ldots, \omega_2, \omega_1) : A_{\omega_{i-1} \omega_i} = 1 \text{ for all } i \},
$$
and similarly for $ \widetilde{E}_A^n $ and $ \widetilde{E}_A^{\ast} $.

Finally, general symbolic spaces, spaces of infinite type, and their subspaces have duals defined in an analogous way.

\vfill
\eject

\section{$ C^{1+\alpha} $ function systems}
\label{Conf}
In this chapter we will present graph-directed pseudo-Markov systems, their limit sets, and some of their associated thermodynamic formalism.
This theory is parallel to that of Stratmann and Urba\'{n}ski \cite{Str}, but altered to account for the symbolic dynamics of the Kuperberg pseudogroup, which will be studied in detail in Chapter \ref{Kupmin}.

We assume that each space is a compact subinterval of $ [0,1] $ and that the maps have regularity $ C^{1+\alpha} $.
From this we will deduce the important properties of bounded variation and distortion in this context, which are analogues of the corresponding properties in the setting of the cookie-cutter Cantor sets of Sullivan \cite{Sul}, \cite{Bed1}.

We will then introduce general function systems-- a natural generalization of pseudo-Markov systems-- and their limit sets.
We conclude by presenting interlaced limit sets of two general function systems satisfying a disjointness condition.

\subsection{Graph directed pseudo-Markov systems}
\label{GDPM}
Let $ X $ be a bounded metric space.
Let $ E $ be a countable alphabet and $ A: E \times E \rightarrow \{0,1\} $ an incidence matrix determining the admissible words $ E_A^{\infty} $.
Assume that for each $ i \in E $ we have injective maps $ f_i : X \rightarrow X $ with a common Lipschitz constant $ 0 < s < 1 $.
We denote $ \Delta_i = f_i(X) $, and further assume that these images satisfy the separation condition
$$
\Delta_i \cap \Delta_j = \emptyset \; \text{ if } \; i \neq j.
$$
The following definition is given in terms of the above notation.

\begin{definition}
\label{GDPMdef}
A \textit{graph directed pseudo-Markov system}-- or pseudo-Markov system for short-- is a set
$$
\bigcup_{\substack{i,j \in E \\ A_{ij} = 1}} \{ \phi_{i,j} : \Delta_j \rightarrow X \}
$$
of injective maps satisfying the following properties.
\begin{itemize}
\item \textit{Lipschitz}: For each $ i $, the maps $ \phi_{i,j} : \Delta_j \rightarrow X $ have a common Lipschitz constant $ 0 < s < 1 $.
\vspace{0.2cm}
\item \textit{Separation}: For each $ i,j \in E $ with $ A_{ij} = 1 $ we have
$$
\phi_{i,j}(\Delta_j) \cap \phi_{i',j'}(\Delta_{j'}) = \emptyset
$$
when $ i \neq i' $ or $ j \neq j' $.
\vspace{0.2cm}
\item \textit{Graph directed property}: For all $ i,j \in E $ with $ A_{ij} = 1 $, we have
$$
\phi_{i,j}(\Delta_j) \subset \Delta_i.
$$
\vspace{0.2cm}
\end{itemize}
\end{definition}

By the graph directed property and Equation \ref{admis}, for each $ n \geq 1 $ and $ \omega \in E_A^n $ we have a map $ \phi_{\omega} : X \rightarrow X $ given by the composition 
\begin{equation}
\label{comp}
\phi_{\omega} = \phi_{\omega_1,\omega_2} \circ \phi_{\omega_2,\omega_3} \circ \cdots \circ \phi_{\omega_{n-1},\omega_n} \circ f_{\omega_n}.
\end{equation}
For convenience, define
\begin{equation}
\label{comp2}
\Delta_{\omega} = \phi_{\omega}(X).
\end{equation}
In this notation, we deduce the \textit{nesting property} $ \Delta_{\omega,i} \subset \Delta_{\omega} $ for all $ \omega \in E_A^{\ast} $ and $ i \in E $ such that $ (\omega,i) \in E_A^{\ast} $.

Since each map $ \phi_{\omega_i,\omega_{i+1}} $ and $ f_i $ has Lipschitz constant $ 0 < s < 1 $, we have for each $ n \geq 1 $ that
$$
\text{diam}\left(\Delta_{\omega|_n}\right) \leq s^n \: \text{diam}(X).
$$
From the nesting property we see $ \Delta_{\omega|_n} \supset \Delta_{\omega|_{n+1}} $.
By this and the above equation, $ \bigcap_{n=1}^{\infty} \Delta_{\omega|_n} $ is necessarily a singleton.
This defines a bijective coding map $ \pi : E_A^{\infty} \rightarrow X $ given by
$$
\pi(\omega) = \bigcap_{n=1}^{\infty} \phi_{\omega |_n}(X) .
$$
The \textit{limit set} $ J $ of the pseudo-Markov system $ \{ \phi_{i,j} \} $ is 
\begin{align}
\label{J}
J &= \pi(E_A^{\infty}) \\
&= \bigcup_{\omega \in E_A^{\infty}} \bigcap_{n=1}^{\infty} \Delta_{\omega|_n} \nonumber \\
&= \bigcap_{n=1}^{\infty} \bigcup_{\omega \in E_A^n} \Delta_{\omega}. \nonumber
\end{align} 
Note: the above description of $ J $ is only true when the pseudo-Markov system is of \textit{finite multiplicity}, which is a consequence of our separation condition.
For a definition of this term and details, see Lemma 3.2 of \cite{Str}.

\subsection{Topological pressure}

\subsubsection{Pressure of continuous potentials}
Fix an alphabet $ E $ and incidence matrix $ A $, and let $ f : E_A^{\infty} \rightarrow \mathbb{R} $ be a continuous function; we will refer to such as a \textit{potential}. 
For any $ n \geq 1 $, denote by $ S_n f : E^n_A \rightarrow \mathbb{R} $ the sum
$$
S_n f(\omega) = \sup_{\tau \in [\omega]} \sum_{j=0}^{n-1} f(\sigma^j \tau),
$$
and from this we form the $ n $th \textit{partition function}
$$
Z_n(f) = \sum_{\omega \in E_A^n} \exp S_n f(\omega).
$$
From the cocycle relation $ S_{m+n}f(\omega) = S_m f(\omega) + S_n f(\sigma^m \omega) $ we deduce that $ Z_{m+n}(f) \leq Z_n(f) Z_m(f) $ and so the following limit exists, which we call the \textit{topological pressure} of the potential $ f $
$$
P(f) = \lim_{n \to \infty} \frac{1}{n} \log Z_n(f).
$$
There is a natural generalization of this notion, to families of potentials.

\subsubsection{Pressure of summable H\"{o}lder families of potentials}
We use the notation
$$ 
F = \{ g_i : X \rightarrow \mathbb{R}, h_{i,j} : \Delta_j \rightarrow \mathbb{R} \} 
$$
to denote a family of H\"{o}lder continuous functions of the same H\"{o}lder order.
Also assume that $ F $ satisfies the summability conditions
$$
\sum_{i \in E} \left\| e^{g_i} \right\| < \infty, \; \text{ and } \; \sum_{\substack{i,j \in E \\ A_{ij}=1}} \left\| e^{h_{i,j}} \right\| < \infty.
$$
We refer to such a family as a \textit{summable  H\"{o}lder family}.
For any $ n \geq 1 $, word $ \omega \in E_A^n $, and summable  H\"{o}lder family $ F $, denote by $ S_n F(\omega) : X \rightarrow \mathbb{R} $ the function
$$
S_n F(\omega) = \sum_{j=1}^n h_{\omega_j, \omega_{j+1}} \circ \phi_{\sigma^j \omega} + g_{\omega_n}.
$$
Similar to above, the following cocycle relation holds:
\begin{align*}
S_{m+n} F(\omega) &= \sum_{j=1}^{n+m} h_{\omega_j, \omega_{j+1}} \circ \phi_{\sigma^j \omega} + g_{\omega_{n+m}} \\
&= \sum_{j=1}^m h_{\omega_j, \omega_{j+1}} \circ \phi_{\sigma^j \omega} + \sum_{j=m+1}^{m+n} h_{\omega_j, \omega_{j+1}} \circ \phi_{\sigma^j \omega} + g_{\omega_{n+m}} \\
&= \sum_{j=1}^m h_{\omega_j, \omega_{j+1}} \circ \phi_{\sigma^j \omega} + \sum_{j=1}^n h_{\omega_{j+m}, \omega_{j+m+1}} \circ \phi_{\sigma^{j+m} \omega} + g_{\omega_{n+m}} \\
&= S_m F(\omega) + S_n F(\sigma^m \omega).
\end{align*}
This implies that the following limit exists:
\begin{equation}
\label{fampres}
P(F) = \lim_{n \to \infty} \frac{1}{n} \log \sum_{\omega \in E_A^n} \left\| \exp S_n F(\omega) \right\|.
\end{equation}
This is called the \textit{topological pressure} of the family $ F $.

\subsection{$ C^{1+\alpha} $ graph directed systems in dimension one}
The pseudo-Markov formalism outlined above is very general.
To apply this formalism to the Kuperberg minimal set, we will make the following assumptions on $ X $, the images $ \Delta_i = f_i(X) $, and maps $ \phi_{i,j} : \Delta_j \rightarrow X $.

\subsubsection{Dimension one.}
From now on we assume that $ X $ is an interval in $ [0,1] $, and that each $ \Delta_i $ is a closed subinterval.
Let $ | \cdot | $ be usual distance on $ [0,1] $, and set $ |U| = \text{diam}(U) $ when $ U \subset [0,1] $.
For any function $ f: X \rightarrow X $ or $ X \rightarrow \mathbb{R} $, we denote its uniform norm in this distance by
$$
\| f \|_{\infty} = \sup_{x \in X} | f(x) |.
$$

From the  condition $ \lim_{n \to \infty} |\Delta_{\omega|_n}| = 0 $ for all $ \omega \in E_A^{\infty} $ we see that the limit set $ J $ from Equation \ref{J} is perfect.
From the separation condition on pseudo-Markov systems, $ J $ is totally disconnected. 
By these facts and our above assumption on $ X $ and $ \Delta_i $, we see that $ J $ is a Cantor set in the line.
See Figure \ref{GDMSLimit} for a picture of a limit set of pseudo-Markov system in the line satisfying these conditions.

\begin{figure}[h]
\includegraphics[width=1.7\linewidth, trim={6.2cm 19.5cm 0 3cm}, clip]{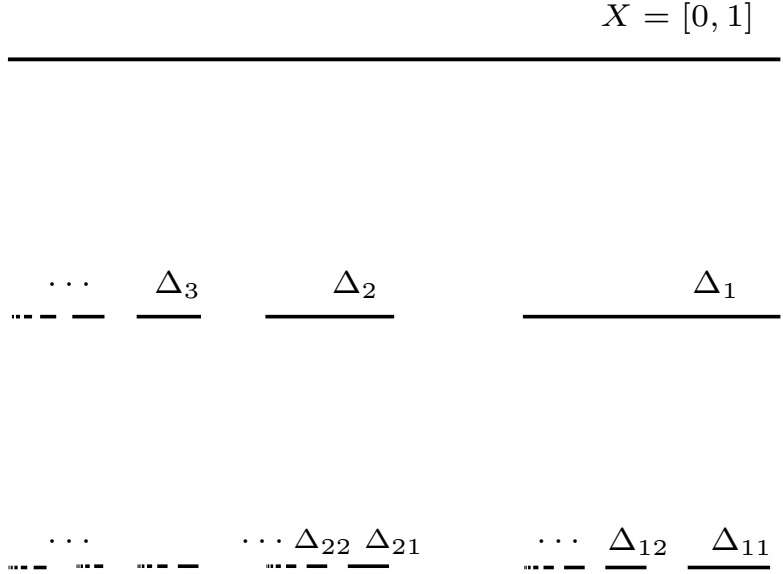};
\caption{The first two steps of the recursive construction of $ J $ in the notation of Equation \ref{J}. The alphabet is $ E=\{1,2,3,\ldots\} $, and the incidence matrix is $ A_{ij} = 1 $ for all $ i,j \in E $. Note the separation condition $ \Delta_i \cap \Delta_j = \emptyset $ and nesting property $ \Delta_{\omega, i} \subset \Delta_{\omega} $.}
\label{GDMSLimit}
\end{figure}

\subsubsection{$C^{1+\alpha} $ regularity}
In general, to develop thermodynamic formalism we need a conformality condition.
Since we are assuming $ \Delta_i \subset X \subset [0,1] $, this can be replaced by the weaker condition of $ C^{1+\alpha} $ regularity.

\begin{definition}
\label{conf}
A pseudo-Markov system $ \{ \phi_{i,j} : \Delta_j \rightarrow X \} $ is said to be $ C^{1+\alpha} $ if there exists an $ \alpha > 0 $ such that
\begin{itemize}
\item for all $ i \in E $, the map $ f_i: X \rightarrow X $ defining $ \Delta_i $ has regularity $ C^{1+\alpha} $.
\item For all $ i,j \in E $ such that $ A_{ij}=1 $, the map $ \phi_{i,j}: \Delta_j \rightarrow X $ has regularity $ C^{1+\alpha} $.
\end{itemize}
\end{definition}

A pseudo-Markov system satisfying this assumption is referred to as a $ C^{1+\alpha} $ pseudo-Markov system. 
Henceforth we will assume this regularity.
The following lemmas are standard in one-dimensional dynamics (see \cite{Shu}, \cite{Bed1}, or the appendix to \cite{Sul}).
Our proofs are based on their analogues for iterated function systems.

\begin{lemma}[Bounded variation]
\label{BV}
Let $ F= \{ g_i, h_{i,j} \} $ be a summable H\"{o}lder family of potentials.
Then there exists a constant $ M>0 $ such that for any $ n \geq 1 $ and all $ \omega \in E_A^n $ we have
$$
\left|S_n F(\omega) (x) - S_n F(\omega)(y) \right| < M
$$
for all $ x,y \in X $.
\end{lemma}

\begin{proof}
Let $ \alpha>0 $ be the H\"{o}lder order of each $ g_i $ and $ h_{i,j} $.
Since these maps have Lipschitz constant $ 0 < s < 1 $, we know for all $ x,y \in X $ that
$$
|\phi_{\omega}(x) - \phi_{\omega}(y)| \leq s^{|\omega|} |X|.
$$
By this and the H\"{o}lder continuity of each potential we have
\begin{align*}
\left|S_n F(\omega) (x) - S_n F(\omega)(y) \right| &\leq \sum_{j=1}^{n-1} \left| h_{\omega_j, \omega_{j+1}} \left( \phi_{\omega_{j+1},\ldots,\omega_n}(x)\right) - h_{\omega_j, \omega_{j+1}} \left( \phi_{\omega_{j+1},\ldots,\omega_n}(y)\right) \right| \\
&\qquad \qquad \qquad \qquad \qquad +|g_{\omega_n}(x)-g_{\omega_n}(y)| \\
&\leq C \sum_{j=1}^{n-1} \left| \phi_{\omega_{j+1}, \ldots, \omega_n}(x) - \phi_{\omega_{j+1}, \ldots, \omega_n}(y) \right|^{\alpha} + C|x-y|^{\alpha} \\
&\leq C \sum_{j=0}^{n-1} s^{(n-j-1) \alpha} |X| \\
&< \frac{C|X|}{1-|s|^{\alpha}}.
\end{align*}
\end{proof}

For $ C^{1+\alpha} $ pseudo-Markov systems in dimension one, we obtain the important \textit{bounded distortion property} from the bounded variation property.

\begin{lemma}[Bounded distortion of derivatives]
\label{BD1}
Let $ \{ \phi_{i,j} \} $ be a $ C^{1+\alpha} $ pseudo-Markov system. Then there exists a constant $ K>1 $ such that for all $ n \geq 1 $ and $ \omega \in E_A^n $,
$$
K^{-1} < \frac{|\phi'_{\omega}(x)|}{|\phi'_{\omega}(y)|} < K
$$
for all $ x, y \in X $.
\end{lemma}

\begin{proof}
Consider the family $ F = \{ g_i, h_{i,j} \} $, where
$$
g_i(x) = \log | f'_i(x) |, \; \text{ and } \; h_{i,j}(x) = \log | \phi'_{i,j}(x) |.
$$
By our $ C^{1+\alpha} $ assumption in Definition \ref{conf}, each $ f_i $ and $ \phi'_{i,j} $ is H\"{o}lder continuous on a compact set and bounded away from zero, so $ F $ is a H\"{o}lder family.
Note that the summability conditions on $ F $ are
$$
\sum_{i \in E} \| f'_i \| < \infty, \; \text{ and } \; \sum_{\substack{i,j \in E \\ A_{ij}=1}} \| \phi'_{i,j} \| < \infty.
$$
The first is a consequence of the mean value theorem and the separation conditions on the images $ \Delta_i = f_i(X) $.
The second is a consequence of that, together with the nesting property $ \Delta_{i,j} \subset \Delta_i $ when $ A_{ij}=1 $.
 
Since $ F $ is a summable H\"{o}lder family, we may apply Lemma \ref{BV} to say there exists a constant $ M>0 $ such that for all $ n \geq 1 $ and all $ \omega \in E_A^n $ we have $ | S_n F(\omega)(x) - S_n F(\omega)(y) | < M $ for all $ x,y \in X $.
For our choice of $ F $, by the chain rule we have
$$
S_n F(\omega)(x) = \sum_{j=1}^{n-1} \log \left| \phi'_{\omega_j,\omega_{j+1}}(\phi_{\omega_{j+1}, \ldots, \omega_n}(x)) \right| + \log \left| f'_{\omega_n}(x) \right| = \log |\phi'_{\omega}(x)|,
$$
so the conclusion of Lemma \ref{BV} states that
$$
e^{-M} < \frac{|\phi'_{\omega}(x)|}{|\phi'_{\omega}(y)|} < e^M.
$$
Let $ K = e^M > 1 $.
\end{proof}

From the bounded distortion of derivatives and the mean value theorem, we obtain bounded distortion of the intervals $ \Delta_{\omega} $.

\begin{lemma}[Bounded distortion of intervals]
\label{BD2}
Let $ K \geq 1 $ be the constant defined in Lemma \ref{BD1}.
Then for all $ n \geq 1 $ and $ \omega \in E_A^n $ we have
$$
K^{-1}|X| < \frac{|\Delta_{\omega}|}{|\phi'_{\omega}(x)|} < K |X|
$$
for all $ x \in X $.
\end{lemma}

\begin{proof}
By the mean value theorem applied to $ \phi_{\omega}: X \rightarrow X $ we have
$$
\inf_{x \in X} |\phi'_{\omega}(x)| \leq \frac{|\Delta_{\omega}|}{|X|} \leq \sup_{x \in X} |\phi'_{\omega}(x)|.
$$
Let $ x^-, x^+ \in X $ be the points on which $ \phi'_{\omega} $ takes its infimum and supremum respectively, and let $ x \in X $ be arbitrary.
By Lemma \ref{BD1} and the above inequality,
$$
K^{-1} |\phi'_{\omega}(x)| < |\phi'_{\omega}(x^-)| \leq \frac{|\Delta_{\omega}|}{|X|} \leq |\phi'_{\omega}(x^+)| < K |\phi'_{\omega}(x)|.
$$
\end{proof}

\subsection{Asymptotically stationary pseudo-Markov systems}
\label{statmark}
In the last chapter, we showed that pseudo-Markov systems with regularity $ C^{1+\alpha} $ have bounds on the distortion of their derivatives and intervals.
In this chapter, we will introduce a simpler class of pseudo-Markov systems with zero distortion, called stationary systems.
Then we will introduce asymptotically stationary systems, a simple generalization of these.

\begin{definition}[Ratio geometry]
Let $ \{ \phi_{i,j} \} $ be a pseudo-Markov system. 
For each $ i \in E $ let $ R_i : E_A^{\ast} \rightarrow \mathbb{R}_{\geq 0} $ be given by
$$
R_i(\omega) = \frac{|\Delta_{\omega,i}|}{|\Delta_{\omega}|}.
$$
The function $ E_A^{\ast} \rightarrow \mathbb{R}_{\geq 0}^{\mathbb{N}} $ defined by $ \omega \mapsto \{ R_i(\omega) \}_{i \in E} $ is called the \textit{ratio geometry} of the pseudo-Markov system.
\end{definition}

The simplest pseudo-Markov systems are those whose ratio geometry is constant.
Following Pesin and Weiss (see \cite{Pes1}, \cite{Pes2}, \cite{Bar2}) we refer to such systems as \textit{stationary}.

\begin{definition}
\label{statdef}
Let $ \{ \phi_{i,j} \} $ be a pseudo-Markov system with ratio geometry $ R_i $.
Suppose that there exist positive real constants $ \{ r_i \}_{i \in E} $ such that for all $ \omega \in E_A^{\ast} $ with $ |\omega| > 1 $, we have
$$
R_i(\omega) = r_i.
$$
Such a pseudo-Markov system is called \textit{stationary}, and the numbers $ \{r_i\}_{i \in E} $ are called the \textit{ratio coefficients} of the system.
\end{definition}

For example, consider a pseudo-Markov system for which $ f_i $ and $ \phi_{i,j} $ are similarities for all $ i,j \in E $ (i.e. $ f'_i $ and $ \phi'_{i,j} $ are everywhere constant); this is a stationary system.

For each $ i \in E $ let $ s_i = |\Delta_i| $.
Then for each $ \omega \in E_A^n $, by Equations \ref{comp} and \ref{comp2}, the lengths of the intervals $ \Delta_{\omega} $ of a stationary pseudo-Markov system are simply a product of the ratio coefficients.
\begin{equation}
\label{ratcoef}
|\Delta_{\omega}| = s_{\omega_1} r_{\omega_2} \cdots r_{\omega_n}
\end{equation}
In Chapter \ref{Dime} we will see that stationary systems have a particularly simple dimension theory, in terms of their ratio coefficients.

We now introduce a class of pseudo-Markov systems whose ratio geometry differs from that of a stationary system by some explicit error functions.

\begin{definition}
\label{asympstat1}
Let $ \{ \phi_{i,j} \} $ be a pseudo-Markov system.
Suppose that there exist positive real constants $ \{ r_i \}_{i \in E} $ and functions $ a^{\pm}: E_A^{\ast} \rightarrow \mathbb{R}_{\geq 0} $ such that for all $ n \geq 1 $ and $ \omega \in E_A^n $,
\begin{equation}
\label{asympratcoef}
s_{\omega_1} r_{\omega_2} \cdots r_{\omega_n} - a^-(\omega) < |\Delta_{\omega}| < s_{\omega_1} r_{\omega_2} \cdots r_{\omega_n} + a^+(\omega)
\end{equation}
Such a pseudo-Markov system is called \emph{asymtotically stationary with error $ a^{\pm} $}.
\end{definition}

To relate these systems to their simpler stationary counterparts, it is necessary to impose some conditions on the error functions $ a^{\pm} $.
With these conditions, we will see later that the dimension theory of limit sets of asymptotically stationary systems can also be analyzed using their ratio coefficients.
\begin{itemize}
\item \textit{Summability}: Assume for all $ n \geq 1 $ that
$$
\sum_{\omega \in E_A^n} a^{\pm}(\omega) < \infty.
$$
\vspace{0.1cm}
\item \textit{Monotonicity}: Assume that the error functions $ a^{\pm} $ depend on an external parameter $ \delta \in \mathbb{R}_{\geq 0} $-- which we notate as $ a^{\pm} = a_{\delta}^{\pm} $-- such that the following holds.
$$
\lim_{\delta \to 0} a_{\delta}^{\pm} = 0.
$$
\vspace{0.1cm}
\end{itemize}
Henceforth when referring to an asymptotically stationary pseudo-Markov system with summable monotone error, we mean a system in the sense of Definition \ref{asympstat1} satisfying these two properties.

\subsection{General function systems}
\label{genfun}
We will now present general function systems and their limit sets.
These are generalizations of graph-directed systems, and their dynamics are not necessarily conjugate to a shift.

Let $ E $ be a countable alphabet and let $ \Sigma \subset E^{\infty} $ be a symbolic space of infinite type as defined in Chapter \ref{Thermo}.
This implies that $ \Sigma_n \neq \emptyset $ for all $ n \geq 1 $.
Let $ X $ be a bounded metric space, and for each $ i \in E $ assume that there exist injective maps $ f_i : X \rightarrow X $ with a common Lipschitz constant $ 0 < s < 1 $.
We denote $ \Delta_i = f_i(X) $ and assume the separation condition
$$
\Delta_i \cap \Delta_j = \emptyset \; \text{ when } \; i \neq j.
$$
In terms of this notation, we give the following definition.

\begin{definition}
\label{genfundef}
A \textit{general function system modeled by $ \Sigma $} is a set
$$
\{ \phi_{i,j} : \Delta_j \rightarrow X \}_{(i,j) \in \Sigma_2}
$$
of injective maps satisfying the following properties.
\begin{itemize}
\item \textit{Lipschitz}: For each $ (i,j) \in \Sigma_2 $, the maps
$$
\{ \phi_{i,j} : \Delta_j \rightarrow X \}
$$
have a common Lipschitz constant $ 0 < s < 1 $.
\vspace{0.2cm}

\item \textit{Separation}: For each $ (i,j) \in \Sigma_2 $ we have
$$
\phi_{i,j}(\Delta_j) \cap \phi_{i',j'}(\Delta_{j'}) = \emptyset
$$
when $ i \neq i' $ or $ j \neq j' $.
\vspace{0.2cm}

\item \textit{Nesting property}: For all $ n \geq 1 $ and $ \omega \in \Sigma_n $ we have
$$
\phi_{\omega_i, \omega_{i+1}}(\Delta_{\omega_{i+1}}) \subset \Delta_{\omega_i}
$$
for all $ 1 \leq i \leq n-1 $.
\vspace{0.2cm}
\end{itemize}
\end{definition}

By the nesting property, for any $ n \geq 1 $ and $ \omega \in \Sigma_n $ we have a map $ \phi_{\omega} : X \rightarrow X $ given by the composition
$$
\phi_{\omega} = \phi_{\omega_1,\omega_2} \circ \phi_{\omega_2,\omega_3} \circ \cdots \circ \phi_{\omega_{n-1},\omega_n} \circ f_{\omega_n}.
$$

Setting $ \Delta_{\omega} = \phi_{\omega}(X) $, we have the following consequence of the nesting property.
$$
\Delta_{\omega,i} \subset \Delta_{\omega}, \; \text{ and } \; \Delta_{\omega,i} \cap \Delta_{\omega,j} \neq \emptyset
$$
for all $ \omega \in \Sigma \cap E^{\ast} $ and $ i \neq j \in E $ such that $ (\omega,i) $ and $ (\omega,j) \in \Sigma \cap E^{\ast} $.

Because the maps $ \phi_{i,j} $ have global Lipschitz constant $ 0 < s < 1 $, we have for each $ n \geq 1 $ that
$$
\text{diam} \left( \Delta_{\omega |_n} \right) \leq s^n \: \text{diam}(X).
$$
As with the graph-directed systems, the compact sets $ \Delta_{\omega|_n} $ are nested, so $ \bigcap_{n=1}^{\infty} \Delta_{\omega|_n} $ is necessarily a singleton and nonempty by our assumption on $ \Sigma $.
This defines a bijective coding map $ \pi : \Sigma \rightarrow X $ given by
$$
\pi(\omega) = \bigcap_{n=1}^{\infty} \Delta_{\omega |_n}.
$$
The \textit{limit set} $ J $ of the general function system $ \{ \phi_{i,j} \} $ is 
\begin{align}
\label{J}
J &= \pi(\Sigma) \\
&= \bigcup_{\omega \in \Sigma} \bigcap_{n=1}^{\infty} \Delta_{\omega|_n} \nonumber \\
&= \bigcap_{n=1}^{\infty} \bigcup_{\omega \in \Sigma_n} \Delta_{\omega}. \nonumber
\end{align} 

As with graph-directed systems, for our applications we will only consider the case when $ X \subset [0,1] $ is compact and each $ \Delta_i \subset X $ is a closed subinterval.

We will impose the same regularity conditions on general function systems as we did on graph-directed systems.
Namely, we assume that there exists $ \alpha > 0 $ such that the maps $ f_i : X \rightarrow X $ and $ \phi_{i,j} : \Delta_j \rightarrow X $ have regularity $ C^{1+\alpha} $.
We call such a function system a \textit{$ C^{1+\alpha} $ general function system modeled by $ \Sigma $}.

If $ A : E \times E \rightarrow \{0,1\} $ is an incidence matrix, the space of admissible words $ E_A^{\infty} $ defined in Chapter \ref{Thermo} is a symbolic space of infinite type, so for the choice $ \Sigma = E_A^{\infty} $, the general function system is a graph directed pseudo-Markov system as in Chapter \ref{GDPM}. 

\subsection{Interlaced limit sets}
\label{sectioninter}
Suppose we have general function systems modeled by two disjoint copies of the same symbolic space, with a mutual disjointness condition on their images.
These two systems can naturally combined to create a function system modeled by a ``joint" sequence space.
The limit set of this function system is said to be the \textit{interlacing} of the limit sets of the two original systems.

In this chapter we will give a precise definition of these terms in the context of limit sets of the $ C^{1+\alpha} $ general function systems from Chapter \ref{genfun}, and then the special case of pseudo-Markov systems from Chapter \ref{GDPM}.

\subsubsection{Interlaced limit sets of general function systems}
Let $ E $ be a countable alphabet, and $ \Sigma \subset E^{\infty} $ a symbolic space of infinite type.
Let $ X \subset [0,1] $ be compact, and consider two $ C^{1+\alpha} $ general function systems $ \{\phi_{i,j} : X_j \rightarrow X\} $ and $ \{ \psi_{i,j} : Y_j \rightarrow X \} $, modeled by $ \Sigma $.
To distinguish between the maps in the two function system, define $ E $ and $ E' $ to be disjoint copies of $ E $, define $ \Sigma \subset E^{\infty} $ and $ \Sigma' \subset {E'}^{\infty} $ two disjoint copies of the same symbolic space, and say that $ \{\phi_{i,j} : X_j \rightarrow X\} $ and $ \{ \psi_{i,j} : Y_j \rightarrow X \} $ are modeled by $ \Sigma $ and $ \Sigma' $, respectively.

Separation conditions on $ X_i = f_i(X) $ and $ Y_j = g_j(X) $ are implicit in the definition presented in Chapter \ref{genfun}.
Assume further that $ X_i $ and $ Y_j $ satisfy the \textit{joint} separation property
$$
X_i \cap Y_j = \emptyset \; \text{ when } \; i \in E \text{ and } j \in F.
$$

For each $ n \geq 1 $ and $ \omega \in \Sigma_n $, $ \tau \in \Sigma'_n $ we have composition maps $ \phi_{\omega}, \psi_{\tau} : X \rightarrow X $ with images $ X_{\omega} = \phi_{\omega}(X) $ and $ Y_{\tau} = \phi_{\tau}(Y) $.
The nesting property satisfied by each function system, together with this joint separation condition, ensures that
$$
X_{\omega} \cap Y_{\tau} = \emptyset \; \text{ for all } \; \omega \in \Sigma_n \text{ and } \tau \in \Sigma'_n
$$
for all $ n \geq 1 $.

These two function systems have Cantor limit sets $ J_{\Sigma} $, $ J_{\Sigma'} $, respectively.
See Figure \ref{JEJF} for a picture of two such limit sets.

\begin{figure}[h]
\minipage{0.51\textwidth}
\includegraphics[width=1.6\linewidth, trim={7cm 20.8cm 3cm 2cm}, clip]{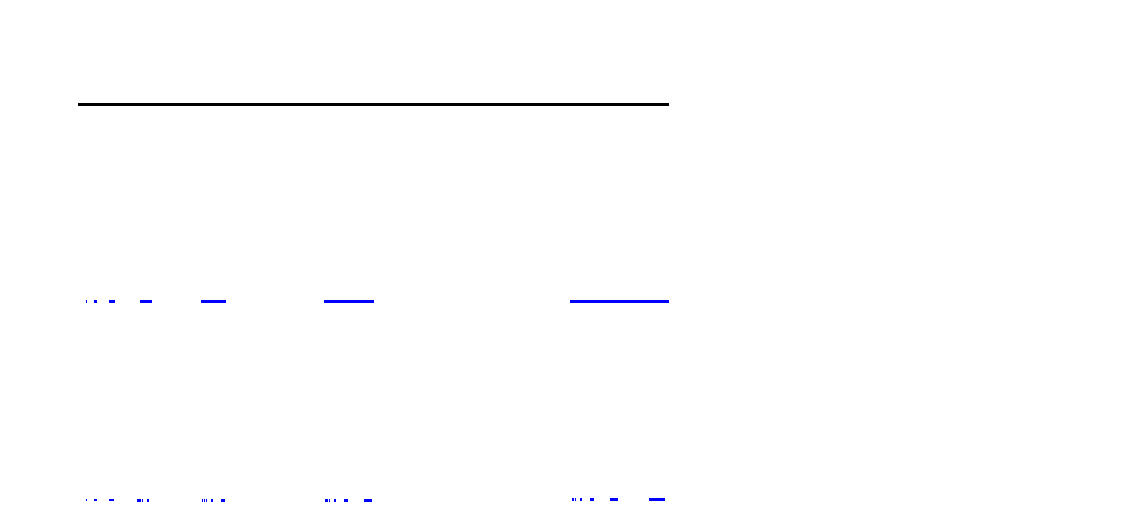}
  \caption*{}
\endminipage
\minipage{0.51\textwidth}%
\includegraphics[width=1.6\linewidth, trim={7cm 20.8cm 3cm 2cm}, clip]{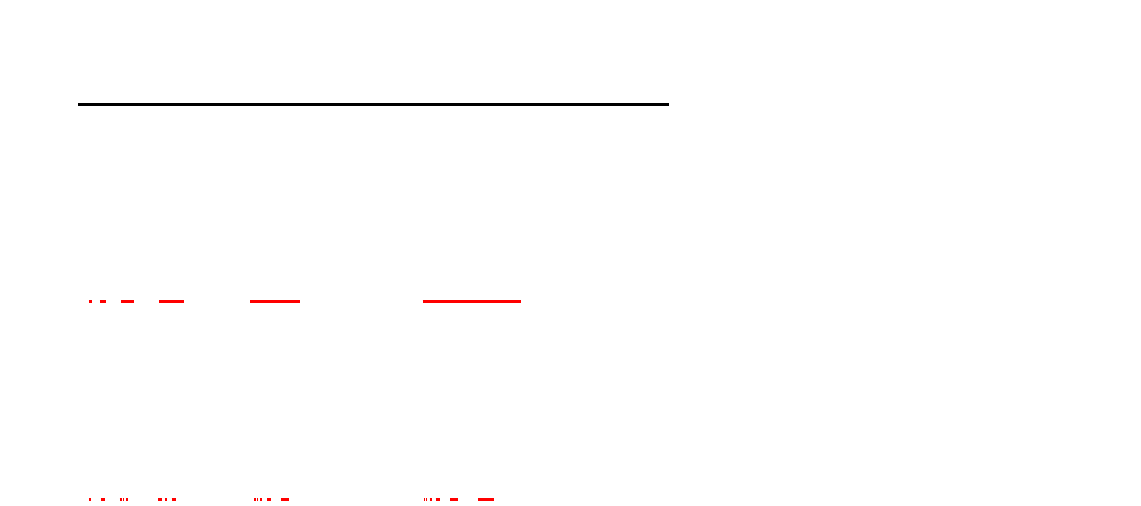}
  \caption*{}
\endminipage
\caption{Two Cantor limit sets $ J_{\Sigma} $ and $ J_{\Xi} $ of general function systems in the line, satisfying the joint separation condition.}
\label{JEJF}
\end{figure}

Let $ \Sigma \ast \Sigma' \subset (E \cup E')^{\infty} $ be the set of all infinite words on the alphabet $ E \cup E' $ comprised of admissible subwords of $ \Sigma $ and $ \Sigma' $.
This is called the \textit{joint} sequence space of $ \Sigma $ and $ \Sigma' $.

From the general function systems $ \{ \phi_{i,j} \} $ and $ \{ \psi_{i,j} \} $ modeled by $ \Sigma $ and $ \Sigma' $, we will now construct a general function system modeled by $ \Sigma \ast \Sigma' $.
For $ i \in E $ and $ j \in E' $, assume we have an extension
$$
\widetilde{\psi}_{i,j} : Y_j \rightarrow X \; \text{  satisfying  } \; \widetilde{\psi}_{i,j}(Y_j) \subset X_i.
$$
Similarly, for $ i \in E' $ and $ j \in E $ assume an extension
$$
\widetilde{\phi}_{i,j} : X_j \rightarrow X \; \text{  satisfying  } \; \widetilde{\phi}_{i,j}(X_j) \subset Y_i.
$$

Now consider the function system
$$
\{ \gamma_{i,j} : Z_j \rightarrow X\}_{(i,j) \in (\Sigma \ast \Sigma')_2}
$$
modeled by $ \Sigma \ast \Sigma' $, where
$$
Z_j =
\left\{
     \begin{array}{lr}
       X_j & : j \in E \\
       Y_j & : j \in E'
    \end{array}
  \right.
$$
and
$$
\gamma_{i,j} =
\left\{
     \begin{array}{lr}
       \phi_{i,j} & : i,j \in E \\
       \widetilde{\phi}_{i,j} & : i \in E', j \in E \\
       \widetilde{\psi}_{i,j} & : i \in E, j \in E' \\
       \psi_{i,j} & : i,j \in E'
    \end{array}
  \right.
$$

Then for any $ \omega \in (\Sigma \ast \Sigma')_n $ we have a composition map $ \gamma_{\omega} : X \rightarrow X $ given by 
$$
\gamma_{\omega} = \gamma_{\omega_1, \omega_2} \circ \gamma_{\omega_2, \omega_3} \circ \cdots \circ \gamma_{\omega_{n-1},\omega_n} \circ h_{\omega_n},
$$
where $ h_{\omega_n} = f_{\omega_n} $ if $ \omega_n \in E $, and $ h_{\omega_n} = g_{\omega_n} $ if $ \omega_n \in E' $.

As in Chapter \ref{genfun}, we let $ \Delta_{\omega} = \gamma_{\omega}(X) $ so that the Cantor limit set of $ \{ \gamma_{i,j} \} $ is
$$
J_{\Sigma \ast \Sigma'} = \bigcap_{n=1}^{\infty} \bigcup_{\omega \in (\Sigma \ast \Sigma')_n} \Delta_{\omega}.
$$
We say the Cantor set $ J_{\Sigma \ast \Sigma'} $ is the \textit{interlacing} of the Cantor sets $ J_{\Sigma} $ and $ J_{\Sigma'} $.
See Figure \ref{interlacefig}.

\begin{figure}[h]
\includegraphics[width=1.5\linewidth, trim={5.5cm 20.5cm 0cm 2.5cm}, clip]{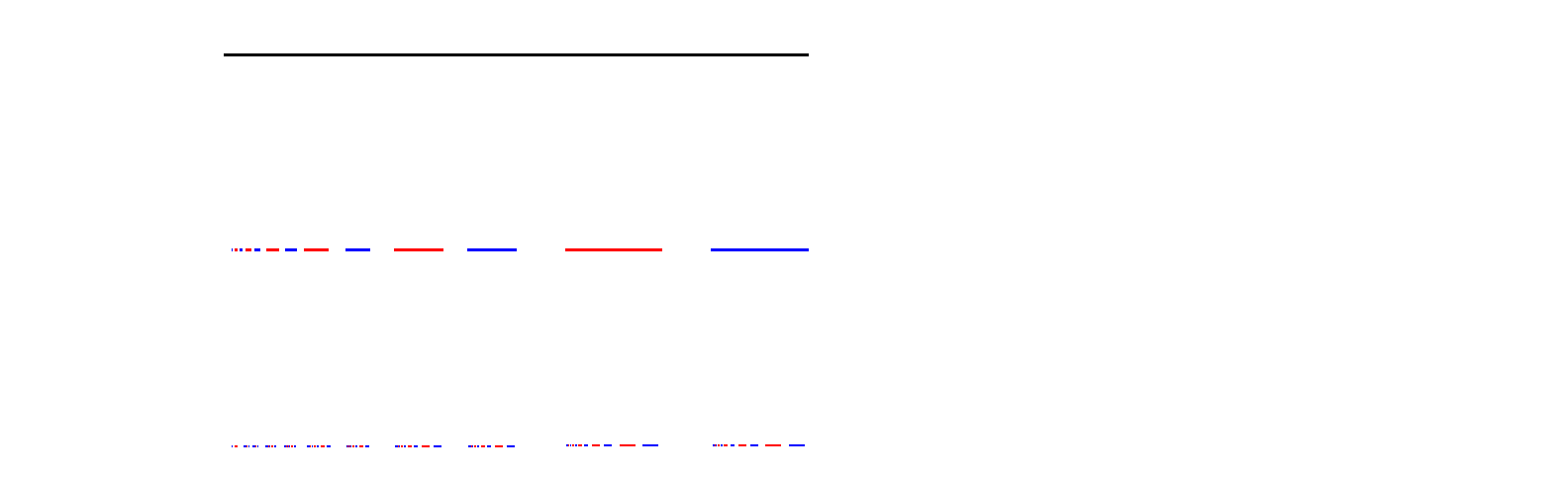};
\caption{The interlacing $ J_{\Sigma \ast \Xi} $ of the Cantor sets $ J_{\Sigma} $ and $ J_{\Xi} $ from Figure \ref{JEJF}.}
\label{interlacefig}
\end{figure}

\subsubsection{Interlaced limit sets of pseudo-Markov systems}
If we have incidence matrices $ A^E : E \times E \rightarrow \{0,1\} $ and $ A^{E'} : E' \times E' \rightarrow \{0,1\} $ such that $ E_{A^E}^n = \Sigma_n $ and $ {E'}_{A^{E'}}^n = \Sigma'_n $ for all $ n \geq 1 $, the above general function systems are graph directed pseudo-Markov systems as studied in Chapter \ref{GDPM}.

Then the joint sequence space $ \Sigma \ast \Sigma' $ defined above is $ (E \cup E')_{A^{E \cup E'}}^{\infty} $, where $ A^{E \cup E'} : (E \cup E') \times (E \cup E') \rightarrow \{0,1\} $ is the joint incidence matrix given by $ A^{E \cup E'}(i,j) = A^E(i,j) $, i.e. the joint words are admissible according to $ E $.

For each $ n \geq 1 $, consider the intervals $ X_{\omega} = \phi_{\omega}(X) $ where $ \omega \in E_{A^E}^n $, and $ Y_{\tau} = \psi_{\tau}(X) $ where $ \tau \in {E'}_{A^{E'}}^n $.
Then by Equation \ref{J} we have the following descriptions of the limit sets of the respective pseudo-Markov systems.
$$
J_E = \bigcap_{n=1}^{\infty} \bigcup_{\omega \in E_{A^E}^n} X_{\omega}, \; \text{ and } \; J_{E'} = \bigcap_{n=1}^{\infty} \bigcup_{\tau \in {E'}_{A^{E'}}^n} Y_{\omega}.
$$
The interlacing $ J_{E \cup E'} $ of $ J_E $ and $ J_{E'} $ is the limit set of the joint pseudo-Markov system, and is given by 
$$
J_{E \cup E'} = \bigcap_{n=1}^{\infty} \bigcup_{\omega \in (E \cup E')_{A^{E \cup E'}}^n} \Delta_{\omega},
$$
where $ \Delta_{\omega} = \gamma_{\omega}(X) $ and $ \gamma_{\omega} $ is the composition of the maps $ \phi_{i,j} $ and $ \psi_{i,j} $ indexed by admissible words $ \omega $ in the joint sequence space $ (E \cup E')_{A^{E \cup E'}}^n $.
Each point in $ J_{E \cup E'} $ corresponds to a unique word in $ (E \cup E')^{\infty}_{A^{E \cup E'}} $.

\vfill
\eject

\section{Dimension theory of limit sets}
\label{Dime}
The Hausdorff dimension of a limit set is related to the pressure by Bowen's equation.
In regularity $ C^{1+\alpha} $, the pressure has uniformity properties that can be deduced from the bounded variation and distortion properties in Lemmas \ref{BV} and \ref{BD2}.
We present these properties for pseudo-Markov systems and then state Bowen's equation in this context.
We then apply this to the dimension theory of the asymptotically stationary pseudo-Markov systems of Chapter \ref{statmark}.

\subsection{The pressure function}
Let $ E $ be a countable alphabet, $ A $ an incidence matrix, and $ \{ \phi_{i,j} : \Delta_j \rightarrow X \} $ a $ C^{1+\alpha} $ pseudo-Markov system as in Chapter \ref{GDPM}.
For any $ t \in (0,\infty) $ consider the family $ F_t = \{ g_i, h_{i,j} \} $, where
$$
g_i(t) = t \log |f'_i(x)|, \; \text{ and } \; h_{i,j}(x) = t \log |\phi'_{i,j}(x)|.
$$
This is a summable H\"{o}lder family of potentials as defined in Chapter \ref{Conf}, and as such has a well-defined topological pressure $ P(F_t) $.
We define $ p(t) = P(F_t) $ and call $ p $ the \textit{pressure function} determined by the system $ \{\phi_{i,j}\} $.
From the proof of Lemma \ref{BD1}, for all $ \omega \in E_A^n $ we have
$$
S_n F_t(\omega)(x) = t \log | \phi'_{\omega}(x) |.
$$
Substituting this into Equation \ref{fampres} we obtain
\begin{equation}
\label{pressurefunction}
p(t) = \lim_{n \to \infty} \frac{1}{n} \log \sum_{\omega \in E_A^n} \| \phi'_{\omega} \|^t.
\end{equation}
Notice that $ p = \lim_{n \to \infty} \frac{1}{n} \log p_n $, where
$$
p_n(t) = \sum_{\omega \in E_A^n} \| \phi'_{\omega} \|^t.
$$
Because $ p_{m+n}(t) \leq p_m(t) p_n(t) $ for all $ t \in [0,\infty) $, we have that $ p_n(t) < \infty $ if and only if $ p_1(t) < \infty $.
Let $ \theta = \inf \{ t: p(t) < \infty \} $, so that the set of finiteness of $ p $ is $ (\theta, \infty) $.
A summary of the properties of $ p $ are collected below.

\begin{proposition}[Proposition 4.10 from \cite{Str}] 
\label{presprop}
The topological pressure function $ p(t) $ is non-increasing on $ [0,\infty) $, and is continuous, strictly decreasing, and convex on $ (\theta, \infty) $.
\end{proposition}

The definition of topological pressure given in Equation \ref{pressurefunction} can be difficult to use in practice.
Fortunately, the assumption of $ C^{1+\alpha} $ regularity and its consequences yields a more useful definition.
Applying Proposition \ref{BD1},
$$
p(t) = \lim_{n \to \infty} \frac{1}{n} \log \sum_{\omega \in E_A^n} \left|\phi'_{\omega}(x)\right|^t
$$
for any $ x \in X $. 
By Proposition \ref{BD2},
\begin{equation}
\label{pressurefunction2}
p(t) = \lim_{n \to \infty} \frac{1}{n} \log \sum_{\omega \in E_A^n} |\Delta_{\omega} |^t.
\end{equation}

\subsection{Bowen's equation for pressure}
A generalization of Bowen's equation (\cite{Bow1}) is proved in \cite{Str} for what are termed ``weakly thin" pseudo-Markov systems.
Weak thinness is a general notion, but in our setting it is equivalent to $ p_1(1) = \sum_{i \in E} |\Delta_i| < \infty $, which is a consequence of the separation and compactness conditions from Chapter \ref{Conf}.

\begin{theorem}[Proposition 4.13 of \cite{Str}]
\label{Bow}
Let $ \{ \phi_{i,j} \} $ be a $ C^{1+\alpha} $ pseudo-Markov system with limit set $ J $ and associated pressure function $ p(t) $.
Then the Hausdorff dimension $ \text{dim}_H(J) $ satisfies
$$
\text{dim}_H(J) = \inf \{ t \geq 0 : p(t) < 0 \},
$$
and if $ p(t) = 0 $ then $ t $ is the only zero of $ p(t) $ and $ t = \text{dim}_H(J) $.
\end{theorem}

\subsection{Dimension of limit sets of asymptotically stationary pseudo-Markov systems}
In Chapter \ref{statmark} we introduced the asymptotically stationary pseudo-Markov systems, with error $ a_{\delta}^{\pm} $.
We assume that this error is summable and monotone, as specified in that chapter.
The dimension theory of stationary systems is particularly simple and goes back to Moran (\cite{Mor}).
The dimension theory of asymptotically stationary systems is similar.

\begin{theorem}
\label{statGDMS}
Let $ \{ \phi_{i,j} \} $ be an asymptotically stationary pseudo-Markov system, with summable monotone error $ a_{\delta}^{\pm} $, and let $ J_{\delta} $ be its limit set. Then the Lebesgue measure of $ J_{\delta} $ satisfies 
$$
\lim_{\delta \to 0} \mu(J_{\delta}) = 0,
$$ 
and the Hausdorff dimension $ \text{dim}_H(J_{\delta}) $ satisfies 
$$ 
0 < \lim_{\delta \to 0} \text{dim}_H(J_{\delta}) < 1.
$$
\end{theorem}

\begin{proof}
Let $ \mu $ be Lebesgue measure on $ [0,1] $.
By the nesting and separation conditions on $ \Delta_{\omega} $,
$$
\mu(J_{\delta}) = \mu \left( \bigcap_{n=1}^{\infty} \bigcup_{\omega \in E_A^n} \Delta_{\omega} \right)
= \lim_{n \to \infty} \mu \left( \bigcup_{\omega \in E_A^n} \Delta_{\omega} \right)
= \lim_{n \to \infty} \sum_{\omega \in E_A^n} |\Delta_{\omega}|.
$$
We then substitute Equation \ref{asympratcoef} to obtain
\begin{align*}
\mu(J_{\delta}) &\leq \lim_{n \to \infty} \sum_{\omega \in E_A^n} s_{\omega_1} r_{\omega_2} \cdots r_{\omega_n} + \lim_{n \to \infty} \sum_{\omega \in E_A^n} a_{\delta}^+(\omega) \\
&\leq \lim_{n \to \infty} \sum_{\omega \in E^n} s_{\omega_1} r_{\omega_2} \cdots r_{\omega_n} + \lim_{n \to \infty} \sum_{\omega \in E_A^n} a_{\delta}^+(\omega) \\
&= \lim_{n \to \infty} \left( \sum_{i \in E} s_i \right) \left( \sum_{i \in E} r_i \right)^{n-1} + \lim_{n \to \infty} \sum_{\omega \in E_A^n} a_{\delta}^+(\omega)
\end{align*}
By the separation condition in Definition \ref{GDPM} we know $ \sum_{i \in E} r_i < 1 $.
By the summability and monotonicity conditions on $ a_{\delta}^+ $ the right term decreases to $ 0 $ as $ \delta \rightarrow 0 $, so $ \lim_{\delta \to 0} \mu(J_{\delta}) = 0 $, as desired.

We now turn to the Hausdorff dimension. 
Let $ p_{\delta}(t) $ be the pressure function associated to this pseudo-Markov system. 
Substituting Equation \ref{asympratcoef} into the pressure function in Equation \ref{pressurefunction2} we obtain that $ p_{\delta}^- < p_{\delta} < p_{\delta}^+ $, where
$$
p_{\delta}^{\pm}(t) = \lim_{n \to \infty} \frac{1}{n} \log \left( \sum_{\omega \in E_A^n} (s_{\omega_1} r_{\omega_2} \cdots r_{\omega_n})^t \pm \sum_{\omega \in E_A^n} a_{\delta}^{\pm}(\omega) \right)
$$
Then by the monotonicity of $ a_{\delta}^{\pm} $ we have that $ \lim_{\delta \to 0} p_{\delta}^{\pm} = p $, where
$$
p(t) = \lim_{n \to \infty} \frac{1}{n} \log \sum_{\omega \in E_A^n} (s_{\omega_1} r_{\omega_2} \cdots r_{\omega_n})^t
$$
For the upper bound, we calculate
\begin{align*}
p(t) &< \lim_{n \to \infty} \frac{1}{n} \log \sum_{\omega \in E^n} s_{\omega_1}^t r_{\omega_2}^t \cdots r_{\omega_n}^t \\
&= \lim_{n \to \infty} \frac{1}{n} \log \left( \sum_{i \in E} s_i^t \right) \left( \sum_{i \in E} r_i^t \right)^{n-1} \\
&= \log \sum_{i \in E} r_i^t.
\end{align*}
Let $ t_{\ast} $ be the unique solution to $ \sum_{i \in E} r_i^t = 1 $, and notice that $ t_{\ast} < 1 $.
Applying Bowen's theorem (\ref{Bow}), we have $ \dim_H(J) < t_{\ast} < 1 $.

For the lower bound, recall that for all $ n \geq 1 $, $ E_A^n $ contains more than one word, say $ \omega = (\omega_1, \ldots, \omega_n) $.
$$
p(t) > \lim_{n \to \infty} \frac{1}{n} \log \left(s_{\omega_1} r_{\omega_2} \cdots r_{\omega_n} \right)^t = \lim_{n \to \infty} \frac{1}{n} \left( \log s_{\omega_n}^t + \sum_{j=1}^{n-1} \log r_{\omega_j}^t \right).
$$
Setting the right hand side $ =0 $, we see that $ t_{\ast} = 0 $ is a solution.
So again by Bowen's theorem, we have $ \text{dim}_H(J) > t_{\ast} = 0 $.
\end{proof}

\vfill
\eject

\section{The Wilson flow}
\label{Wilsflow}

Wilson's flow (\cite{Wil}) is defined on a \textit{plug}, a closed manifold that traps orbits.
First we will define general plugs, and then present the construction of Wilson's plug.
Then we will introduce Wilson's vector field, and study its dynamics in the plug.

\subsection{Plugs}
\label{plugs}
Let $ M $ be a compact orientable manifold with nonempty boundary. 
A \textit{plug} is a product $ M \times [0,1] $, supporting a vector field $ \mathcal{V} $ with flow $ \phi_t $.

For the plugs we consider, $ M $ will have dimension two, so $ M \times [0,1] $ is an oriented three-manifold with boundary $ \partial M \times [0,1] $. Let $ (x,z) $ be a coordinate system on $ M \times [0,1] $. We will orient the plug vertically, so that $ M \times \{0\} $ is the ``bottom" of the plug, and $ M \times \{1\} $ the ``top."
If $ (x,0) \in M \times \{0\} $ and $ (x',1) \in M \times \{1\} $ satisfy $ x=x' $, then these two points are said to be \textit{facing}.

A plug is a local dynamical system designed to be inserted into a global one.
For the plug to be inserted into a manifold with a flow there are several important assumptions it must satisfy.
These ensure that the dynamics inside the plug are compatible with the dynamics outside, and that the plug traps a set of orbits of the flow on the manifold.
\begin{itemize}
\item \textit{Matched ends property}: If a flowline of $ \phi_t $ passes through the points $ (x, 0) $ and $ (x', 1) $, then these points are facing, i.e. $ x=x' $.
\item \textit{Trapped orbit property}: There exists a flowline of $ \phi_t $ passing through $ (x,0) $ but not intersecting $ M \times \{1\} $.
\end{itemize}
If a plug satisfies the following additional symmetry condition, we call it a \textit{mirror-image plug}.
\begin{itemize}
\item \textit{Mirror-image property}: The reflection of the field $ \mathcal{V} $ over the center $ M \times \left\{\frac{1}{2}\right\} $ is the negative of $ \mathcal{V} $.
\end{itemize}
Flowlines in a mirror-image plug are symmetric over $ M \times \left\{\frac{1}{2}\right\} $.
Notice that the mirror-image property implies the matched-ends property.

The Wilson plug is a mirror-image plug with $ M $ a closed annulus. For the original description see \cite{Wil}, and for subsequent descriptions \cite{Kup1}, \cite{Ghy}, \cite{Hur}.
Our notation does not differ much from this literature's.

\subsection{The Wilson plug}
Define the closed rectangle $ E = [1,3] \times [-2,2] $ in coordinates $ (r, z) $, and the closed rectangular solid $ E \times [0,2\pi] $, in coordinate $ (r, \theta, z) $.
Denote by $ c_1, c_2 \in E $ the points $ (2,-1) $ and $ (2,+1) $, respectively.
Then $ l_i = c_i \times [0,2\pi] $ are two line segments in the rectangular solid.

Finally, define closed neighborhoods $ B_i $ of $ c_i  $, so that $ B_i \times [0,2\pi] $ is a tubular neighborhood of each $ l_i $.
The Wilson plug $ W $ is the image of the region $ E \times [0,2\pi] $ under the embedding $ (r, \theta, z) \mapsto (r \cos \theta, r \sin \theta, z) $.

\begin{figure}[!h]
\includegraphics[width=\linewidth, trim={2cm 14.5cm 4cm 4cm}, clip]{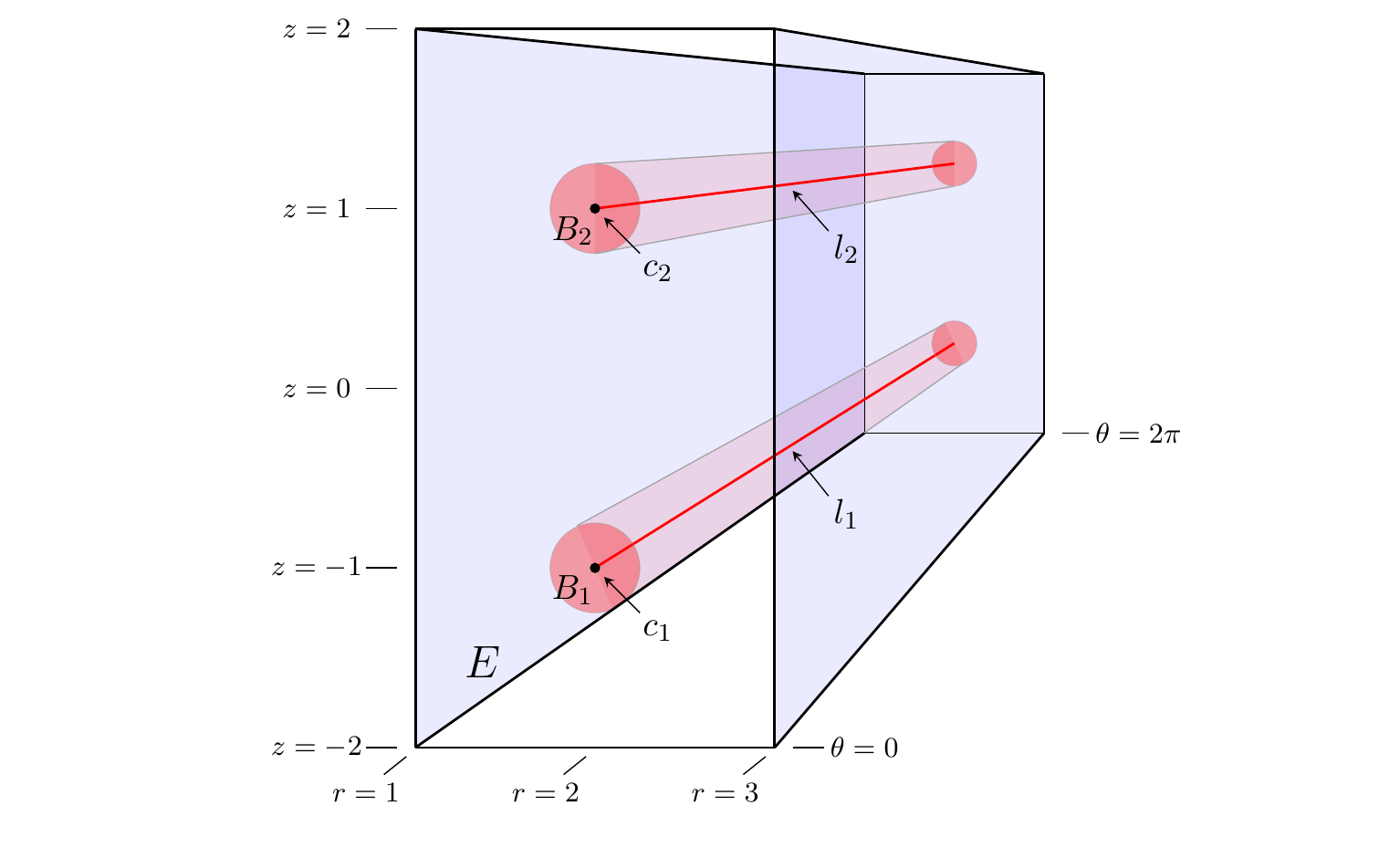};
\caption{The rectangular region $ E \times [0,2\pi] $ with the coordinates and special regions indicated}
\label{rectreg}
\end{figure}

See Figures \ref{rectreg} and \ref{plug1} for a picture of the rectangle and the embedded plug, respectively.
Notice that the lines $ l_i $ map to circles under the embedding, and the tubes $ B_i \times [0,2\pi] $ map to torii containing the corresponding circles $ l_i $. 

\begin{figure}[!h]
\includegraphics[width=\linewidth, trim={2cm 13.5cm 2cm 2.6cm}, clip]{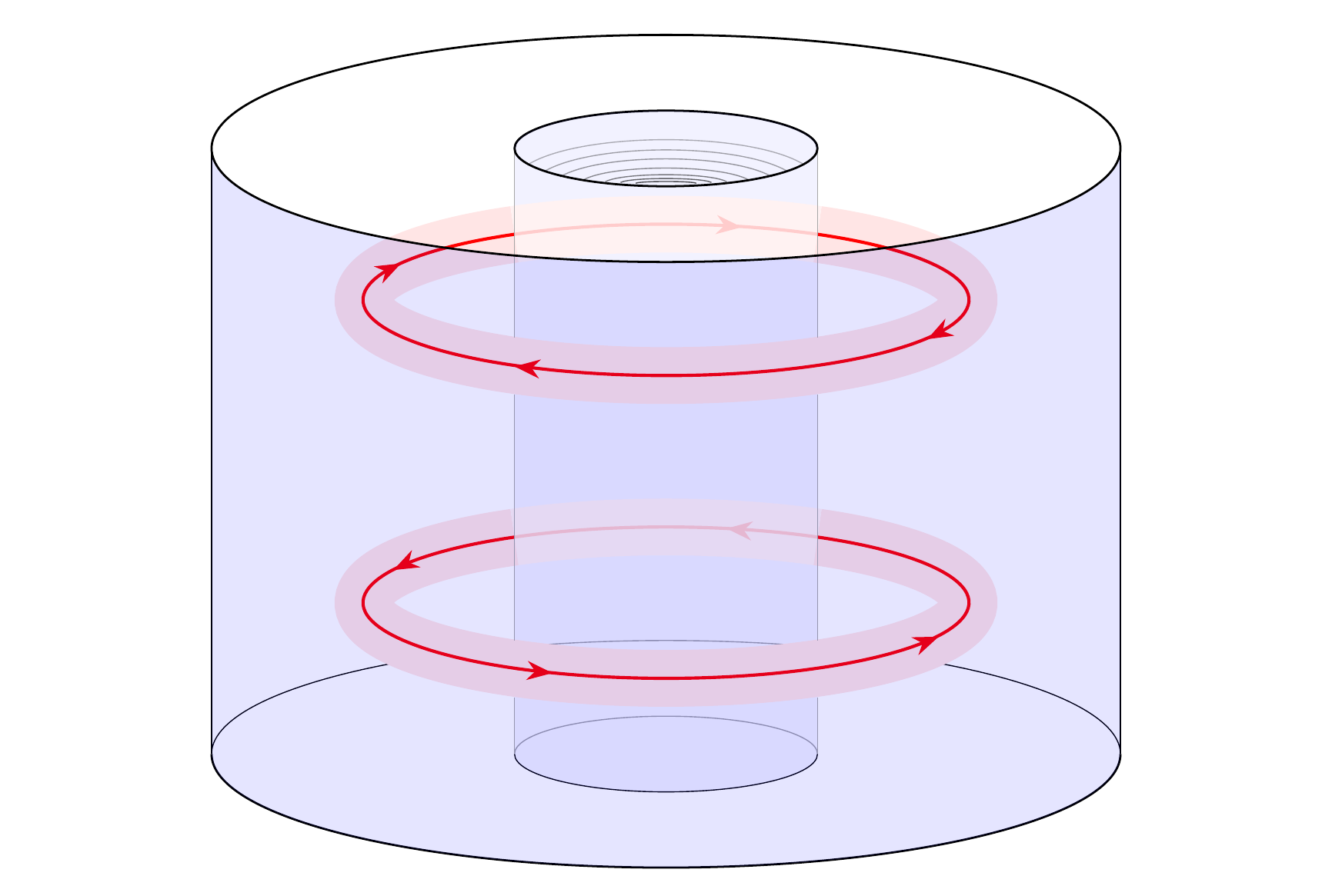};
\caption{The embedded Wilson plug. The lines $ l_i $ from Figure \ref{rectreg} map to periodic orbits of the Wilson flow.}
\label{plug1}
\end{figure}

Under this embedding, $ M=\{(r,\theta) : 2 \leq r \leq 3, 0 \leq \theta \leq 2\pi \} $ is an annulus, and $ W=M \times [-2,2] $ is a plug in the notation of Chapter \ref{plugs}.
The bottom of the plug is $ M \times \{-2\} $ and the top is $ M \times \{2\} $.

\subsection{The Wilson vector field}
For convenience, we will describe the dynamics in the coordinates $ (r, \theta, z) $ and suppress the embedding.
On $ E \times [0,2\pi] $, we define a vector field $ \mathcal{W} $.
\begin{equation}
\mathcal{W} = f \: \frac{\partial}{\partial \theta} + g \: \frac{\partial}{\partial z},
\label{W}
\end{equation}
where $ f $ and $ g $ are $ C^{\infty} $ real-valued functions of the rectangle $ E $, constructed as follows. First, fix $ a > 0 $, and define $ f : E \rightarrow \mathbb{R} $ by

\begin{equation}
f(r,z)=
\left\{
     \begin{array}{lr}
       a & : z < 0 \\
       -a & : z \geq 0
    \end{array}
  \right.
\label{f}
\end{equation}

Notice that this function is not $ C^{\infty} $-- not even continuous-- but can be made so by adjusting it in an arbitrarily small neighborhood of $ \{z=0\} \subset R $.

To construct $ g $, for $ i=1,2$ let $ p_i : B_i \rightarrow [0,1] $ be $ C^{\infty} $ functions satisfying
\begin{equation}
\label{p}
p_i(c_i) = 0, \quad \quad p_i \equiv 1 \text{ on } \partial B_i, \quad \quad p_i(x) > 0 \text{ for all } x \in B_i \setminus \{c_i\}
\end{equation}
Then we define $ g : E \rightarrow [0,1] $ by

\begin{equation}
\label{g}
g(x) =  \left\{
     \begin{array}{lr}
       p_i(x) & : x \in B_i, \; i=1, 2 \\
       1 & : x \in E \setminus (B_1 \cup B_2)
     \end{array}
   \right.
\end{equation}
Notice that $ g \equiv 1 $ outside the regions $ B_i $. Inside each $ B_i $, $ g $ decreases smoothly to zero, reaching zero (by definition of $ p_i $) at precisely $ c_i $.

Since $ g \equiv 0 $ at the two points $ c_i \in E $, the $ z $ component of the Wilson field $ \mathcal{W} $ (equation \ref{W}) is singular on the circles $ l_i $.
The field $ \mathcal{W} $ preserves these circles, forming two periodic orbits inside the plug. 
These are referred to as the \textit{special orbits}, and are illustrated in Figure \ref{plug1}.
The torii $ B_i \times [0,2\pi] $ that contain them are referred to as the \textit{critical torii}.

Finally, we define the \textit{Reeb cylinder} as $ \mathcal{R} = \{r=2\} $ and for any $ \epsilon > 0 $ we define the \textit{critical region} $ C_{\epsilon} $ as an $ \epsilon $-neighborhood of $ \mathcal{R} $-- explicitly, $ C_{\epsilon} = \{ 2 \leq r \leq 2+\epsilon \} \subset W $.
All the interesting dynamics will occur inside this critical region.

\subsection{Dynamics of the Wilson flow}

\subsubsection{Orbits of points: Helices}
Let $ \phi_t $ be the flow of $ \mathcal{W} $.
By definition of $ \mathcal{W} $, the radial coordinate of each orbit is preserved, so that flowlines are helical in shape.

At the base annulus $ \{z=-2\} $ we have $ f \equiv a $ and $ g \equiv 1 $ in equation \ref{W}, so the orbit spirals upward counter-clockwise from the base annulus to the central annulus $ \{z=0\} $.
At this point, $ f \equiv -a $, so the $ \theta $ component of the flow direction is reversed; now the orbit spirals upward clockwise until it reaches the upper annulus $ \{z=2\} $ and escapes the plug.

Since $ f $ is anti-symmetric across the line $ \{z=0\} \subset E $, flowlines are symmetric about the annulus $ \{z=0\} \subset W $.
This implies that $ W $ is a mirror-image plug. In particular, it satisfies the matched-ends property (See Chapter \ref{plugs}).
Wilson orbits that originate in the base $ \{z=-2\} $ of the plug have three orbit types, as shown in Table \ref{wilsontable}.
The third orbit type shows that $ W $ satisfies the trapped orbit property.

\begin{table}
\begin{tabular}{ m{7cm} m{11cm} }
\begin{itemize}
\item \textbf{Disjoint from critical torii}: In this case $ f \equiv \pm a $ and $ g \equiv 1 $ in equation \ref{W}, and the orbit helix spirals at a constant speed. 
The orbit takes a short time to escape the plug.
\end{itemize}
&
\includegraphics[width=0.8\linewidth, trim={1cm 13cm 1cm 1cm}, clip]{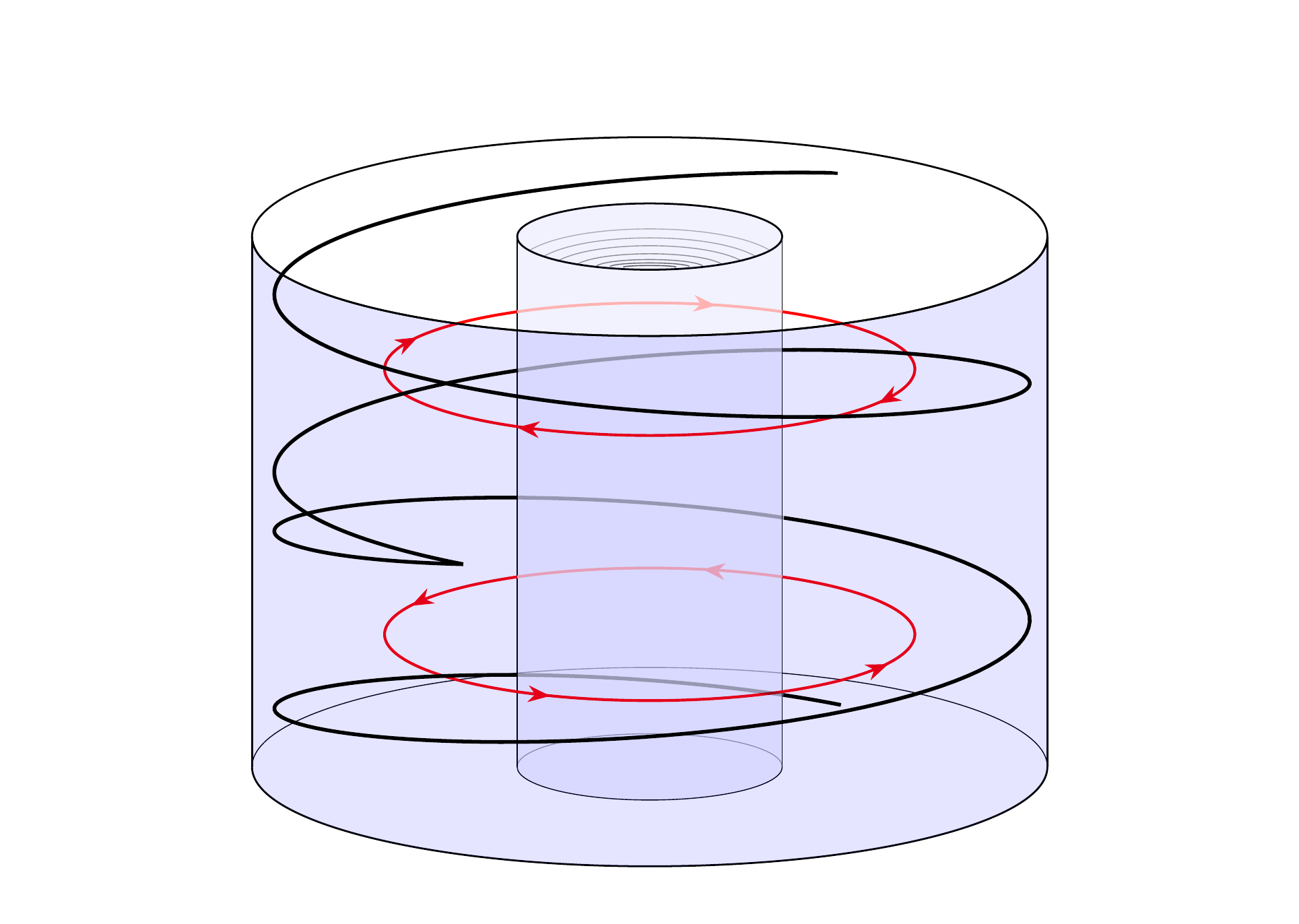}
\\
\begin{itemize}
\item \textbf{Intersecting critical torii}: Inside the critical torii, $ g \equiv p_i $ which is zero at $ c_i $. 
Thus the vertical component of the orbit slows dramatically inside the torii, at a speed depending on the orbit's radial proximity to the special orbit.
The orbit takes a long time to escape the plug.
\end{itemize}
&
\includegraphics[width=0.8\linewidth, trim={1cm 13cm 1cm 1cm}, clip]{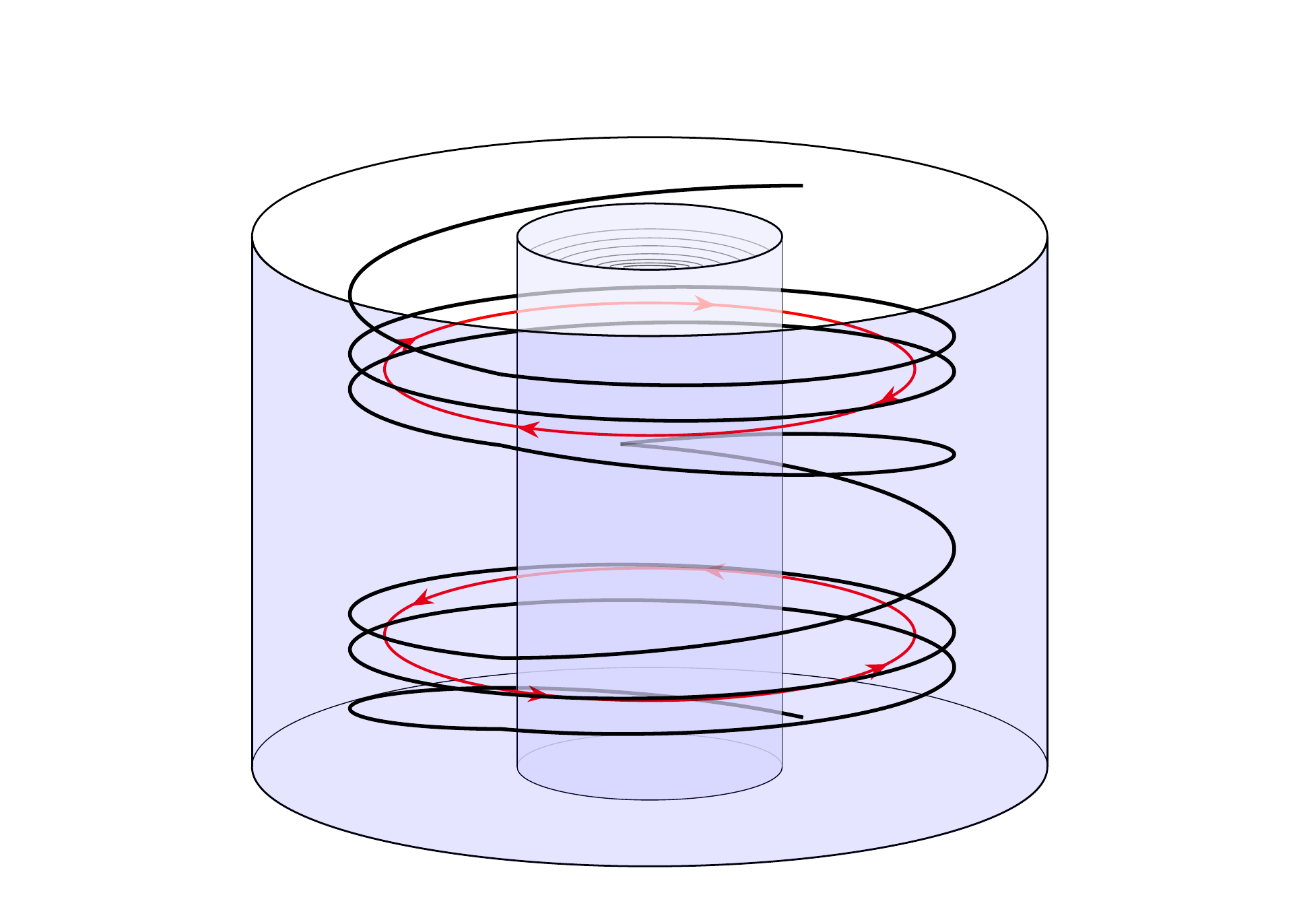}
\\
\begin{itemize}
\item \textbf{In the Reeb cylinder $ \mathcal{R} $}: Upon entering the first torus, $ g \equiv p_1 $ and the orbit spirals towards the special orbit $ c_1 $. 
As the orbit approaches $ c_1 $ the speed of its vertical component approaches zero.
The orbit is trapped and remains in the plug for infinite time. 
\end{itemize}
&
\includegraphics[width=0.8\linewidth, trim={1cm 13cm 1cm 1cm}, clip]{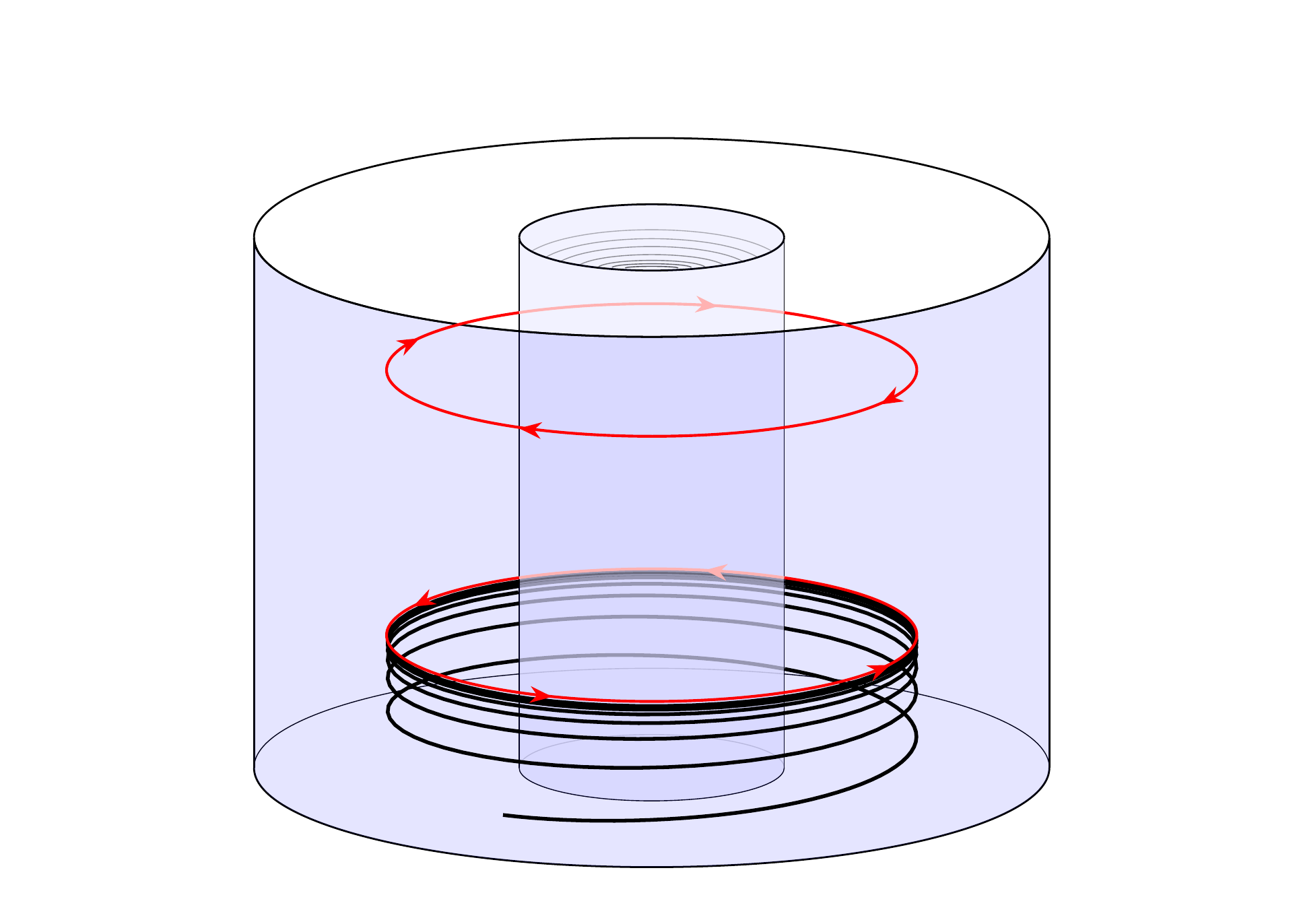}
\\
\end{tabular}
\caption{Classification of Wilson orbits originating in the base $ \{z=-2\} $}
\label{wilsontable}
\end{table}

\subsubsection{Orbits of curves: Propellers}
\label{propellers}
Following \cite{Hur} we make the following definition.

\begin{figure}[h]
\minipage{0.5\textwidth}
\includegraphics[width=0.9\linewidth]{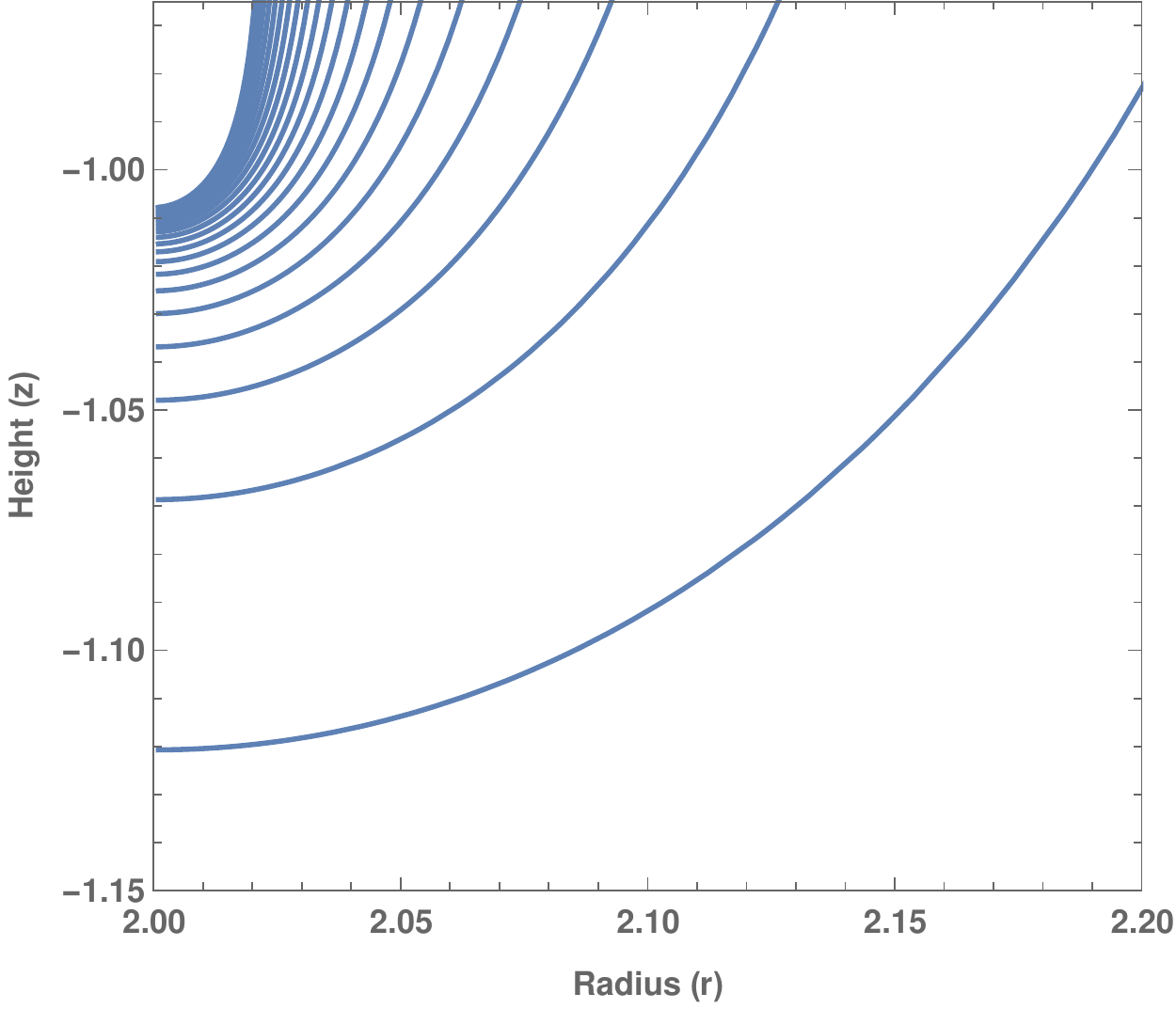}
  \caption*{Section of single propeller}
\endminipage
\minipage{0.5\textwidth}%
\includegraphics[width=0.9\linewidth]{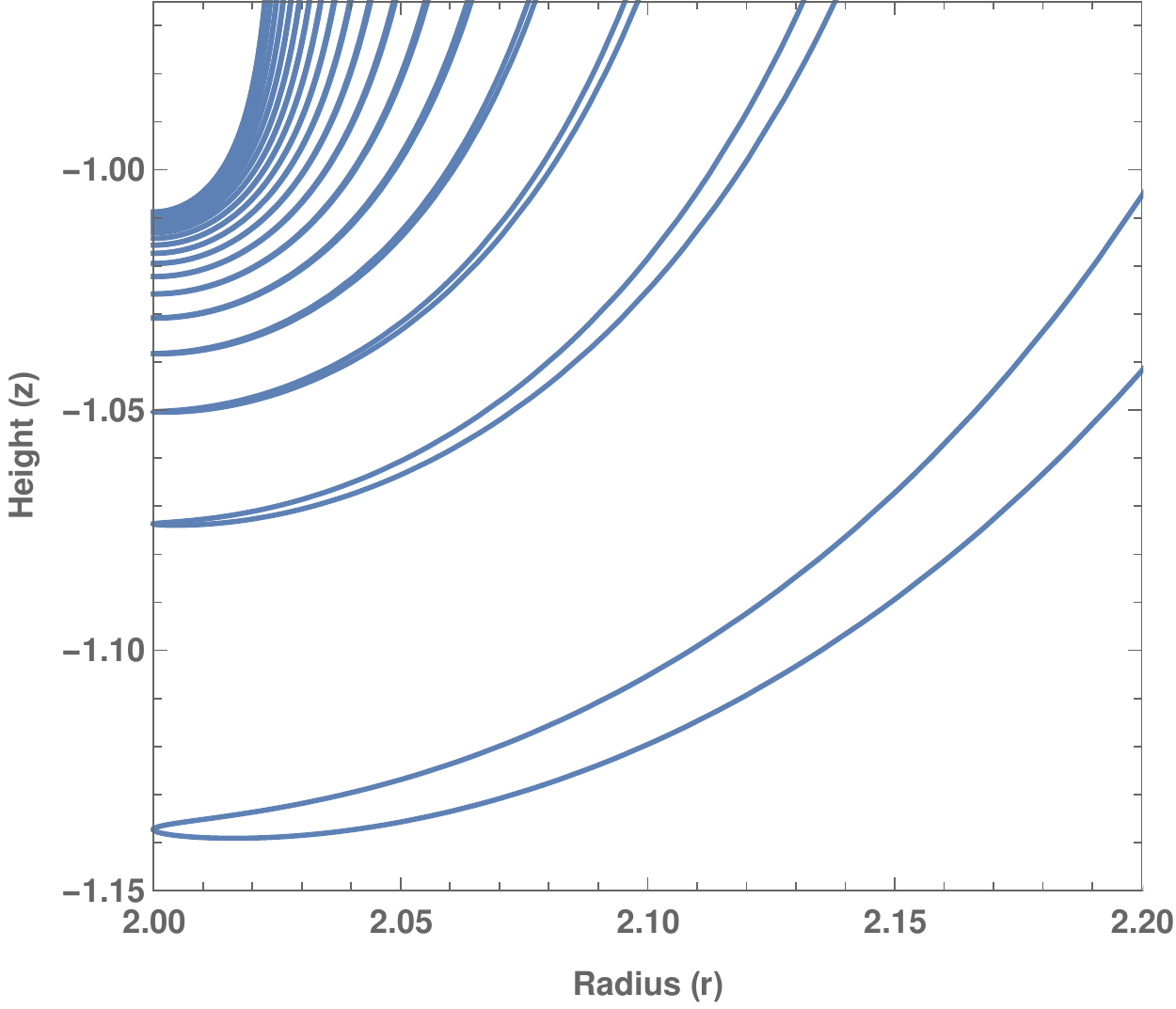}
  \caption*{Section of double propeller}
\endminipage
\caption{The intersections of a single and double propeller with a transverse section $ \{\theta= \text{constant}\}$ of the Wilson plug. The inside edge at $ r=2 $ is trapped and limits on the intersection of the special orbit $(r,z)=(2,-1)$, while the outside edge(s) at $ r>2 $ escapes.}
\label{prop}
\end{figure}

\begin{definition}[\textit{Single propellers}]
\label{propdef}
Let $ \eta : [s_1,s_2] \rightarrow W $ be a continuous curve such that the radial coordinate of $ \eta(s_1) $ is $ 2 $, and for all $ s_1 < s \leq s_2 $, the radial coordinate of $ \eta(s) $ is strictly greater than $ 2 $. 
A \textit{single propeller} is $ \bigcup_{t \geq 0} \phi_t(\eta) $ for such an $ \eta $.
\end{definition}

\begin{definition}[\textit{Double propellers}]
\label{doublepropdef}
Let $ \eta : [s_1,s_2] \rightarrow W $ be a continuous curve such that there exists $ s_1 < s_c < s_2 $ with $ \eta(s_c) $ having a radial coordinate of $ 2 $, and for all $ s_1 \leq s < s_c $ and $ s_c < s \leq s_2 $ the radial coordinate of $ \eta(s) $ is strictly greater than $ 2 $. 
A \textit{double propeller} is $ \bigcup_{t \geq 0} \phi_t(\eta) $ for such an $ \eta $.
\end{definition}

Notice that a single propeller can be obtained from a double propeller by restricting the parametrization of the generating curve $ \eta $.
We will see later that the minimal set of the Kuperberg flow can be decomposed into a union of single propellers, so understanding how propellers are embedded in $ W $ is the key to understanding the embedding of the minimal set.
A propeller forms a ``helical ribbon'' winding around the Wilson plug.
Its outside edge has an $ r $-coordinate bounded away from $ 2 $, so it forms a helix, the first orbit type. Its inside edge has an $ r $-coordinate of $ 2 $ and thus is trapped in the plug, the third orbit type.
Thus each propeller contains curve that is trapped for infinite time, resulting in a complicated embedding in the plug.
This complexity is illustrated in a cross-section of the Wilson plug shown in Figure \ref{prop}.

\subsection{The Wilson minimal set}
Let $ x \in W $. By Table \ref{wilsontable}, if the radial coordinate of $ x $ is $ >2 $, its orbit escapes through the top $ \{z=2\} $ in finite forward time, and escapes through the bottom $ \{z=-2\} $ in finite backward time.
If the radial coordinate of $ x $ is $ =2 $, its orbit limits on one of the special orbits $ l_i $ in forward and/or backward time, depending on its vertical position in the plug.
Thus the minimal invariant set in $ W $ is the union $ l_1 \cup l_2 $.
See Figure \ref{wilmin}.

\begin{figure}[h]
\includegraphics[width=\linewidth, trim={2cm 13.5cm 2cm 2.6cm}, clip]{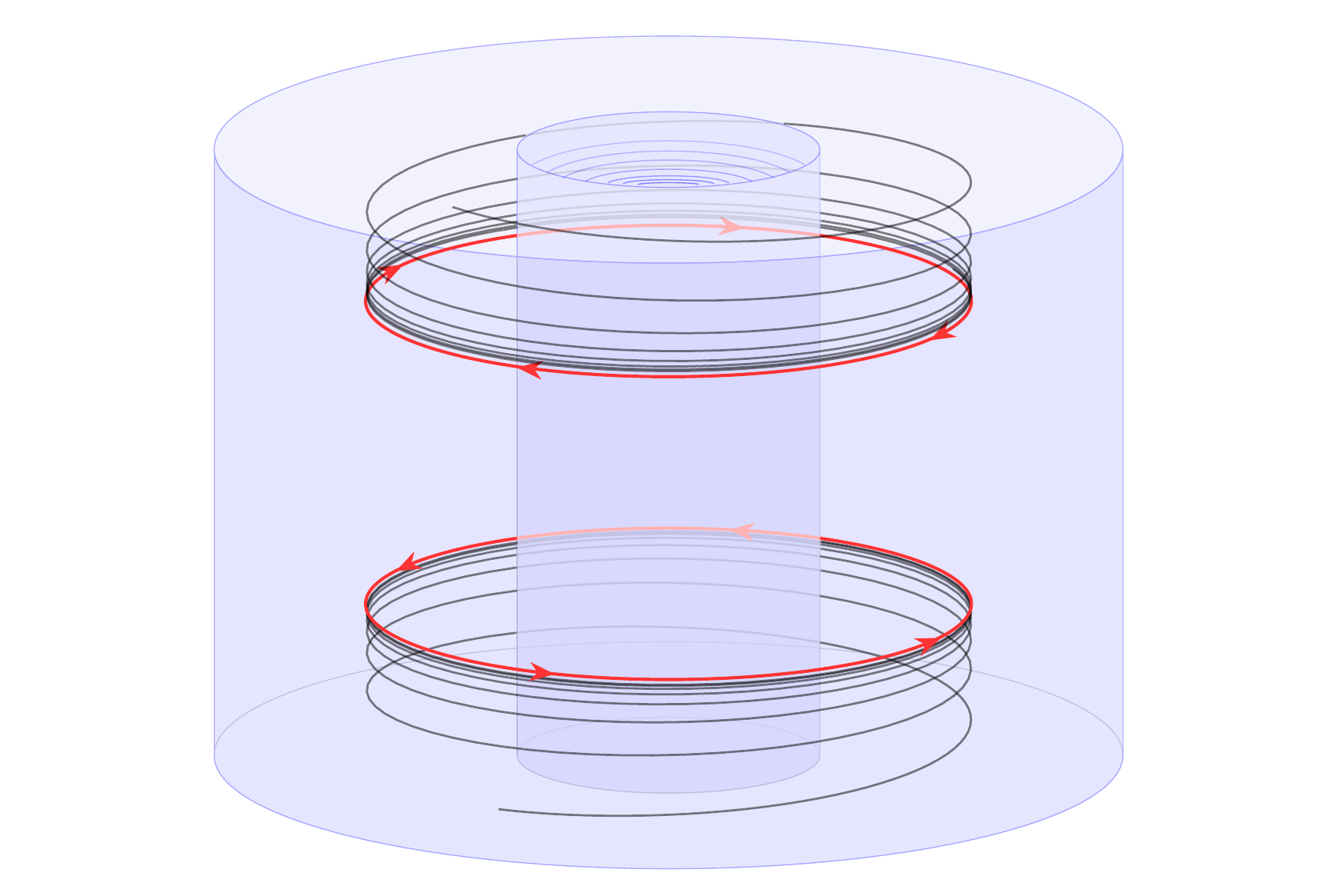}
\caption{The minimal set in the Wilson plug $ W $ is the two special orbits $ l_1 \cup l_2 $.}
\label{wilmin}
\end{figure}

\subsection{The Wilson pseudogroup}
\label{Wilpseudo}
Let $ S \subset W $ be a surface tranverse to the Wilson flow $ \phi_t $. 
For our purposes, it will suffice to consider a small rectangle with a constant $ \theta $-coordinate.
Consider the first return map $ \Phi : S \rightarrow S $ of $ \phi_t $ to $ S $. 
Explicitly, $ \Phi(x) = \phi_T(x) $ where $ T>0 $, $ \phi_T(x) \in S $, and $ T $ is minimal with respect to these properties.
Each such map has a natural inverse, by first-return under the backward orbit.

Notice that $ \Phi $ is not defined on all of $ S $, nor are successive compositions of $ \Phi $ necessarily defined, even where $ \Phi $ is.
Thus $ \Phi $ does not generate a group, but does generate a \textit{pseudogroup} (see \cite{Hur0}, \cite{Wal}) of local homeomorphisms. 
This is the holonomy pseudogroup of the foliation of $ W $ by flowlines of $ \phi_t $.

\vfill
\eject

\section{The Kuperberg flow} 
\label{Kupsection}
The Kuperberg plug is constructed by performing two operations of \textit{self-insertion} on the Wilson plug.
We will summarize this below, but the construction is delicate and we refer to \cite{Kup1} for the details. 

\subsection{Kuperberg's construction and theorem}
\label{KupThm}
First we define two closed disjoint regions $ L_1, L_2 \subset W $, intersecting the outside boundary $ \{r=3\} $ of the plug, the top and bottom of the plug, and the two special orbits.
For $ i=1,2 $ we denote by $ L_i^+ $ the intersection of these regions with the top of the plug, and by $ L_i^- $ the bottom.
We then re-embed the Wilson plug in $ \mathbb{R}^3 $ in a folded figure-eight.
See Figure \ref{kup1}.

\begin{figure}[htp]
\subfloat{
\includegraphics[width=0.8\linewidth, trim={4.5cm 16.6cm 0cm 4.5cm}, clip]{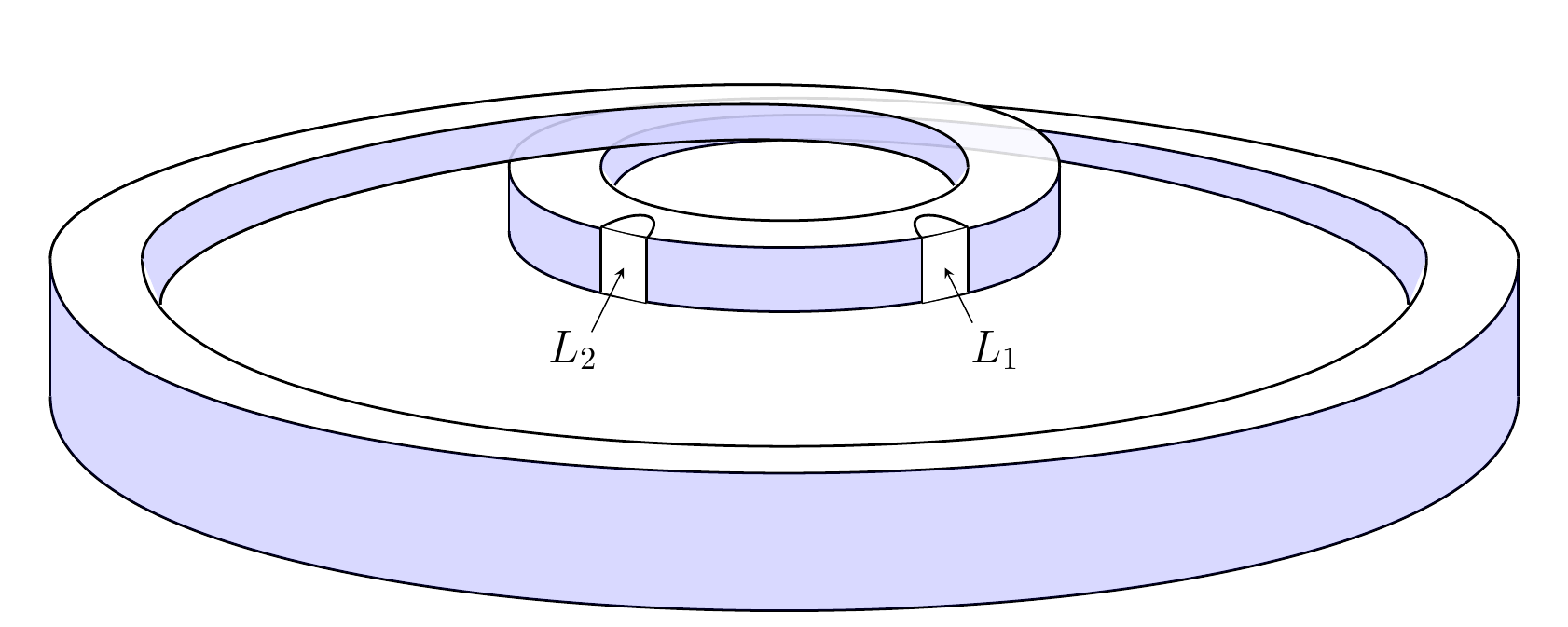}
}

\subfloat{
\includegraphics[width=0.8\linewidth, trim={4.5cm 16.6cm 0cm 4.5cm}, clip]{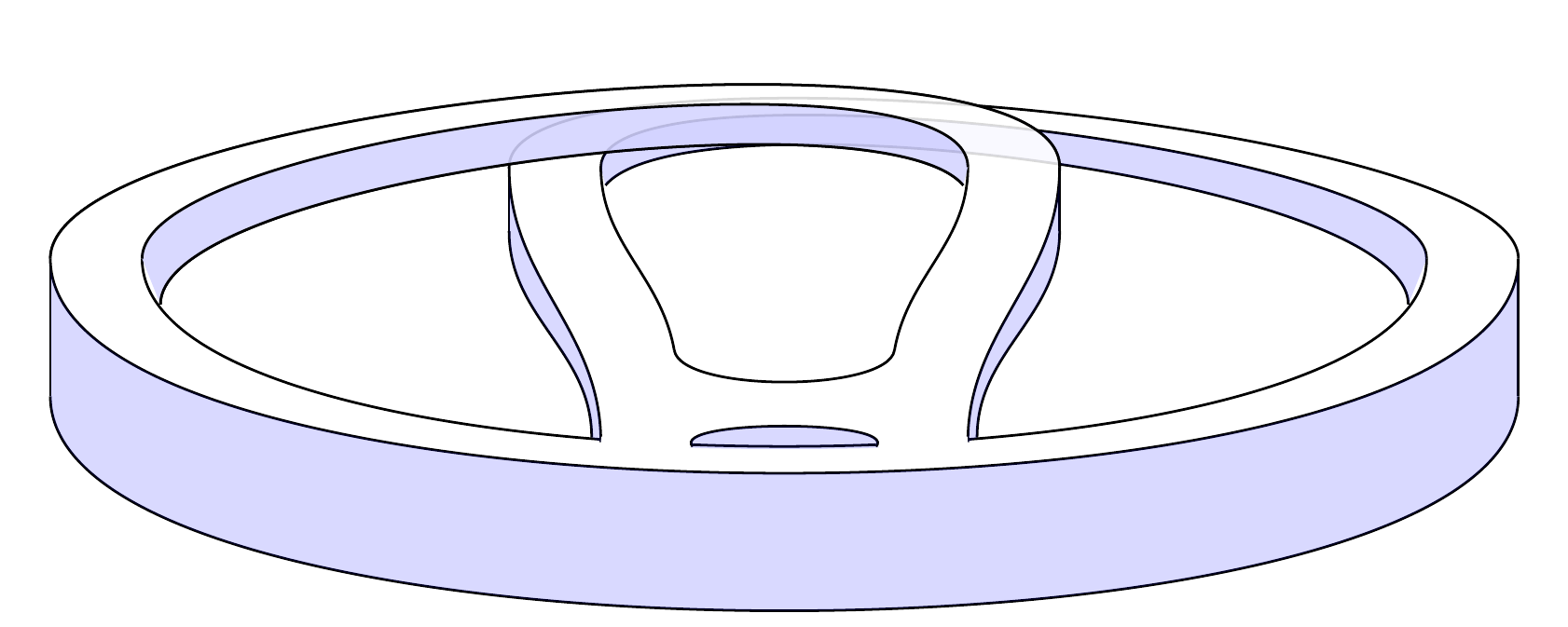}
}
\vspace{0.4cm}
\caption{The re-embedded Wilson plug with the closed regions $ L_1 $ and $ L_2 $, and the quotient Kuperberg plug $ K = \widehat{W}/\sim $}
\label{kup1}
\end{figure}

Now for each $ i = 1,2 $ we define diffeomorphisms $ \sigma_i : L_i \rightarrow W $, called \textit{insertion maps}. 
Denote $ D_i = \sigma_i(L_i) \subset W $, and let $ D_i^{\pm} = \sigma_i(L_i^{\pm}) $.
We make several assumptions about the images $ D_i $.
\begin{itemize}
\item We choose each $ D_i $ to intersect a short segment of the special orbit $ l_i $.
\item The neighborhoods $ D_i $ intersect the inside boundary $ \{r=1\} $ of the plug.
\item The regions $ L_i $ are ``twisted" under $ \sigma_i $ so that special orbits $ l_i $ \textit{enter} through $ D_i^- $ and \textit{exit} through $ D_i^+ $.
\item There is a single angle $ \alpha_i \in [0,2\pi] $ such that the vertical arc $ \{ r=2, \theta=\alpha_i, -2 \leq z \leq 2\} \subset \mathcal{R} \cap L_i $ maps onto the horizontal special orbit segment $ D_i \cap l_i $.
\end{itemize}

We will use the insertion maps to define a new plug as follows.
First we remove the images $ D_i $ of the insertion maps from $ W $, denoting $ \widehat{W} = W \setminus (D_1 \cup D_2) $.
Then, we define an equivalence relation $ \sim $ on $ \widehat{W} $ by setting $ x \sim y $ if $ x $ lies in either $ L_i^+ \cup L_i^- $ or the outside boundary $ L_i \cap \{r=3\} $, and $ y $ lies in the images of these regions under $ \sigma_i $, for both $ i =1,2 $.
The \textit{Kuperberg plug} $ K $ is the quotient $ \widehat{W}/\sim $, a manifold with boundary (See Figure \ref{kup2}).
Let $ \tau: \widehat{W} \rightarrow K $ be the quotient map.

\begin{figure}[h]
 \includegraphics[width=\linewidth, trim={4cm 17cm 0cm 4cm}, clip]{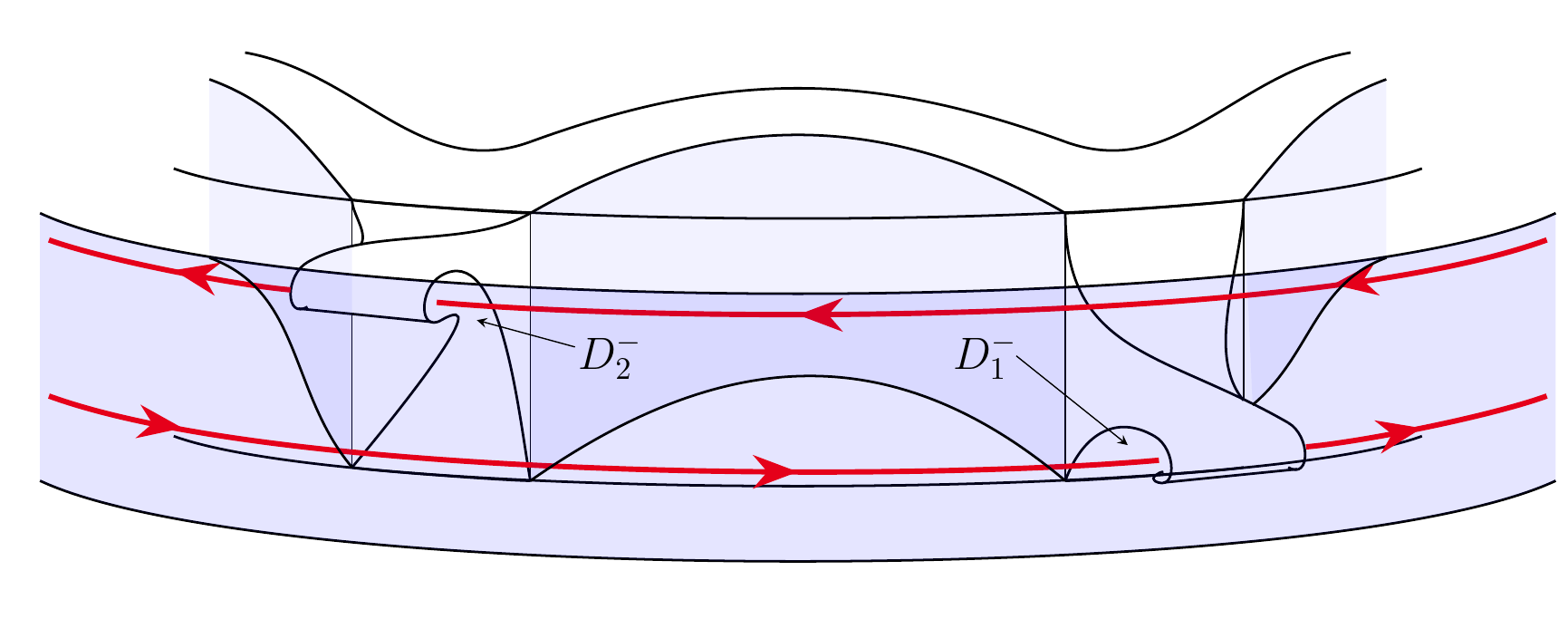}
\caption{The self-insertions defining the Kuperberg plug. The special orbits enter the bottom faces $ D_i = \sigma_i(L_i) $ where $ L_i $ are shown in Figure \ref{kup1}}
\label{kup2}
\end{figure}

The set $ \mathcal{R} \cap \{ |z| \leq 1 \} $ is the sub-cylinder of the Reeb cylinder $ \mathcal{R} $ lying between the two special orbits.
Let $ \mathcal{R}' $ be the closure of $ \mathcal{R} \cap \{ |z| \leq 1 \} \setminus \widehat{W} $.
This is the sub-cylinder with the two ``notches" $ L_i \cap \mathcal{R} $ removed.
We refer to $ \mathcal{R}' $ as the \textit{notched Reeb cylinder}.

Now, for each $ i =1,2 $, we define a rectangular region $ S_i \subset D_i^- $.
We will assume that the the radial coordinate of the inner edge of each $ S_i $ is constant $ =2 $.
Thus $ S_i \cap \mathcal{R} $ is a vertical line segment, which we denote by $ \gamma_i $, and $ S_i \cap \mathcal{R}' $ is the upper half of $ \gamma_i $, which we denote by $ \gamma_i^u $.
Further, each rectangle $ S_i $ is foliated by vertical line segments $ \{\gamma_{c,i}\}_c $, where $ \gamma_{0,i} = \gamma_i $.

We will write each $ S_i $, $ \gamma_i $, and $ \gamma_i^u $ in coordinates in Chapter \ref{InsertAssume}.
For now, we need only specify that each $ S_i $ intersects the special orbit $ l_i $, which is consistent with Kuperberg's construction outlined above.
Using this notation, there are two important assumptions we must make about the insertions $ \sigma_i $ defining $ K $. 
The first is important for proving that the dynamics inside $ K $ are aperiodic.
The second will prove to be crucial for determining properties of the minimal set.

\begin{itemize}
\item \textbf{Radius Inequality}: For $ i=1,2 $, the radial coordinate of each point in $ L_i $ is strictly greater than that of its image under $ \sigma_i $, with one exception.
That is, for points in the inverse image $ \{ r=2, \theta=\alpha_i, -2 \leq z \leq 2\} $ under $ \sigma_i $ of the special orbit $ c_i $, where the radial coordinates agree.
\item \textbf{Quadratic Insertion}: For $ i=1,2 $, the inverse image under $ \sigma_i $ of $ \gamma_i $ is a parabola with vertex $ (2,\alpha_i, -2) $.
Furthermore, the inverse image under $ \sigma_i $ of the rectangular region $ S_i $ is a ``parabolic strip" with vertex $ (2,\alpha_i, -2) $.
More precisely, the inverse image under $ \sigma_i $ of each vertical line segment $\gamma_{c,i} $ in the vertical foliation of $ S_i $ is a parabola with vertex $ (2+c, \alpha_i, -2) $.
\end{itemize}
See Figure \ref{kup2} for an illustration of the quadratic insertion property.

\begin{figure}[htp]
\subfloat{
\includegraphics[width=0.7\linewidth, trim={5.2cm 18.2cm 8cm 4cm}, clip]{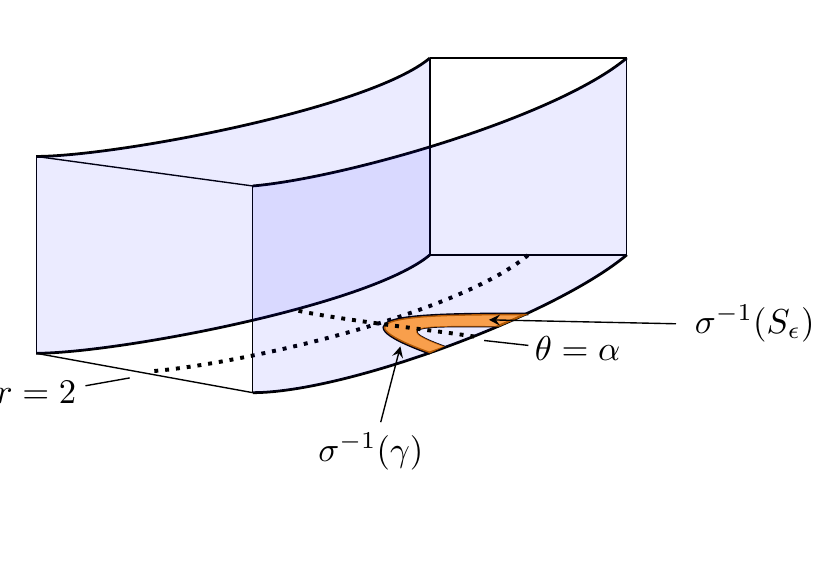}
  }
\vspace{-0.4cm}
\subfloat{
\includegraphics[width=0.8\linewidth, trim={5cm 18cm 6cm 4.5cm}, clip]{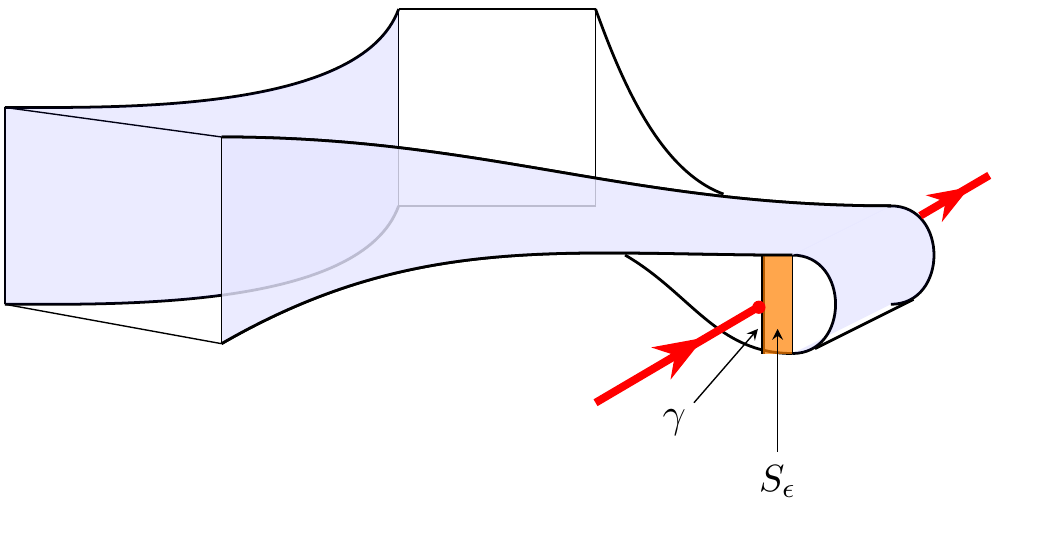}
}
\vspace{-0.5cm}
\caption{The quadratic insertion property}
\label{kup2}
\end{figure}

If a closed manifold carries the dynamics of a smooth vector field, we may insert a plug-- supporting a separate smooth vector field-- into the interior of this manifold.
Assume that the plug has the \textit{matched ends property}, and that the ends of the plug are transverse to the field on the manifold.
Then the theory of plugs and insertions developed in \cite{Wil}, \cite{Sch}, \cite{Kup1} and \cite{Kup2} show that a smooth global field on the plugged manifold, compatible with the dynamics of both the manifold and the plug, can be defined by smoothly altering the dynamics in a tubular neighborhood of the boundary of the plug.
The construction is delicate and we refer to \cite{Kup1} for the details.
By these facts, the Wilson field $ \mathcal{W} $ induces a smooth vector field $ \mathcal{K} $ on the Kuperberg plug, which we call the \textit{Kuperberg field}.
Kuperberg proved that the self-insertions defining $ K $ break the periodic orbits $ l_i $, without creating new periodic orbits.

\begin{theorem} \emph{(Theorem 4.4 from \cite{Kup1})}
\label{Kuptheorem}
The $ C^{\infty} $ vector field $ \mathcal{K} $ defined on $ K $ has no closed orbits.
\end{theorem}

Kuperberg's theorem is true under very flexible assumptions; in fact, the proof uses only the radius inequality and does not require the quadratic insertion property.
However, to determine finer aspects of the dynamics of the Kuperberg flow on its minimal set, we will need to make several more assumptions.

\subsection{Further insertion assumptions}
\label{InsertAssume}
In this chapter, we will impose more restrictive versions of the assumptions we have already made, to obtain explicit formulas for the insertion maps $ \sigma_i $ and the Wilson flow $ \phi_t $.
To simplify the exposition, we will write these formulas only for $ \sigma_1 $, the lower insertion map. 
In the following chapter, we denote by $ \sigma $, $ D^- $, $ B $, $ p $, $ \gamma $, $ \alpha $, $ S $, $ \gamma_{c} $ and $ l $ the quantities $ \sigma_1 $, $ D_1^- $, $ B_1 $, $ p_1 $, $ \gamma_1 $, $ \alpha_1 $, $ S_1 $, $ \gamma_{c, 1} $ and $ l_1 $ respectively.
Identical assumptions will be made (but not written down) for the upper insertion $ \sigma_2 $.

\subsubsection{Rectangular intersection}
First, we assume that the rectangular region $ S $ has a constant angular coordinate $ \theta = \beta $, width $ 0 < b < 1 $, and height $ 2R $ for some $ R>0 $.
Explicitly,
\begin{equation}
\label{Sstrip}
S = \{ (r,\beta, z) : 0 \leq r-2 \leq b, |z+1| \leq R \}.
\end{equation}
The upper and lower boundaries of this rectangle are
\begin{equation}
\label{Spm}
S^{\pm} = \{ (r,\beta, -1 \pm R) : 0 \leq r-2 \leq b \}
\end{equation}
Both intervals $ S^{\pm} $ can be identified with $ [0,b] $ and will be used extensively later when describing the transverse minimal set.
The inner edge of this rectangle is the intersection $ S \cap \mathcal{R} $, the vertical line $ \gamma $ we defined earlier:

\begin{equation}
\label{gammadef}
\gamma = \{ (2,\beta,z) : |z+1|\leq R \}
\end{equation}

Also, we define $ \gamma^u $ and $ \gamma^l $ to be the upper and lower half of $ \gamma $, so $ \gamma = \gamma^u \cup \gamma^l $. 
By definition of $ \mathcal{R}' $, we have $ \mathcal{R}' \cap S = \gamma^u $.
See Figure \ref{rectfig}.

\begin{align}
\label{gammauldef}
\gamma^u &= \{ (2,\beta,z) : 0 \leq z+1 \leq R\} \\
\gamma^l &= \{ (2,\beta,z) : -R \leq z+1 \leq 0\} \nonumber
\end{align}

\begin{figure}[h]
    \includegraphics[width=1.2\linewidth, trim={4.7cm 15.7cm 1cm 3.1cm}, clip]{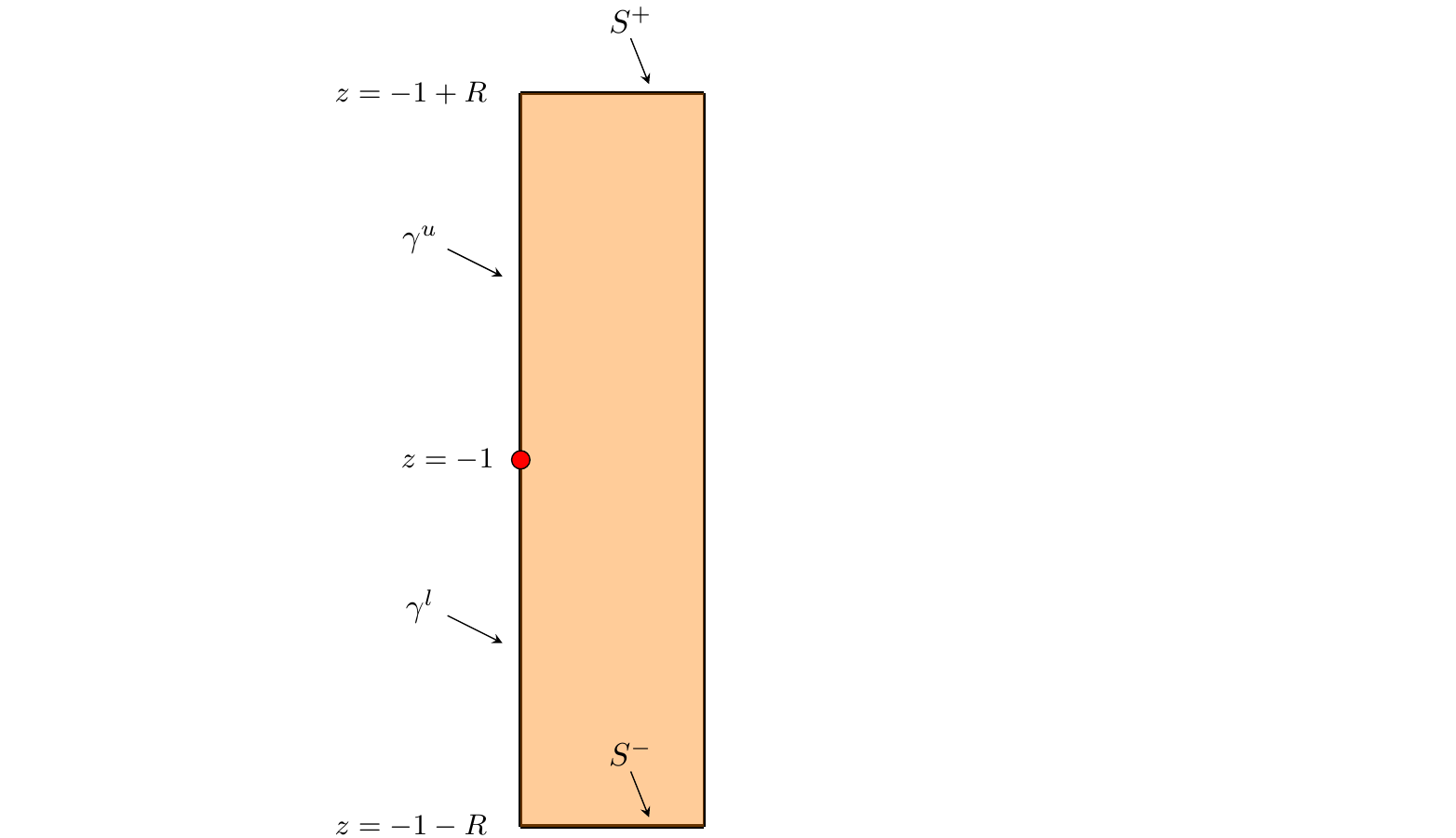}
  \caption{The rectangle $ S $, with the inner edge $ \gamma = \gamma^u \cup \gamma^l $ and the upper and lower boundaries $ S^{\pm} $.}
  \label{rectfig}
\end{figure}
 
Additionally, we assume that $ (B \times [0,2\pi]) \cap (S \times [0,2\pi]) = (S \times [0,2\pi])$.
Recall that the vertical component $ g $ (defined in Equations \ref{W} and \ref{g}) of the Wilson flow changes from $ g=1 $ to $ g=p $ precisely at $ \partial B $.
This assumption will simplify the boundary conditions that arise when integrating $ \mathcal{W} $, since the upper and lower boundaries of the critical torus $ B \times [0,2\pi] $ must now coincide with the two annuli
$$
C^{\pm} = \{(r, \theta, -1 \pm R): 0 \leq r-2 \leq b, \; 0 \leq \theta \leq 2\pi \}.
$$
The intersection of the annuli $ C^{\pm} $ with $ S $ are the upper and lower boundary intervals $ S^{\pm} $ of the strip $ S $.

\subsubsection{Quadratic decay}
Recall the function $ p $ defined in Equation \ref{p}. 
We now assume that $ p $ decays quadratically inside the critical strip $ S $.
\begin{equation}
\label{quadraticp}
p |_S(r,z) = \frac{1}{R^2}((r-2)^2+(z+1)^2)
\end{equation}
By the rectangular intersection assumption, this is compatible with the boundary condition $ p = 1 $ on $ \partial B $ from Equation \ref{p}.

\subsubsection{Quadratic insertion formula}
Recall the quadratic insertion assumption made in Chapter \ref{KupThm}. In this chapter, we will make these assumptions more specific; in particular we will write the inverse of the insertion map $ \sigma $ in coordinates.

By equation \ref{Sstrip}, any point in the rectangle $ S $ can be written as $ (2+r, \beta, -1+z) $, where $ 0 \leq r \leq b $ and $ -R \leq z \leq R $.
We will assume that $ \sigma^{-1} $ takes $ S $ to a parabolic strip in the base $ {z=-2} $, its vertex having a constant $ \theta $ coordinate of $ \alpha $, in the following way:
\begin{equation}
\label{sigmainverse}
\sigma^{-1}(2+r, \beta, -1+z) = (2+r+z^2, \alpha-z, -2)
\end{equation}
See Figure \ref{kup2}. 

In light of Equation \ref{gammadef}, we can parametrize $ \gamma $ as

\begin{equation}
\label{gamma}
\gamma: [-R,R] \rightarrow S \quad \quad \gamma(s) = (2, \beta, -1+s), \\
\end{equation}
and by Equation \ref{gammauldef}, $ \gamma^u $ and $ \gamma^l $ are parametrized as $ \gamma^u = \gamma |_{[0,R]} $ and $ \gamma^l = \gamma |_{[-R,0]} $.
Referring to equation \ref{sigmainverse}, we can parametrize parabolic the curve $ \sigma^{-1} \gamma $ as follows. 

\begin{equation}
\label{sigmagamma}
\sigma^{-1}\gamma(s) = (2+s^2, s+\alpha, -2)
\end{equation}

Observe that $ S = \bigcup_{0 \leq c \leq b} \gamma_c $, where
$$
\gamma_c = \{ (2+c, \beta, z) : |z+1| \leq R \}.
$$
The collection $ \{ \gamma_c \}_{0 \leq c \leq b} $ is the foliation of $ S $ by vertical lines, introduced in the statement of the quadratic insertion property from Chapter \ref{KupThm}.
We parametrize each vertical line $ \gamma_c $ as follows.
\begin{equation}
\label{gammac}
\gamma_c: [-R,R] \rightarrow S \quad \quad \gamma_c(s) = (2+c,\beta,-1+s)
\end{equation}
Equation \ref{sigmainverse} implies that for each $ c \in [0,b] $, the curve $ \sigma^{-1} \gamma_c $ is parabolic in the base $ \{z=-2\} $ of the plug, with the parametrization
\begin{equation}
\label{sigmagammac}
\sigma^{-1}\gamma_c(s) = (2+c+s^2, s+\alpha, -2).
\end{equation}
Since $ \gamma_0 = \gamma $, this parametrization is compatible with the above parametrization of $ \gamma $.

\subsection{Integrals of $ \mathcal{W} $}
Our quadratic decay assumption allows us to integrate $ \mathcal{W} $ explicitly.
At points $ (r, \theta, -2) \in \{z=-2\} $, the Wilson vector field $ \mathcal{W} $ has $ f\equiv +a $ and $ g \equiv 1 $, resulting in the simple expression

\begin{equation}
\label{wilout}
\phi_t(r, \theta, z) = (r, \theta+at, z+t) \text{ when } 0 \leq z(t) \leq -1-R 
\end{equation}

A flowline looks like the first case in Table \ref{wilsontable}, a helix rising with constant vertical speed $ \frac{2\pi}{a} $.
The upper bound on $ z $ in Equation \ref{wilout} is the point at which the orbit intersects the lower annulus $ C^- $.
At this point, we have $ g = p $ by our rectangular intersection assumption, and use Equation \ref{quadraticp} to integrate $ \mathcal{W} $.

\begin{multline}
\label{wilin}
\phi_t(r, \theta, z) = \left( r, \theta+At, -1+(r-2) \tan \left( \frac{r-2}{R^2}t+\tan^{-1}\left(\frac{z+1}{r-2}\right)\right)\right) \\
\text{ when } |z(t)+1| \leq R.
\end{multline}

In this region, a flowline looks like the second case in Table \ref{wilsontable}, a helix rising at a variable speed depending on its radial proximity to the Reeb cylinder $ \{ r=2 \} $ and its vertical proximity to $ z=-1 $.

\subsection{Transition and level}
\label{levo}
Let $ \psi_t $ be the flow of the Kuperberg vector field.
Flowlines of $ \psi_t $ are very complicated and do not admit a classification as simple as those of the Wilson flow given in Table \ref{wilsontable}.
However, since the $ K $ is a quotient of $ W $, the dynamics of $ \psi_t $ resemble the dynamics of $ \phi_t $.
To see this resemblance, we begin by embedding $ K $ in $ \mathbb{R}^3 $ as we did $ W $ in Figure \ref{plug1}, suppressing the more complicated embedding as in Figure \ref{kup1}, but retaining the interior self-insertions defining $ K $.
See Figure \ref{kupwil} for this embedding.

\begin{figure}[h]
\includegraphics[width=1.1\linewidth, trim={3cm 13.5cm 2cm 2.6cm}, clip]{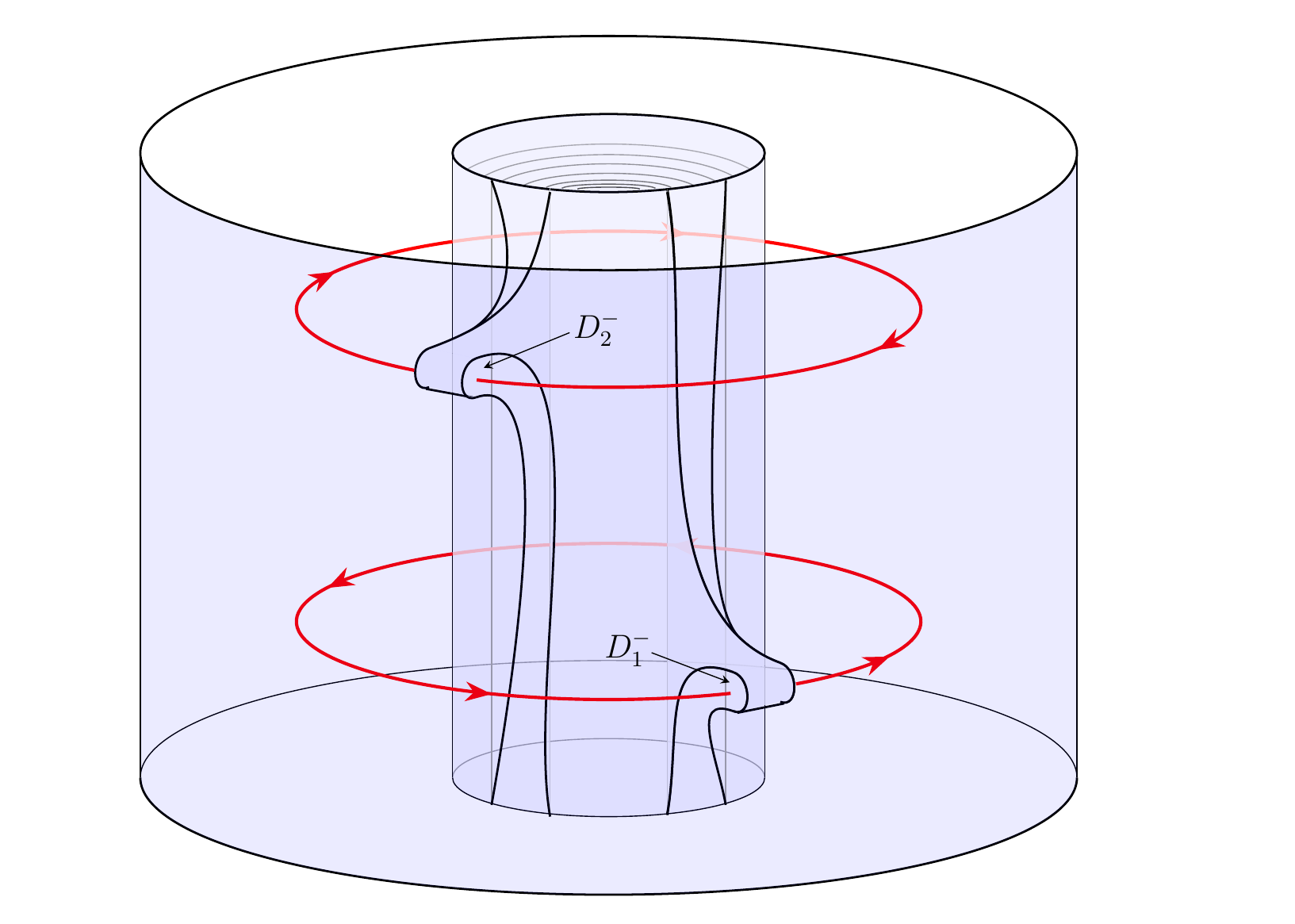}
\caption{The Kuperberg plug embedded as the Wilson plug.}
\label{kupwil}
\end{figure}

Each orbit of the Kuperberg flow $ \psi_t $ contains \textit{transition points}.
These are intersections of the orbit with an insertion region.
Between these transition points, the flowline coincides with one of the flowlines of the Wilson flow $ \phi_t $.
The hierarchy of \textit{levels} will be used to keep track of these transition points. 
By studying levels and the dynamics of the Wilson flow, we can understand the dynamics of the Kuperberg flow.

\subsubsection{Transition points and the level function for orbits}
\begin{definition}[\textit{Orbit segments and orbits}]
\label{orbpoint}
For any $ x \in K $, we denote its \textit{closed orbit segment} for time $ t_2-t_1>0 $ by 
$$ 
\mathcal{O}[x,t_1, t_2] = \bigcup_{t_1 \leq t \leq t_2} \psi_t(x).
$$
Its \textit{open orbit segment} is
$$ 
\mathcal{O}(x,t_1, t_2) = \bigcup_{t_1 < t < t_2} \psi_t(x),
$$
and its \textit{half-open orbit segment} is
$$
\mathcal{O}(x,t_1, t_2] = \bigcup_{t_1 < t \leq t_2} \psi_t(x).
$$
Its \textit{orbit} $ \mathcal{O}(x) $, \textit{forward orbit} $ \mathcal{O}^+(x) $, and \textit{backward orbit} $ \mathcal{O}^-(x) $ are 
$$ 
\mathcal{O}(x) = \bigcup_{-\infty < t < \infty} \psi_t(x), \qquad
\mathcal{O}^+(x) = \bigcup_{t \geq 0} \psi_t(x), \qquad
\mathcal{O}^-(x) = \bigcup_{t \leq 0} \psi_t(x).
$$
\end{definition}

Depending on the location of $ x $ in the plug, its orbit $ \mathcal{O}(x) $ may be finite or infinite (see Table \ref{wilsontable}).
An orbit's intersection with the bottom $ \{z=-2\} $, the top $ \{z=+2\} $, or either of the four insertion faces $ D_i^{\pm} $ ($i=1,2$), is called a \textit{transition point}.
There are four types of transition points.

\begin{itemize}
\item \textit{primary entry points} are transition points in $ \{z=-2\} $.
\item \textit{primary exit points} are transition points in $ \{z=2\} $.
\item \textit{secondary entry points} are transition points in $ D_i^+ $ for $ i=1,2 $.
\item \textit{secondary exit points} are transition points in $ D_i^- $ for $ i=1,2 $.
\end{itemize}

For each $ x \in K $, there is a natural orbit decomposition
\begin{equation}
\label{orbitdecomp}
\mathcal{O}(x) = \bigcup_{i \in I} \mathcal{O}(x, t_i, t_{i+1}],
\end{equation}
into disjoint half-open orbit segments, where for all $ i \in I $, $ \psi_{t_i}(x) $ is a transition point and $ \mathcal{O}(x, t_i, t_{i+1}) $ contains no interior transition points. 
The indexing set $ I $ is countable if $ x $ has an infinite orbit, and is finite if the orbit is.
The \textit{level function} $ n_x(t) $ along the orbit of $ x $ indexes how many insertions an orbit has passed through at time $ t $, measured from zero.

\begin{definition}[\textit{Level function along orbits}]
\label{leveldef}
Let $ x \in K $, let $ n_x^+(t) $ be the number of secondary entry points in $ \mathcal{O}(x,0,t] $, and let $ n_x^-(t) $ be the number of secondary exit points in $ \mathcal{O}(x,0,t] $. 
Define the \textit{level function} $ n_x : \mathcal{O}(x) \rightarrow \mathbb{N} $ by $ n_x(\psi_t(x)) = n_x^+(t) - n_x^-(t) $.
\end{definition}

For a fixed $ x \in K $, we say that $ y \in \mathcal{O}(x) $ \textit{has level k} if $ y=\psi_T(x) $ with $ n_x(T) = k $.
The following lemma appears in \cite{Ghy} (Lemme, pg. 300) and is formulated more precisely in Lemma 6.5 of \cite{Hur}; the only secondary entrance points that are trapped have a radial coordinate $=2$; the rest escape the insertion in finite time.

\begin{lemma}
\label{escapeorbit}
Suppose $ x $ has a radial coordinate $>2 $, and the orbit $ \mathcal{O}(x) $ contains a secondary entrance point $ \psi_T(x) $ for some $ T>0 $. Then there exists $ S>T $ such that $ \psi_S(x) $ is a secondary exit point, $ \psi_T(x) $ and $ \psi_S(x) $ are facing, and $ n_x(T) = n_x(S) $.
\end{lemma}

The next lemma appears in various forms in the literature (Proposition 4.1 of \cite{Kup1}, Lemma 5.1 of \cite{Hur}, and Lemma 7.1 of \cite{Ghy}) and is crucial in relating orbits of the Kuperberg flow to orbits of the Wilson flow.
Recall that $ \tau: \widehat{W} \rightarrow K $ is the quotient map defining the Kuperberg plug.

\begin{lemma}[\textit{short-cut lemma}]
\label{shortcutorbit}
Suppose that a secondary entrance point $ x_- \in D_i^- $ and a secondary exit point $ x_+ \in D_i^+ $ are facing.
Then there exists a point $ y_- $ in the base $ \{z=-2\} \subset W $ and $ y^+ $ in the top $ \{z=2\} \subset W $ of the Wilson plug such that $ \tau(y_{\pm}) = x_{\pm} $, and a finite time $ T>0 $ such that $ y_+ = \phi_T(y_-) $.
\end{lemma}

In this way, the dynamics of a Kuperberg orbit segment between secondary entrance and exit points reduces to the dynamics of a finite Wilson orbit from the base to the top of the plug.

Finally, for orbits of curves we have an analogous definition to that of Definition \ref{orbpoint}.

\begin{definition}[\textit{Orbit strips and surfaces}]
\label{orbsurf}
For any $ \eta $ be a curve with image in $ K $. 
Its \textit{closed orbit strip} for time $ t_2-t_1>0 $ is 
$$ 
\mathcal{O}[\eta,t_1, t_2] = \bigcup_{t_1 \leq t \leq t_2} \psi_t(\eta).
$$
Its \textit{open orbit strip} is
$$ 
\mathcal{O}(\eta,t_1, t_2) = \bigcup_{t_1 < t < t_2} \psi_t(\eta),
$$
and its \textit{half-open orbit strip} is
$$
\mathcal{O}(\eta,t_1, t_2] = \bigcup_{t_1 < t \leq t_2} \psi_t(\eta).
$$
Its \textit{orbit surface} $ \mathcal{O}(\eta) $, \textit{forward orbit surface} $ \mathcal{O}^+(\eta) $, and \textit{backward orbit surface} $ \mathcal{O}^-(\eta) $ are 
$$ 
\mathcal{O}(\eta) = \bigcup_{-\infty < t < \infty} \psi_t(\eta), \qquad
\mathcal{O}^+(\eta) = \bigcup_{t \geq 0} \psi_t(\eta), \qquad
\mathcal{O}^-(\eta) = \bigcup_{t \leq 0} \psi_t(\eta).
$$
\end{definition}

As we will see in Section \ref{Kupmin}, the minimal set of the Kuperberg flow is the closure of a union of \textit{propellers}, and each propeller is an orbit surface in this sense.

\vfill
\eject

\section{The Kuperberg pseudogroup}
\label{Kuppseudo}
In Chapter \ref{Wilpseudo} we introduced the Wilson pseudogroup generated by $ \Phi $, the first-return map of the Wilson flow to a tranverse section.
In this chapter, we will study the Kuperberg pseudogroup $ \Psi $, defined in the same way using the Kuperberg flow.
The domains of the generators of the pseudogroup we define will be subsets of the two transverse rectangles $ S_i $, $ i=1,2 $, defined in Chapter \ref{InsertAssume}.

\begin{figure}[h!]
\includegraphics[width=1.1\linewidth, trim={3cm 13.5cm 2cm 2.6cm}, clip]{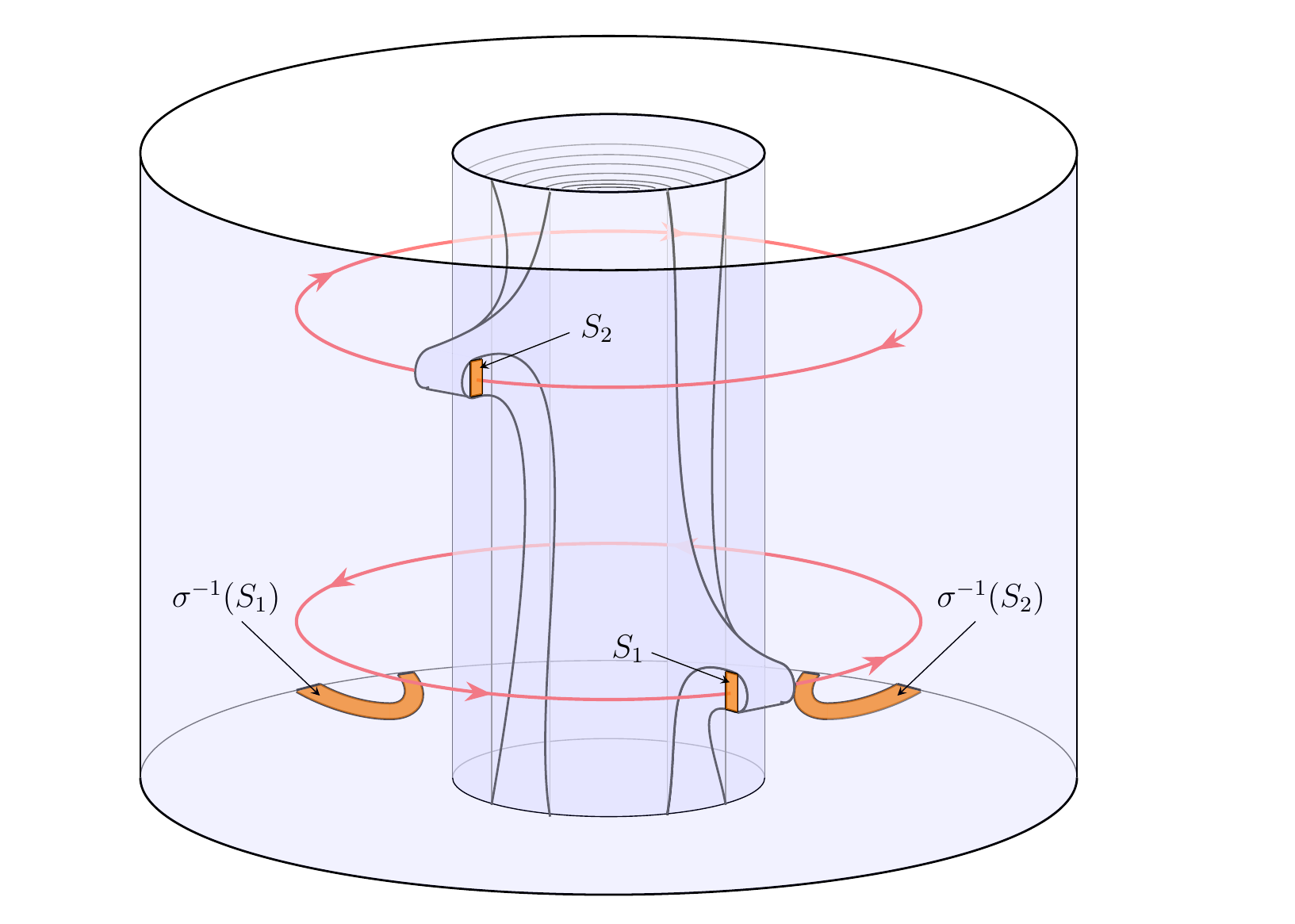}
\caption{The rectangular regions $ S_i \subset D_i^- $ and their inverse images $ \sigma^{-1}(S_i) \subset \{z=-2\} $. Compare with Figures \ref{kup2} and \ref{kupwil}.}
\label{kup3}
\end{figure}

In Chapter \ref{Kupsection}, we showed that every orbit decomposes into segments whose endpoints are transition points, having no interior transition points.
At a secondary transition point, the orbit intersects an insertion region $ D_i^{\pm} $, which is identified via $ \sigma^{-1} $ with $ L_i^{\pm} \subset \{ z = \pm 2 \} $ in the base or the top of the plug.
The dynamics of the orbit changes drastically at transition points, and these dynamics are determined by $ \sigma^{-1} $.
The interior of the orbit segment follows the helical Wilson flow $ \phi_t $ studied in Chapter \ref{Wilsflow}.

The transverse rectangles $ S_i $ lie in $ D_i^- $, so the Kuperberg first-return of a point to $ S_i $ is a secondary entrance point, by definition.
This first-return map follows the Wilson flow.
At the transition point, it is mapped via $ \sigma^{-1} $ into a parabolic strip in the base of the plug, which then follows the Wilson flow up to more intersections with $ S_i $.
So the Kuperberg pseudogroup $ \Psi $ of first-return maps to the rectangles $ S_1 \cup S_2 $ is generated by the Wilson pseudogroup to $ S_i $ from $ S_i $ or the base, and the insertion maps from $ S_i $ to the base, for $ i=1,2 $.
In this section, we will construct these generators for $ \Psi $.

In Chapter 9 of \cite{Hur}, the full Kuperberg pseudogroup to a larger transverse section was studied. 
This pseudogroup is very complicated, and in subsequent chapters of \cite{Hur} its properties were used to study the dynamics of the Kuperberg flow on the entire plug $ K $. 
In this paper we are concerned only with the dynamics of the Kuperberg flow in small neighborhoods of the special orbits $ l_i $, which is why we choose the sections $ S_i $.
The pseudogroup $ \Psi $ we consider is a restriction of the full pseudogroup studied in \cite{Hur}.

In the second part of this chapter, we explore the symbolic dynamics of the $ \Psi $ on an orbit.
For simplicity, we focus on the lower rectangle $ S_1 $, by considering a suitable sub-pseudogroup of $ \Psi $.
For any $ x \in K $, the intersection $ \mathcal{O}(x) \cap S $
is a sequence of points ordered by the flow direction, on which the Kuperberg pseudogroup acts faithfully.
Using the notion of \textit{level} introduced Chapter \ref{Kupsection}, we will decompose this intersection into level sets, and show that the pseudogroup generators permute this level decomposition.
Finally, we will construct a sequence space $ \Sigma \subset \mathbb{N}^{\mathbb{N}} $ and a bijective coding map $ \Sigma \rightarrow \mathcal{O}(x) \cap S $, and study the induced dynamics of the pseudogroup on this space.
This is the symbolic dynamics of the Kuperberg pseudogroup, which will be instrumental later when studying the minimal set.

\subsection{Generators of the pseudogroup}
Recall the rectangular regions $ S_i \subset D_i^- $ defined in Equation \ref{Sstrip} of Chapter \ref{InsertAssume}.
In the quotient $ K $, these regions are identified with the parabolic regions $ \sigma^{-1}(S_i) \subset \{z=-2\} $ in the base of the plug.
See Figure \ref{kup3}.

We now list the generators of the Kuperberg pseudogroup restricted to the rectangles $ S_1 \cup S_2 $.

\begin{figure}[h!]
\includegraphics[width=1.1\linewidth, trim={3cm 13.5cm 2cm 2.6cm}, clip]{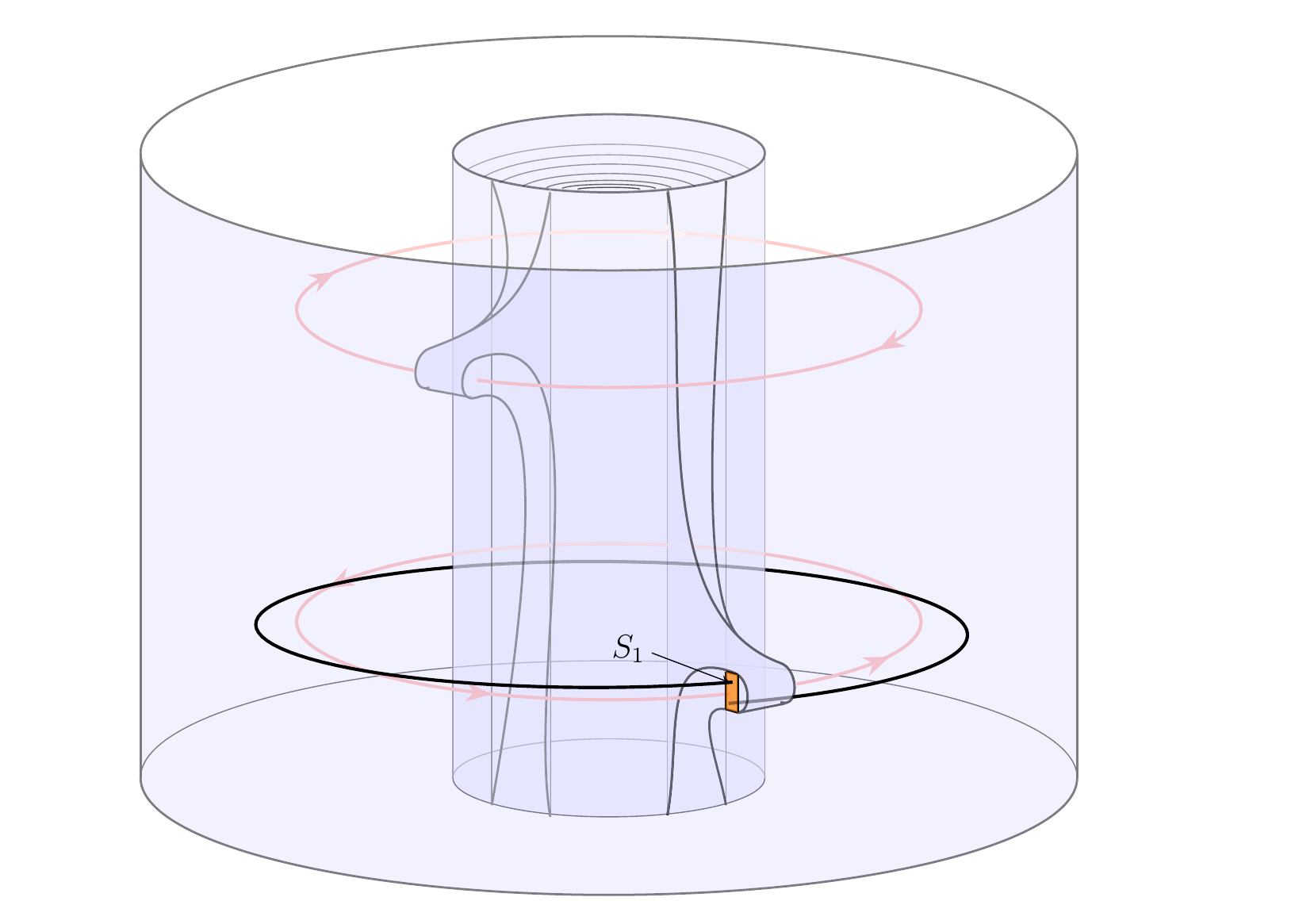}
\caption{The map $ \Phi_1 $ follows the Wilson flow $ \phi_t $ through the lower insertion region $ D_1 $ and around to its first return to $ S_1 $.
There is a similar picture for $ \Phi_2 $.}
\label{kup4}
\end{figure}

\subsubsection{The Wilson maps $ \Phi_i : D_{\Phi_i} \rightarrow R_{\Phi_i} $}
Consider a point $ x \in S_i \subset D_i^- $ for $ i = 1,2 $.
We assume that $ x $ is not the intersection point of the special orbit $ l_i $ with $ S_i $, i.e. $ x \neq (2, \beta_i, \pm 1) $.
We define $ \Phi_i(x) $ as the first return to $ S_i $ under the Wilson flow $ \phi_t $.
Explicitly, $ \Phi_i(x) = \phi_T(x) $,
where $ T>0 $, $ \phi_T(x) \in S_i $, and $ T $ is minimal with respect to these properties.

In the Kuperberg plug, $ x $ is identified with $ \sigma^{-1}(x) $ in the base. 
By the assumption that $ x $ is not the intersection point of $ l_i $ with $ S_i $, we know by the radius inequality that $ \sigma^{-1}(x) $ has radius $ >2 $.
Applying the short-cut lemma (Lemma \ref{shortcutorbit}), there exists a facing point $ x' \in D_i^+ $, and the flow from $ x $ to $ x' $ is a finite union of Wilson flow segments.
From $ x' $, the orbit follows the Wilson flow around the plug and back to $ \Phi_i(x) $, its first-return to $ S_i $.
See Figure \ref{kup4} for an illustration of $ \Phi_1 $.

\begin{figure}[h!]
    \includegraphics[width=1.2\linewidth, trim={4.7cm 13.5cm 1cm 3.1cm}, clip]{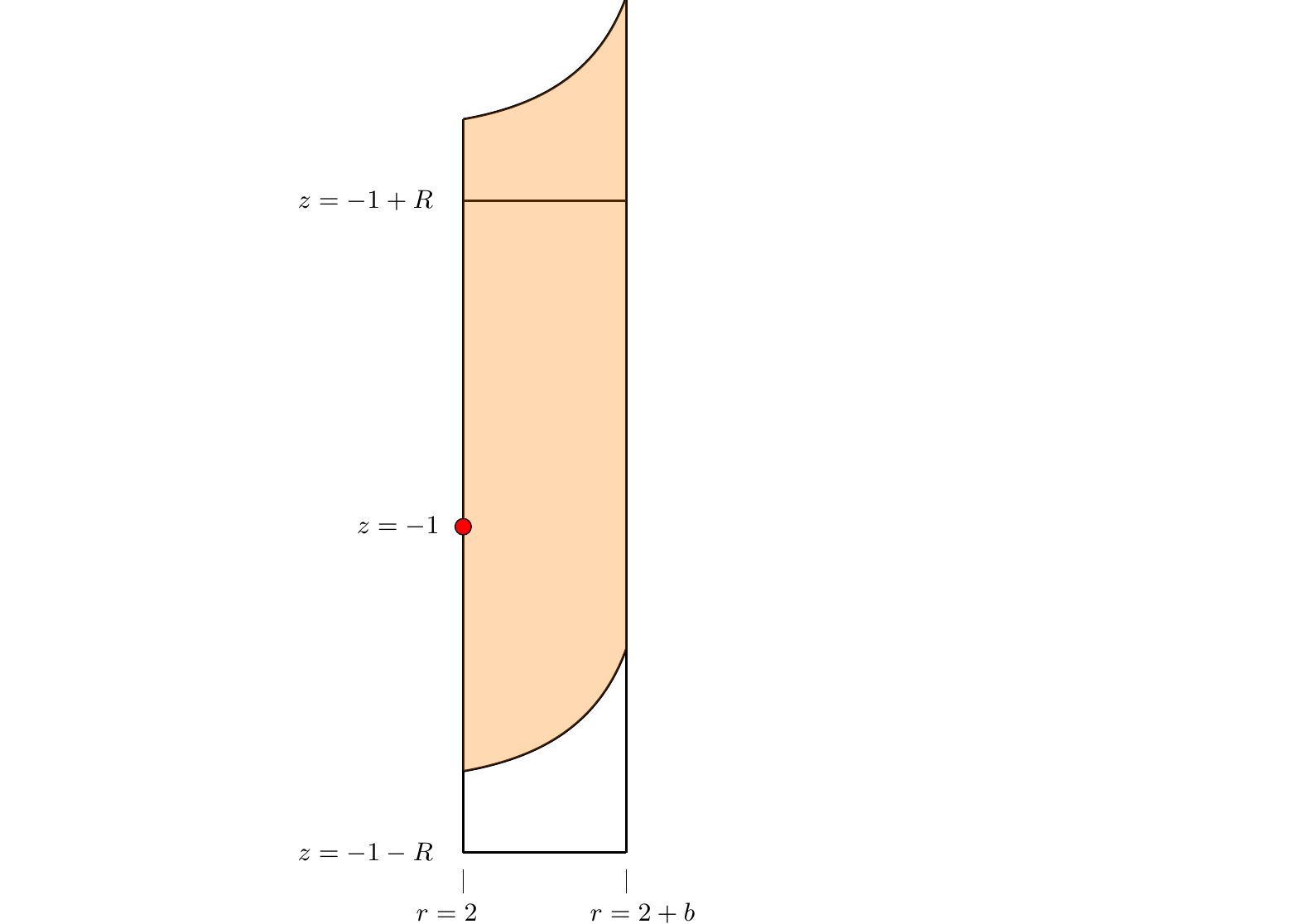}
  \caption{The image $ R_{\Phi_i} = \Phi_i(D_{\Phi_i}) $ is the orange region. The points $ x \in S_i \setminus D_{\Phi_i} $ are near the top of $ S_i $; the Wilson flow of these points does not return to $ S_i $.}
\label{kup5}
\end{figure}

Note that $ \Phi_i(x) $ is not defined for all $ x \in S_i $.
This is because the Wilson flow of a point has a monotonically increasing $ z $-coordinate, so there are points near the top of $ S_i $ that never return to $ S_i $ under $ \Phi_i $.
However, there is a subset $ D_{\Phi_i} \subset S_i $ of points $ x $ for which $ \Phi_i(x) $ is defined.
Denote the image by $ R_{\Phi_i} = \Phi_i(D_{\Phi_i}) \subset S_i $ (see Figure \ref{kup5}).

\begin{figure}[h!]
\includegraphics[width=1.1\linewidth, trim={3cm 13.5cm 2cm 2.6cm}, clip]{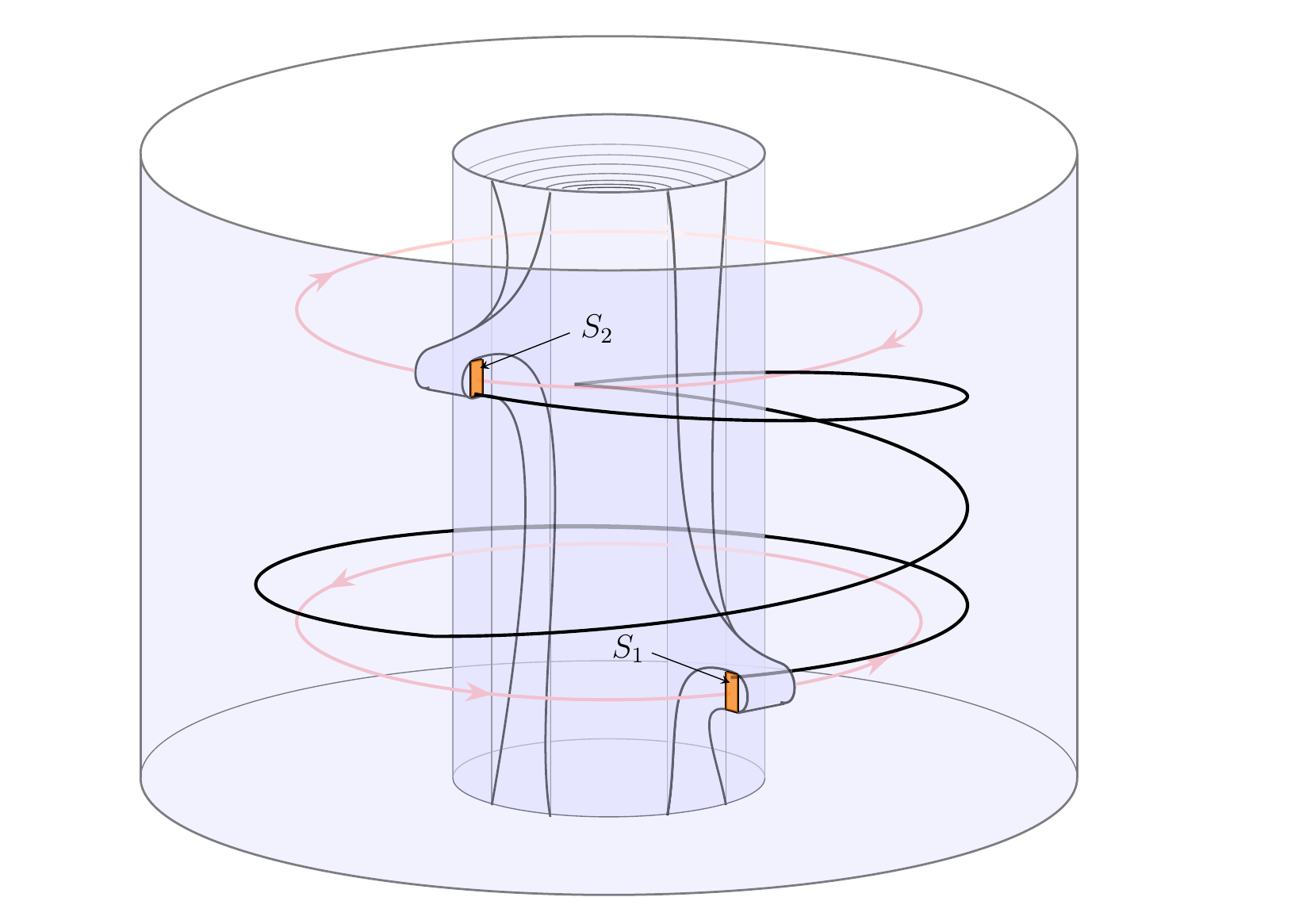}
\caption{The map $ \Phi_{1,2} $ follows the Wilson flow $ \phi_t $ from the upper region of $ S_1 $, through the insertion region $ D_1 $ and up to its first return to $ S_2 $.}
\label{kup6}
\end{figure}

\subsubsection{The Wilson map $ \Phi_{1,2} : D_{\Phi_{1,2}} \rightarrow R_{\Phi_{1,2}} $}
As discussed in the previous paragraph, there is a set of points near the upper boundary of $ S_1 $ that do not return to $ S_1 $ under the Wilson flow.
However, the Wilson flow of these points does intersect the upper rectangle $ S_2 $.
This defines a map $ \Phi_{1,2} : D_{\Phi_{1,2}} \rightarrow R_{\Phi_{1,2}} $ given by $ \Phi_{1,2}(x) = \phi_T(x) $, where $ T>0 $, $ \phi_T(x) \in S_2 $, and $ T $ is minimal with respect to these properties (See Figure \ref{kup6}).

\begin{figure}[h!]
\includegraphics[width=1.1\linewidth, trim={3cm 13.5cm 2cm 2.6cm}, clip]{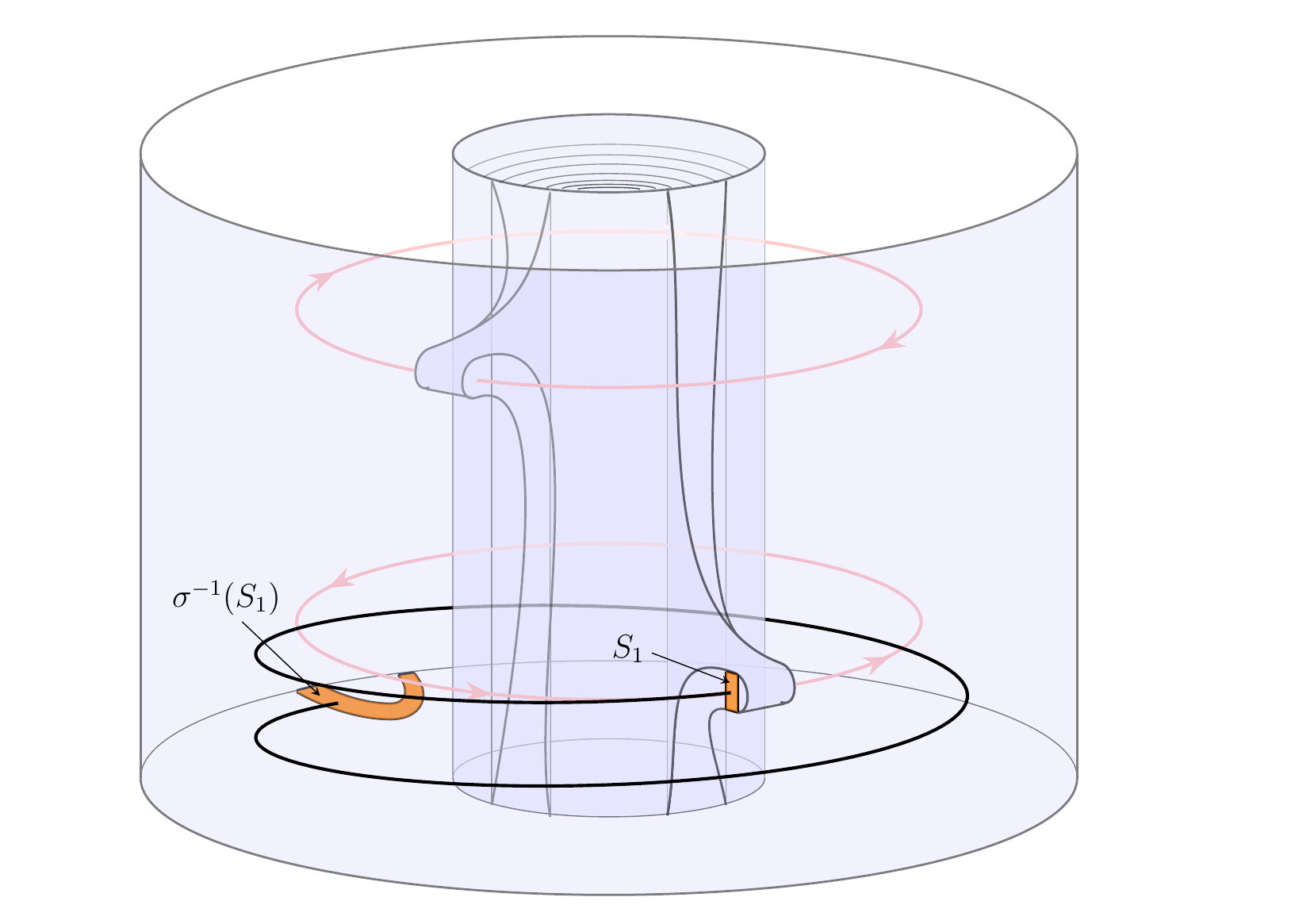}
\caption{The map $ \Theta_1 $ follows the Wilson flow from $ \sigma^{-1}(S_1) $ in the base $ \{ z=-2 \} $ up to its first return to $ S_1 $. 
There is a similar picture for $ \Theta_2 $.}
\label{kup7}
\end{figure}

\subsubsection{The insertion maps $ \Theta_i : D_{\Theta_i} \rightarrow R_{\Theta_i} $}
In the Kuperberg plug $ K $, the quotient map $ \tau $ identifies each rectangle $ S_i $ with the parabolic strip $ \sigma^{-1}(S_i) $ in the base (See Figure \ref{kup3}).
Thus an orbit that intersects a rectangle $ S_i $ is identified via $ \sigma^{-1} $ with the base, after which the orbit follows the Wilson flow back up to the lower rectangle $ S_1 $.
See Figure \ref{kup7} for an illustration of this.
We will now define a map $ \Theta_i : D_{\Theta_i} \rightarrow R_{\Theta_i} $, where $ D_{\Theta_i} \subset S_i $, and $ R_{\Theta_i} \subset S_1 $.

\begin{figure}[h!]
    \includegraphics[width=1.2\linewidth, trim={4.7cm 15.7cm 1cm 3.1cm}, clip]{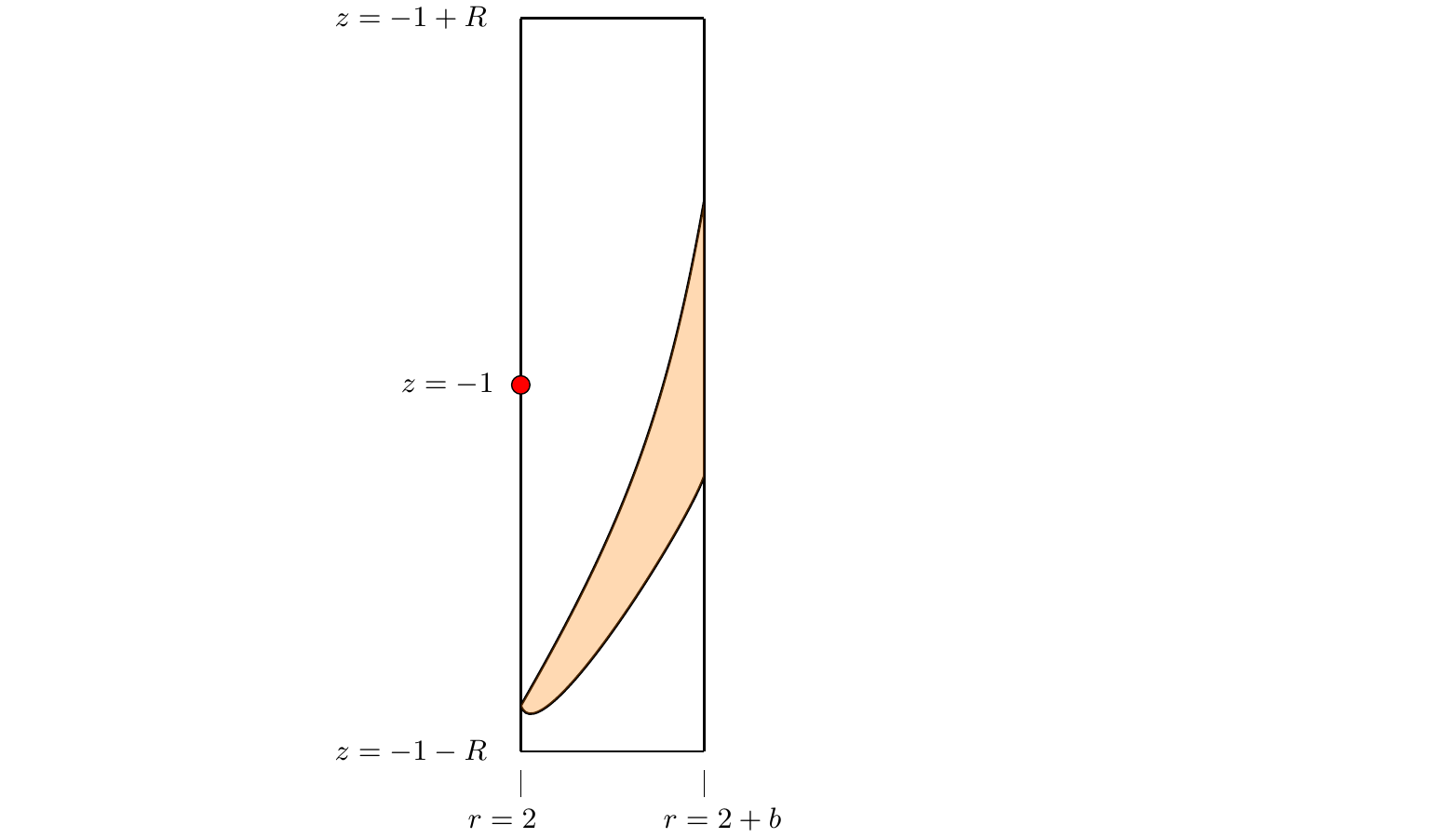}
  \caption{The tip of the twisted parabolic strip $ \Theta_i(S_i) $ in $ S_i $. The vertex $ \phi_{T_k}(x_i) $ of the strip lies in the Reeb cylinder $ \{ r=2 \} $.}
  \label{kup8}
\end{figure}

The vertex of a parabolic strip $ \sigma^{-1}(S_i) $ is $ x_i = (2, \alpha_i, -2) $ (See Chapter \ref{InsertAssume} and Figure \ref{kup2}).
Because the radial coordinate of $ x_i $ is $ 2 $, its Wilson orbit $ \phi_t(x_i) $ is trapped in $ K $.
The orbit $ \phi_t(\sigma^{-1}(S_i)) $ of the entire parabolic strip $ \sigma^{-1}(S_i) $ intersects the lower rectangle $ S_1 $ at a sequence of times $ T_j $, $ j=1,2,\ldots $.
These intersections are ``twisted" parabolic strips, resembling the propellers' cross-sections in Figure \ref{prop}.
The vertices of these parabolic regions are the intersections of the orbit of $ x_i $ with $ S_1 $, which is the ordered sequence of points $ \phi_{T_j}(x_i) $ limiting on the special orbit intersection, whose $ z $-coordinate monotonically increases with $ j $.
Because $ S $ contains the special orbit intersection $ (2, \beta, -1) $, there is a critical time $ T_k $ such that this sequence of points remains in $ S_i $ for all $ j \geq k $.
In other words, let $ k $ be the minimal value of $ j $ such that $ \phi_{T_j}(x_i) \in S_i $ for all $ j \geq k $.
In terms of this fixed $ k $, define 
$$ 
\Theta_i(x) = \phi_{T_k} (\sigma^{-1}(x)).  
$$
Define $ D_{\Theta_i} \subset S_i $ the set of points $ x $ for which the above equation is defined, and define the image $ R_{\Theta_i} = \Theta_i(D_{\Theta_i}) \subset S_1 $ (See Figure \ref{kup8}).

\subsection{Restriction to a sub-pseudogroup}
To summarize, we have constructed the Kuperberg pseudogroup $ \Psi $ on $ S_1 \cup S_2 $, generated by five elements:
\begin{equation}
\label{Psi}
\Psi = \langle \Phi_1, \Phi_2, \Phi_{1,2}, \Theta_1, \Theta_2 \rangle
\end{equation}
The dynamics of the full pseudogroup $ \Psi $ are complicated. 
To simplify the study, we will consider the sub-pseudogroup generated by the two maps $ \Phi_1 , \Theta_1 : S_1 \rightarrow S_1 $.
To save on notation, we denote $ \Phi = \Phi_1 $ and $ \Theta = \Theta_1 $. 
In terms of these, we define
\begin{equation}
\label{Psi1}
\Psi_1 = \langle \Phi, \Theta \rangle.
\end{equation}

Following the shorthand used in Section \ref{InsertAssume}, we will refer to $ S_1 $, $ D_1 $, $ \sigma_1 $, $ l_1 $, $ \gamma_1 $, $ \gamma_{c,1} $, and $ \beta_1 $ simply by $ S $, $ D $, $ \sigma $, $ l $, $ \gamma $, $ \gamma_c $, and $ \beta $, respectively.
These conventions will be observed for the remainder of this chapter, and throughout Chapters \ref{Kupmin} -- \ref{FunctionSystems}.
We will return to the dynamics of the full pseudogroup in Chapter \ref{Transcant} when we discuss \textit{interlacing}.

\subsection{Orbit intersections with a transversal}
In this and subsequent sections we will introduce the intersection of a forward Kuperberg orbit $ \mathcal{O}^+(x) $ with $ S $.
We then study its level decomposition and the symbolic dynamics of the pseudogroup action on this intersection.

Let $ x \in S $ be any point \textit{other than} the intersection $ (2,\beta,-1) $ of $ S $ with the special orbit $ l $, and consider its forward orbit $ \mathcal{O}^+(x) $ in $ K $.
In Section \ref{levo}, we introduced the level function along an orbit.
This induces a level decomposition of the intersection $ \mathcal{O}^+(x) \cap S $.

\begin{equation}
\label{orbint1}
\mathcal{O}^+(x) \cap S = \bigcup_k \mathcal{O}^+(x)_k \cap S, \; \text{ where } \; \mathcal{O}^+(x)_k \cap S = \{ y \in \mathcal{O}^+(x) \cap S : n_x(y) = k\}.
\end{equation}

\subsubsection{The pseudogroup action}
The next two lemmas show that $ \Phi $ preserves level, while $ \Theta $ increases level by one.

\begin{lemma}
\label{Phipermute}
If $ x \in S $ is not in the special orbit, the map $ \Phi $ restricted to $ \mathcal{O}^+(x) \cap S $ maps  $ \mathcal{O}^+(x)_k \cap S $ into $ \mathcal{O}^+(x)_k \cap S $.
\end{lemma}

\begin{proof} Let $ y \in \mathcal{O}^+(x)_k \cap S $, so there exists $ T>0 $ such that $ y =\psi_T(x) $ and $ n_x(T) = k $. 
Since $ S \subset D^- $, $ y $ is necessarily a secondary entrance point.
In the quotient plug $ K $, $ y $ is identified with $ \sigma^{-1}(y) $.
Because we assumed that $ \mathcal{O}^+(x) $ is not the special orbit, $ y $ is not in the special orbit's intersection with $ S $.
Applying the radius inequality (Section \ref{Kupsection}), we see that $ y $ has a radial coordinate $ >2 $, and so we may apply Lemma \ref{escapeorbit} to say there exists $ S > T $ with $ \psi_S(x) \in D^+ $ a secondary exit point, $ y $ and $ \psi_S(x) $ are facing, and $ n_x(T) = n_x(S) = k $.
The orbit of $ \psi_S(x) $ now follows the Wilson flow forward to its first return to $ S $ (if it exists), which by definition is $ \Phi(y) $ (if it is defined).
Because the Wilson flow preserves the level, $ \Phi(y) $ has level $ k $, so $ \Phi(y) \in \mathcal{O}^+(x)_k \cap S $, which concludes the proof.
\end{proof}

\begin{lemma}
\label{Thetapermute}
If $ x \in S $ is not in the special orbit, the map $ \Theta $ restricted to $ \mathcal{O}^+(x) \cap S $ maps  $ \mathcal{O}^+(x)_k \cap S $ into $ \mathcal{O}^+(x)_{k+1} \cap S $.
\end{lemma}

\begin{proof} Let $ y \in \mathcal{O}^+(x)_k \cap S $, so there exists $ T>0 $ such that $ y =\psi_T(x) $ and $ n_x(T) = k $. 
As in Lemma \ref{Phipermute}, $ y $ is a secondary entrance point.
If it exists, $ \Theta(y) $ is the first return of $ \sigma^{-1}(y) $ to $ S $.
Notice that points on the forward orbit of $ \sigma^{-1}(y) $ have level $ k $. 
Because $ S \subset D^- $, $ \Theta(y) $ is a secondary entrance point and thus has level $ k+1 $.
This shows that $ \Theta(y) \in \mathcal{O}^+(x)_{k+1} \cap S $, which concludes the proof.
\end{proof}

These lemmas demonstrate that the pseudogroup $ \langle \Phi, \Theta \rangle $ acts faithfully on the intersection $ \mathcal{O}^+(x) \cap S $ by permuting the level.

\subsection{Symbolic dynamics of orbits}
\label{symborbits}
In this section, we will define a natural sequence space coding the points in an intersection $ \mathcal{O}^+(x) \cap S $.
This space will consist of finite words, whose word length is equal to the level of the corresponding point.
The action of the pseudogroup $ \langle \Phi, \Theta \rangle $ on the level decomposition of $ \mathcal{O}^+(x) \cap S $ will induce a faithful action on this sequence space.

Fix $ x \in K $ with $ \mathcal{O}^+(x) \cap S \neq \emptyset $, and let $ y \in \mathcal{O}^+(x) \cap S $ be a point of level zero, i.e. $ y = \psi_T(x) $ with $ n_x(T) = 0 $.
Then by Lemma \ref{Thetapermute}, $ \Theta(y) $ has level one.

\begin{enumerate}[(1)]
\vspace{0.1cm}
\item \textit{Points of level one}:
For $ 1 \leq i_1 \leq M(x) $, let 
$$ 
y_{i_1} = (\Phi^{i_1-1} \Theta)(y),
$$ 
where $ M(x) $ is the minimum positive integer such that $ (\Phi^{M(x)} \Theta)(y) $ is not defined, i.e. $ (\Phi^{M(x)}\Theta)(y) $ does not return to $ S $ under the Kuperberg flow.
We call $ M(x) $ the \textit{escape time} of $ \Theta(y) $ from $ S $.
By Lemma \ref{Phipermute}, each point $ y_{i_1} $ has level one.
\vspace{0.2cm}

\item \textit{Points of level two}: For each $ 1 \leq i_1 \leq M(x) $, let 
$$
y_{i_1,i_2} = (\Phi^{i_2-1} \Theta)(y_{i_1}),
$$
where $ 1 \leq i_2 \leq M_{i_1}(x) $ and $ M_{i_1}(x) $ is the minimum positive integer such that $ (\Phi^{M_{i_1}(x)+1} \Theta)(y_{i_1}) $ is not defined.
Note that $ M_{i_1}(x) \neq \infty $, because $ \Theta(y_{i_1}) $ has a radial coordinate $ >2 $, so the Wilson orbit of the points $ \Theta(y_{i_1}) $ escape in finite time.
By Lemmas \ref{Phipermute} and \ref{Thetapermute}, each point $ y_{i_1, i_2} $ has level two.
\vspace{0.2cm}

\item \textit{Points of level k}: For each $ 1 \leq i_{k-1} \leq M_{i_1,\ldots,i_{k-1}}(x) $, let
$$
y_{i_1,\ldots,i_k} = (\Phi^{i_k-1} \Theta) (y_{i_1,\ldots,i_{k-1}}),
$$
where $ 1 \leq i_n \leq M_{i_1,\ldots,i_{k-1}}(x) $ and $ M_{i_1,\ldots,i_{k-1}}(x) $ is the minimum positive integer such that $ (\Phi^{M_{i_1,\ldots,i_{k-1}}(x)+1} \Theta)(y_{i_1,\ldots,i_{k-1}}) $ is not defined.
By Lemmas \ref{Phipermute} and \ref{Thetapermute}, each point $ y_{i_1,\ldots,i_k} $ has level $ k $.
\end{enumerate}

We have recursively defined the symbolic dynamics of a forward orbit. 
For a finite orbit, this process must terminate, resulting in a finite sequence space.
Naturally, the sequence space for an infinite orbit is infinite.
We now make this precise.

\begin{figure}[h]
    \includegraphics[width=1.4\linewidth, trim={4.7cm 15.7cm 0 3.1cm}, clip]{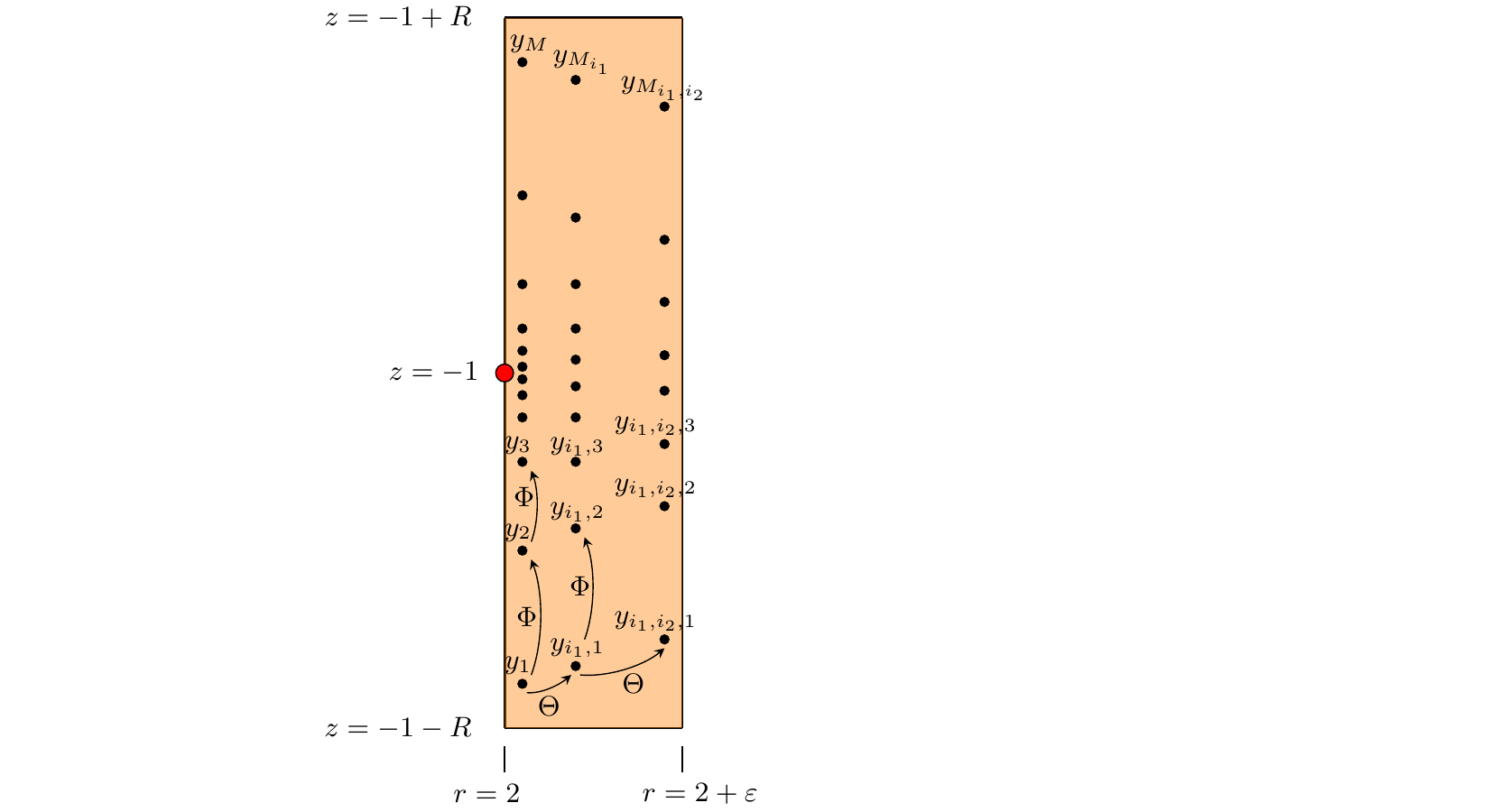}
  \caption{Symbolic dynamics of finite orbit of $ \Theta(y) $ on the rectangle $ S $. The points are labeled according to Equation \ref{kuppermute}. The map $ \Phi $ moves points up along the Wilson flow, preserving the radial coordinate. The map $ \Theta $ moves points outward, through the insertion. The radius inequality implies that $ \Theta $ increases the radius.}
  \label{symborbit}
\end{figure}

\subsubsection{Symbolic dynamics of a finite orbit}
Assume that $ O^+(x) $ is finite, and let $ y \in \mathcal{O}^+(x) \cap S $ be a point of level zero as above.
Since a finite orbit must intersect $ S $ at most a finite number of times, there exists $ N \in \mathbb{N} $ such that
$$
\mathcal{O}^+(\Theta(y)) \cap S = \bigcup_{j=1}^N \bigcup_{i_j=1}^{M_{i_1,\ldots,i_{j-1}}(x)} y_{i_1,\ldots,i_j},
$$
where $ M_{i_1,i_0}(x) = M(x) $, and $ M_{i_1,\ldots,i_{N-1}}(x) < \infty $ for all $ i_1,\ldots, i_{N-1} $.
From Lemmas \ref{Phipermute} and \ref{Thetapermute}, we have that for all $ 1 \leq j \leq N $, the maps $ \Phi $ and $ \Theta $ in the Kuperberg pseudogroup permute these points in the following way.
\begin{align}
\label{kuppermute}
\Phi(y_{i_1,\ldots,i_j}) &= y_{i_1,\ldots,i_j + 1} \\
\Theta(y_{i_1,\ldots,i_{j-1}}) &= y_{i_1,\ldots,i_j, 1} \nonumber
\end{align}
See Figure \ref{symborbit} for a picture of part of a finite orbit's intersection with $ S $ and its permutation by $ \Phi $ and $ \Theta $.
We now have a sequence space $ \Sigma \subset \mathbb{N}^N $ given by
$$
\Sigma = \bigcup_{j=1}^N \bigcup_{i_j=1}^{M_{i_1,\ldots,i_{j-1}}(x)} (i_1, \ldots, i_j).
$$
The Kuperberg pseudogroup acts faithfully on this space by Equation \ref{kuppermute} and we have a bijective coding map 
$$ 
\pi: \Sigma \rightarrow O^+(\Theta(y)) \cap S,
$$
given by $ \pi(\omega) = y_{\omega} $.
In this correspondence, the length of $ \omega $ is equal to the level of $ \pi(\omega) $.

\subsubsection{Symbolic dynamics of an infinite orbit}
We have similar symbolic dynamics for an infinite orbit, but the orbit now has points of arbitrary level so $ N = \infty $.
The sequence space is now $ \Sigma \subset \mathbb{N}^{\mathbb{N}} $, given by
\begin{equation}
\label{inforbsigma}
\Sigma = \bigcup_{j=1}^{\infty} \bigcup_{i_j=1}^{M_{i_1,\ldots,i_{j-1}}(x)} (i_1, \ldots, i_j).
\end{equation}
The coding map $ \pi :\Sigma \rightarrow O^+(\Theta(y)) \cap S $ is still bijective and the Kuperberg pseudogroup acts on $ \Sigma $ as defined in Equation \ref{kuppermute}.

For finite or infinite orbits, the sequence space $ \Sigma $ coding the points in the transverse section was constructed iteratively; we added symbols to the right of words of length $ k-1 $ to define the words of length $ k $.
This implies that the sequence space $ \Sigma $ satisfies the \textit{extension admissibility condition} from Definition \ref{extadm}, and thus is a general symbolic space as defined in Chapter \ref{Thermo}, over the alphabet $ E = \mathbb{N} $.

\vfill
\eject

\section{The Kuperberg minimal set}
\label{Kupmin}
The Kuperberg flow $ \psi_t $ preserves a unique minimal set $ \mathcal{M} \subset K $, with the following characterization.

\begin{theorem}(\cite{Hur}, Theorem 17.1) 
\label{minchar}
Let $ \mathcal{M} \subset K $ be the Kuperberg minimal set, and let $ \mathcal{R}' $ be the notched Reeb cylinder. 
Then $ \mathcal{M} $ is a codimension one lamination with Cantor transversal $ \tau $, and
$$
\mathcal{M} = \overline{\bigcup_{-\infty < t < \infty} \psi_t(\mathcal{R}')}.
$$
\end{theorem}

This theorem is proved under \textit{generic} assumptions on the insertions and flow, detailed in Chapter 17 of \cite{Hur}.
The assumptions we made in Section \ref{Kupsection} are special cases of these generic assumptions, so the above theorem applies to the plug $ K $ that we have constructed.
We will use this theorem as a point of departure in studying $ \mathcal{M} $.

\subsection{The level decomposition}
First, define
\begin{equation}
\label{MN0}
\mathcal{N}_0 = \bigcup_{-\infty < t < \infty} \psi_t(\mathcal{R}), \; \text{ and } \; \mathcal{M}_0 = \bigcup_{-\infty < t < \infty} \psi_t(\mathcal{R}'),
\end{equation}
so that $ \mathcal{M} = \overline{\mathcal{M}_0} $.
In the notation of orbit surfaces from Definition \ref{orbsurf}, we have $ \mathcal{N}_0 = \mathcal{O}(\mathcal{R}) $ and $ \mathcal{M}_0 = \mathcal{O}(\mathcal{R}') $.

An important part of the analysis in \cite{Hur} that we will require is the level decomposition.
In Definition \ref{leveldef}, for any $ x \in K $ we defined the level function $ n_x : \mathcal{O}(x) \rightarrow \mathbb{N} $ along the orbit of $ x $.
We extend $ n_x $ to a level function $ n_0 $ on $ \mathcal{M}_0 $ in the following way.

\begin{definition}[\textit{level function of $ \mathcal{M}_0 $}]
Let $ x \in \mathcal{M}_0 $.
By Equation \ref{MN0}, there exists $ T \geq 0 $ and $ y \in \gamma^u $ such that $ x = \psi_T(y) $.
In terms of the level function $ n_y : \mathcal{O}(y) \rightarrow \mathbb{N} $, define
$$
n_0(x) = n_y(T).
$$
\end{definition}

The following proposition appears as Proposition 10.1 in \cite{Hur}.

\begin{proposition}
\label{levelfunction}
The function $ n_0 : \mathcal{M}_0 \rightarrow \mathbb{N} = \{0, 1, 2, \ldots \} $ is well-defined.
\end{proposition}

As a consequence, the following level decomposition is well-defined.

\begin{equation}
\label{leveldecompM0}
\mathcal{M}_0 = \bigcup_{k \geq 0} \mathcal{M}_0^n, \text{ where } \mathcal{M}_0^k = \{ x \in \mathcal{M}_0 : n_0(x) = k \}
\end{equation}

Note that the level function $ n_0 $ extends to $ \mathcal{N}_0 $, hence $ \mathcal{N}_0 $ also has a well-defined level decomposition. 

\subsection{The intersections $ \mathcal{N}_0 \cap S $ and $ \mathcal{M}_0 \cap S $}
The lower insertion rectangle $ S $ is transverse to the Kuperberg flow $ \psi_t $.
Recall from Equation \ref{Psi} that $ \Psi $ is the Kuperberg pseudogroup of first-return maps of the flow $ \psi_t $.
By construction $ \gamma = S \cap \mathcal{R} $ and $ \gamma^u = S \cap \mathcal{R}' $.
From this and Equation \ref{MN0} we have

\begin{equation}
\label{N0S}
\mathcal{N}_0 \cap S = \bigcup_{g \in \Psi} g(\gamma), \; \text{ and } \; \mathcal{M}_0 \cap S = \bigcup_{g \in \Psi} g(\gamma^u).
\end{equation}

In the quotient plug $ K $, $ \gamma $ is identified with $ \sigma^{-1} \gamma $, a parabolic curve parametrized in Equation \ref{sigmagamma}.
By this equation, we see that $ \sigma^{-1} \gamma(0) $ is the only point of $ \sigma^{-1} \gamma $ with a radial coordinate of $ 2 $; every other point in $ \sigma^{-1} \gamma $ has a radial coordinate $ >2 $.
Thus the Wilson orbit of $ \gamma $ is a double propeller (see Definition \ref{doublepropdef}), and by restricting the parametrization of $ \gamma $ to $ \gamma^u $, we obtain a single propeller (see Definition \ref{propdef}).

Under the Kuperberg flow, each intersection of the orbit of $ \gamma $ with $ S $ is identified with a curve in the base of the plug. 
The Kuperberg flow of these curves then follows the Wilson flow up into the interior of the plug, and each of these orbits is a propeller.
Each of these surfaces then may intersect the insertion regions $ S_i $ again, and the process repeats, creating an infinitely branching union of propellers with a complicated embedding in $ K $.

As Kuperberg orbits of $ \gamma $ and $ \gamma^u $, the surfaces $ \mathcal{N}_0 $ and $ \mathcal{M}_0 $ are unions of double and single propellers, respectively.
These branching surfaces are termed ``choux-fleurs" in \cite{Ghy}, and are extensively studied in \cite{Hur}.
Each single propeller in $ \mathcal{M}_0 $ is a restriction of a double propeller in $ \mathcal{N}_0 \supset \mathcal{M}_0 $.
The embedding and transverse dynamics of $ \mathcal{N}_0 $ and $ \mathcal{M}_0 $ are complicated.

By Equation \ref{N0S}, to study $ \mathcal{N}_0 \cap S $ we must compute the image of $ \gamma $ under the full Kuperberg pseudogroup $ \Psi $ as defined in Equation \ref{Psi}.
However, determining the admissible compositions of the generators of $ \Psi $ is difficult.
So we will begin by focusing on the dynamics of $ \Psi_1 = \langle \Phi, \Theta \rangle $ (see Equation \ref{Psi1}), defining

\begin{equation}
\label{N01S}
\mathcal{N}_{0,1} \cap S = \bigcup_{g \in \Psi_1} g(\gamma).
\end{equation}

Of course, we have $ \mathcal{M}_{0,1} $ defined similarly.

\subsection{The pseudogroup action on the level decomposition}
\label{pseudomin}
In Chapter \ref{Kuppseudo} we studied the action of the pseudogroup $ \langle \Phi, \Theta \rangle $ on the transverse section $ \mathcal{O}^+(\Theta(y)) \cap S $ of an orbit.
We defined a general sequence space $ \Sigma $ and a bijective coding map $ \pi : \Sigma \rightarrow \mathcal{O}^+(\Theta(y)) \cap S $, and studied the dynamics of the pseudogroup on the section and the induced dynamics on the sequence space.

In this section, we will develop similar symbolic dynamics for $ \langle \Phi, \Theta \rangle $ acting on $ \mathcal{N}_{0,1} \cap S $.
In addition to labeling the curves in $ \mathcal{N}_{0,1} \cap S $ according to a sequence space, we will explicitly parametrize these curves in coordinates $ (r, \theta, z) $.

To do this, we return to the assumptions made in Chapter \ref{InsertAssume}.
These will allow us to explicitly parametrize the propellers $ \mathcal{O}^+(\gamma, 0, t) $ for $ t \geq 0 $. 
The images of $ \gamma $ under the maps $ \Phi $ and $ \Theta $ can be explicitly calculated from these parametrizations, by studying their intersections with the lower insertion rectangle $ S $.
As we studied in Chapter \ref{Kuppseudo}, the sequence space coding an orbit is defined in terms of escape times.
By analyzing the parametrizations of the curves in $ \mathcal{N}_{0,1} \cap S $, we will obtain precise estimates on their escape times, and thus the sequence space coding these curves.

The level decomposition of $ \mathcal{N}_0 $ (identical to that of $ \mathcal{M}_0 $ given in Equation \ref{leveldecompM0}) induces the following level decomposition of $ \mathcal{N}_{0,1} \cap S $.
\begin{equation}
\label{leveldecompNS}
\mathcal{N}_{0,1} \cap S = \bigcup_{k \geq 0} \mathcal{N}_{0,1}^k \cap S, \; \text{ where } \; \mathcal{N}_{0,1}^k \cap S = \{ x \in \mathcal{N}_{0,1} \cap S : n_0(x) = k \}
\end{equation}
By Propositions \ref{Phipermute} and \ref{Thetapermute}, the pseudogroup $ \langle \Phi, \Theta \rangle $ acts on this level decomposition in the following way.
\begin{align}
\label{kuppermute2}
\Phi : \mathcal{N}_{0,1}^k \cap S &\mapsto \mathcal{N}_{0,1}^k \cap S \\
\Theta : \mathcal{N}_{0,1}^k \cap S &\mapsto \mathcal{N}_{0,1}^{k+1} \cap S \nonumber
\end{align}

\subsubsection{The level-zero curve $ \mathcal{N}_{0,1}^0 \cap S $}
By Equation \ref{MN0}, $ \mathcal{N}_0 = \mathcal{O}(\mathcal{R}) $. 
Points in the Reeb cylinder $ \mathcal{R} $ have level zero, as do points in the intersection $ \mathcal{R} \cap S = \gamma $, hence $ \mathcal{N}_{0,1}^0 \cap S = \{ \gamma \} $.
Recall the parametrization of $ \gamma $ given in Equation \ref{gamma}:  
$$ 
\gamma(s) = (2, \beta, -1+s), \; \text{ with } \; s \in [-R, R].
$$
We refer to the midpoint $ \gamma(0) $ as the \textit{vertex} of $ \gamma $, which is the intersection $ l \cap S $ of the special orbit with $ S $.

\subsubsection{The level-one curves $ \mathcal{N}_{0,1}^1 \cap S $}
For all $ i_1 \in \mathbb{N} $ let
\begin{equation}
\label{gammai1}
\gamma_{i_1} = (\Phi^{i_1-1} \Theta)(\gamma).
\end{equation}
Because $ \gamma $ has level zero, by Equation \ref{kuppermute2} $ \gamma_{i_1} $ has level one for all $ i_1 $.
We now use the assumptions we made in Chapter \ref{InsertAssume} to parametrize each $ \gamma_{i_1} $.

\begin{proposition}
\label{gamma1param}
For all $ i_1 \in \mathbb{N} $ there exist $ s_{i_1}^{\pm} $ with $ -R < s_{i_1}^- < 0 < s_{i_1}^+ < R $ such that the parametrization of $ \gamma_{i_1}:[s_{i_1}^-, s_{i_1}^+]\setminus 0 \rightarrow S $ is
\begin{align}
\label{gamma1}
\displaystyle 
\gamma_{i_1}(s) &= \left(2+s^2, \beta, -1+q_{i_1}(s)\right) \text{, where } \\
& \displaystyle q_{i_1}(s) = s^2 \tan \left(\frac{s^2}{R^2} T_{i_1}(s)- \tan^{-1}\left(\frac{R}{s^2}\right)\right), \nonumber \\
& \displaystyle T_{i_1}(s) = a^{-1}(2\pi i_1+ \beta-\alpha+s)+R-1. \nonumber
\end{align}
\end{proposition} 

\begin{proof}
By definition $ \gamma_{i_1} $ is the $ i_1 $-th return time of $ \sigma^{-1} \gamma $ to $ S $.
Recall from Equation \ref{sigmagamma} the the parametrization:
$$
\sigma^{-1}\gamma (s) = (2+s^2, \alpha-s, -2), \; \text{ with } \; s \in [-R,R]
$$
From $ \{z=-2\} $ to $ \{z=-1-R\} $, the Kuperberg flow is given by Wilson's flow in Equation \ref{wilout}.
Applying this to the above parametrization, we obtain a parametrization of the following orbit strip:
$$
\mathcal{O}^+(\sigma^{-1}\gamma, 0, 1-R) = (2+s^2, \alpha-s+at, -2+t), \text{ where } s \in [-R, R] \text{ and } t \in [0,1-R].
$$
As the Wilson orbit of the parabolic curve $ \sigma^{-1} \gamma $, we see that $ \mathcal{O}^+(\sigma^{-1}\gamma, 0, 1-R) $ is a double propeller parametrized by $ s $ and $ t $.
Its intersection with the bottom annulus $ C^- $ is the curve $ \psi_{1-R} \sigma^{-1} \gamma $. 
For $ |z+1| \leq R $ the Kuperberg flow is now given by Wilson's flow in Equation \ref{wilin}.
Applying this to the parametrization of $ \psi_{1-R} \sigma^{-1} \gamma $, we obtain a parametrization of the double propeller inside this region.
\begin{multline*}
\mathcal{O}^+(\psi_{1-R} \sigma^{-1} \gamma, 0, T) = \left(2+s^2, \alpha-s+a(1-R+t), -1+\left(\frac{s^2}{R^2} t-\tan^{-1} \left(\frac{r}{s^2}\right)\right) \right), \text{ where } \\
s \in [-R,R] \setminus \{0\} \text{ and } t \in [0,T].
\end{multline*}

By Equation \ref{gammai1} and the definition of $ \Phi $, each curve $ \gamma_{i_1} $ is the $ i_1 $-th intersection of this double propeller with $ S $ as $ T $ increases.
To find parametrizations of these curves, recall by Equation \ref{Sstrip} that $ S $ has a constant angular coordinate $ \theta=\beta $.
Setting the $ \theta $ coordinate in the parametrization of $ \mathcal{O}_1^+(\psi_{1-R} \sigma^{-1} \gamma, 0, T) $ to $ \beta+2\pi i_1 $ and solving for $ t>0 $, we find that the $ i_1 $-th return time of $ \psi_{1-R} \sigma^{-1} \gamma $ to $ S $ is
$$
T_{i_1}(s) = a^{-1}(2\pi i_1 +\beta-\alpha +s) + R-1
$$
Substituting this back into the parametrization of $ \mathcal{O}^+(\psi_{1-R} \sigma^{-1} \gamma, 0, T) $ we obtain the desired formula given in Equation \ref{gamma1}.

However, these parametrizations are not valid for all $ s \in [-R,R] $ or $ i_1 \in \mathbb{N} $; because $ \gamma_{i_1} $ is defined by $ \Phi, \Theta : S \rightarrow S $ we must restrict to values of $ s $ and $ i_1 $ such that $ \gamma_{i_1}(s) \in S $.
The upper boundary $ S^+ $ of $ S $ has a constant $ z $-coordinate $ z=-1+R $.
Thus in the notation of Equation \ref{gamma1}, our restriction should be such that $ q_{i_1}(s) \leq R $.
Define $ s_{i_1}^+ $ and $ s_{i_1}^- $ as the unique solutions to the equation $ q_{i_1}(s) = R $ on the domains $ s>0 $ and $ s<0 $, respectively.
Using the parametrization for $ q_{i_1} $ given in Equation \ref{gamma1}, it is easy to show that by the intermediate value theorem that these exist, and that by monotonicity of $ q_i $ they are unique.

In Chapter \ref{Transversal}, we will prove that the radial coordinates of the endpoints $ \gamma_{i_1}(s_{i_1}^{\pm}) $ decrease monotonically as $ i_1 \to \infty $.
Referring to Equation \ref{Sstrip}, we see that $ S $ has a fixed radial width of $ b>0 $, so there exists a minimal $ N_b \in \mathbb{N} $ such that $ \gamma_{i_1}(s_{i_1}^{\pm}) \in S $ for all $ i_1 \geq N_b $.
In terms of $ N_b $, we define
\begin{equation}
\label{Sigma1}
\Sigma_{b,1} = \{ N_b, N_b+1, \ldots, \} \subset \mathbb{N}.
\end{equation}
The index $ i_1 $ ranges through  all $ \Sigma_{b,1} $ because the double propeller $ \mathcal{O}^+(\psi_{1-R} \sigma^{-1} \gamma, 0, T) $ is trapped.
We conclude by restricting our parametrization to $ \gamma_{i_1}: [s_{i_1}^-, s_{i_1}^+] \setminus 0 \rightarrow S $, and indices to $ i_1 \in \Sigma_{b,1} $.
\end{proof}

As with the level-zero curve $ \gamma $, for all $ i_1 \in \Sigma_{b,1} $, we call $ \gamma_{i_1}(0) $ the \textit{vertex} of $ \gamma_{i_1} $.
Defining $ v_{i_1} = \lim_{s \to 0} q_{i_1}(s) $, we see by Equation \ref{gamma1} that the vertex of $ \gamma_{i_1} $ is 
$$ 
\lim_{s \to 0} \gamma_{i_1}(s) = (2, \beta, -1+v_{i_1})
$$ 
Notice that the vertices of the level-one curves $ \gamma_{i_1} $ lie on the level-zero curve $ \gamma $. 
Explicitly, $ \gamma(v_{i_1}) = (2, \beta, -1+v_{i_1}) $ using Equation \ref{gamma}. 
This relation will imply a nesting property for higher-level curves.

Using the parametrization given in Equation \ref{gamma1}, it can be shown that $ v_{i_1}<0 $ for all $ i_1 $ and that $ \lim_{i_1 \to \infty} v_{i_1} = 0 $. 
So as $ i_1 \rightarrow \infty $, these vertices limit on intersection $ l \cap S = (2, \beta, -1)$, the vertex of $ \gamma $.
See Figure \ref{levelone} for a plot of these curves.

\begin{figure}[h]
\includegraphics[width=0.9\linewidth, trim={3.2cm 10.5cm 1cm 2.8cm}, clip]{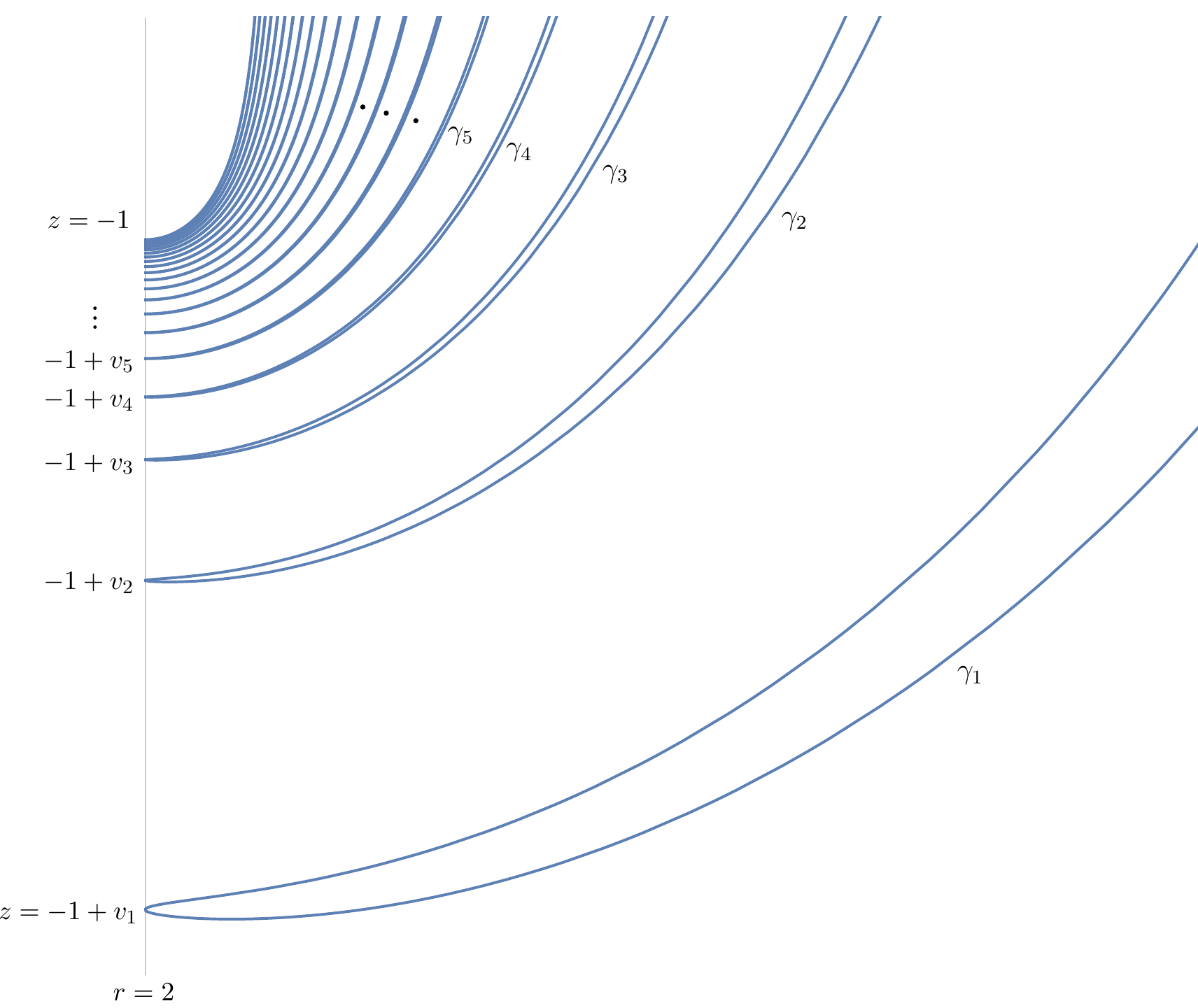}
\caption{A plot of the level-one curves $ \gamma_{i_1} \subset S $ for $ i_1 = 1,2, \ldots, 20 $, and $ a=R=1 $, $ \alpha = \beta = 0 $. The vertices $ v_{i_1} $ form a vertical sequence on the Reeb cylinder $ \{r=2\} $, limiting on the special point $ (2,\beta,-1) $.}
\label{levelone}
\end{figure}

\subsubsection{The level-two curves $ \mathcal{N}_{0,1}^2 \cap S $}
For each $ i_1 \in \Sigma_{b,1} $ define 
\begin{equation}
\label{gammai1i2}
\gamma_{i_1, i_2} = (\Phi^{i_2-1} \Theta)(\gamma_{i_1}).
\end{equation}
Because each $ \gamma_{i_1} $ has level one, by Equation \ref{kuppermute2} each $ \gamma_{i_1, i_2} $ has level two.
We can parametrize each $ \gamma_{i_1,i_2} $ as follows.

{\small
\begin{align}
\label{gamma2}
\gamma_{i_1,i_2}(s) &= \left(2+(s^2+q_{i_1}^2(s)), \beta, -1+q_{i_1,i_2}(s)\right) \text{, where } \\
& \displaystyle q_{i_1,i_2}(s) = (s^2+q_{i_1}^2(s)) \tan \left(\left(\frac{s^2+q_{i_1}^2(s)}{R^2}\right) T_{i_1,i_2}(s)- \tan^{-1}\left(\frac{R}{s^2+q_{i_1}^2(s)}\right)\right), \nonumber \\
& \displaystyle T_{i_1,i_2}(s) = a^{-1}(2\pi i_2+ \beta-\alpha+q_{i_1}(s))+R-1 \nonumber
\end{align}
}
Here $ \gamma_{i_1,i_2} : [s_{i_1,i_2}^-, s_{i_1,i_2}^+]\setminus 0 \rightarrow S $, where $ s_{i_1,i_2}^- < 0 < s_{i_1,i_2}^+ $ are the solutions to the equation $ q_{i_1,i_2}(s) = R $.
The derivation of the parametrization in Equation \ref{gamma2} goes exactly like the proof of Proposition \ref{gamma1param}; we follow the orbit surface of each $ \gamma_{i_1} $ through the insertion and calculate its $ i_2 $-th intersection with $ S $.
We omit the details.
See Figure \ref{leveltwo} for a plot of these curves.

It remains to determine the admissible words $ (i_1,i_2) $ coding the level-two curves in $ \mathcal{N}_{0,1}^2 $.
To determine these words, we will need to estimate the escape times of the vertices of the level-two curves $ \gamma_{i_1,i_2} $, which we now define.
As with the level-one curves, we call $ \gamma_{i_1,i_2}(0) $ the \textit{vertex} of $ \gamma_{i_1,i_2} $, and define $ v_{i_1,i_2} = \lim_{s \to 0} q_{i_1, i_2}(s) $, so that the vertex is
$$
\lim_{s \to 0} \gamma_{i_1,i_2}(s) = \left(2+v_{i_1}^2, \beta, -1+v_{i_1,i_2} \right),
$$
using Equation \ref{gamma2}.
The level-two curves satisfy an important nesting property that we now describe.

\begin{definition}
Let $ \eta $ be a curve in $ S $, and suppose that $ \eta \cup S^+ $ bounds a closed region in $ S $. If $ \zeta $ is another curve in $ S $, we say that $ \zeta $ is \textit{nested in} $ \eta $ if the image of $ \zeta $ is contained in this closed region.
\end{definition}
Notice in Figure \ref{levelone} that each $ \gamma_{i_1} \cup S^+ $ bounds a closed region.

\begin{proposition}
\label{nesting2}
For each $ (i_1, i_2) $, the level-two curve $ \gamma_{i_1, i_2} $ is nested in $ \gamma_{i_2} $.
\end{proposition}

\begin{proof}
Recall that $ v_{i_1} = \lim_{s \to 0} q_{i_1}(s) $.
Using Equations \ref{gamma1} and \ref{gamma2}, it is easy to show that $ \lim_{s \to 0} T_{i_1,i_2}(s) = T_{i_2}(v_{i_1}) $ and $ \lim_{s \to 0} q_{i_1, i_2}(s) = q_{i_2}(v_{i_1}) $. 
From this we obtain that 
\begin{align*}
\lim_{s \to 0} \gamma_{i_1,i_2}(s) &= (2+v_{i_1}^2, \beta, -1+q_{i_2}(v_{i_1})) \\
&= \gamma_{i_2}(v_{i_1}).
\end{align*}
This shows that the vertex of $ \gamma_{i_1, i_2} $ is located on the image of $ \gamma_{i_2} $.
By the radius inequality, $ \gamma_{i_1, i_2} $ is nested in $ \gamma_{i_2} $.
\end{proof}

\begin{figure}[h]
\includegraphics[width=0.9\linewidth, trim={3.2cm 10.5cm 1cm 2.8cm}, clip]{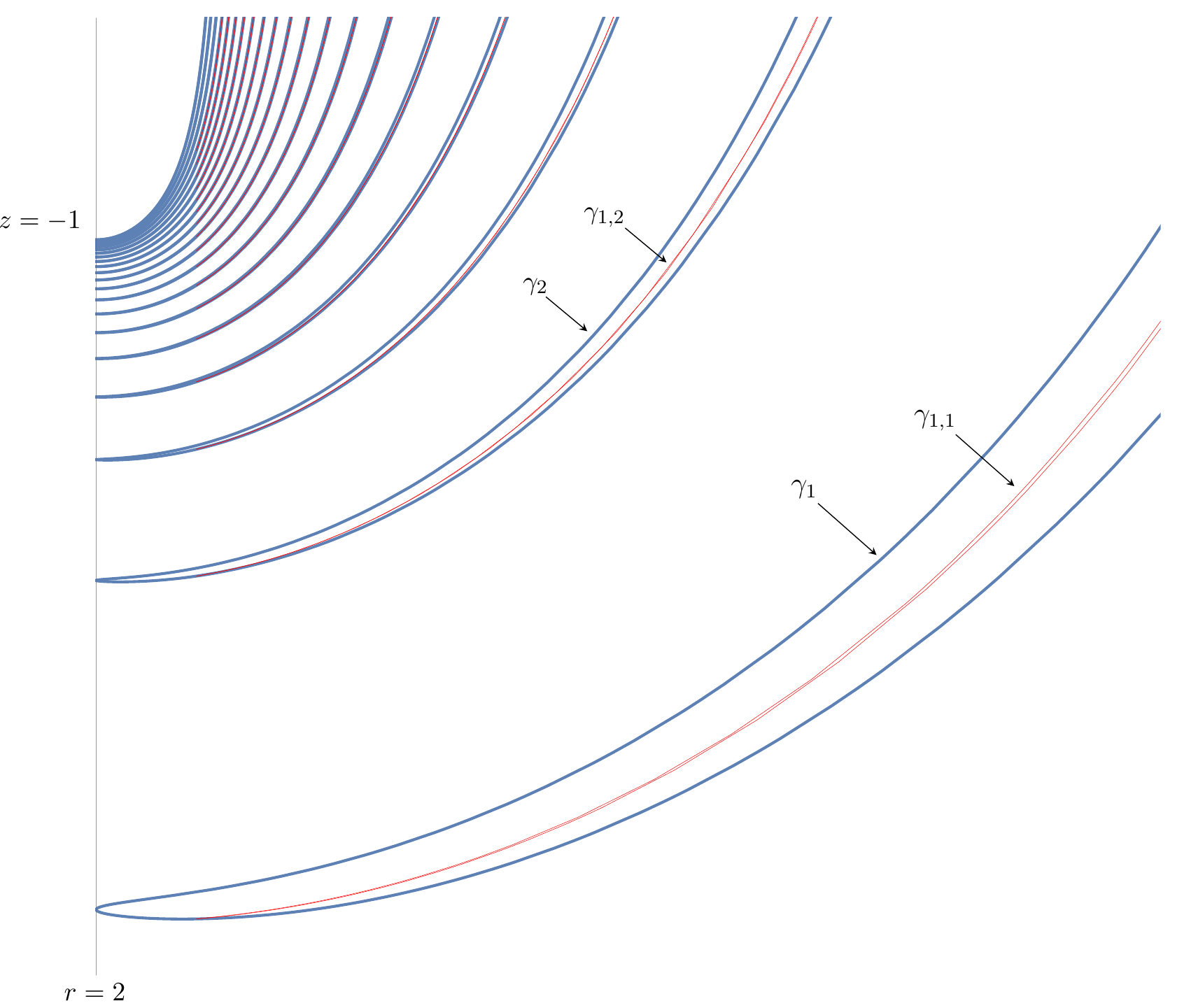}
\caption{A plot of the level-two curves $ \gamma_{i_1,i_2} $ in $ S_{\epsilon} $ where $ i_1 = 1 $. Note that each $ \gamma_{i_1,i_2} $ is nested in $  \gamma_{i_2} $.}
\label{leveltwo}
\end{figure}

Inspecting the parametrization in Equation \ref{gamma2}, we see that for a fixed $ (i_1, i_2) $, the radial coordinate of each $ \gamma_{i_1,i_2} $ is bounded away from 2.
Each $ \gamma_{i_1,i_2} $ is the $i_2$-th intersection of the orbit surface $ \mathcal{O}^+(\psi_{1-R} \sigma^{-1} \gamma_{i_1}, 0, T) $ with $ S $, so this surface is \textit{not} a double propeller and escapes the plug in finite time.
In particular it has finitely many intersection curves, thus for a fixed $ i_1 \in \Sigma_{b,1} $ there are only finitely many values of $ i_2 $ such that $ (\Phi^{i_2-1} \Theta) (\gamma_{i_1}) \cap S \neq \emptyset $.
The minimal value of $ i_2 $ is $ N_b $, because $ \gamma_{i_1,i_2} $ is nested in $ \gamma_{i_2} $ by Proposition \ref{nesting2}.
So for each $ i_1 \in \Sigma_{b,1} $ there exists $ M_{i_1} $ such that $ N_b \leq i_2 \leq M_{i_1} $, hence the admissible words defining $ \gamma_{i_1,i_2} $ are
\begin{equation}
\label{Sigma2}
\Sigma_{b,2} = \bigcup_{i_1 \in \Sigma_{b,1}} \bigcup_{i_2=N_b}^{M_{i_1}} (i_1,i_2) = \bigcup_{i_1=N_b}^{\infty} \bigcup_{i_2=N_b}^{M_{i_1}} (i_1,i_2).
\end{equation}

Using the parametrization in Equation \ref{gamma2}, we can show that the vertex of each curve is its point of minimal $ z $-coordinate.
Recall that $ \gamma_{i_1,i_2} $ is defined (i.e. the parametrization in Equation \ref{gamma2} is valid) if and only if $ q_{i_1,i_2}(s) = R $ has a solution; equivalently, if $ q_{i_1,i_2}(s) \leq R $ for some $ s $.
Since $ -1+q_{i_1,i_2}(s) $ is the $ z $-coordinate of $ \gamma_{i_1,i_2}(s) $, we see that $ \gamma_{i_1,i_2} $ is defined if and only if $ v_{i_1,i_2} \leq R $.
Thus $ M_{i_1} $ coincides with the \textit{escape time} of the vertex as defined in Chapter \ref{Kuppseudo}; for a fixed $ i_1 \in \Sigma_{b,1} $, it is the maximal $ i_2 $ such that $ v_{i_1,i_2} \leq R $.
Using this, we can find explicit bounds on $ M_{i_1} $ by estimating these escape times.

\begin{proposition}
\label{escape2}
For each $ i \in \Sigma_{b,1} $, let $ M_i $ be the greatest positive integer such that $ v_{i, M_i} \leq R $.
Then there exist constants $ C,K>0 $ such that $ M_i $ is asymptotic to $ C+Ki^2 $.
More precisely, for any $ \delta>0 $ there is an integer $ N_1>0 $ with
$$
C+(K-\delta)i^2 < M_i < (C+\delta)+Ki^2
$$
for all $ i\geq N_1 $.
\end{proposition}

\begin{proof}
We will prove the upper bound; the lower bound is similar.
Recall from the proof of Proposition \ref{nesting2} the nesting property $ v_{i_1,i_2} = q_{i_2}(v_{i_1}) $.
In particular, $ v_{i,M_i} = q_{M_i}(v_i) $.
Since the upper boundary $ S^+ $ of $ S $ has a constant $ z $-coordinate of $ -1+R $, we have that $ M_i $ is the greatest positive integer such that $ q_{M_i}(v_i) \leq R $.
Referring to the parametrization in Equation \ref{gamma1param}, this inequality is equivalent to
$$
2\pi M_i \leq \alpha-\beta+a(1-R)-v_i + \frac{2aR^2}{v_i^2} \tan^{-1} \left(\frac{R}{v_i^2}\right).
$$
Recall that $ \lim_{i \to \infty} v_i = 0 $. Thus for any $ \delta>0 $, there exists $ N>0 $ such that $ 0 < -v_i < 2\pi \delta $ for all $ i\geq N $.
Also note that $ \displaystyle \tan^{-1}\left(\frac{R}{v_i^2}\right) < \frac{\pi}{2} $ for all $ i $.
Substituting these into the above inequality, we obtain
\begin{equation}
\label{ineq}
2\pi M_i < \alpha-\beta+a(1-R)+2\pi \delta + \frac{\pi aR^2}{v_i^2}.
\end{equation}
Using the definition $ v_i = \lim_{s \to 0} q_i(s) $, it is easy to show that there exists a constant $ p>0 $ such that $ \displaystyle v_i = -\frac{p}{i} $.
We define the following constants.
\begin{equation}
\label{defCK}
C = \frac{\alpha-\beta+a(1-R)}{2\pi}, \hskip 2cm K = \frac{aR^2}{2p^2}
\end{equation} 
Substituting these into Equation \ref{ineq}, we obtain
$$
M_i < (C+\delta)+K i^2.
$$
\end{proof}

Recall that the vertices $ v_{i_1} $ of $ \gamma_{i_1} $ limit on the special orbit intersection $ l \cap S $, the vertex of $ \gamma $.
From the recursive definition in Equation \ref{gammai1i2} and the parametrizations in Equation \ref{gamma2}, one can show that the vertices of $ \gamma_{i_1,i_2} $ limit on the vertices of $ \gamma_{i_1} $.
See Figure \ref{leveltwo2} for another picture of the level-two curves inside the level-one curves, and observe this limiting behavior.

\begin{figure}[h]
\includegraphics[width=0.9\linewidth, trim={0 0.5 0 0.8cm}, clip]{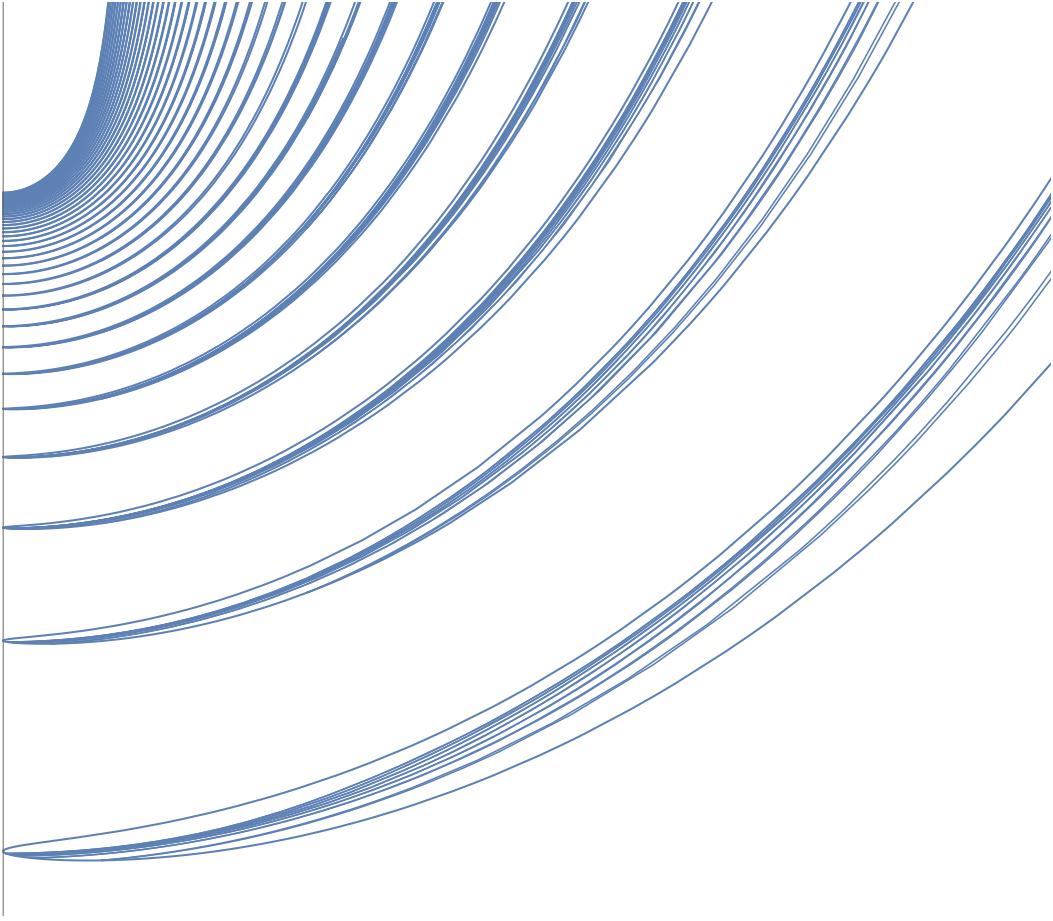}
\caption{The level-two curves limiting on the level-one curves.}
\label{leveltwo2}
\end{figure}

\subsubsection{The level-k curves $ \mathcal{N}_{0,1}^k \cap S $}
Let $ \Sigma_{b,k-1} $ denote the admissible words of level $k-1$ defining the curves $ \gamma_{i_1,\ldots,i_{k-1}} $.
As before, we define 
\begin{equation}
\label{gammai1ik}
\gamma_{i_1, \ldots, i_k} = \Phi^{i_k-1} \Theta(\gamma_{i_1, \ldots, i_{k-1}}),
\end{equation}
and observe that $ \gamma_{i_1, \ldots, i_k} \in \mathcal{N}_{0,1}^k \cap S $ by Equation \ref{kuppermute2}.
As with levels one and two, we can explicitly parametrize these curves.
{\Small
\begin{align}
\label{gammak}
\gamma_{i_1, \ldots, i_k}(s) &= \left(2+\left(s^2+\sum_{j=1}^{k-1} q_{i_1, \ldots,i_j}^2(s)\right), \beta, -1+q_{i_1,\ldots,i_k}(s)\right) \text{, where } \\
&\displaystyle q_{i_1,\ldots,i_k}(s) = \left(s^2+\sum_{j=1}^{k-1} q_{i_1, \ldots,i_j}^2(s) \right) \tan \left(\left(\frac{s^2+\sum_{j=1}^{k-1} q_{i_1, \ldots,i_j}^2(s)}{R^2} \right) T_{i_1,\ldots,i_k}(s)-\tan^{-1}\left(\frac{R}{s^2+\sum_{j=1}^{k-1} q_{i_1, \ldots,i_j}^2(s)}\right)\right), \nonumber \\
&\displaystyle T_{i_1,\ldots,i_k}(s) = a^{-1}(2\pi i_k+ \beta-\alpha+q_{i_1,\ldots,i_{k-1}}(s))+R-1 \nonumber
\end{align}
}
For each $ \omega = (i_1,\ldots,i_k) $, we have $ \gamma_{\omega} : [s_{\omega}^-, s_{\omega}^+]\setminus 0 \rightarrow S $, where $ s_{\omega}^{\pm} $ are the unique solutions to the equation $ q_{\omega}(s) = R $.

As with levels one and two, we call $ \gamma_{\omega}(0) $ the \textit{vertex} of $ \gamma_{\omega} $.
The $ z $-coordinate of the vertex is $ -1+v_{\omega} $, where 
$$ 
v_{\omega} = \lim_{s \to 0} q_{\omega}(s).
$$
The proof of the following proposition is identical to the proof of Proposition \ref{nesting2}.

\begin{proposition}
\label{nestingk}
For each $ (i_1,\ldots,i_k) \in \Sigma_{b,k} $, the level-$k$ curve $ \gamma_{i_1,\ldots,i_k} $ is nested in the level-$ (k-1) $ curve $ \gamma_{i_2, \ldots, i_k} $.
\end{proposition}

It remains to recursively determine the admissible words $ \Sigma_{b,k} $ from $ \Sigma_{b,k-1} $.
For fixed values of $ (i_1, \ldots, i_{k-1}) \in \Sigma_{b,k-1} $, the curve $ \gamma_{i_1,\ldots,i_k} $ is defined for finitely many $ N_b \leq i_k \leq M_{i_1,\ldots, i_{k-1}} $, resulting in a sequence space
\begin{align}
\label{Sigmak}
\Sigma_{b,k} &= \bigcup_{(i_1,\ldots,i_{k-1}) \in \Sigma_{b,k-1}} \bigcup_{i_k = N_b}^{M_{i_1,\ldots,i_{k-1}}} (i_1,\ldots,i_k) \\
&= \bigcup_{i_1 = N_b}^{\infty} \bigcup_{i_2 = N_b}^{M_{i_1}} \cdots \bigcup_{i_k=N_b}^{M_{i_1,\ldots,i_{k-1}}} (i_1,\ldots,i_k) \nonumber.
\end{align}

As in Proposition \ref{escape2} we will estimate $ M_{i_1,\ldots,i_{k-1}} $ via escape times of vertices $ v_{\omega} $.
Recall from Equation \ref{defCK} the constants $ C,K>0 $ determining the admissible words of level two.

\begin{proposition}
\label{escapek}
For each $ (i_1, \ldots, i_{k-1}) \in \Sigma_{b,k-1} $ let $ M=M_{i_1,\ldots,i_{k-1}} $ be the greatest positive integer such that $ v_{i_1,\ldots,i_{k-1}, M} \in S $.
Then for large values of $ i_1,\ldots,i_{k-1} $, $ M_{i_1,\ldots,i_{k-1}} $ is asymptotic to $ C+K i_{k-1}^2 $.
More precisely, for any $ \delta>0 $ there is an integer $ N_{k-1} >0 $ with
$$
C+(K-\delta)i_{k-1}^2 < M_{i_1,\ldots,i_{k-1}} < (C+\delta)+K i_{k-1}^2
$$
when $ i_1,\ldots,i_{k-1} \geq N_{k-1} $.
\end{proposition}

\begin{proof}
We will prove the upper bound; the lower bound is similar.
By Proposition \ref{nestingk} we have the nesting property $ v_{i_1,\ldots,i_k} = q_{i_2,\ldots,i_k}(v_{i_1}) $.
In particular, $ v_{i_1,\ldots,i_{k-1},M} = q_{i_2,\ldots,i_{k-1},M}(v_{i_1}) $.
Then $ M=M_{i_1,\ldots,i_{k-1}} $ is the greatest positive integer such that $ q_{i_2,\ldots,i_{k-1},M}(v_{i_1}) \leq R $.
By Equation \ref{gammak} this is equivalent to
{\small
$$
2\pi M_{i_1,\ldots,i_{k-1}} < \alpha-\beta+a(1-R) - q_{i_2,\ldots,i_{k-1}}(v_{i_1}) + \frac{2R^2}{v_{i_1}^2+\sum_{j=2}^{k-1}q_{i_2,\ldots,i_j}^2(v_{i_1})} \tan^{-1} \left(\frac{R}{v_{i_1}^2+\sum_{j=2}^{k-1}q_{i_2,\ldots,i_j}^2(v_{i_1})} \right)
$$
}
Recall that $ \lim_{i \to \infty} v_i = 0 $ and $ v_{\omega} = \lim_{s \to 0} q_{\omega} $.
Combining this with the nesting property we obtain that
$$
\lim_{i_1 \to \infty} q_{i_2,\ldots,i_{k-1}}(v_{i_1}) = v_{i_2,\ldots,i_{k-1}} = q_{i_3,\ldots,i_{k-1}}(v_{i_2}),
$$
and by induction, for all $ 1 \leq j \leq k-1 $ that
$$
\lim_{i_1,\ldots,i_{j-1} \to \infty} q_{i_2,\ldots,i_j}(v_{i_1}) = v_{i_j}
$$
Then for a sufficiently large integer $ N_{k-1} $, we have for all $ i_1, \ldots, i_{k-1} \geq N_{k-1} $ that
$$
0 < -q_{i_2,\ldots,i_{k-1}}(v_{i_1}) < 2\pi \delta, \; \text{ and } \; 
0 < |q_{i_2,\ldots,i_j}(v_{i_1})| < \frac{\sqrt{\delta}}{k}.
$$
Substituting this into the first inequality and using that $ \displaystyle \tan^{-1}(\cdot) < \frac{\pi}{2} $, we obtain for $ i_1, \ldots, i_{k-1} \geq N_k $ that $ M_{i_1,\ldots,i_{k-1}} $ is the greatest positive integer such that
$$
2\pi M_{i_1,\ldots,i_{k-1}} < \alpha-\beta+a(1-R)+2\pi \delta + \frac{\pi aR^2}{v_{i_{k-1}}^2+\delta}.
$$
Since $ \displaystyle v_{i_{k-1}} = -\frac{p}{i_{k-1}} $ from the proof of Proposition \ref{escape2}, this is equivalent to
$$
M_{i_1,\ldots,i_{k-1}} < (C+\delta)+K i_{k-1}^2.
$$
\end{proof}

Finally, from the recursive definition in Equation \ref{gammai1ik}, and Equation \ref{kuppermute2}, the pseudogroup $ \langle \Phi, \Theta \rangle $ permutes the curves $ \gamma_{i_1,\ldots,i_k} $ in the following way.
\begin{align}
\label{kuppermutegamma}
\Phi(\gamma_{i_1,\ldots,i_k}) &= \gamma_{i_1,\ldots,i_k + 1} \\
\Theta(\gamma_{i_1,\ldots,i_{k-1}}) &= \gamma_{i_1,\ldots,i_k, 1} \nonumber
\end{align}

\subsection{Symbolic dynamics of the rectangle $ S $}
\label{symbrect}
In the previous section we described the level decomposition of the intersection $ \mathcal{N}_{0,1} \cap S $, parametrized the curves in each level set, and labeled the curves by words in a sequence space.
The parametrization of the parabolic curve $ \sigma^{-1} \gamma $ in Equation \ref{sigmagamma} was crucial in this analysis.

In this section, we recall the additional assumption from Chapter \ref{InsertAssume} that $ \sigma^{-1}(S) $ is a parabolic strip (see Figure \ref{kup2}).
More specifically, we assumed that the foliation of $ S $ by the vertical lines $ \{ \gamma_c \}_{0 \leq c \leq b} $ (where $ \gamma_c $ is parametrized in Equation \ref{gammac}) is mapped under $ \sigma^{-1} $ into the parabolic foliation $ \{ \sigma^{-1} \gamma_c \}_{0 \leq c \leq b} $ of $ \sigma^{-1}(S) $, 
where $ \sigma^{-1}\gamma_c $ is parametrized in Equation \ref{sigmagammac}.
See Figure \ref{stripfol}.

\begin{figure}[h]
\includegraphics[width=0.9\linewidth, trim={-1cm -0.5cm 4cm -0.5cm}, clip]{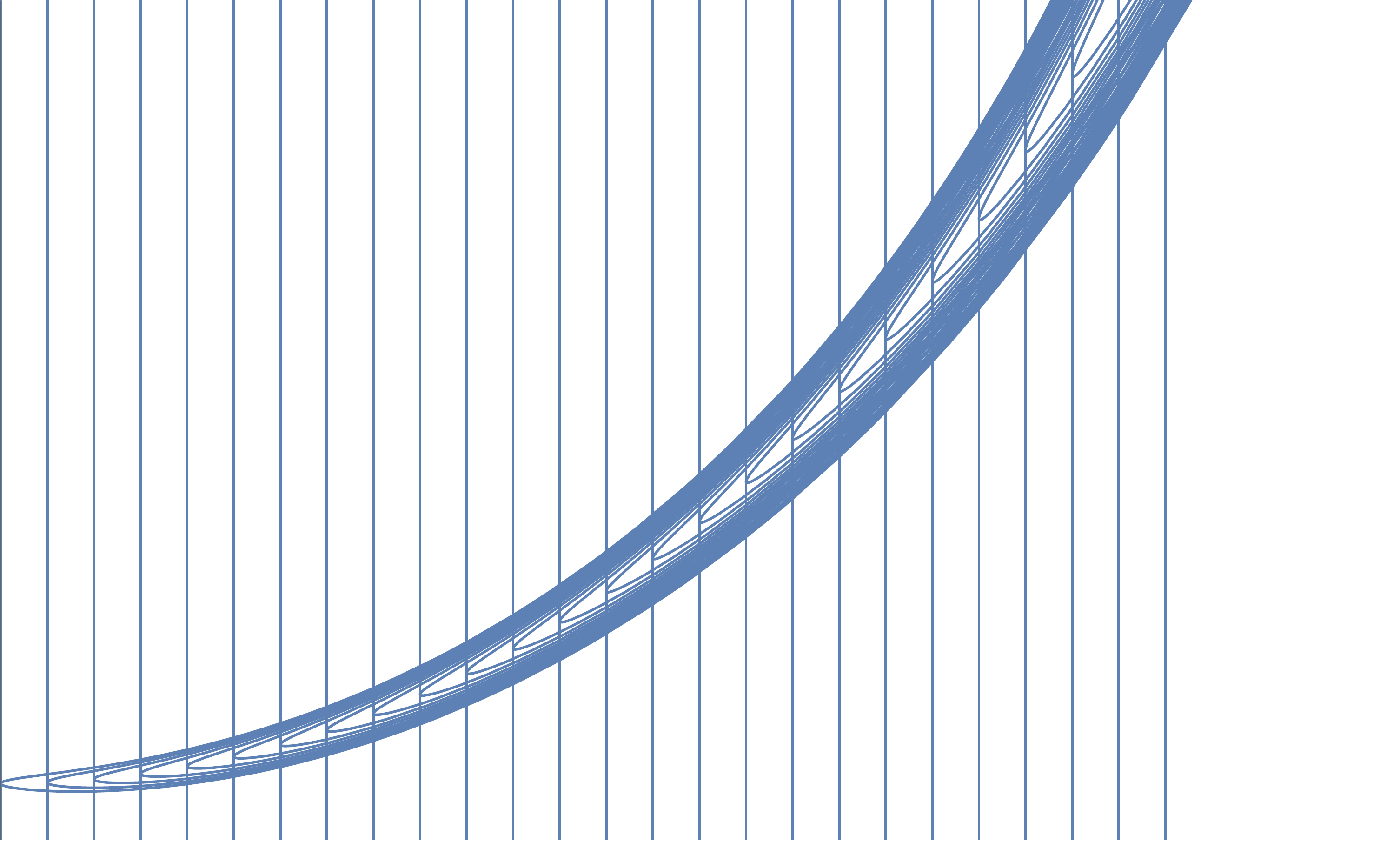}
\caption{The vertical curves $ \gamma_c $ and their images $ \sigma^{-1}\gamma_c $.}
\label{stripfol}
\end{figure}
For each $ 0 \leq c \leq b $, define
\begin{equation}
\label{Nc}
\mathcal{N}_c = \bigcup_{-\infty < t < \infty} \psi_t(\gamma_c).
\end{equation}
Since $ \gamma_0 = \gamma $, this agrees with our definition of $ \mathcal{N}_0 $ given in Equation \ref{MN0}.

For $ 0 \leq c \leq b $, and for $ i=1,2 $ we have a definition of $ \mathcal{N}_{c,1} \cap S $ identical to the definition of $ \mathcal{N}_{0,1} \cap S $ given in Equation \ref{N01S}.
As in Equation \ref{leveldecompNS}, for each $ c $ there is a well-defined level decomposition of $ \mathcal{N}_{c,1} \cap S $.
\begin{equation}
\label{leveldecompc}
\mathcal{N}_{c,1} \cap S = \bigcup_{k \geq 0} \mathcal{N}_{c,1}^k \cap S.
\end{equation}
Each level set $ \mathcal{N}_{c,1}^k \cap S $ is comprised of curves $ \gamma_{c,(i_1,\ldots,i_k)} $ recursively defined by pseudogroup elements as 
\begin{equation}
\label{gammack}
\gamma_{c,(i_1,\ldots,i_k)} = \Phi^{i_k-1} \Theta\left(\gamma_{c,(i_1,\ldots,i_{k-1})}\right),
\end{equation}
exactly as we defined $ \gamma_{i_1,\ldots,i_k} $ in Equation \ref{gammai1ik}.

The symbolic dynamics of the action of $ \langle \Phi, \Theta \rangle $ on $ \mathcal{N}_{c,1} \cap S $ is similar to that of its action on $ \mathcal{N}_{0,1} \cap S $.
For $ c=0 $ and each $ k \geq 1 $ we recover the sequence space $ \Sigma_{b,k} $ from Equation \ref{Sigmak} coding the curves $ \gamma_{i_1,\ldots,i_k} \in \mathcal{N}_{0,1} \cap S $.
For $ 0 < c \leq b $, there is a similar sequence space $ \Sigma_{c,k} $ coding the curves $ \gamma_{c,(i_1,\ldots,i_k)} \in \mathcal{N}_{c,1} \cap S $, but this sequence space has fewer admissible words because the escape times of $ \gamma_c $ under the action of $ \Phi $ decrease as $ c \rightarrow b $. 
This is evident from Figure \ref{stripfol}.

Finally, let $ A_i = \Phi^{i-1} \Theta(S) $, and recursively define
$$
A_{i_1,\ldots,i_k} = \Phi^{i_k-1} \Theta \left(A_{i_1,\ldots,i_{k-1}}\right).
$$
Notice that the admissible words $ \omega $ coding the sets $ A_{\omega} $ are the same as those coding the curves $ \gamma_{\omega} $ because $ \gamma_{\omega} \subset \partial A_{\omega} $, so their escape times are equal.
This is an important point that we will return to later, when defining function systems on the transversal.
The nesting property for curves $ \gamma_{\omega} $ established in Proposition \ref{nestingk} implies that the sets $ A_{\omega} $ are nested.

\begin{proposition}
\label{nestingstrip}
For each $ (i_1,\ldots,i_k) \in \Sigma_k $, we have
$$
A_{i_1,\ldots,i_k} \subset A_{i_2,\ldots,i_k}.
$$
\end{proposition}
See Figure \ref{transstripfig} for a picture of these sets $ A_{\omega} $ for level-one $ \omega $.

\begin{figure}[h]
\includegraphics[width=0.9\linewidth, trim={3.2cm 10.5cm 1cm 2.8cm}, clip]{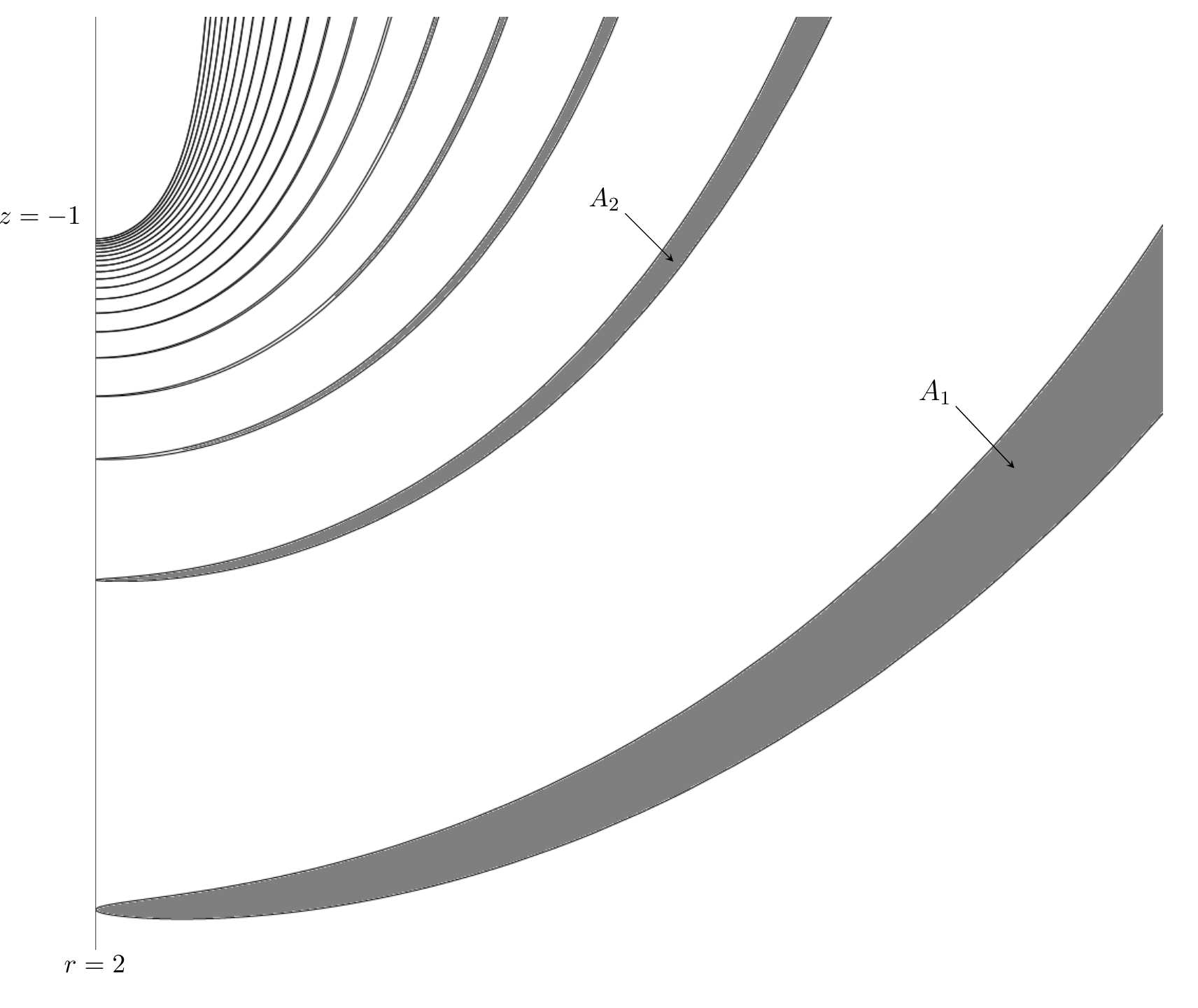}
\caption{The sets $ A_{\omega} $ for $ \omega \in \Sigma_{b,1} $ of level one. Notice that each curve $ \gamma_{\omega} $ is the lower boundary of each $ A_{\omega} $. Compare with Figure \ref{levelone}.}
\label{transstripfig}
\end{figure}

\subsection{Summary of symbolic dynamics}
From Equation \ref{leveldecompNS}, the transverse intersection $ \mathcal{N}_{0,1} \cap S $ has a level decomposition
\begin{equation}
\label{leveldecompNS2}
\mathcal{N}_{0,1} \cap S = \bigcup_{k \geq 0} \mathcal{N}_0^k \cap S.
\end{equation}
Each level set is a collection of curves 
\begin{equation}
\label{Ok}
\mathcal{N}_{0,1}^k \cap S = \bigcup_{\omega \in \Sigma_{b,k}} \gamma_{\omega},
\end{equation}
where $ \Sigma_{b,k} \subset \mathbb{N}^k $ is the space of admissible words of length $ k $ (see Equation \ref{Sigmak}) depending on $ b $, the width of the transverse section $ S $.
Each curve $ \gamma_{\omega} \in \mathcal{N}_{0,1} \cap S $ corresponds to a word $ \omega $ whose length $ |\omega| $ is the level of $ \gamma_{\omega} $.
Define the space of all finite admissible words as
\begin{equation}
\label{Sigmab}
\Sigma_b = \bigcup_{k=0}^{\infty} \Sigma_{b,k},
\end{equation}
where $ \Sigma_{b,0} $ is a singleton (because there is only one curve of level zero, namely $ \gamma $).
Referring to Equation \ref{Sigmak}, a word $ (i_1,\ldots,i_k) $ is in $ \Sigma_{b,k} $ only if $ (i_1,\ldots,i_{k-1}) $ is in $ \Sigma_{b,k-1} $.
By Definition \ref{extadm}, $ \Sigma_b $ satisfies the extension admissibility property, and thus is a general symbolic space as defined in Chapter \ref{Thermo}.

Substituting Equations \ref{Ok} and \ref{Sigmab} into Equation \ref{leveldecompNS2}, we obtain
\begin{equation}
\label{leveldecompgamma}
\mathcal{N}_{0,1} \cap S = \bigcup_{k \geq 0} \bigcup_{\omega \in \Sigma_{b,k}} \gamma_{\omega} = \bigcup_{\omega \in \Sigma_b} \gamma_{\omega}.
\end{equation}

From Equation \ref{MN0}, $ \mathcal{M}_0 $ is the orbit of the curve $ \gamma^u $ which is obtained by restricting the parametrization of $ \gamma $. 
By restricting the parametrization in Equation \ref{leveldecompgamma} we obtain 
\begin{equation}
\label{leveldecompgammau}
\mathcal{M}_{0,1} \cap S = \bigcup_{\omega \in \Sigma_b} \gamma_{\omega}^u.
\end{equation}

The faithful action of the pseudogroup $ \langle \Phi, \Theta \rangle $ on $ \mathcal{N}_{0,1} \cap S $ given in Equation \ref{kuppermutegamma} induces a faithful action on $ \Sigma_b $:
\begin{align}
\label{KupSigmak}
& \Phi : \Sigma_{b,k} \rightarrow \Sigma_{b,k} \qquad \Phi(i_1,\ldots,i_k) = (i_1,\ldots,i_k+1) \\
& \Theta : \Sigma_{b,k} \rightarrow \Sigma_{b,k+1} \qquad \Theta(i_1,\ldots,i_k) = (i_1,\ldots,i_k,1) \nonumber
\end{align}

For each $ 0 \leq c \leq b $ and each curve $ \gamma_c $ in the vertical foliation of $ S $, we have a similar level decomposition of $ \mathcal{N}_{c,1} \cap S $ as a collection of curves coded by a smaller space $ \Sigma_c $ of admissible words.
Together, this gives a level decomposition of $ \bigcup_{t \geq 0} \psi_t(S) \cap S $ in terms of the sets $ A_{\omega} $.

\begin{equation}
\label{leveldecompS}
\bigcup_{-\infty < t < \infty} \psi_t(S) \cap S = \bigcup_{\omega \in \Sigma_b} A_{\omega}.
\end{equation}

\subsection{Dual symbolic dynamics}
\label{dualsymb}
In Chapter \ref{dualcant} we introduced the dual $ \widetilde{\Sigma} $ of a symbolic space $ \Sigma $.
In this section we will compute the admissible words in the dual space $ \widetilde{\Sigma}_{b,k} $.
We first recall the conventions; if $ \omega = (i_1,\ldots,i_k) \in \Sigma_{b,k} $ is an admissible word, then we denote its dual by $ \widetilde{\omega} = (i_k,\ldots,i_1) $.
For any $ k \geq 1 $, the dual of $ \Sigma_{b,k} $ is
$$
\widetilde{\Sigma}_{b,k} = \{ \widetilde{\omega} : \omega \in \Sigma_{b,k} \}.
$$
By Equation \ref{Sigmak},
\begin{equation}
\label{dualSigmak}
\widetilde{\Sigma}_{b,k} = \bigcup_{i_1=N_b}^{\infty} \bigcup_{i_2=N_b}^{M_{i_1}} \cdots \bigcup_{i_k=N_b}^{M_{i_1,\ldots,i_{k-1}}} (i_k,\ldots,i_1).
\end{equation}
The space of all finite dual words is
\begin{equation}
\label{dualSigmab}
\widetilde{\Sigma}_b = \bigcup_{k=0}^{\infty} \widetilde{\Sigma}_{b,k},
\end{equation}
For every $ \omega \in \Sigma_b $ there is a corresponding curve $ \gamma_{\omega} $.
The curve dual to $ \gamma_{\omega} $ is $ \widetilde{\gamma}_{\omega} = \gamma_{\widetilde{\omega}} $.
From the action of $ \langle \Phi, \Theta \rangle $ on $ \Sigma_{b,k} $ shown in Equation \ref{KupSigmak}, we obtain an obviously defined action on $ \widetilde{\Sigma}_b $.
Also, the nesting property for curves $ \gamma $ given in Proposition \ref{nestingk} implies a nesting property for dual curves $ \widetilde{\gamma} $.

\begin{proposition}
\label{dualnestingk}
For each $ (i_1,\ldots,i_k) \in \widetilde{\Sigma}_{b,k} $, the level-$k$ curve $ \gamma_{i_1,\ldots,i_k} $ is nested in the level-$(k-1)$ curve $ \gamma_{i_1,\ldots,i_{k-1}} $.
\end{proposition}

Finally, recall the sets $ A_{\omega} $ coded by $ \omega \in \Sigma_b $ introduced in Chapter \ref{symbrect}.
For each set $ A_{\omega} $ there is a corresponding dual set $ \widetilde{A}_{\omega} = A_{\widetilde{\omega}} $.
These dual sets satisfy a nesting property similar to that in Proposition \ref{nestingstrip}.

\begin{proposition}
\label{dualnestingstrip}
For each $ (i_1,\ldots,i_k) \in \widetilde{\Sigma}_{b,k} $ we have
$$
A_{i_1,\ldots,i_k} \subset A_{i_1,\ldots,i_{k-1}}.
$$
\end{proposition}

\vfill
\eject

\section{Transverse dynamics} 
\label{Transversal}
In the previous chapter we coded the curves in the intersection of the surface $ \mathcal{N}_{0,1} $ with the transverse rectangle $ S $, and studied the pseudogroup dynamics on this intersection.
In this chapter we choose a one-dimensional transversal in $ S $, and study the induced pseudogroup dynamics on its intersection with $ \mathcal{N}_{0,1} $.
Our choice of transversal is the upper boundary $ S^+ $ of $ S $, as defined in Section \ref{InsertAssume}:
\begin{equation}
\label{Sp}
S^+ = \{ (r,\beta,-1+R): 0 \leq r-2 \leq b \}.
\end{equation}
Note that $ S^+ $ can be identified with $ [0,b] $.
We will introduce the \textit{transverse distances} of the curves $ \gamma_{\omega} $ measured along $ S^+ $.
Then we will use the parametrizations of the curves derived in Section \ref{Kupmin} to asymptotically estimate these transverse distances.
These will be important for later estimates of the Hausdorff dimension of the minimal set.

\subsection{The transverse set $ \mathcal{N}_{0,1} \cap S^+ $}
Recall from Equation \ref{gammak} and the remarks afterwards that for each $ k \geq 1 $ and $ \omega \in \Sigma_{b,k} $ there exist unique $ s_{\omega}^{\pm} $ with $ s_{\omega}^- < 0 < s_{\omega}^+ $ such that $ q_{\omega}(s_{\omega}^{\pm}) = R $.
By the definition of $ S^+ $ above and the parametrizations of $ \gamma_{\omega} $ in Equation \ref{gammak}, this is equivalent to $ \gamma_{\omega}(s_{\omega}^{\pm}) \in S^+ $, so each curve $ \gamma_{\omega} $ has two unique points of intersection with $ S^+ $.
Because $ \gamma = \gamma^l \cup \gamma^u $ as defined in Equation \ref{gammauldef}, we see that for all $ k \geq 1 $ and each $ \omega \in \Sigma_{b,k} $, $ \gamma_{\omega}^l $ and $ \gamma_{\omega}^u $ each have one unique intersection point with $ S^+ $.
We define $ a_{\omega}^{\pm} $ as the radial distances of these points from the Reeb cylinder, measured along $ S^+ $.
In coordinates,
\begin{align}
\label{awdef}
\gamma_{\omega}^u \cap S^+ = (2+a_{\omega}^-, \beta, -1+R) \\
\gamma_{\omega}^l \cap S^+ = (2+a_{\omega}^+, \beta, -1+R). \nonumber
\end{align}
With this choice, it is easy to see from the parametrization in Equation \ref{gammak} that $ a_{\omega}^- < a_{\omega}^+ $ for each $ \omega $.

From Equation \ref{leveldecompgamma} we have
\begin{equation}
\label{leveldecompgamma1}
\mathcal{N}_{0,1} \cap S^+ = \bigcup_{\omega \in \Sigma_b} a_{\omega}^{\pm},
\end{equation}
and by Equation \ref{leveldecompgammau} we have
\begin{equation}
\label{transverse1}
\mathcal{M}_{0,1} \cap S^+ = \bigcup_{\omega \in \Sigma_b} a_{\omega}^-.
\end{equation}
From the parametrization in Equation \ref{gammak},
\begin{equation}
\label{a}
a_{i_1,\ldots,i_k}^{\pm} = \left(s_{i_1,\ldots,i_k}^{\mp}\right)^2+\sum_{j=1}^{k-1} q_{i_1, \ldots,i_j}^2\left(s_{i_1,\ldots,i_k}^{\mp}\right)
\end{equation}
for each $ \omega = (i_1,\ldots,i_k) \in \Sigma_{b,k} $.
See Figure \ref{transa} for a picture of $ a_{\omega}^{\pm} $ for words $ \omega \in \Sigma_{b,1} $ of level one.

\begin{figure}[h!]
\includegraphics[width=0.9\linewidth, trim={6cm 7.1cm 1cm 6cm}, clip]{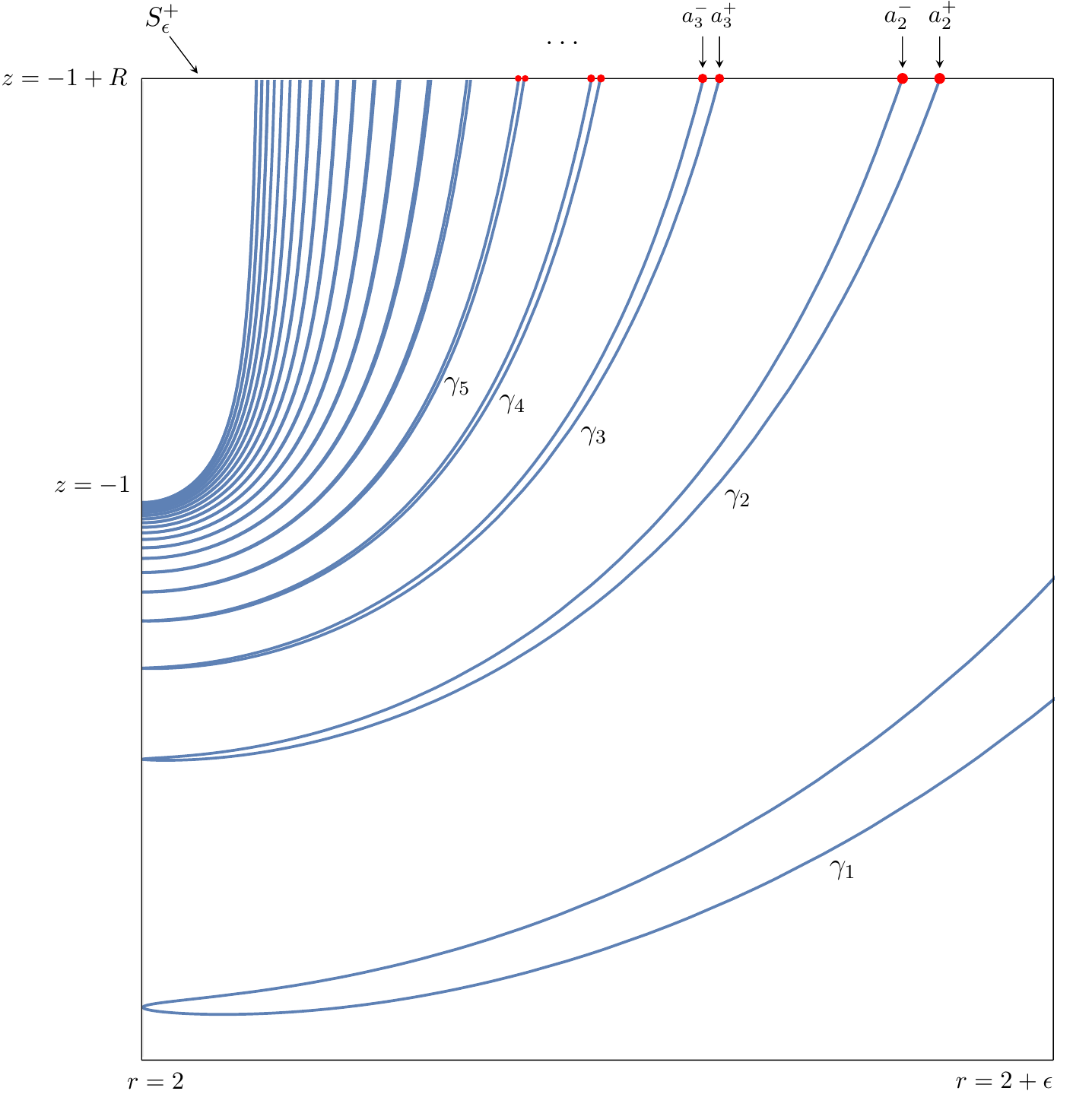}
\caption{The points $ a_i^{\pm} $ as intersections of the level-one curves $ \gamma_i $ with the upper boundary $ S^+ $ of $ S $. In this case $ N_b = 2 $, the minimal value of $ i $ such that $ q_i(s) = R $ has a solution.}
\label{transa}
\end{figure}

In Figures \ref{levelone} and \ref{leveltwo} it appears that $ \gamma_{\omega} $ becomes radially narrower as $ |\omega| \rightarrow \infty $, as does $ \gamma_{i,\omega} $ as $ i \rightarrow \infty $ for $ \omega $ fixed.
In the next section we measure the asymptotics of these widths more precisely.

\subsection{Transverse distances}
Define the function $ a: \Sigma_b \rightarrow \mathbb{R}^+ $ by
$$
a(\omega) = |a_{\omega}^+ - a_{\omega}^-|.
$$
This function gives the transverse width of the curve $ \gamma_{\omega} $ measured along $ S^+ $. 
We say that $ \{ a(\omega) \}_{\omega \in \Sigma_{b,k}} $ are the \textit{transverse distances} of level $ k $.
We will now estimate the transverse distances of each level.

\subsubsection{Transverse distances of level one}
By Equation \ref{a}, we have
\begin{equation}
\label{ai}
a(i) = \left| (s_i^+)^2 - (s_i^-)^2 \right|,
\end{equation}
where $ s_i^{\pm} $ are the unique solutions to $ q_i(s) = R $.
Recall the constants $ C $ and $ K $ from Equation \ref{defCK}.

\begin{proposition}
\label{trans1}
For all $ \delta>0 $ there exists $ L_1 \in \mathbb{N} $ such that for all $ i \geq L_1 $,
$$
\left| a(i)-\left(\frac{\pi^{-1} K^{\frac{3}{2}}}{i^{\frac{5}{2}}} \right) \right| < \frac{\delta}{i^2}.
$$
\end{proposition}

\begin{proof}
Using the parametrization given in Equation \ref{gamma1}, the equation $ q_i(s) = R $ is equivalent to $ f_i(s) = 0 $, where
$$
f_i(s) = 2\pi C + s +\frac{4K}{s^2} \tan^{-1} \left( \frac{R}{s^2} \right) - 2\pi i.
$$ 
So $ s_i^{\pm} $ are the unique roots of $ f_i $.
We now claim that for any $ \delta>0 $ there exists $ N \in \mathbb{N} $ such that for all $ i \geq N $, 
$$ 
s_i^+, -s_i^- \in \left[ \sqrt{ \frac{K(1-\delta)}{i} } , \sqrt{ \frac{K(1+\delta)}{i-1} }, \; \right].
$$
We will prove this for $ s_i^+ $; the proof for $ -s_i^- $ is identical.

First, restrict parameter values $ s $ to the interval 
$$ 
\sqrt{ \frac{K(1-\delta)}{i-C} } < s < \sqrt{ \frac{K(1+\delta)}{i-1-C} }.
$$
We will show that $ f_i $ has a root on this interval; by uniqueness it must be $ s_i^+ $.
Notice as $ i \rightarrow \infty $ that $ s \searrow 0 $ on this interval, so for large enough $ i $, $ \displaystyle \tan^{-1}\left(\frac{R}{s^2}\right) \sim \frac{\pi}{2} $.
From this, we can show for sufficiently large $ i $ that $ f_i^-(s) < f_i(s) < f_i^+(s) $ for all $ s $ on this interval, where
$$
f_i^{\pm}(s) = 2\pi C + s +\frac{2\pi K (1\pm \delta)}{s^2} - 2\pi i.
$$
Note that $ f_i^{\pm} $ are monotonically decreasing, and that 
$$ 
f_i^- \left( \sqrt{\frac{K(1-\delta)}{i-C}} \right)>0, \quad \text{ and } \quad f_i^+ \left( \sqrt{\frac{K(1+\delta)}{i-1-C}} \right) < 0, 
$$
so $ f_i $ must have a root on this interval and we obtain the desired bounds on $ s_i^+ $, after absorbing $ C $ into the constant $ N $.

As an immediate corollary, notice that for any $ \delta>0 $ and sufficiently large $ i $,
\begin{equation}
\label{sisi}
\frac{2\sqrt{K(1-\delta)}}{i^{\frac{1}{2}}} < \left| s_i^+ - s_i^- \right| < \frac{2\sqrt{K(1+\delta)}}{i^{\frac{1}{2}}}
\end{equation}

We now turn to the proof of the proposition.
Substituting the equations $ f_i(s_i^{\pm}) = 0 $ into Equation \ref{ai} yields
$$
a(i) = \left| \frac{4K}{2\pi (i-C) - s_i^+} \tan^{-1} \left(\frac{R}{(s_i^+)^2}\right) - \frac{4K}{2\pi(i-C) - s_i^-} \tan^{-1} \left(\frac{R}{(s_i^-)^2}\right) \right|.
$$
It is easy to show using the parametrization in Equation \ref{gamma1} that $ (s_i^+)^2 > (s_i^-)^2 $.
Applying this to the above expression for $ a(i) $ we obtain
$$
 \tan^{-1} \left(\frac{R}{(s_i^+)^2}\right) < \frac{a(i)}{\left|\frac{4K}{2\pi (i-C) - s_i^+} - \frac{4K}{2\pi(i-C) - s_i^-}\right|} <  \tan^{-1} \left(\frac{R}{(s_i^-)^2}\right).
$$
In light of the bounds we established on $ s_i^{\pm} $ we know that $ s_i^{\pm} \rightarrow 0 $ and therefore that 
$$ 
\tan^{-1} \left(\frac{R}{(s_i^{\pm})^2}\right) \rightarrow \frac{\pi}{2} $$
as $ i \to \infty $.
Then for any $ \delta>0 $,
$$
\frac{K(1-\delta)}{2\pi i^2}\:|s_i^+-s_i^-| < a(i) < \frac{K(1+\delta)}{2\pi i^2} \: |s_i^+-s_i^-|
$$
for sufficiently large $ i $. Combining this with Equation \ref{sisi}, we obtain the desired result.
\end{proof}

From this proof we deduce the following corollary.
\begin{corollary}
\label{limita1}
The following limit exists
$$
\lim_{i \to \infty} a_i^- = 0.
$$
\end{corollary}

\begin{proof}
By Equation \ref{a}, $ a_i^- = (s_i^+)^2 $. From the proof of Proposition \ref{trans1}, $ \lim_{i \to \infty} s_i^+ = 0 $.
\end{proof}
Corollary \ref{limita1} is analytic confirmation of one of the heuristic facts evident in Figure \ref{levelone}; that the level-one curves $ \gamma_i $ limit on the Reeb cylinder $ r=2 $ as $ i \to \infty $.
From this and the nesting properties, we will later deduce that the level-two curves limit on the level-one curves, and inductively that the level-$ k $ curves limit on the level-$ (k-1) $ curves.

\subsubsection{Transverse distances of level two}
By Equation \ref{a}, for all $ (i,j) \in \Sigma_{b,2} $ we have
\begin{equation}
\label{aij}
a(i,j) = \left| (s_{i,j}^+)^2+q_i^2(s_{i,j}^+) - (s_{i,j}^-)^2 - q_i^2(s_{i,j}^-) \right|,
\end{equation}
where $ s_{i,j}^{\pm} $ are the unique solutions to $ q_{i,j}(s) = R $.

\begin{proposition}
\label{trans2}
For all $ \delta>0 $ there exists $ L_2 \in \mathbb{N} $ such that for all $ (i,j) \in \Sigma_{b,2} $ with $ i,j \geq L_2 $,
$$
\left| a(i,j)- \left(\frac{\pi^{-1} K^{\frac{3}{2}}}{j^{\frac{5}{2}}} \cdot \frac{(2\pi)^{-2} aR^2}{i^2} \right)\right| < \frac{\delta}{i^2 j^2}.
$$
\end{proposition}

\begin{proof}
Using the parametrization given in Equation \ref{gamma2}, the equation $ q_{i,j}(s) = R $ is equivalent to $ f_{i,j}(s) = 0 $, where
$$
f_{i,j}(s) = 2\pi C+ q_i(s) + \frac{4K}{s^2+q_i^2(s)} \tan^{-1} \left( \frac{R}{s^2+q_i^2(s)} \right) -2\pi j.
$$
The unique roots of $ f_{i,j} $ are $ s_{i,j}^{\pm} $.

Recall that $ \lim_{s \to 0} q_i(s) = v_i $ by definition and that $ \lim_{i \to \infty} v_i = 0 $. 
Applying this fact to the above expression for $ f_{i,j} $ and using a method similar to that in the proof of Proposition \ref{trans1}, we can show for any $ \delta $ that 
$$ 
s_{i,j}^+, -s_{i,j}^- \in \left[ \sqrt{\frac{K(1-\delta)}{j+1-C}}, \sqrt{\frac{K(1+\delta)}{j-1-C}} \right] 
$$ 
for sufficiently large $ i,j $.
As a corollary we obtain that 
\begin{equation}
\label{sij-sij}
\frac{2\sqrt{K(1-\delta)}}{j^{\frac{1}{2}}} < | s_{i,j}^+ - s_{i,j}^- | < \frac{2\sqrt{K(1+\delta)}}{j^{\frac{1}{2}}}
\end{equation}
for sufficently large $ i,j $.

We now claim that for small enough parameter values $ s $ we have
\begin{equation}
\label{qibound}
\frac{-aR^2-\frac{\delta}{i^2}}{2\pi(i-C)-s +aR+\delta} < q_i(s) < \frac{-aR^2+\frac{\delta}{i^2}}{2\pi(i-C)-s +aR-\delta},
\end{equation}
for sufficiently large $ i $.
To prove this we use the parametrization from Equation \ref{gamma1}:
$$
q_i(s) = \frac{s^2 \tan\left(\frac{s^2}{R^2}T_i(s)\right)-R}{1+\frac{R}{s^2}\tan\left(\frac{s^2}{R^2}T_i(s)\right)}.
$$
For small $ x $, $ \tan x \sim x $. Then for any $ \delta>0 $ we have
$$
\frac{-R-\frac{\delta}{i}}{1+\frac{1}{R}T_i(s)+\delta} < q_i(s) < \frac{-R+\frac{\delta}{i}}{1+\frac{1}{R}T_i(s)-\delta}
$$
for large enough $ i $ and small enough $ s $.
Multiplying the top and bottom of each fraction by $ aR $, substituting $ a T_i(s) = 2\pi(i-C)-s $ from Equation \ref{gamma1} and re-scaling $ \delta $ we obtain the desired bounds.

As an immediate corollary we obtain for any $ \delta>0 $ and small enough parameter values $ u,v $ that
\begin{equation}
\label{qi-qi}
\frac{(2\pi)^{-2} aR^2}{i^2}\: |u-v|-\frac{\delta}{i^2} < |q_i(u)-q_i(v)| < \frac{(2\pi)^{-2} aR^2}{i^2}\: |u-v|+\frac{\delta}{i^2}
\end{equation}
for large enough $ i $.

We now turn to the proof of the proposition. 
Substituting the equations $ f_{i,j}(s_{i,j}^{\pm}) = 0 $ into Equation \ref{aij} yields
{\small
$$
a(i,j)=\left| \frac{4K}{2\pi (j-C) - q_i(s_{i,j}^+)} \tan^{-1} \frac{R}{(s_{i,j}^+)^2+q_i^2(s_{i,j}^+)} - \frac{4K}{2\pi(j-C) - q_i(s_{i,j}^-)} \tan^{-1} \frac{R}{(s_{i,j}^-)^2+q_i^2(s_{i,j}^-)} \right|,
$$
}
and since $ (s_{i,j}^+)^2 > (s_{i,j}^-)^2 $ this implies
{\small
$$
 \tan^{-1} \left(\frac{R}{(s_{i,j}^+)^2+q_i^2(s_{i,j}^+)}\right) < \frac{a(i,j)}{\left|\frac{4K}{2\pi (j-C) - q_i(s_{i,j}^+)} - \frac{4K}{2\pi(j-C) - q_i(s_{i,j}^-)}\right|} <  \tan^{-1} \left(\frac{R}{(s_{i,j}^-)^2+q_i^2(s_{i,j}^-)}\right).
$$
}
But from the bounds we established on $ s_{i,j}^{\pm} $ we know that $ s_{i,j}^{\pm} \rightarrow 0 $ and thus that 
$$ 
\tan^{-1} \left(\frac{R}{(s_{i,j}^{\pm})^2+q_i^2(s_{i,j}^{\pm})}\right) \rightarrow \frac{\pi}{2} 
$$ 
as $ i,j \rightarrow \infty $.
Then we can eliminate the inverse tangent terms and simplify to
$$
\frac{K(1-\delta)}{2\pi j^2} \: \left| q_i(s_{i,j}^+)-q_i(s_{i,j}^-) \right| < a(i,j) < \frac{K(1+\delta)}{2\pi j^2} \: \left| q_i(s_{i,j}^+)-q_i(s_{i,j}^-) \right|.
$$
Substituting in Equation \ref{qi-qi}, we improve the bounds to
{\small
$$
\frac{K(1-\delta)}{2\pi j^2} \: \left( \frac{(2\pi)^{-2} aR^2}{i^2}\: | s_{i,j}^+ - s_{i,j}^- |-\frac{\delta}{i^2} \right) < a(i,j) < \frac{K(1+\delta)}{2\pi j^2} \: \left( \frac{(2\pi)^{-2} aR^2}{i^2}\:| s_{i,j}^+ - s_{i,j}^- |+\frac{\delta}{i^2} \right).
$$
}
Finally, we substitute in Equation \ref{sij-sij} and re-scale $ \delta $ to obtain the desired bounds.
\end{proof}

Using the estimates in the above proof, we now show that the level-two points $ a_{i,j} $ limit on the level-one points $ a_i $ in the following way.

\begin{corollary}
\label{limita2}
For $ (i,j) \in \Sigma_{b,2} $ the limit exists
$$
\lim_{i,j \to \infty} a_{i,j}^- = 0,
$$
and for $ j $ sufficiently large,
$$ 
\lim_{i \to \infty} a_{i,j}^- = a_j^-.
$$
\end{corollary}

\begin{proof}
By Equation \ref{a}, 
$$ 
a_{i,j}^- = (s_{i,j}^+)^2 + q_i^2(s_{i,j}^+).
$$
From the proof of Proposition \ref{trans1} we know that $ \lim_{i,j \to \infty} s_{i,j}^+ = 0 $.
Using this, furthermore we have
$$
\lim_{i,j \to \infty} q_i(s_{i,j}^+) = \lim_{i \to \infty} v_i = 0,
$$
which proves the first statement.
To prove the second statement, we first claim that for a sufficiently large $ j $,
$$
\lim_{i \to \infty} s_{i,j}^+ = s_j^+.
$$
To prove this, recall from the proof of Proposition \ref{trans1} that $ s_j^+ $ is the unique root of $ f_j $ on $ s>0 $, so
$$
f_j(s_j^+) = 2\pi C + s_j^+ +\frac{4K}{(s_j^+)^2} \tan^{-1} \left( \frac{R}{(s_j^+)^2} \right) - 2\pi j = 0,
$$
and from the proof of Proposition \ref{trans2} that $ s_{i,j}^+ $ is the unique root of $ f_{i,j} $, so
$$
f_{i,j}(s_{i,j}^+) = 2\pi C+ q_i(s_{i,j}^+) + \frac{4K}{(s_{i,j}^+)^2+q_i^2(s_{i,j}^+)} \tan^{-1} \left( \frac{R}{(s_{i,j}^+)^2+q_i^2(s_{i,j}^+)} \right) -2\pi j=0.
$$
For sufficiently large $ j $, $ \lim_{i \to \infty} q_i(s_{i,j}^+) =0 $ from the proof of Proposition \ref{trans2}.
Using this and comparing the above parametrizations, we see that $ \lim_{i \to \infty} s_{i,j}^+ $ is a root of $ f_j $ for sufficiently large $ j $.
Since the root of $ f_j $ is unique on $ s>0 $ and equals $ s_j^+ $, we obtain the desired result.
\end{proof}

\subsubsection{Transverse distances of level $ k $}

\begin{proposition}
\label{transk}
For all $ \delta>0 $ there exists $ L_k \in \mathbb{N} $ such that for all $ (i_1,\ldots,i_k) \in \Sigma_{b,k} $ with $ i_1,\ldots,i_k \geq L_k $,
$$
\left| a(i_1,\ldots,i_k)- \left(\frac{\pi^{-1} K^{\frac{3}{2}}}{i_k^{\frac{5}{2}}} \cdot \frac{\left((2\pi)^{-2} aR^2\right)^{k-1}}{i_1^2 \cdots i_{k-1}^2} \right)\right| < \frac{\delta}{i_1^2 \cdots i_k^2}.
$$
\end{proposition}

\begin{proof}
By the parametrization given in Equation \ref{gammak}, the equation $ q_{i_1,\ldots,i_k}(s) = R $ is equivalent to $ f_{i_1,\ldots,i_k}(s) = 0 $, where
{\small
$$
f_{i_1,\ldots,i_k}(s) = 2\pi C+q_{i_1,\ldots,i_{k-1}}(s) + \frac{4K}{s^2+\sum_{j=1}^{k-1} q_{i_1, \ldots,i_j}^2(s)} \tan^{-1} \left(\frac{R}{s^2+\sum_{j=1}^{k-1} q_{i_1, \ldots,i_j}^2(s)} \right)- 2\pi i_k,
$$
}

Recall that $ \lim_{s \to 0} q_{i_1,\ldots,i_{k-1}}(s) = v_{i_1,\ldots,i_{k-1}} $ and that $ \lim_{i_1,\ldots,i_{k-1} \to \infty} v_{i_1,\ldots,i_{k-1}} =0 $.
Applying this to the above expression for $ f_{i_1,\ldots,i_k} $ we can show that 
$$ 
s_{i_1,\ldots,i_k}^+, -s_{i_1,\ldots,i_k}^- \in \left[ \sqrt{\frac{K(1-\delta)}{i_k+1-C}}, \sqrt{\frac{K(1+\delta)}{i_k-1-C}} \right] 
$$ 
for sufficiently large $ i_1,\ldots,i_k $.
As a corollary we obtain that 
\begin{equation}
\label{sik-sik}
\frac{2\sqrt{K(1-\delta)}}{i_k^{\frac{1}{2}}} < | s_{i_1,\ldots,i_k}^+ - s_{i_1,\ldots,i_k}^- | < \frac{2\sqrt{K(1+\delta)}}{i_k^{\frac{1}{2}}}
\end{equation}
for sufficently large $ i_1,\ldots,i_k $.

We now claim that for sufficiently small parameter values $ s $ and sufficiently large values of $ i_1,\ldots,i_k $ we have
{\small
$$
\frac{-aR^2-\frac{\delta}{i_1^2 \cdots i_k^2}}{2\pi(i_k-C)-q_{i_1,\ldots,i_{k-1}}(s)+aR+\delta} < q_{i_1,\ldots,i_k}(s) < \frac{-aR^2+\frac{\delta}{i_1^2 \cdots i_k^2}}{2\pi(i_k-C)-q_{i_1,\ldots,i_{k-1}}(s)+aR-\delta},
$$
}
for sufficiently large $ i_1,\ldots,i_k $.
The proof of this uses the expression for $ q_{i_1,\ldots,i_k} $ in terms of $ q_{i_1,\ldots,i_{k-1}} $ given in Equation \ref{gammak}, together with precisely the same method of proof as the corresponding claim in the proof of Proposition \ref{trans2}.
As a corollary, we have for sufficiently large $ i_1, \ldots, i_k $ and small enough parameter values $ u,v $ that
{\small
\begin{multline}
\frac{(2\pi)^{-2} aR^2}{i_k^2}\: |q_{i_1,\ldots,i_{k-1}}(u)-q_{i_1,\ldots,i_{k-1}}(v)|-\frac{\delta}{i_1^2 \cdots i_k^2} < |q_{i_1,\ldots,i_k}(u)-q_{i_1,\ldots,i_k}(v)| \nonumber \\
< \frac{(2\pi)^{-2} aR^2}{i_k^2}\: |q_{i_1,\ldots,i_{k-1}}(u)-q_{i_1,\ldots,i_{k-1}}(v)|+\frac{\delta}{i_1^2 \cdots i_k^2}.
\end{multline}
}
From this recursive expression we can prove by induction on $ k $ for sufficiently large $ i_1,\ldots,i_k $ and small $ u,v $ that
{\small
\begin{equation}
\label{qikinduct}
\frac{\left((2\pi)^{-2} aR^2\right)^k}{i_1^2 \cdots i_k^2}\: |u-v|-\frac{\delta}{i_1^2 \cdots i_k^2} < |q_{i_1,\ldots,i_k}(u)-q_{i_1,\ldots,i_k}(v)| \\
< \frac{\left((2\pi)^{-2} aR^2\right)^k}{i_1^2 \cdots i_k^2}\: |u-v|+\frac{\delta}{i_1^2 \cdots i_k^2}
\end{equation}
}

We now turn to the proof of the proposition.
Substituting the equations $ f_{i_1,\ldots,i_k}(s_{i_1,\ldots,i_k}^{\pm}) = 0 $ into Equation \ref{a} and eliminating the inverse tangent terms as in the proof of Proposition \ref{trans2} yields
{\small
\begin{multline}
\frac{K(1-\delta)}{2\pi i_k^2} \: \left| q_{i_1,\ldots,i_{k-1}}(s_{i_1,\ldots,i_k}^+)-q_{i_1,\ldots,i_{k-1}}(s_{i_1,\ldots,i_k}^-) \right| < a(i_1,\ldots,i_k) \nonumber \\
< \frac{K(1+\delta)}{2\pi i_k^2} \: \left| q_{i_1,\ldots,i_{k-1}}(s_{i_1,\ldots,i_k}^+)-q_{i_1,\ldots,i_{k-1}}(s_{i_1,\ldots,i_k}^-) \right| \nonumber.
\end{multline}
}
Substituting Equation \ref{qikinduct} we obtain
{\small
\begin{multline}
\frac{K(1-\delta)}{2\pi i_k^2} \: \left( \frac{\left((2\pi)^{-2} aR^2\right)^{k-1}}{i_1^2 \cdots i_{k-1}^2}\: |u-v|-\frac{\delta}{i_1^2 \cdots i_{k-1}^2} \right) < a(i_1,\ldots,i_k) \nonumber \\
< \frac{K(1+\delta)}{2\pi i_k^2} \: \left( \frac{\left((2\pi)^{-2} aR^2\right)^{k-1}}{i_1^2 \cdots i_{k-1}^2}\: |u-v|+\frac{\delta}{i_1^2 \cdots i_{k-1}^2} \right) \nonumber.
\end{multline}
}
Finally, we substitute in Equation \ref{sik-sik} and re-scale $ \delta $ to obtain the desired bounds.
\end{proof}

The proof of the following corollary is a straightforward generalization of the proof of Corollary \ref{limita2}.

\begin{corollary}
\label{limitak}
For $ (i_1,\ldots,i_k) \in \Sigma_{b,k} $ the limit exists
$$
\lim_{i_1,\ldots,i_k \to \infty} a_{i_1,\ldots,i_k}^- = 0,
$$
and for $ \omega = (i_1,\ldots,i_k) $ with $ i_1,\ldots,i_k $ sufficiently large,
$$ 
\lim_{j \to \infty} a_{j, \omega}^- = a_{\omega}^-.
$$
\end{corollary}

\subsection{The projection action}
\label{projaction}
In Chapter \ref{Kupmin}, we exhibited a faithful action of $ \Psi_1 = \langle \Phi, \Theta \rangle $ on $ \mathcal{N}_{0,1} \cap S $. 
This does not restrict to an action on the transversal $ \mathcal{N}_{0,1} \cap S^+ $, because the maps $ \Phi $ and $ \Theta $ do not preserve $ S^+ $. 
Nevertheless, we will show that the $ \langle \Phi, \Theta \rangle $ acts faithfully on $ \mathcal{N}_{0,1} \cap S^+ $ in this section.

By Theorem \ref{minchar}, $ \mathcal{N} \cap S $ is a codimension-one lamination in $ S $, and $ \mathcal{N}_{0,1} \cap S $ is a collection of leaves of this lamination.
By Equation \ref{leveldecompgamma}, the leaves in $ \mathcal{N}_{0,1} \cap S $ are the collection of curves $ \gamma_{\omega} $ indexed by $ \omega \in \Sigma_b $.

To obtain a faithful action of $ \langle \Phi, \Theta \rangle $ on $ \mathcal{N}_{0,1} \cap S^+ $, we will project to $ S^+ $ along these leaves.
We will call this the \textit{projection action} of $ \langle \Phi, \Theta \rangle $ on $ S^+ $.
To define this action, we will first define the projection maps along the leaves.

\subsubsection{The projection maps}
By Equation \ref{awdef}, each curve $ \gamma_{\omega} $ has two unique intersections with $ S^+ $, whose radial coordinates are $ 2+a_{\omega}^{\pm} $.
Furthermore, each $ \gamma_{\omega} $ has a vertex $ v_{\omega} = \gamma_{\omega}(0) $, as defined in Section \ref{Kupmin}.
For each $ \omega \in \Sigma_b $ we have maps
\begin{equation}
\label{pmaps}
p_{\omega}^{\pm} : a_{\omega}^{\pm} \mapsto v_{\omega},
\end{equation}
and each map $ p_{\omega}^{\pm} $ has a well-defined inverse.

Each leaf $ \gamma_{\omega} $ is the intersection with $ S $ of the orbit of the smooth curve $ \gamma $ under the $ C^{\infty} $ flow $ \psi_t $.
The surface $ S $ is transverse to the flow, so each $ \gamma_{\omega} $ is a $ C^{\infty} $ submanifold of codimension one in $ S $.
Each $ \gamma_{\omega} $ is covered by a finite number of charts of the lamination in $ S $, and each map $ p_{\omega}^{\pm} $ is a finite composition of transition maps of these charts, which are $ C^{\infty} $.
As a consequence, the maps $ p_{\omega}^{\pm} $ are in the holonomy of this lamination and are smooth projections along its leaves.
See Figure \ref{projfig} for a picture of the projection along curves $ \gamma_i $ of level one.

\begin{figure}[h]
\includegraphics[width=1.1\linewidth, trim={4.5cm 8.5cm 0 5.8cm}, clip]{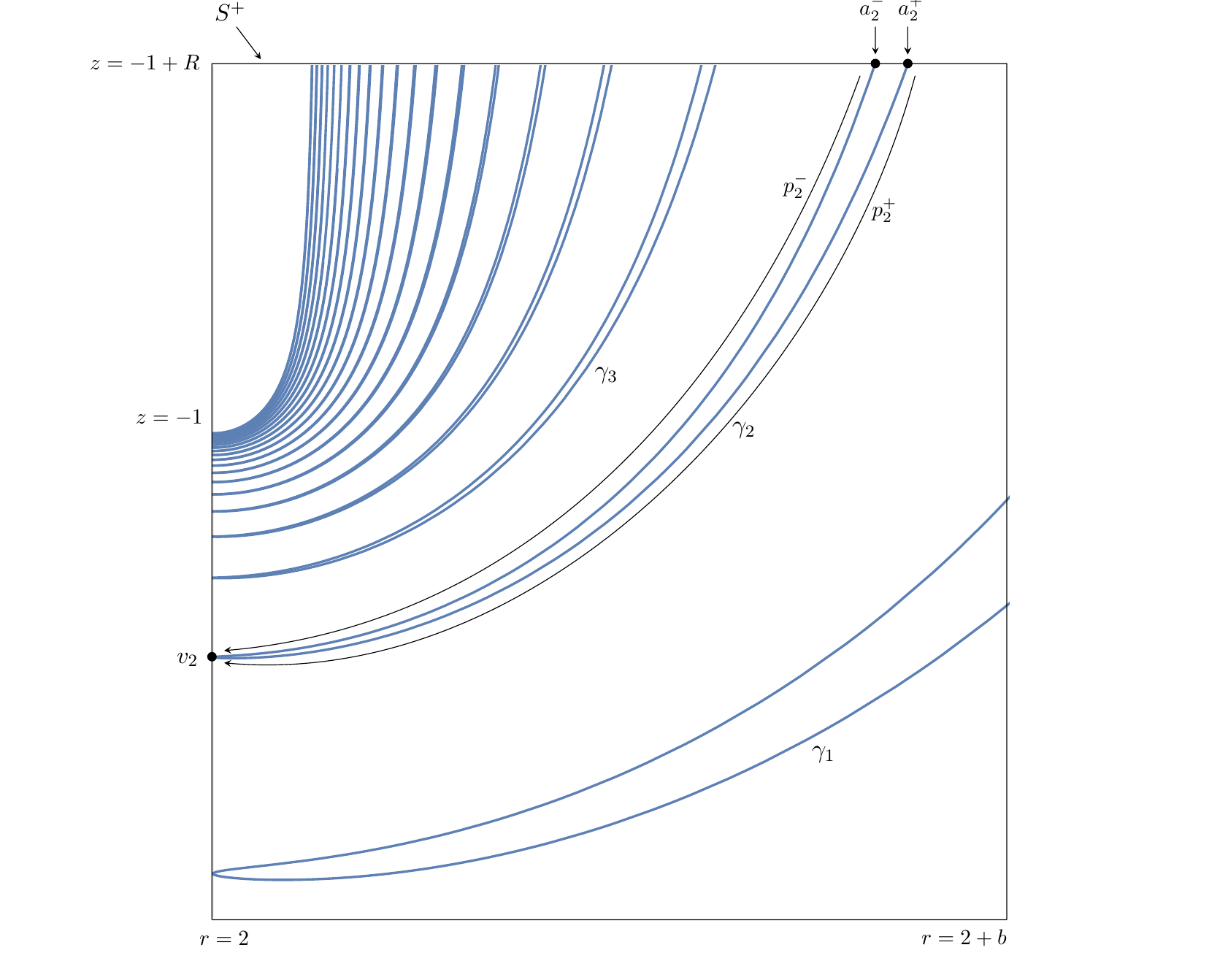}
\caption{The projection maps $ p_2^{\pm} $ projecting the points $ a_2^{\pm} $ along the curve $ \gamma_2 $ to its vertex $ v_2 $.}
\label{projfig}
\end{figure}

\subsubsection{The projection action}
Notice that the vertices $ \bigcup_{\omega \in \Sigma_b} v_{\omega} $ are preserved by $ \langle \Phi, \Theta \rangle $.
By Equation \ref{KupSigmak}, the action on the vertices is
\begin{align}
\label{projvert}
\Phi(v_{i_1,\ldots,i_k}) &= v_{i_1,\ldots,i_k+1} \\
\Theta(v_{i_1,\ldots,i_k}) &= v_{i_1,\ldots,i_k,1}. \nonumber
\end{align}
To define the projection action of $ \langle \Phi, \Theta \rangle $ on $ \mathcal{N}_0 \cap S^+ = \bigcup_{\omega \in \Sigma_b} a_{\omega}^{\pm} $, we conjugate the above action by the projection maps.
\begin{align*}
\Phi \cdot a_{i_1,\ldots,i_k}^{\pm} &= \left( p_{i_1,\ldots,i_k+1}^{\pm} \right)^{-1} \Phi \: p_{i_1,\ldots,i_k}^{\pm} (a_{i_1,\ldots,i_k}^{\pm}) \\
\Theta \cdot a_{i_1,\ldots,i_k}^{\pm} &= \left( p_{i_1,\ldots,i_k,1}^{\pm} \right)^{-1} \Theta \: p_{i_1,\ldots,i_k}^{\pm} (a_{i_1,\ldots,i_k}^{\pm}) \nonumber
\end{align*}
Combining this with Equations \ref{pmaps} and \ref{projvert} we see that
\begin{align}
\label{paction}
\Phi \cdot a_{i_1,\ldots,i_k}^{\pm} &= a_{i_1,\ldots,i_k+1}^{\pm} \\
\Theta \cdot a_{i_1,\ldots,i_k}^{\pm} &= a_{i_1,\ldots,i_k,1}, \nonumber
\end{align}
so the symbolic dynamics of this action on the transversal $ S^+ $ is the same as that of the action on the section $ S $ given in Equation \ref{kuppermutegamma}.

\subsection{Dual transverse distances}
\label{dualtrans}
In Section \ref{dualsymb} we defined the dual space $ \widetilde{\Sigma}_b $ and obtained nesting properties for the curves $ \gamma_{\omega} $ and sets $ A_{\omega} $ when $ \omega \in \widetilde{\Sigma}_b $.
Thus from any statement about $ a_{\omega}^{\pm} $ or $ a(\omega) $ we have a dual version of the statement.
For later use we record two such versions below. 
The first is the dual version of Proposition \ref{transk}.

\begin{proposition}
\label{dualtransk}
For all $ \delta>0 $ there exists $ L_k \in \mathbb{N} $ such that for all $ (i_1,\ldots,i_k) \in \widetilde{\Sigma}_{b,k} $ with $ i_1,\ldots,i_k \geq L_k $,
$$
\left| a(i_1,\ldots,i_k)- \left(\frac{\pi^{-1} K^{\frac{3}{2}}}{i_1^{\frac{5}{2}}} \cdot \frac{\left((2\pi)^{-2} aR^2\right)^{k-1}}{i_2^2 \cdots i_k^2} \right)\right| < \frac{\delta}{i_1^2 \cdots i_k^2}.
$$
\end{proposition}

The second is the dual version of Corollary \ref{limitak}.

\begin{corollary}
\label{duallimitak}
For $ (i_1,\ldots,i_k) \in \widetilde{\Sigma}_{b,k} $ the limit exists
$$
\lim_{i_1,\ldots,i_k \to \infty} a_{i_1,\ldots,i_k}^- = 0,
$$
and for $ \omega = (i_1,\ldots,i_k) $ with $ i_1,\ldots,i_k $ sufficiently large,
$$ 
\lim_{j \to \infty} a_{\omega,j}^- = a_{\omega}^-,
$$
and by induction,
$$
\lim_{j_{k+1},\ldots,j_{k+n} \to \infty} a_{\omega, j_{k+1}, \ldots,j_{k+n}} = a_{\omega}^-
$$
for any $ n \geq 1 $.
\end{corollary}

\vfill
\eject

\section{$ C^{1+\alpha} $ function systems on the transversal}
\label{FunctionSystems}
In this chapter we use the pseudogroup $ \Psi_1 = \langle \Phi, \Theta \rangle $ and the projection maps from Chapter \ref{projaction} to define a function system on the transversal $ S^+ $.
By Equation \ref{Sp}, $ S^+ $ can be identified with $ [0,b] $ via the map 
$$ 
(r, \beta, -1+R) \mapsto r-2.
$$
In this coordinate system, Equation \ref{awdef} reads simply 
$$ 
\gamma_{\omega} \cap S^+ = a_{\omega}^{\pm},
$$
so for ease of notation we will frequently use this coordinate system.

The function system we will define will be a $ C^{1+\alpha} $ general function system on $ [0,b] $ modeled by a general symbolic space in the sense of Chapter \ref{genfun}.
Furthermore, we will prove that for sufficiently small $ \epsilon>0 $, this function system has a pseudo-Markov subsystem on $ [0,\epsilon] \subset [0,b] $, as studied in Chapter \ref{GDPM}.
These function systems will be related to the transverse Kuperberg minimal set in Chapter \ref{Transcant}.

\subsection{A $ C^{1+\alpha} $ function system on $ [0,b] $}
The domain of the projection maps defined in Chapter \ref{projaction} is $$ 
\bigcup_{\omega \in \Sigma_b} a_{\omega}^{\pm} \subset [0,b].
$$
To define a function system on $ [0,b] $, we will need to project along curves $ \gamma_{c,\omega} $ in the parabolic foliations of $ A_{\omega} $ studied in Chapter \ref{symbrect}.

\subsubsection{Extension of the projection maps}
Recall the foliation of $ S $ by vertical lines $ \{ \gamma_c \}_{0 \leq c \leq b} $ parametrized in Equation \ref{gammac}.
Then by Equation \ref{leveldecompc} and the subsequent remarks, we have a level decomposition
$$
\mathcal{N}_{c,1} \cap S = \bigcup_{\omega \in \Sigma_c} \gamma_{c,\omega}.
$$
For each $ \omega \in \Sigma_c $, the curve $ \gamma_{c,\omega} $ has intersection points $ a_{c,\omega}^{\pm} \in [0,b] $, and vertex $ v_{c,\omega} = \gamma_{c,\omega}(0) $.
We extend the projections $ p_{\omega}^{\pm} $ to projections $ p_{c,\omega}^{\pm} $ along the curves $ \gamma_{c,\omega} $ in the same way as in Equation \ref{pmaps}.

\begin{equation}
\label{pcmaps}
p_{c,\omega}^{\pm} : a_{c,\omega}^{\pm} \mapsto v_{c,\omega}
\end{equation}

\subsubsection{Preliminary steps}
The definition of a general function system modeled on a symbolic space, as given in Section \ref{genfun}, has some preliminary steps.
Namely, a compact space $ X \subset [0,1] $, a countable alphabet $ E $, and for each $ i \in E $ a $ C^{1+\alpha} $ map $ f_i : X \rightarrow X $ with Lipschitz constant $ < 1 $ and images $ \Delta_i = f_i(X) $ satisfying the separation property
$$
\Delta_i \cap \Delta_j = \emptyset \; \text{ when } \; i \neq j.
$$

In our setting, we let $ X = S^- \cup S^+ $, where $ S^{\pm} $ are the upper and lower boundaries of $ S $, from Equation \ref{Spm}.
Because $ S^{\pm} $ are both identified with $ [0,b] $, the space $ X $ is naturally identified with two disjoint copies of $ [0,b] $.
Let $ N_b \in \mathbb{N} $ be the constant defined in the proof of Proposition \ref{gamma1param}, and let 
$$ 
E = \Sigma_{b,1} = \{ N_b, N_b+1, \ldots \} 
$$ 
as defined in Equation \ref{Sigma1}.

To define $ f_i : X \rightarrow X $, we define $ f_i(c) $ for $ c $ on each interval $ S^{\pm} $ separately.
If $ c \in S^+ $ then by Equation \ref{gammac}, $ c= a_c^- $, the unique upper endpoint of $ \gamma_c $.
If $ c \in S^- $ then $ c= a_c^+ $, the unique lower endpoint of $ \gamma_c $.
We now define
\begin{equation}
\label{fi}
f_i(a_c^{\pm}) = (p_{c,i}^{\pm})^{-1} \: \Phi^{i-1} \: \Theta \: p_c^{\pm} (a_c^{\pm}).
\end{equation}
In words, $ f_i $ first projects $ a_c^{\pm} $ to the vertex $ v_c $ of $ \gamma_c $, then follows the orbit of $ v_c $ through the insertion to its first intersection $ \Theta(v_c) $ with $ S $.
It then follows the orbit of $ \Theta(v_c) $ to its its $ (i-1) $-th return to $ S $ under $ \Phi $.
By construction, this is the vertex $ v_{c,i} $ of $ \gamma_{c,i} $ which is then inversely projected back along $ \gamma_{c,i} $ to its intersection $ a_{c,i}^{\pm} $ with $ S^+ = [0,b] $.

For any $ i \in \Sigma_{b,1} $, recall from Chapter \ref{symbrect} that $ A_i = (\Phi^{i-1} \Theta) (S) $.
For each $ i $, the curves $ \{\gamma_{c,i} \}_c $ form a parabolic foliation of $ A_i $ (See figures \ref{stripfol} and \ref{transstripfig}).
From this and the definition of the extended projection maps given in Equation \ref{pcmaps}, we see that
$$
f_i(X) = A_i \cap S^+.
$$
Denote $ \Delta_i = f_i(X) $ and note that each $ \Delta_i $ is a closed interval.
Since $ \gamma_i \subset \partial A_i $, and $ a_i^{\pm} $ are the unique intersection points of $ \gamma_i $ with $ [0,b] $, we have
$$
\Delta_i = [a_i^-, a_i^+],
$$
so that $ |\Delta_i| = a(i) $, the transverse distances of level one studied in Section \ref{Transversal}.
See Figure \ref{adeltafig}.

\begin{figure}[h]
\includegraphics[width=1\linewidth, trim={5.5cm 6.8cm 0 6cm}, clip]{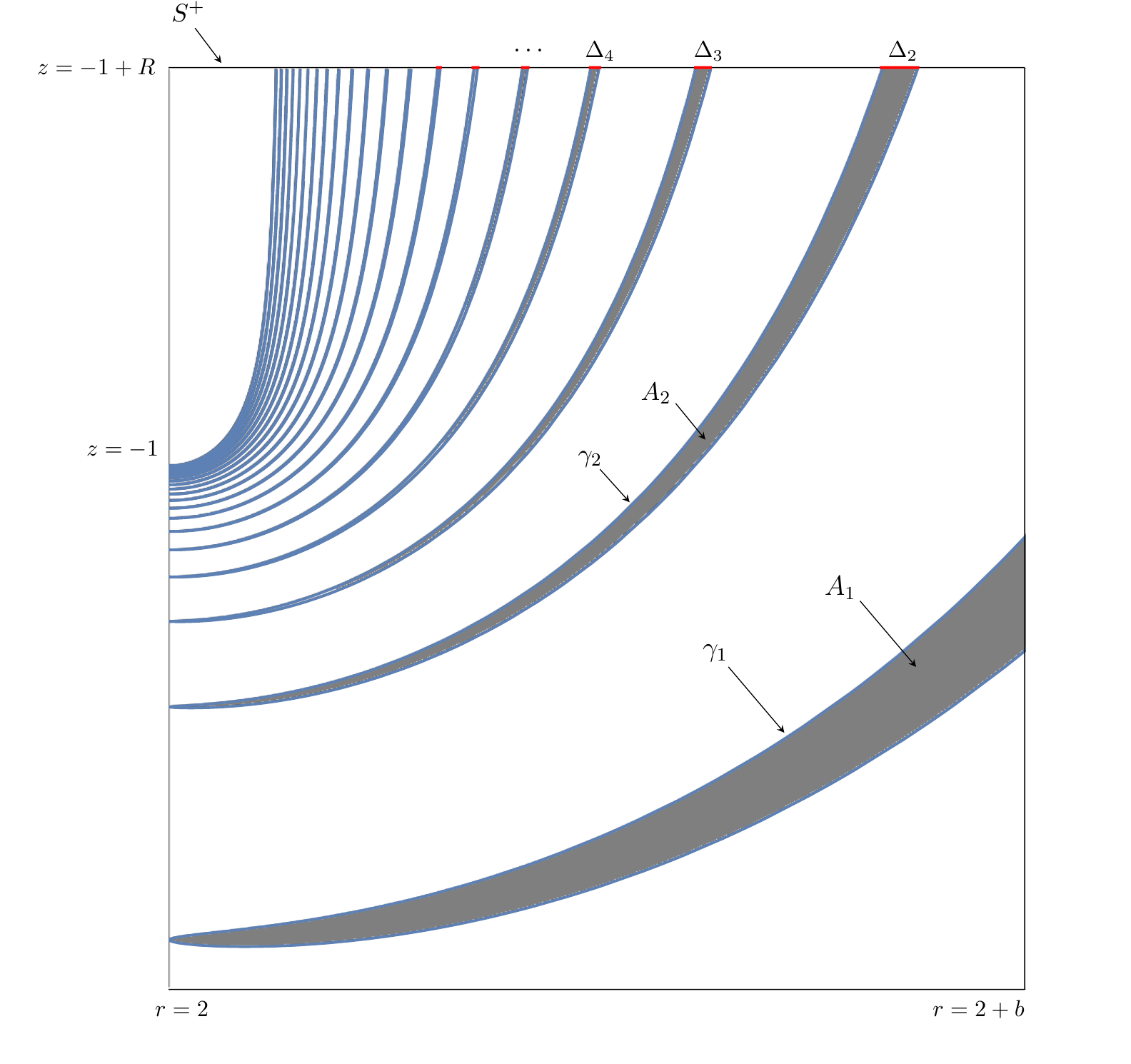}
\caption{The sets $ A_i $ and their intersection intervals $ \Delta_i $ with $ S^+ = [0,b] $.}
\label{adeltafig}
\end{figure}

We now show that $ f_i $ satisfies the properties we imposed in Section \ref{genfun}.

\begin{itemize}

\item \textit{Uniform contraction}:
For all $ i \in \Sigma_{b,1} $, the maps $ f_i $ have a uniform Lipschitz constant $ 0 < s < 1 $.
\vspace{0.1cm}

In fact, more is true. 
First, note that as $ C^1 $ maps of a compact space, each $ f_i $ is individually Lipschitz by the mean value theorem.
Let $ c, c' \in S^+ $.
Then
$$
|f_i(c) - f_i(c')| \leq |\Delta_i| = a(i).
$$
By Proposition \ref{trans1}, $ a(i) \sim i^{-\frac{5}{2}} \rightarrow 0 $ as $ i \rightarrow \infty $.
So as $ i $ increases, the Lipschitz constant of $ f_i $ becomes arbitrarily small.
See Figure \ref{fifig} for a picture of this.
Thus setting $ s $ to be the Lipschitz constant of $ f_1 $ suffices for our purposes.
\vspace{0.2cm}

\begin{figure}[h]
\minipage{0.55\textwidth}
\hspace{-1.5cm}
\includegraphics[width=1.1\linewidth, trim={2cm 7.7cm 4cm 1.8cm}, clip]{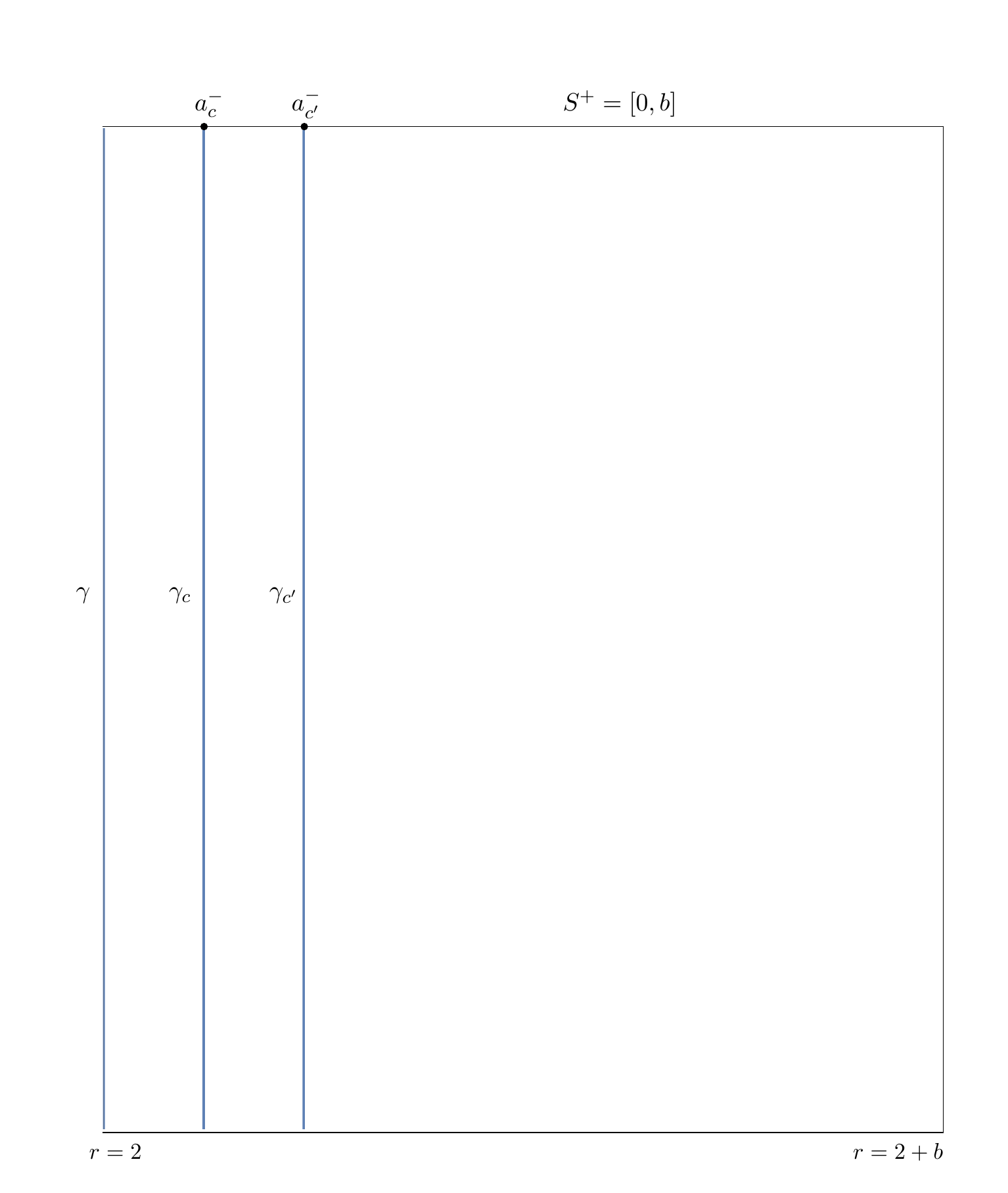}
  \caption*{Two points $ a_c^-, a_{c'}^- \in S^+ = [0,b] $ as endpoints of the vertical segments $ \gamma_c, \gamma_{c'} $.}
\endminipage
\minipage{0.55\textwidth}%
\hspace{-1cm}
\includegraphics[width=1.1\linewidth, trim={2cm 7cm 4cm 3.2cm}, clip]{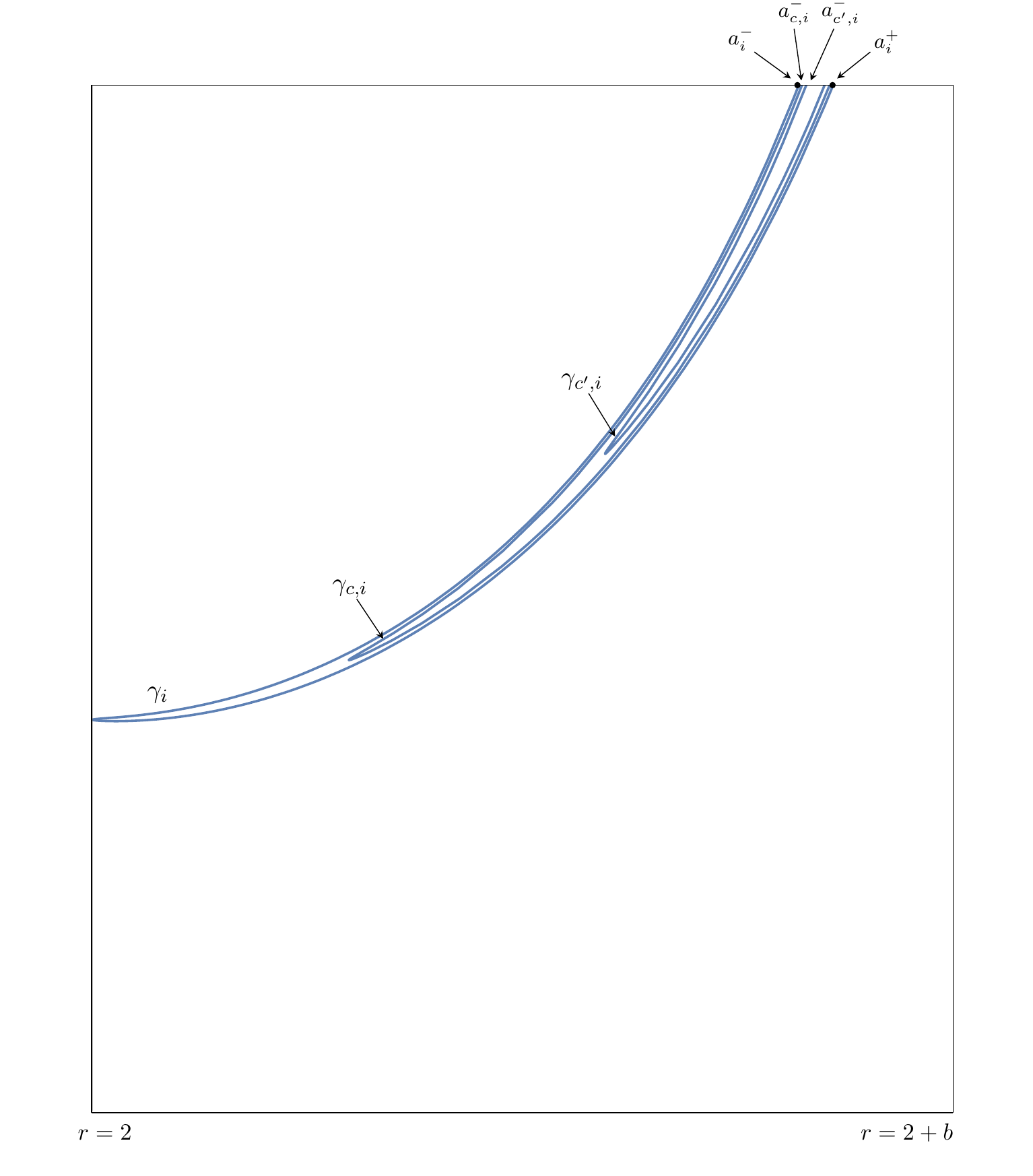}
  \caption*{The images $ a_{c,i}^-, a_{c',i}^- \in [a_i^-, a_i^+] $ under $ f_i $ of $ a_c^-, a_{c'}^- $, respectively.}
\endminipage
\caption{}
\label{fifig}
\end{figure}

\item \textit{$ C^{1+\alpha} $ regularity}: There exists $ \alpha > 0 $ such that for all $ i \in \Sigma_{b,1} $, the maps $ f_i $ have regularity $ C^{1+\alpha} $.
\vspace{0.1cm}

Recall from Section \ref{projaction} that the projection maps $ p_{\omega}^{\pm} $ are $ C^{\infty} $; this argument also holds for the maps $ p_{c,\omega}^{\pm} $.
The maps $ \Theta, \Phi $ are in the holonomy of the Kuperberg flow and as such are also $ C^{\infty} $.
By Equation \ref{fi}, the maps $ f_i $ are compositions of these and as such are $ C^{\infty} $.
By the mean value theorem, a $ C^{\infty} $ map of a compact interval has a uniform bound on its second derivative.
This demonstrates that for all $ i \in \Sigma_{b,1} $, $ f_i $ are uniformly $ C^{1+\alpha} $ for $ \alpha = 1 $.
\vspace{0.2cm}

\item \textit{Separation property}:
As $ \Delta_i = A_i \cap S^+ $, the sets $ \Delta_i $ are pairwise disjoint because the sets $ A_i $ are.
\vspace{0.2cm}
\end{itemize}

\subsubsection{The function system on $ [0,b] $}
In this section we will use the spaces $ \Delta_i $ defined above to define a general function system modeled by a symbolic space of infinite type, in the sense of Definition \ref{genfundef}.

The dual space $ \widetilde{\Sigma}_b $ defined in Equation \ref{dualSigmab} is a general symbolic space, and thus has an infinite extension $ \widetilde{\Sigma}_b^{\infty} $ (see Definition \ref{infext}) which is a symbolic space of infinite type.
We now define a general function system
$$
\{ \phi_{i,j} : \Delta_j \rightarrow X \}_{(i,j) \in \widetilde{\Sigma}_{b,2}}
$$
modeled by $ \widetilde{\Sigma}_b^{\infty} $.

The collection of curves $ \{ \gamma_{c,j} \}_{0 \leq c \leq b} $ forms a parabolic foliation of each $ A_j $ (see figure \ref{stripfol}).
Each point $ x \in \Delta_j $ is the unique intersection point $ a_{c,j}^{\pm} $ of one of these curves $ \gamma_{c,j} $ with the upper boundary $ S^+ = [0,b] $.
For any $ i,j \geq N_b $ we define maps $ \phi_{i,j} : \Delta_j \rightarrow [0,b] $ as follows.

\begin{equation}
\label{phiij}
\phi_{i,j}(a_{c,j}^{\pm}) = \left( p_{c,(j,i)}^{\pm} \right)^{-1} \: \Phi^{i-1} \: \Theta \: p_{c,j}^{\pm} (a_{c,j}^{\pm})
\end{equation}

The definition resembles that of $ f_i $ given in Equation \ref{fi}.
Each point $ a_{c,j}^{\pm} \in \Delta_j $ is projected down to the vertex $ v_{c,j} $ of the parabola $ \gamma_{c,j} $.
It then follows the orbit of $ v_{c,j} $ through the insertion $ \Theta $ and the $ (i-1) $-th return to $ S $ under $ \Phi $, which by definition is the vertex $ v_{c,(j,i)} $, and is then inversely projected back along $ \gamma_{c,(j,i)} $ to its intersection point $ a_{c,(j,i)}^{\pm} \in [0,b] $.

We need to show that this is well-defined for $ (i,j) \in \widetilde{\Sigma}_{b,2} $.
Recall from Equation \ref{Sigma2} that $ \Sigma_{b,2} $ is the sequence space indexing the level-two curves $ \gamma_{c,(i,j)} $ defined in Equation \ref{gammack} by
$$
\gamma_{c,(i,j)} = (\Phi^{j-1} \Theta)(\gamma_{c,i}).
$$
Comparing with Equation \ref{phiij}, we see that $ \phi_{i,j} $ is well-defined with image in $ S $ when $ (i,j) \in \widetilde{\Sigma}_{b,2} $.
We can now state the following theorem.

\begin{theorem}
\label{functhm}
Let $ \Sigma_b $ be the general symbolic space given in Equation \ref{Sigmab}, and let $ \Sigma_b^{\infty} $ be its infinite extension.
Then the collection $ \{ \phi_{i,j} : \Delta_j \rightarrow [0,b] \}_{(i,j) \in \widetilde{\Sigma}_{b,2}} $ is a $ C^{1+\alpha} $ general function system modeled by the dual $ \widetilde{\Sigma}_b^{\infty} $.
\end{theorem}

\begin{proof}
We will show that $ \{ \phi_{i,j} \} $ satisfies the requirements of a $ C^{1+\alpha} $ general function system given in Definition \ref{genfundef}.

\begin{itemize}
\item \textit{Uniform contraction}: For each $ (i,j) \in \widetilde{\Sigma}_{b,2} $ the maps $ \{ \phi_{i,j} : \Delta_j \rightarrow [0,b] \} $ have a common Lipschitz constant $ 0 < s < 1 $.
\vspace{0.1cm}

Recall the above proof that the maps $ f_i $ are uniformly Lipschitz; a similar argument holds here.
Consider the dual transverse distances $ a(i,j) $ of level two defined in Chapter \ref{dualtrans}.
Then for all $ 0 \leq c, c' \leq b $ we have
$$
| \phi_{i,j}(a_{c,j}^{\pm}) - \phi_{i,j}(a_{c',j}^{\pm}) | < a(i,j)
$$
for all $ (i,j) \in \widetilde{\Sigma}_{b,2} $.
Since by $ a(i,j) \rightarrow 0 $ as $ i,j \rightarrow \infty $ by Proposition \ref{trans2}, the Lipschitz constant decreases as $ i,j \rightarrow \infty $.
Thus for a fixed $ i $, we have that $ \phi_{i,j} $ is uniformly Lipschitz for all $ j $, with Lipschitz constant equal to the Lipschitz constant of $ \phi_{i,1} $.
Let $ K_f $ be the uniform Lipschitz constant of the maps $ f_i $, and let $ K_{\phi} $ be the uniform Lipschitz constant of the maps $ \phi_{i,1} $. Taking $ K = \max \{ K_f, K_{\phi} \} $ suffices.
\vspace{0.2cm}

\item \textit{Separation}:
For each $ (i,j), (i',j') \in \widetilde{\Sigma}_{b,2} $ we have
$$
\phi_{i,j}(\Delta_j) \cap \phi_{i',j'}(\Delta_{j'}) = \emptyset
$$
when $ i \neq i' $ or $ j \neq j' $.
\vspace{0.1cm}

This is a consequence of the separation of $ \Delta_i $ and the following nesting property.
\vspace{0.2cm}

\item \textit{Nesting property}:
For all $ k \geq 1 $ and $ \omega \in \widetilde{\Sigma}_{b,k} $ we have
$$
\phi_{\omega_i,\omega_{i+1}} (\Delta_{\omega_{i+1}}) \subset \Delta_{\omega_i}
$$
for all $ 1 \leq i \leq k-1 $.
\vspace{0.1cm}

The dual curves $ \{\gamma_{c,(i,j)}\}_{(i,j) \in \widetilde{\Sigma}_{b,2}} $ form a parabolic foliation of the dual sets $ \{ A_{i,j} \}_{(i,j) \in \widetilde{\Sigma}_{b,2}} $. 
By Proposition \ref{dualnestingk} we know that $ A_{i,j} \subset A_i $ for dual words $ (i,j) $.
By definition, $ \phi_{i,j} $ maps the endpoints $ a_{c,j}^{\pm} $ of each curve $ \gamma_{c,j} \subset A_j $ to the endpoints $ a_{c,(i,j)}^{\pm} $ of the curve $ \gamma_{c,(i,j)} \subset A_{i,j} \subset A_i $.
Since $ \Delta_i = A_i \cap [0,b] $, this can be rewritten as 
$$
\phi_{i,j}(\Delta_j) \subset \Delta_i.
$$
The desired statement then follows by induction on the word length $ k = |\omega| $.
\vspace{0.2cm}

\item \textit{$ C^{1+\alpha} $ regularity}:
There exists $ \alpha > 0 $ such that for all $ (i,j) \in \widetilde{\Sigma}_{b,2} $, the maps $ \phi_{i,j} $ are of class $ C^{1+\alpha} $.
\vspace{0.1cm}

This is identical to the previous argument for the regularity of the maps $ \{f_i\}_{i \in \Sigma_{b,1}} $.
\vspace{0.2cm}
\end{itemize}
\end{proof}

\subsection{A graph directed pseudo-Markov subsystem}
\label{GDPMsub}
The previous section defined a general function system on $ [0,b] $ modeled by $ \widetilde{\Sigma}_b $.
The sequence space $ \widetilde{\Sigma}_b $ is not completely determined, because we do not know the escape times $ M_{i_1,\ldots,i_k} $ defining it.
However, in Proposition \ref{escapek} we obtained explicit estimates on those escape times for large values of $ i_1,\ldots,i_k $.
In this section we will extract a subspace that uses these estimates.
We will then show that the function system modeled by the subspace is a graph directed pseudo-Markov system, as defined in Section \ref{GDPM}.

\subsubsection{The sequence space $ \widetilde{\Sigma}_{\epsilon} $}
For each $ 0 < \epsilon \leq b $, let $ S_{\epsilon} \subset S $ be the rectangle intersecting the Reeb cylinder $ \{ r=2 \} $, and the top $ S^+ $ and bottom $ S^- $ of $ S $, with width $ \epsilon $.
Let $ S_{\epsilon}^+ \subset S^+ $ be the upper boundary of this rectangle, which can be identified with $ [0,\epsilon] $.
See Figure \ref{Sepsilonfig}.

\begin{figure}[h]
\includegraphics[width=\linewidth, trim={1cm 8.3cm 2cm 3cm}, clip]{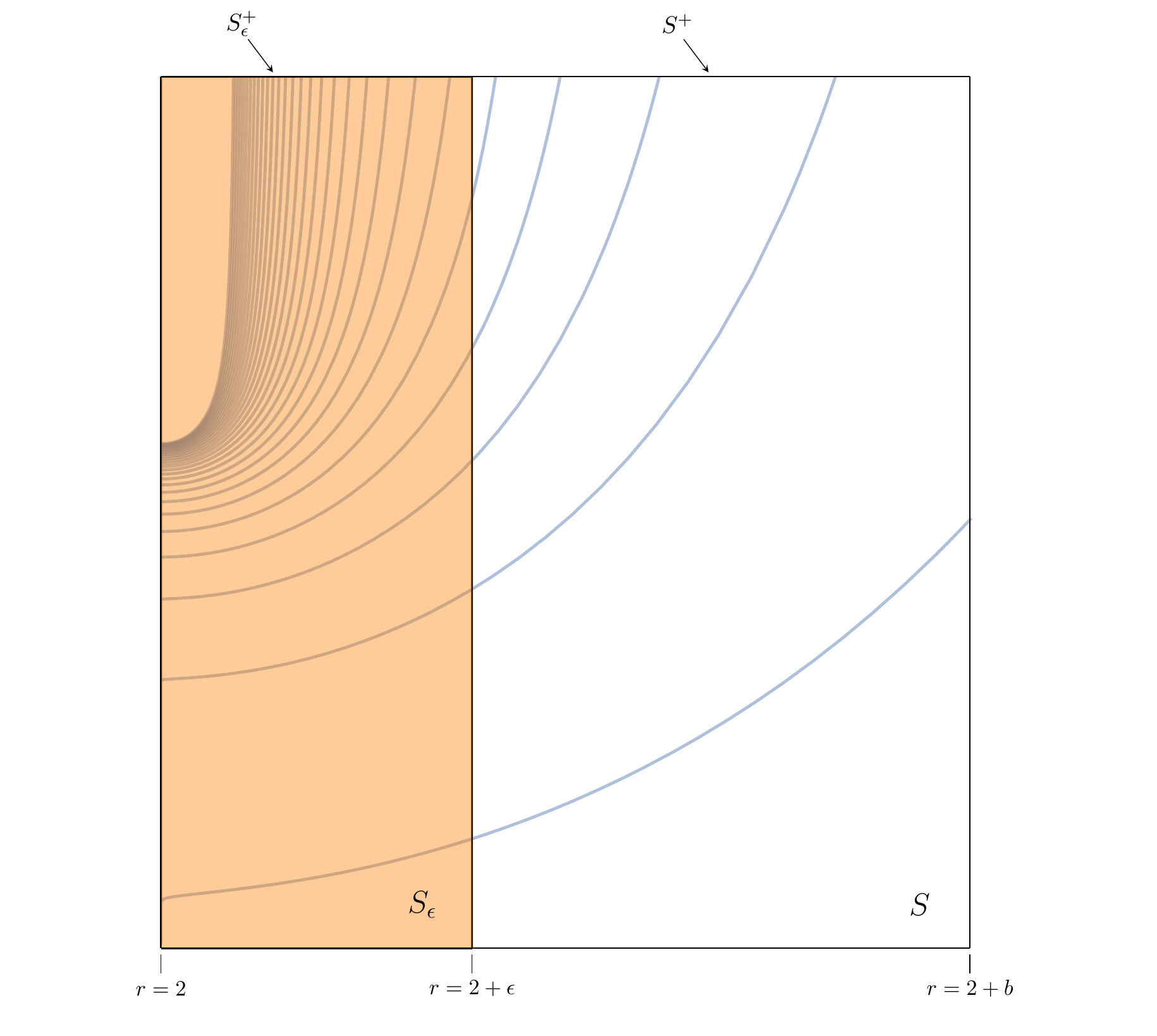}
\caption{The small rectangle $ S_{\epsilon} $ inside the larger rectangle $ S $. Here $ N_b=2 $, the smallest integer such that $ \gamma_i $ intersects $ S^+ $ for all $ i \geq N_b $, and $ N_{\epsilon} = 6 $, the smallest integer such that $ \gamma_i $ intersects $ S_{\epsilon}^+ $ for all $ i \geq N_{\epsilon} $.}
\label{Sepsilonfig}
\end{figure}

As for $ N_b $, let $ N_{\epsilon} $ be the smallest integer such that $ \gamma_i $ intersects $ S_{\epsilon}^+ $ for all $ i \geq N_{\epsilon} $.
This defines a sequence space $ \Sigma_{\epsilon} $ as in Equation \ref{Sigmab}, with dual $ \widetilde{\Sigma}_{\epsilon} $ as in Equation \ref{dualSigmab}.
Since $ \epsilon \leq b $ we have $ N_{\epsilon} \geq N_b $.
Furthermore, $ \lim_{\epsilon \to 0} N_{\epsilon} = \infty $.

We claim that for sufficiently small $ \epsilon > 0 $, the sequence space $ \Sigma_{\epsilon} $ is a graph directed symbolic space in the sense of Chapter \ref{Thermo}; there is a countable alphabet $ E $ and incidence matrix $ A : E \times E \rightarrow \{0,1\} $ such that for each $ n \geq 1 $, $ \Sigma_{\epsilon, n} = E_A^n $ in the notation of Equation \ref{admis}.

By Proposition \ref{escapek}, for small $ \delta>0 $ and for large $ i_1,\ldots,i_n $ we have
$$
C+(K-\delta) i_{n-1}^2 < M_{i_1,\ldots,i_n} < (C+\delta) + K i_{n-1}^2,
$$
where $ C $ and $ K $ are defined in Equation \ref{defCK}.
Let $ \lfloor \cdot \rfloor $ be the integer floor.
Since $ \lim_{\epsilon \to 0} N_{\epsilon} = \infty $, for small enough $ \epsilon $ we may substitute the above estimate into Equation \ref{Sigmak} to obtain

\begin{equation}
\label{Sigmaepn}
\Sigma_{\epsilon, n} = \bigcup_{i_1 = N_{\epsilon}}^{\infty} \bigcup_{i_2=N_{\epsilon}}^{\lfloor C \rfloor + \lfloor K \rfloor i_1^2} \cdots \bigcup_{i_n = N_{\epsilon}}^{\lfloor C \rfloor + \lfloor K \rfloor i_{n-1}^2} (i_1, \ldots, i_n).
\end{equation}

Let $ E = \Sigma_{\epsilon,1} = \{ N_{\epsilon}, N_{\epsilon}+1, \ldots \} $ and define the matrix $ A : E \times E \rightarrow \{0,1\} $ by

\begin{equation}
\label{Aij}
A(i,j) = \left\{
     \begin{array}{lr}
       1 & : j \leq \lfloor C \rfloor + \lfloor K \rfloor i^2 \\
       0 & : j > \lfloor C \rfloor + \lfloor K \rfloor i^2
    \end{array}
  \right.
\end{equation}
Then the admissible words $ E_A^n $ defined in Equation \ref{admis} are 
\begin{align*}
E_A^n &= \{ (i_1, \ldots, i_n) \in E^n : A_{i_j i_{j+1}} = 1 \text{ for all } 1 \leq j \leq n-1 \} \\
&=  \{ (i_1, \ldots, i_n) \in \{ N_{\epsilon}, N_{\epsilon}+1, \ldots \}^n : i_{j+1} \leq \lfloor C \rfloor + \lfloor K \rfloor i_j^2 \text{ for all } 1 \leq j \leq n-1 \} \\
&= \Sigma_{\epsilon, n},
\end{align*}
by comparing with Equation \ref{Sigmaepn}.
Taking the dual, we have $ \widetilde{\Sigma}_{\epsilon, n} = \widetilde{E}_A^n $ for each $ n \geq 1 $.

\subsubsection{The limit set $ J_{\epsilon} $}
By Theorem \ref{functhm}, for any $ 0 < \epsilon \leq b $ the function system
$$
\{ \phi_{i,j} : \Delta_j \rightarrow [0,\epsilon] \}_{(i,j) \in \widetilde{\Sigma}_{\epsilon,2}}
$$
is a well-defined $ C^{1+\alpha} $ subsystem of $ \{ \phi_{i,j} : \Delta_j \rightarrow [0,b] \}_{(i,j) \in \widetilde{\Sigma}_{b,2}} $ modeled by the dual space $ \widetilde{\Sigma}_{\epsilon} $.

By the above discussion, for sufficiently small $ \epsilon > 0 $ there exists an incidence matrix $ A $ such that 
$$
\widetilde{\Sigma}_{\epsilon,n} = \widetilde{E}_A^n.
$$

Let $ J_{\epsilon} \subset J_b $ be the limit set of this subsystem.
By definition, 
\begin{equation}
\label{Jep}
J_{\epsilon} = \bigcap_{n=1}^{\infty} \bigcup_{\omega \in \widetilde{E}_A^n} \Delta_{\omega}.
\end{equation}

\vfill
\eject

\section{The transverse Cantor set}
\label{Transcant}
In this section we will relate this transverse Cantor set $ \tau $ of the Kuperberg minimal set $ \mathcal{M} $ to limit sets of the function systems defined in Section \ref{FunctionSystems}.
We will use the previous symbolic dynamics developed in Section \ref{Kupmin} for the sets $ \mathcal{N}_0 \cap S $ and $ \mathcal{M}_0 \cap S $ to define bijective coding maps between these Cantor sets and the appropriate symbolic spaces.

\subsection{Sections of the minimal set}
Re-stating Equation \ref{leveldecompgammau},
\begin{equation}
\label{min1char}
\mathcal{M}_{0,1} \cap S = \bigcup_{\omega \in \Sigma_b} \gamma_{\omega}^u.
\end{equation}
Define $ \mathcal{M}_1 = \overline{\mathcal{M}_{0,1}} $,
so that

\begin{equation}
\label{M1}
\mathcal{M}_1 \cap S = \overline{\mathcal{M}_{0,1} \cap S} = \overline{ \bigcup_{\omega \in \Sigma_b} \gamma_{\omega}^u }.
\end{equation}
Recall the subspace $ \Sigma_{\epsilon} \subset \Sigma_b $ defined in Section \ref{GDPMsub}.
Replacing $ S $ with $ S_{\epsilon} $ and $ b $ with $ \epsilon $ in Equation \ref{leveldecompgammau}, we obtain

\begin{equation}
\label{min1char2}
\mathcal{M}_1 \cap S_{\epsilon} = \overline{ \bigcup_{\omega \in \Sigma_{\epsilon}} \gamma_{\omega}^u }.
\end{equation}

\subsection{The transverse Cantor set in $ [0,b] $}
In this section we will prove preliminary versions of Theorem $ \mathbf{A} $ and Corollary $ \mathbf{B} $ from Chapter \ref{Intro}.
The full versions will require the notion of \textit{interlacing} and will be given in the subsequent section.

In the lamination charts for $ \mathcal{M} $ constructed in Chapter 19 of \cite{Hur}, the transverse Cantor set $ \tau $ has a variable radial coordinate.
By Equation \ref{Sp}, our choice of transversal $ S^+ $ is compatible with these lamination charts. 
Then with $ \tau $ defined in Theorem \ref{minchar}, we may set
\begin{equation}
\label{tau}
\tau = \mathcal{M} \cap S^+.
\end{equation}
As in the previous section we denote
\begin{equation}
\label{tau1}
\tau_1 = \mathcal{M}_1 \cap S^+.
\end{equation}
Combining Equations \ref{awdef}, \ref{transverse1} and \ref{min1char2} we obtain
$$
\tau_1 = \overline{ \bigcup_{\omega \in \Sigma_b} a_{\omega}^- }.
$$
Re-indexing the points $ a_{\omega}^- $ using the bijection $ (\omega_1, \omega_2, \ldots) \mapsto (\ldots, \omega_2, \omega_1) $ yields
\begin{equation}
\label{min1trans}
\tau_1 = \overline{ \bigcup_{\omega \in \widetilde{\Sigma}_b} a_{\omega}^- }.
\end{equation}
See Figure \ref{transa2} for an illustration the intersections of the level-one curves in $ \mathcal{M}_1 \cap S $ with $ S^+ $.
\begin{figure}[h!]
\includegraphics[width=\linewidth, trim={5cm 9cm 1cm 5.8cm}, clip]{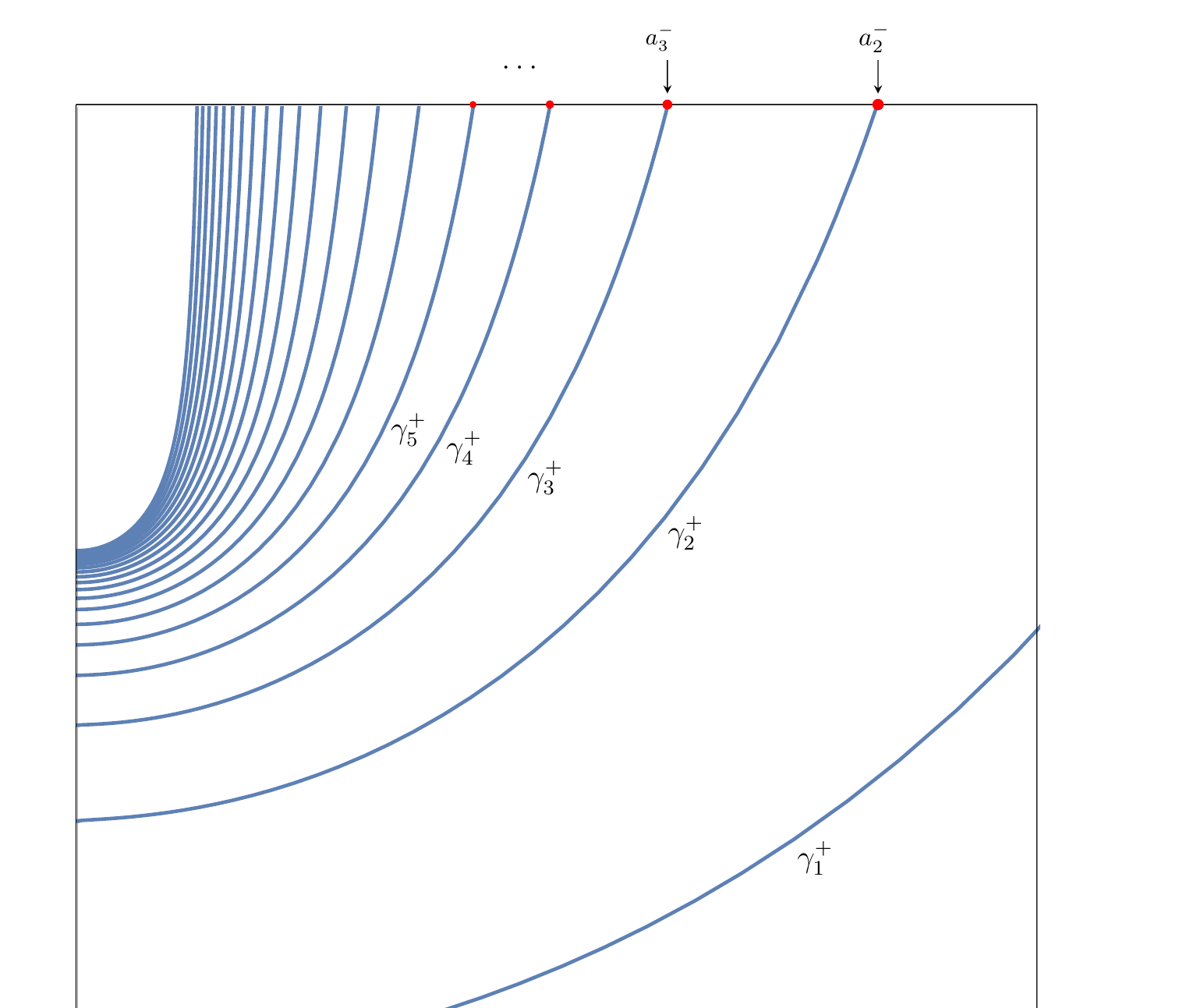}
\caption{The points $ a_i^- $ as intersections of the level-one curves $ \gamma^u_i $ with the upper boundary $ S^+ $ of $ S $. We obtain this from Figure \ref{transa} by restricting the parametrization of $ \gamma_i $.}
\label{transa2}
\end{figure}

In Section \ref{FunctionSystems}, we showed that the general symbolic space $ \Sigma_b $ has the extension admissibility property.
Since its dual $ \widetilde{\Sigma}_b $ also does, it has a well-defined infinite extension $ \widetilde{\Sigma}_b^{\infty} $.
We can now state the following theorem, which will be used later to prove Theorem \textbf{A} from Chapter \ref{Intro}.

\begin{theorem*}[$\mathbf{A}_0$]
There is a $ C^{1+\alpha} $ general function system on $ [0,b] $ modeled by $ \widetilde{\Sigma}_b^{\infty} $ with limit set $ \tau_1 $.
\end{theorem*}

\begin{proof}
In Theorem \ref{functhm} we defined a $ C^{1+\alpha} $ general function system on $ [0,b] $ modeled by $ \widetilde{\Sigma}_b^{\infty} $ and proved that its limit set is
$$
J_b = \bigcap_{n=1}^{\infty} \bigcup_{\omega \in \widetilde{\Sigma}_{b,n}} \Delta_{\omega}.
$$
Thus it suffices to show that $ \tau_1 = J_b $.
By Equation \ref{min1trans}, this is equivalent to
$$
\overline{ \bigcup_{\omega \in \widetilde{\Sigma}_b} a_{\omega}^- } = \bigcap_{n=1}^{\infty} \bigcup_{\omega \in \widetilde{\Sigma}_{b,n}} \Delta_{\omega}.
$$
We will show both containments.

First, let 
$$
x \in \overline{ \bigcup_{\omega \in \widetilde{\Sigma}_b} a_{\omega}^- } = \overline{ \bigcup_{n=1}^{\infty} \bigcup_{\omega \in \widetilde{\Sigma}_{b,n}} a_{\omega}^- }.
$$
Then there is a sequence of finite words $ \omega_n \in \widetilde{\Sigma}_{b,n} $ with $ a_{\omega_n}^- \rightarrow x $.

Note that $ J_b $ contains each point $ a_{\omega_n}^- $, because by construction, $ a_{\omega_n}^- $ is the left endpoint of the interval $ \Delta_{\omega_n} $.
Because $ J_b $ is a Cantor set it must contain all its limit points.
In particular it must contain $ x $, which concludes the forward containment.

For the reverse containment, let 
$$
x \in J_b = \bigcap_{n=1}^{\infty} \bigcup_{\omega \in \widetilde{\Sigma}_{b,n}} \Delta_{\omega}.
$$
By Theorem \ref{functhm}, $ J_b $ is the limit set of a general function system modeled by $ \widetilde{\Sigma}_b^{\infty} $, so $ x $ corresponds to a unique word $ \omega \in \widetilde{\Sigma}_b^{\infty} $ via the coding map $ \pi $:
$$
x = \pi(\omega) = \bigcap_{n=1}^{\infty} \Delta_{\omega |_n}.
$$
For details, see Section \ref{genfun}.
Consider the finite restriction of $ \omega $; this is the sequence $ \omega_n = \omega |_n \in \widetilde{\Sigma}_{b,n} $ (see Section \ref{extrestwords}).
By definition of the sets $ \Delta_{\omega} $ we have $ a_{\omega_n}^- \in \Delta_{\omega_n} $ and thus
$$
\lim_{n \to \infty} a_{\omega_n}^- = \lim_{n \to \infty} \bigcap_{k=1}^n \Delta_{\omega_k} = x.
$$
Then $ x $ is a limit point of a sequence $ a_{\omega_n}^- $ with $ \omega_n \in \widetilde{\Sigma}_b $, so
$$ 
x \in \overline{ \bigcup_{\omega \in \widetilde{\Sigma}_b} a_{\omega}^- }
$$ 
as desired.
\end{proof}

In the above proof, we used that the limit set of a general function system modeled by a symbolic space of infinite type has a bijective coding to that space.
For details, see Section \ref{genfun}.
As an immediate corollary to Theorem $ \mathbf{A}_0 $ we obtain the following.

\begin{corollary*}[$ \mathbf{B}_0 $]
There is a symbolic space $ \Sigma_1 $ of infinite type and a bijective coding
$$
\pi_1 : \Sigma_1 \rightarrow \tau_1
$$
\end{corollary*}

\begin{proof}
By Theorem $ \mathbf{A}_0 $, $ \tau_1 $ is the limit set of a general function system on $ [0,b] $ modeled by $ \widetilde{\Sigma}_b^{\infty} $.
By the results of Section \ref{genfun}, there is a bijective coding
$$
\pi_1 : \widetilde{\Sigma}_b^{\infty} \rightarrow \tau_1.
$$
So it suffices to take $ \Sigma_1 = \widetilde{\Sigma}_b^{\infty} $.
\end{proof}

\subsection{The transverse Cantor set $ \tau $ as an interlaced Cantor set}
Recall from Equation \ref{Psi} that the full Kuperberg pseudogroup $ \Psi $ on $ S_1 \cup S_2 $ is
$$
\Psi = \langle \Phi_1, \Phi_2, \Phi_{1,2}, \Theta_1, \Theta_2 \rangle.
$$
Up to now, we have only considered the sub-pseudogroup $ \Psi_1 = \langle \Phi_1, \Theta_1 \rangle $, using the shorthand notation $ \Theta = \Theta_1 $ and $ \Phi = \Phi_1 $.
The results of Chapters \ref{Kupmin} -- \ref{FunctionSystems} gave us a complete description of the transverse set $ \mathcal{M}_{0,1} \cap S $, its closure $ \mathcal{M}_1 \cap S $, and its transverse Cantor set $ \tau_1 $.

However, to account for the full transverse minimal set $ \mathcal{M} \cap S $ it is necessary to incorporate the other maps $ \Phi_2, \Theta_2 $, and $ \Phi_{1,2} $.
In this section we will use the symmetry of the plug $ K $ to show that these maps generate a Cantor set $ \tau_2 $ identical to $ \tau_1 $, and that the interlacing of $ \tau_1 $ and $ \tau_2 $ (see Section \ref{sectioninter}) is the transverse Kuperberg minimal set $ \tau $ from Theorem \ref{minchar}.

In this section, we will dispense with our shorthand notation $ S, D, \gamma $ for $ S_1, D_1, \gamma_1 $, and return to considering $ S_i, D_i $, and $ \gamma_i $ for $ i=1,2 $, as we did in the first half of Chapter \ref{Kuppseudo}.

\subsubsection{The Cantor set $ \tau_2 $ in $ S_1^+ $}
The intersection of the notched Reeb cylinder $ \mathcal{R}' $ with the upper insertion rectangle $ S_2 $ is $ \gamma_2 $, a vertical line with a parametrization similar to that of $ \gamma_1 $ given in Equation \ref{gamma}. 
Consider the curves
$$
\overline{\gamma}_i = \Phi_1^{i-1} \Theta_2 (\gamma_2).
$$
The images of $ \overline{\gamma}_i $ lie in $ S_1 $, and have a similar parametrization to those of $ \gamma_i $ given in Proposition \ref{gamma1param}.
A derivation of this fact closely resembles the proof there, which we will not repeat. 
The only quantitative difference appearing in the parametrization is the constant $ 0 \leq \alpha < 2\pi $, which is the angular coordinate of the vertex of the parabola $ \sigma^{-1} \gamma_1 \subset \{ z=-2 \} $ (see Equation \ref{sigmagamma}).
If $ \alpha_1 $ and $ \alpha_2 $ are the vertices of the parabolas $ \sigma^{-1} \gamma_1 $ and $ \sigma^{-1} \gamma_2 $ respectively, then $ 0 \leq \alpha_1 < \alpha_2 < 2\pi $.

We can then adjust the proof of Proposition \ref{gamma1param} to show that the $ i $th return time of the Wilson orbit of $ \sigma^{-1} \gamma_2 $ to $ S_1 $ is strictly between the $ i $th and $ (i+1) $th return times of $ \sigma^{-1} \gamma_1 $, which define the curves $ \gamma_i $ and $ \gamma_{i+1} $.
As $ i $ increases, the $ r $ and $ z $-coordinate of the curves $ \gamma_i $ and $ \overline{\gamma}_i $ increases.
Thus the curve $ \overline{\gamma}_i $ is between $ \gamma_i $ and $ \gamma_{i+1} $ for all $ i \in \Sigma_{b,1} $, hence these curves alternate as $ i $ increases.
For an illustration of this, see Figure \ref{curveintfig}.

\begin{figure}[h]
\includegraphics[width=\linewidth, trim={3.2cm 10.5cm 1cm 2.8cm}, clip]{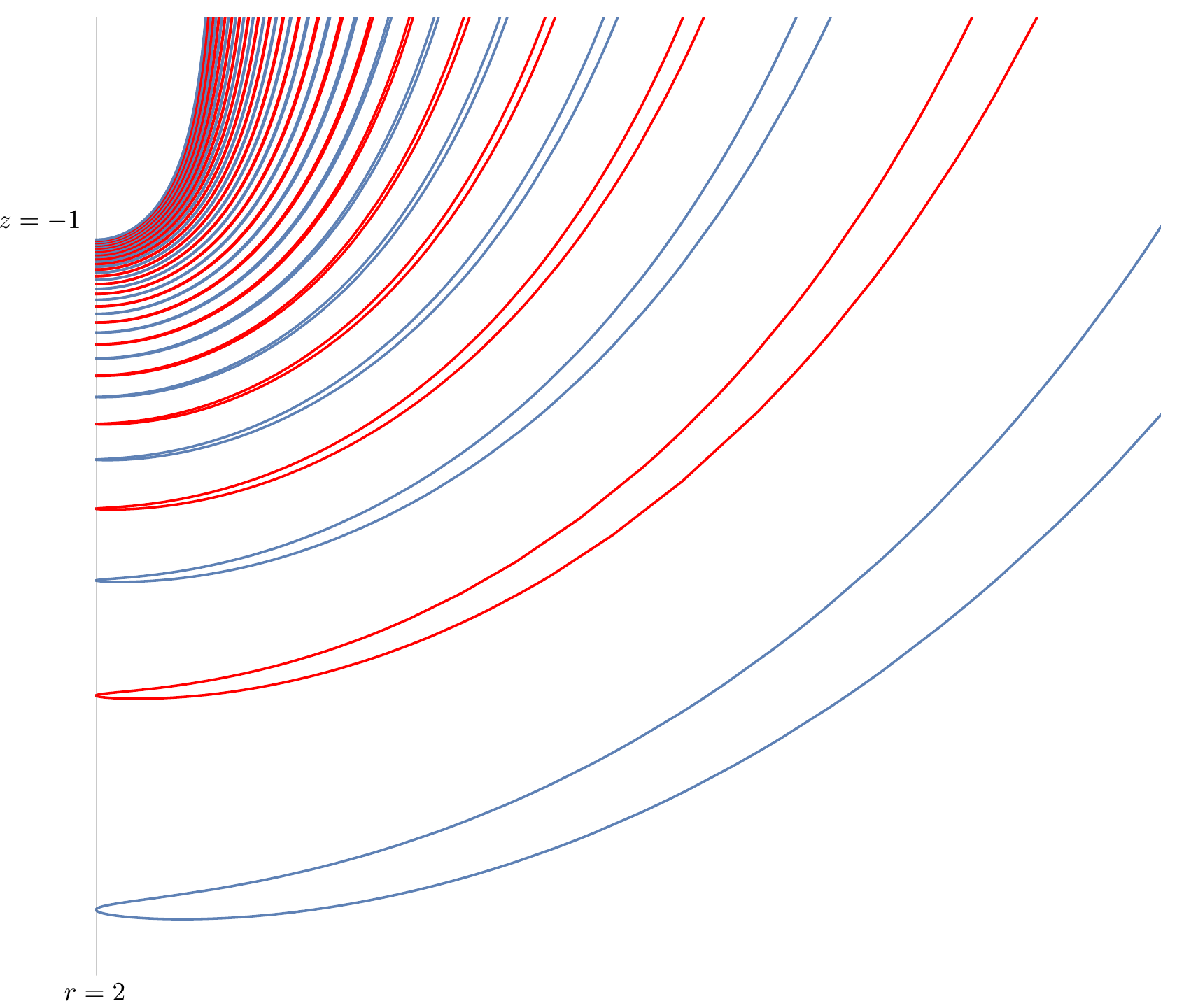}
\caption{A plot of the level-one curves $ \gamma_i $ generated by $ \Phi_1 $ and $ \Theta_1 $, shown in blue, and the level-one curves $ \overline{\gamma}_i $ generated by $ \Phi_1 $ and $ \Theta_2 $, shown in red.
Compare with Figure \ref{levelone}.}
\label{curveintfig}
\end{figure}

For any $ (i_1,\ldots,i_k) \in \Sigma_{b,k} $ we recursively define
$$
\overline{\gamma}_{i_1,\ldots,i_k} = \Phi^{i_k -1} \: \Theta_1 (\overline{\gamma}_{i_1,\ldots,i_{k-1}}),
$$
exactly as we defined $ \gamma_{i_1,\ldots,i_k} $.
Because their construction is identical to that of $ \gamma_{i_1,\ldots,i_k} $, the curves $ \overline{\gamma}_{i_1,\ldots,i_k} $ are also coded by $ \Sigma_{b,k} $, and their transverse distances $ \overline{a}(i_1,\ldots,i_k) $ are identical to the transverse distances $ a(i_1,\ldots,i_k) $ estimated in Chapter \ref{Transversal}.
The following definitions are similar to Equations \ref{min1char2} and \ref{tau1}:
$$
\mathcal{M}_2 \cap S = \overline{ \bigcup_{\omega \in \Sigma_b} \overline{\gamma}_{\omega} }, \qquad \qquad \tau_2 = \mathcal{M}_2 \cap S^+.
$$
We then define a function system using conjugations of pseudogroup elements by projections along $ \overline{\gamma}_{\omega} $, and show that $ \tau_2 $ is a Cantor limit set of a general function system modeled by $ \widetilde{\Sigma}_b^{\infty} $, exactly as we did for $ \tau_1 $ in Theorem $ \mathbf{A}_0 $.

\subsubsection{The level-one curves in $ S_1 \cup S_2 $}
Consider the collection $ \{ \gamma_i, \overline{\gamma}_j \}_{i,j \in \Sigma_{b,1}} $ of interlaced curves in $ S_1 $. 
We may naturally re-index this collection to $ \eta_i $, where $ i \in (\Sigma_{b,1} \ast \Sigma_{b,1})_1 $, the joint sequence space (see Section \ref{sectioninter}).

There is an identical family of curves of interlaced curves $ \eta'_i $ in $ S_2 $.
The reason for this is the symmetry of $ K $; we may invert the plug, reverse time, and obtain a similar analysis to Chapter \ref{Kupmin}.
Visually, there is an inverted Figure \ref{curveintfig} in the upper insertion rectangle $ S_2 $.

\subsubsection{The level-two curves in $ S_1 \cup S_2 $}
For each $ i_1 \in (\Sigma_{b,1} \ast \Sigma_{b,1})_1 $, define
$$
\eta_{i_1,i_2} = \Phi_1^{i_2-1} \: \Theta_1(\eta_{i_1}), \qquad \qquad \overline{\eta}_{i_1,i_2} = \Phi_1^{i_2-1} \: \Theta_2 (\eta'_{i_1}).
$$
An argument identical to the construction of $ \Sigma_{b,2} $ shows that the admissible words coding $ \eta_{i_1,i_2} $ and $ \overline{\eta}_{i_1,i_2} $ are two disjoint copies of $ \Sigma_{b,2} $.

The generator $ \Theta_1 $ maps $ S_1 $ into $ A_1 $, the boundary of which is $ \gamma_1 $. 
Because they are defined using $ \Theta_1 $, the curves $ \eta_{i_1,i_2} $ are nested in $ \gamma_{i_2} $. 
Similarly, $ \overline{\eta}_{i_1,i_2} $ are nested in $ \overline{\gamma}_{i_2} $ because they are defined using $ \Theta_2 $.
A proof of these facts follows Proposition \ref{nesting2}.

Consider the collection $ \{ \eta_{i_1,i_2}, \overline{\eta}_{j_1, j_2} \} $ of all level-two curves in $ S_1 $, where $ (i_1, i_2) $ and $ (j_1, j_2) $ range through $ \Sigma_{b,2} $.
As with the level-one curves, we re-index this to a single collection $ \{ \eta_{i_1,i_2} \} $, where $ (i_1,i_2) $ range through the joint sequence space $ (\Sigma_{b,1} \ast \Sigma_{b,1})_2 $.
Finally, note that we have identical level-two interlaced curves $ \eta'_{i_1, i_2} $ in the upper rectangle $ S_2 $, also indexed by $ (\Sigma_{b,1} \ast \Sigma_{b,1})_2 $.

\subsubsection{The level-$k$ curves in $ S_1 \cup S_2 $}
We continue recursively defining the interlaced curves $ \eta_{i_1,\ldots,i_k} $ and $ \overline{\eta}_{i_1,\ldots,i_k} $ in $ S_1 $ as in Equation \ref{gammai1ik}:
\begin{align*}
\eta_{i_1,\ldots,i_k} &= \Phi_1^{i_k-1} \: \Theta_1(\eta_{i_1,\ldots,i_{k-1}}) \\
\overline{\eta}_{i_1,\ldots,i_k} &= \Phi_1^{i_k-1} \: \Theta_2(\eta'_{i_1,\ldots,i_{k-1}})
\end{align*}
After each such definition, we re-index the individual collections $ \eta_{\omega}, \overline{\eta}_{\omega} $ to $ \eta_{\omega} $ by the joint sequence space.
We then note the identical families $ \eta'_{\omega} \in S_2 $, and continue.

\subsubsection{The function system on the interlaced curves}
Recall the transverse Cantor set $ \tau $ defined in Theorem \ref{minchar} and Equation \ref{tau}.
By the interlacing of curves studied above, $ \tau $ is the interlacing of the Cantor sets $ \tau_1 $ and $ \tau_2 $, in the sense of Section \ref{sectioninter}.
Using the interlaced curves, we define a general function system modeled by the infinite extension of the joint sequence space $ \Sigma_{b,1} \ast \Sigma_{b,1} $, whose limit set is $ \tau $.
This is the content of the following theorem, whose proof follows that of Theorem $ \mathbf{A}_0 $.

\begin{theorem*}[\textbf{A}]
Let $ \tau $ be the Cantor transversal of the Kuperberg minimal set.
Then there is a $ C^{1+\alpha} $ general function system on $ [0,b] $ modeled by $ \widetilde{\Sigma}_b^{\infty} \ast \widetilde{\Sigma}_b^{\infty \prime} $ with limit set $ \tau $.
\end{theorem*}

In the notation of Section \ref{sectioninter}, this limit set is
$$
\tau = \bigcap_{n=1}^{\infty} \bigcup_{\omega \in (\widetilde{\Sigma}_b \ast \widetilde{\Sigma}'_b)_n} \Delta_{\omega}.
$$

Just as with Corollary $ \mathbf{B}_0 $, we obtain the following from Theorem \textbf{A}.

\begin{corollary*}[\textbf{B}]
Let $ \tau $ be the Cantor transversal of the Kuperberg minimal set. Then there is a bijective coding map
$$
\pi : \widetilde{\Sigma}_b^{\infty} \ast \widetilde{\Sigma}_b^{\infty \prime} \rightarrow \tau.
$$
\end{corollary*}

\subsection{The transverse Cantor set in $ [0,\epsilon] $}
As in Equation \ref{tau}, we define $ \tau_{\epsilon} \subset \tau $ by
\begin{equation}
\label{tauep}
\tau_{\epsilon} = \mathcal{M} \cap S_{\epsilon}^+.
\end{equation}

For $ i=1,2 $, let $ \tau_{i, \epsilon} $ be the intersection of $ \tau_i $ with an $ \epsilon $-neighborhood of the critical orbit in the Kuperberg plug.
By applying Theorem $ \textbf{A}_0 $ to a suitably small transversal $ [0,\epsilon] $, we can prove Theorem $ \mathbf{C}_0 $.
This will be used to prove Theorem \textbf{C}.

\begin{theorem*}[$ \mathbf{C}_0 $]
For each $ i=1,2 $ and sufficiently small $ \epsilon > 0 $, there is a $ C^{1+\alpha} $ graph directed pseudo-Markov system on $ [0,\epsilon] $ with limit set $ \tau_{i,\epsilon} $.
\end{theorem*}

\begin{proof}
We will give the proof only for $ \tau_{1,\epsilon} $. 
By Theorem $ \mathbf{A}_0 $, $ \tau_{1,\epsilon} = J_{\epsilon} $, the limit set of a $ C^{1+\alpha} $ function system modeled by the dual $ \widetilde{\Sigma}_{\epsilon}^{\infty} $.
By the results of Section \ref{GDPMsub}, this is a $ C^{1+\alpha} $ graph directed pseudo-Markov system.
\end{proof}

Again by considering the interlacing of $ \tau_{1,\epsilon} $ with $ \tau_{2,\epsilon} $ we obtain $ \tau_{\epsilon} $ with the following characterization.

\begin{theorem*}[\textbf{C}]
Let $ \tau $ be the Cantor transversal of the Kuperberg minimal set, and let $ \tau_{\epsilon} $ be the intersection of $ \tau $ with an $ \epsilon $-neighborhood of the critical orbit.
For sufficiently small $ \epsilon > 0 $ there is a $ C^{1+\alpha} $ graph directed pseudo-Markov system on $ [0,\epsilon] $ with limit set $ \tau_{\epsilon} $.
\end{theorem*}

We conclude by displaying the limit set of this pseudo-Markov system.
In Equation \ref{Aij} we defined the incidence matrix $ A $ defining the admissible words $ E_A^n $ of length $ n $.
For two copies $ E $ and $ E' $ of the sequence space $ \Sigma_{b,1} $, we have a joint incidence matrix $ A^{E \cup E'} $ coding the admissible words in the interlaced Cantor set, as defined in Section \ref{sectioninter}.
This joint matrix is
\begin{equation}
\label{jointE}
A^{E \cup E'}(i,j) =
\left\{
     \begin{array}{lr}
       1 & : i,j \in E \text{ or } i,j \in E' \text{ and } j \leq \lfloor C \rfloor + \lfloor K \rfloor i^2 \\
       0 & : i,j \in E \text{ or } i,j \in E' \text{ and } j > \lfloor C \rfloor + \lfloor K \rfloor i^2 \\
       1 & : i \in E \text{ and } j \in E' \text{ or } i \in E' \text{ and } j \in E
    \end{array}
  \right.
\end{equation}

This matrix defines admissible words $ (E \cup E')_{A^{E \cup E'}}^n $, and by Theorem \textbf{C} we have
\begin{equation}
\label{minep}
\tau = \bigcap_{n=1}^{\infty} \bigcup_{\omega \in (E \cup E')_{A^{E \cup E'}}^n} \Delta_{\omega}.
\end{equation}

\vfill
\eject

\section{Dimension of the Cantor set}
\label{Dimtrans}
In this section we will apply the dimension theory developed in Section \ref{Dime} to study the Hausdorff dimension of $ \tau $, the transverse minimal set in Kuperberg's plug.
We will then use the product structure of the lamination to extend this to the dimension of $ \mathcal{M} $.

By Theorem \textbf{A} we know that $ \tau $ is the limit set of a $ C^{1+\alpha} $ general function system.
By Theorem \textbf{C}, for sufficiently small $ \epsilon > 0 $, $ \tau_{\epsilon} $ is the limit set of a pseudo-Markov subsystem.
Limit sets of pseudo-Markov systems have a well-developed dimension theory as exposed in Section \ref{Dime}.
We wish to apply this theory to the transverse minimal set $ \tau $, but to do this we must first relate the dimension of $ \tau $ to that of $ \tau_{\epsilon} $.

\subsection{The Hausdorff dimension of $ \tau $}
The next lemma uses minimality of $ \mathcal{M} $ to show that the Hausdorff dimension of $ \tau $ can be calculated inside a small neighborhood of an arbitrary point.
For any $ x \in \tau $, let $ B_{\epsilon}(x) \subset M$ denote the closed ball of radius $ \epsilon $ centered at $ x $, and let $U_{\epsilon}(x) \subset  B_{\epsilon}(x)$ denote its open interior.

\begin{lemma}[\textbf{D}]
\label{constdim}
Let $ \tau $ be the Cantor transversal of the Kuperberg minimal set.
For   $ \epsilon, \epsilon' > 0 $ sufficiently small, and any $ x,y \in \tau $, we have
$$
{\rm dim}_H(\tau \cap B_{\epsilon}(x)) = {\rm dim}_H(\tau \cap B_{\epsilon'}(y)).
$$
\end{lemma}

\begin{proof}
We show that $\displaystyle {\rm dim}_H(\tau \cap B_{\epsilon'}(y)) \geq {\rm dim}_H(\tau \cap B_{\epsilon}(x))$, with the reverse inequality following by the same method, which proves the claim.

\medskip

For each $z \in \tau \cap B_{\epsilon}(x)$   there exists $T_z > 0$ such that $ \psi_{T_z}(z) \in   B_{\epsilon'}(y)$. Let $\varepsilon_z> 0$ be sufficiently small so that 
$\displaystyle  \psi_{T_z}(B_{\varepsilon_z}(z)) \subset  B_{\epsilon'}(y)$. 

\medskip
The collection of open balls $\{U_{\varepsilon_z}(z) \mid z \in \tau \cap B_{\epsilon}(x)\}$ is an open covering of the compact set $\tau \cap B_{\epsilon}(x)$ so there exists a finite subcovering, centered at points $\{z_1, \ldots, z_k\} \subset \tau \cap B_{\epsilon}(x)$.
By standard properties of Hausdorff dimension, we have that
$${\rm dim}_H(\tau \cap B_{\epsilon}(x)) = \max ~ \{{\rm dim}_H(B_{\varepsilon_{z_i}}(z_i) \cap \tau)) \mid 1 \leq i \leq k \}. $$
The flow $\psi$ is $C^{\infty}$ so for each $i$ we have that
$${\rm dim}_H(\psi_{T_z}(B_{\varepsilon_{z_i}}(z_i) \cap \tau)) = {\rm dim}_H( B_{\varepsilon_{z_i}}(z_i) \cap \tau) .$$ 

Now assume that $\epsilon'> 0$ is sufficiently small so that the projection $\Pi_{\mathcal F}$ along the leaves of the foliation is 1-1 when restricted to the ball $ B_{\epsilon'}(y)$, 
$$\Pi_{\mathcal F} \colon \psi_{T_z}(B_{\varepsilon_{z_i}}(z_i) \cap \tau) \to \tau \cap B_{\epsilon'}(y) .$$
The value of $\epsilon'> 0$ depends only on the construction of the flow, and not on the choice of the point $y'$.
The holonomy projection map $\Pi_{\mathcal F}$ is $C^1$ by results in HR, so we then have
$${\rm dim}_H(\psi_{T_z}(B_{\varepsilon_{z_i}}(z_i) \cap \tau)) = {\rm dim}_H(\Pi_{\mathcal F}(\psi_{T_z}(B_{\varepsilon_{z_i}}(z_i) \cap \tau))) \leq {\rm dim}_H(\tau \cap B_{\epsilon'}(y))  . $$
The claim follows.
\end{proof}

Now consider the point $ x = (2, \beta, -1) $ in the Kuperberg plug. 
This is the intersection of the lower critical orbit with the rectangle $ S $.
By the definition in Section \ref{GDPMsub}, $ x $ is the left endpoint of the transversal $ S_{\epsilon}^+ $.
Then for any $ \epsilon > 0 $, $ \tau_{\epsilon} = \tau \cap B_{\epsilon}(x) $.
Taking $ \epsilon' = b $ in the statement of Lemma \textbf{D}, we obtain that
\begin{equation}
\label{tauepeq}
\text{dim}_H(\tau) = \text{dim}_H (\tau_{\epsilon})
\end{equation}
for any $ \epsilon > 0 $.
This reduces the calculation of the Hausdorff dimension of $ \tau $ to that of $ \tau_{\epsilon} $.
We now combine this with the estimates from Chapter \ref{Transversal} on the transverse distances, to prove the following theorem.

\begin{theorem*}[\textbf{E}]
Let $ \tau $ be the Cantor transversal of the Kuperberg minimal set.
Then the Lebesgue measure of $ \tau $ is zero, and $ 0 < \dim_H (\tau) < 1 $.
\end{theorem*}

\begin{proof}
By Equation \ref{tauepeq}, it suffices to prove the statement for $ \tau_{\epsilon} $ for any $ \epsilon > 0 $.
By Theorem \textbf{C} and Equation \ref{minep} we know that for sufficiently small $ \epsilon > 0 $,
$$
\tau_{\epsilon} = \bigcap_{n=1}^{\infty} \bigcup_{\omega \in (E \cup E')_{A^{E \cup E'}}^n} \Delta_{\omega}.
$$

By construction of the pseudo-Markov system from Section \ref{GDPMsub}, $ |\Delta_{\omega}| = a(\omega) $, the transverse distances studied in Section \ref{Transversal}.
By Proposition \ref{dualtransk}, for any $ \delta > 0 $ there exist $ L_n \in \mathbb{N} $ such that for all $ i_1,\ldots,i_n \geq L_n $ we have
\begin{equation}
\label{ain}
\left| a(i_1,\ldots,i_n)- \left(\frac{\pi^{-1} K^{\frac{3}{2}}}{i_1^{\frac{5}{2}}} \cdot \frac{\left((2\pi)^{-2} aR^2\right)^{n-1}}{i_2^2 \cdots i_n^2} \right)\right| < \frac{\delta}{i_1^2 \cdots i_n^2}
\end{equation}
Taking the dual of Equation \ref{Sigmaepn} yields
\begin{equation}
\label{dualSigmaepn}
\widetilde{E}_{A^E}^n = \bigcup_{i_1 = N_{\epsilon}}^{\infty} \bigcup_{i_2=N_{\epsilon}}^{\lfloor C \rfloor + \lfloor K \rfloor i_1^2} \cdots \bigcup_{i_n = N_{\epsilon}}^{\lfloor C \rfloor + \lfloor K \rfloor i_{n-1}^2} (i_n, \ldots, i_1).
\end{equation}
By the definition of $ N_{\epsilon} $ given in the proof of Proposition \ref{gamma1param}, we know that $ N_{\epsilon} \rightarrow \infty $ as $ \epsilon \rightarrow 0 $.
So taking a sequence $ \epsilon_n \rightarrow 0 $ with $ N_{\epsilon_n} \geq L_n $ for all $ n $, we have that Equation \ref{ain} holds for all $ (i_1,\ldots,i_n) \in \widetilde{E}_{A^E}^n $ for small enough $ \epsilon $, and $ \delta \rightarrow 0 $ as $ \epsilon \rightarrow 0 $.

Substituting $ |\Delta_{\omega}| = a(\omega) $ into Equation \ref{ain} and rewriting, we have that for any $ \delta>0 $ and small enough $ \epsilon > 0 $,
\begin{equation}
\label{another}
\frac{\pi^{-1} K^{\frac{3}{2}}}{i_1^{\frac{5}{2}}} \cdot \frac{\left((2\pi)^{-2} aR^2\right)^{n-1}-\delta}{i_2^2 \cdots i_n^2}
< |\Delta_{i_1,\ldots,i_n}|
< \frac{\pi^{-1} K^{\frac{3}{2}}}{i_1^{\frac{5}{2}}} \cdot \frac{\left((2\pi)^{-2} aR^2\right)^{n-1}+\delta}{i_2^2 \cdots i_n^2}
\end{equation}
for all $ (i_1,\ldots,i_n) \in (E \cup E')_{A^{E \cup E'}}^n $ and $ \delta \rightarrow 0 $ as $ \epsilon \rightarrow 0 $.

To simplify notation, let
$$
s_i = \frac{\pi^{-1} K^{\frac{3}{2}}}{i^{\frac{5}{2}}}, \; \text{ and } \; r_i = \frac{(2\pi)^{-2} a R^2}{i^2}.
$$
Referring to Section \ref{statmark}, we see that for $ \delta > 0 $ there exists sufficiently small $ \epsilon > 0 $ such that $ \tau_{\epsilon} $ is the limit set of an asymptotically stationary pseudo-Markov system with ratio coefficients $ r_i $ given above, and summable monotone error
$$
a_{\delta}^{\pm}(i_1,\ldots,i_n) = \pm \frac{\delta}{i_1^2 \cdots i_n^2}. $$
By Theorem \ref{statGDMS}, we obtain that the Lebesgue measure of $ \tau_{\epsilon} $ is zero, and that $ 0 < \text{dim}_H(\tau_{\epsilon}) < 1 $.
\end{proof}

\subsection{Estimating the dimension via the pressure}
The following theorem is an application of the thermodynamic formalism developed in Section \ref{Dime} to the dimension theory of $ \tau $.

\begin{theorem*}[\textbf{F}]
Let $ \tau $ be the Cantor transversal of the Kuperberg minimal set. 
Let $ t = \text{dim}_H(\tau) $ be its Hausdorff dimension, and $ a > 0 $ the angular speed of the Kuperberg flow.
\begin{itemize}
\item $ t = \text{dim}_H(\tau) $ is the unique zero of a dynamically defined pressure function, 
\item $ t $ depends continuously on $ a $, 
\item For any $ a $ we may compute $ t $ to a desired level of accuracy.
\end{itemize}
\end{theorem*}

\begin{proof}
By Equation \ref{tauepeq}, it suffices to prove the statement for $ \tau_{\epsilon} $ for any $ \epsilon > 0 $.
By Theorem \textbf{C}, we know that for small enough $ \epsilon $, $ \tau_{\epsilon} $ is the interlaced limit set of a pseudo-Markov system, with limit set given in Equation \ref{minep}.
From Section \ref{Dime}, the pressure function determined by this pseudo-Markov system is
$$
p(t) = \lim_{n \to \infty} \frac{1}{n} \log \sum_{\omega \in (E \cup E')_{A^{E \cup E'}}^n} | \Delta_{\omega} |^t.
$$

Since $ E $ and $ E' $ are equal, for each interval $ \Delta_{\omega} $ coded by a word $ \omega \in (E \cup E')_{A^{E\cup E'}}^n $ there are two intervals $ \Delta_{\omega} $ for $ \omega \in E_{A^E}^n $, and these two intervals have equal lengths.
From this we obtain
$$
p(t) = \lim_{n \to \infty} \frac{1}{n} \log \sum_{\omega \in E_A^n} | 2 \Delta_{\omega} |^t.
$$
Applying Equations \ref{dualSigmaepn} and \ref{another} we obtain that $ p^-(t) < p(t) < p^+(t) $, where
\begin{align}
\label{another2}
p^{\pm}(t) &= \lim_{n \to \infty} \frac{1}{n} \log \sum_{(i_1,\ldots,i_n) \in E_A^n} \left| \frac{2\pi^{-1} K^{\frac{3}{2}}}{i_1^{\frac{5}{2}}} \cdot \frac{\left((2\pi)^{-2} aR^2\right)^{n-1} \pm \delta}{i_2^2 \cdots i_n^2} \right|^t \\
&= \lim_{n \to \infty} \frac{1}{n} \log \sum_{i_1 = N_{\epsilon}}^{\infty} \sum_{i_2=N_{\epsilon}}^{\lfloor C \rfloor + \lfloor K \rfloor i_1^2} \cdots \sum_{i_n=N_{\epsilon}}^{\lfloor C \rfloor + \lfloor K \rfloor i_{n-1}^2} \left| \frac{2\pi^{-1} K^{\frac{3}{2}}}{i_1^{\frac{5}{2}}} \cdot \frac{\left((2\pi)^{-2} aR^2\right)^{n-1} \pm \delta}{i_2^2 \cdots i_n^2} \right|^t. \nonumber
\end{align}
By Bowen's theorem (Theorem \ref{Bow}),
$$
\text{dim}_H(\tau_{\epsilon}) = \inf \{ t \geq 0 : p(t) \leq 0 \}.
$$
It is easy to see that $ p^{\pm}(t) $ have the same properties as $ p(t) $ specified in Theorem \ref{presprop}; in particular they are strictly decreasing and have unique zeros on $ (0,1) $.
Then $ t = \text{dim}_H(\tau_{\epsilon}) $ is bounded between these zeros, by Bowen's theorem.
Furthermore, as $ \epsilon \rightarrow 0 $ in the sequence space $ E_{A^E}^n $, we have $ \delta \rightarrow 0 $, so these two zeros approach $ \text{dim}_H(\tau_{\epsilon}) $.
From Equation \ref{another2}, the zeros of $ p^{\pm}(t) $ vary continuously with $ a $, and thus $ \text{dim}_H(\tau_{\epsilon}) $ also does.
For the final statement, we refer to the explicit formula for $ p^{\pm}(t) $ given in Equation \ref{another2}.
For a specific choice of $ \epsilon $, $ \delta $, and $ a $, we can estimate the roots of $ p^{\pm}(t) $.
These are upper and lower bounds on $ \text{dim}_H(\tau_{\epsilon}) $, which improve as $ \epsilon \rightarrow 0 $.
\end{proof}

\subsection{Numerical results for dimension}
\label{Numres}
Finally, we turn to the numerical problem of estimating the Hausdorff dimension of $ \tau $.
As before, by Equation \ref{tauepeq} it suffices to estimate the Hausdorff dimension of $ \tau_{\epsilon} $ for any $ \epsilon > 0 $.
In this section, we will make specific choices of $ \epsilon $ and $ a $, and derive explicit upper and lower estimates on $ \text{dim}_H(\tau_{\epsilon}) $.

Consider $ p^+(t) $ as defined in Equation \ref{another2}.
The following establishes an upper bound on $ p^+(t) $ and hence on $ p(t) $.
\begin{align*}
p^+(t) &< \lim_{n \to \infty} \frac{1}{n} \log \sum_{i_1 = N_{\epsilon}}^{\infty} \sum_{i_2 = N_{\epsilon}}^{\infty} \cdots \sum_{i_n = N_{\epsilon}}^{\infty} \left| \frac{2\pi^{-1} K^{\frac{3}{2}}}{i_1^{\frac{5}{2}}} \cdot \frac{\left((2\pi)^{-2} aR^2\right)^{n-1} + \delta}{i_2^2 \cdots i_n^2} \right|^t \\
&= \lim_{n \to \infty} \frac{1}{n} \log \sum_{i=N_{\epsilon}}^{\infty} \left( \frac{2\pi^{-1} K^{\frac{3}{2}}}{i^{\frac{5}{2}}} \right)^t + \lim_{n \to \infty} \frac{1}{n} \log \left( \sum_{j=N_{\epsilon}}^{\infty} \left( \frac{(2\pi)^{-2} aR^2 + \delta}{j^2} \right)^t \right)^{n-1} \\
&=\log \sum_{j=N_{\epsilon}}^{\infty} \left( \frac{(2\pi)^{-2} aR^2 + \delta}{j^2} \right)^t
\end{align*}
Let $ t=t^{\ast} $ be the unique zero of this upper bound.
Since $ p(t) $ and $ p^+(t) $ are strictly decreasing, we have that $ \text{dim}_H(\tau_{\epsilon}) < t^{\ast} $ by Bowen's theorem.

A lower bound for $ p^-(t) $ is more delicate.
For a given $ \epsilon $, choose $ M \in \mathbb{N} $ with $ M > N_{\epsilon} $.
Then we have the following lower bound.
$$
p^-(t) > \lim_{n \to \infty} \frac{1}{n} \log \sum_{i_1 = N_{\epsilon}}^M \sum_{i_2=N_{\epsilon}}^{\lfloor C \rfloor + \lfloor K \rfloor i_1^2} \cdots \sum_{i_n=N_{\epsilon}}^{\lfloor C \rfloor + \lfloor K \rfloor i_{n-1}^2} \left| \frac{2\pi^{-1} K^{\frac{3}{2}}}{i_1^{\frac{5}{2}}} \cdot \frac{\left((2\pi)^{-2} aR^2\right)^{n-1} - \delta}{i_2^2 \cdots i_n^2} \right|^t
$$
Let $ t=t_{\ast} $ be the unique zero of the right and side.
Again since $ p(t) $ and $ p^-(t) $ are strictly decreasing, we have that $ t_{\ast} < \text{dim}_H(\tau_{\epsilon}) $.

Recall that the constants $ C,K $ are defined in terms of $ a $ in Equation \ref{defCK}.
The constant $ N_{\epsilon} $ is defined in the proof of Proposition \ref{gamma1param}, and from the proof of Proposition \ref{trans2} we can show that $ N_{\epsilon} \sim \lceil \frac{K}{\epsilon} \rceil $.
Let $ \delta > 0 $ be small, and choose $ \epsilon > 0 $ small enough that Equation \ref{another2} holds.
Substituting the values of the constants $ C,K $ and $ N_{\epsilon} $ into this equation, we can use a computer algebra system to numerically estimate $ t^{\ast} $ and $ t_{\ast} $.

For example, choose the following numerical values:
$$
\delta = \epsilon = 0.01, \quad a=10, \quad R=0.5.
$$
Substituting these into the values of $ C,K, N_{\epsilon} $ in Equation \ref{another2}, and numerically estimating $ t_{\ast} $ and $ t^{\ast} $ in Mathematica, we obtain
$$
0.40105 < \text{dim}_H(\tau) < 0.51826.
$$
The lower bound can be improved by choosing larger values of $ M $ and $ n $ in the lower approximation of $ p^-(t) $ above.

\subsection{The Hausdorff dimension of $ \mathcal{M} $}
From the dimension results for $ \tau $ we obtain results for $ \mathcal{M} $. 
First, we have a corollary of Theorem \textbf{E}.

\begin{corollary*}
Let $ \mathcal{M} $ be the Kuperberg minimal set.
Then the three-dimensional Lebesgue measure of $ \mathcal{M} $ is zero, and $ 2 < \text{dim}_H(\mathcal{M}) < 3 $.
\end{corollary*}

\begin{proof}
By Theorem \ref{minchar}, $ \mathcal{M} $ has a local product structure of $ \mathbb{R}^2 \times \tau $.
As a consequence of Theorem \textbf{E}, the product Lebesgue measure is zero.
A standard result in dimension theory (see \cite{Matt} or \cite{Fal1}) states that if $ X $ and $ Y $ are subsets of Euclidean space, and the Hausdorff dimension of $ Y $ is equal to its upper box dimension, then
$$
\text{dim}_H(X \times Y) = \text{dim}_H(X) + \text{dim}_H(Y).
$$
Applying this to the product structure we obtain
$$
\text{dim}_H(\mathcal{M}) = 2 + \text{dim}_H(\tau),
$$
and the result follows from Theorem \textbf{E}.
\end{proof}

Using the product structure in the above proof, we have the following corollary of Theorem \textbf{F}.

\begin{corollary*}
Let $ \mathcal{M} $ be the Kuperberg minimal set.
Let $ t = \text{dim}_H(\mathcal{M}) $ be its Hausdorff dimension, and $ a > 0 $ the angular speed of the Kuperberg flow.
\begin{itemize}
\item $ t = \text{dim}_H(\mathcal{M}) $ is the unique zero of a dynamically defined pressure function, 
\item $ t $ depends continuously on $ a $, 
\item For any $ a $ we may compute $ t $ to a desired level of accuracy.
\end{itemize}
\end{corollary*}

Because of this corollary, for the choice of $ \delta, \epsilon, a $ and $ R $ above we have
$$
2.40105 < \text{dim}_H(\mathcal{M}) < 2.51826.
$$

\vfill
\eject

\section{Further questions}
\label{Furth}
There are many remaining open questions about Kuperberg flows. 
Some of these are surveyed in \cite{Hur3}.
In this section, we will state some open questions that pertain to the dimension theory of minimal sets of Kuperberg flows.

\subsection{Efficient algorithms for dimension estimates}
The method that yields the numerical results from Theorem \textbf{F} is not particularly efficient.
The the zeros of the upper and lower bounds on $ p^{\pm}(t) $ are computationally expensive to estimate.
For this reason, we cannot fully explore the possible range of the Hausdorff dimension over the parameter space.

\begin{question*}
Design a more efficient algorithm for computing the Hausdorff dimension of the transverse Cantor set of the Kuperberg minimal set.
\end{question*}

In the course of the proof of Theorem \textbf{E}, we showed that the ratio geometry of the transverse Cantor set $ \tau_{\epsilon} $ for $ \epsilon > 0 $ is asymptotically stationary.
The dimension theory for limit sets of iterated function systems whose symbolic dynamics are semiconjugate to a subshift of finite type is classical.
For stationary systems, Bowen's equation for dimension reduces to an equation involving the spectral radius of the incidence matrix (see Chapter 7 of \cite{Pes2}).
The proof of this result relies on a theorem of Ruelle relating the pressure to the spectral radius of the Perron-Frobenius operator.
Solving the spectral radius equation is more computationally efficient than calculating the zeros of the pressure, so an answer to this question might be along these lines.

\subsection{Hausdorff measure of the minimal set}
A more delicate problem than determining Hausdorff dimension is proving that the Hausdorff measure at dimension is finite.
In general, for the limit set of a finitely generated iterated function system or a geometric construction to have finite Hausdorff measure at dimension, the dynamics on the sequence space must be topologically mixing.
For subshifts of finite type this is equivalent to transitivity of the incidence matrix (see \cite{Pes2} or \cite{Bar2}).

\begin{question*}
Let $ \mathcal{M} $ be the Kuperberg minimal set, let $ t=\text{dim}_H(\mathcal{M}) $ be its Hausdorff dimension, and let $ H^t $ be the $ t $-dimensional Hausdorff measure.
Show that $ 0 < H^t(\mathcal{M}) < \infty $.
\end{question*}

In \cite{Pes1}, Pesin and Weiss showed that the limit set of a geometric construction has finite Hausdorff measure at dimension, provided that the eigenmeasure of the Perron-Frobenius operator is a Gibbs state.
A Perron-Frobenius operator in the context of pseudo-Markov systems is studied in \cite{Str}.
By transferring the definition there to the notation developed in Sections \ref{Thermo} and \ref{Conf}, it seems possible to prove an analogue of this result for $ \tau $, and then extend to $ \mathcal{M} $ by the product structure.

\subsection{Ergodic properties of invariant measures}
The ergodic theory of measures invariant under the Kuperberg flow appears to be very difficult. However, Theorems \textbf{A} and \textbf{C} appear to offer a foothold onto this problem.
For small $ \epsilon > 0 $ the transverse minimal set $ \tau_{\epsilon} $ is a limit set of a function system modeled on a sequence space $ \Sigma $ that is invariant under the Kuperberg pseudogroup.
Let $ \mu $ be a measure on $ \Sigma $ invariant under the pseudogroup.
Then the pushforward $ \pi_{\ast} \mu $ through the coding map is a measure on the transverse minimal set, invariant under the Kuperberg flow.
From the product structure given in Theorem \ref{minchar}, a global measure on $ \mathcal{M} $ can be disintegrated along the leaves to obtain the product of a measure on the leaves with a measure on the transversal.
As long as the conditional measures on the leaves are absolutely continuous, one can reduce to studying the ergodic properties of the transverse measures on $ \tau $ and therefore to those on $ \Sigma $, which seems more tractable.

A measure invariant under the Kuperberg flow must have zero entropy as a consequence of a theorem of Katok (\cite{Kat}); this was pointed out by Ghys (\cite{Ghy}).
The Kuperberg plug contains an open set of wandering points and as such cannot preserve a measure supported on open sets; this was pointed out by Matsumoto (\cite{Mat}).
Any other question related to the ergodic properties of invariant measures of the Kuperberg flow appears to be wide open.

\subsection{Dimension of minimal sets of perturbations of Kuperberg flows}
Kuperberg flows are not structurally stable.
In \cite{Hur2}, Hurder and Rechtman defined a class of plugs $ K_{\epsilon} $ supporting a $ C^{\infty} $ flow, for which $ K_0 $ is the Kuperberg plug with no periodic orbits, but $ K_{\epsilon} $ for $ \epsilon > 0 $ has infinitely many periodic orbits.
They showed further that the minimal set of $ K_{\epsilon} $ has embedded horseshoes.
It would be simple to construct such a class $ K_{\epsilon} $ compatible with the assumptions we have made in Section \ref{InsertAssume}.
For $ K_0 $ we would recover the symbolic dynamics and dimension results from this paper, and for $ K_{\epsilon} $ with $ \epsilon > 0 $ we would obtain more standard results (positive entropy, uniform hyperbolicity, etc.).
The dimension theory of horseshoes is well studied (see \cite{Man}, \cite{Sim2}, \cite{Sim3}).
It would be interesting to see how the dimension and symbolic dynamics change as $ \epsilon \rightarrow 0 $.

\subsection{Dimension of minimal sets of generic Kuperberg flows}
The minimal set we have studied is that of a very particular Kuperberg flow.
To simplify our calculations, we have made numerous assumptions on the flow, insertion maps, and insertion regions.
These are listed in Section \ref{InsertAssume}.
However, Theorem \ref{minchar} is true under much weaker assumptions, the axioms of a \textit{generic} Kuperberg flow defined by Hurder and Rechtman.
These are listed in Chapter 12 of \cite{Hur}.
It would be interesting to see what results from Theorems \textbf{A} -- \textbf{F} survive in this generality.

\vfill
\eject

\end{document}